\documentclass[acmlarge]{acmart}

\usepackage{booktabs} % For formal tables

\usepackage[ruled]{algorithm2e} % For algorithms

\SetAlFnt{\small}
\SetAlCapFnt{\small}
\SetAlCapNameFnt{\small}
\SetAlCapHSkip{0pt}
\IncMargin{-\parindent}

\acmJournal{TOCL}
\acmVolume{0}
\acmNumber{0}
\acmArticle{}
\acmYear{2018}
\acmMonth{4}
\acmArticleSeq{0}

\setcopyright{usgovmixed}
\acmDOI{0}

\received{April 2018}
\usepackage{tikz}			% tikz package for pictures
\usetikzlibrary{matrix}		% tikz library for diagrams
\usepackage{pgf}
\usetikzlibrary{arrows,automata}
\usepackage{url}	
\usepackage{multirow}
\usepackage{amsmath}
\usepackage{amssymb}
\usepackage{amsfonts}
\usepackage{stmaryrd}
\usepackage{mathrsfs}
\numberwithin{equation}{section}
\usepackage{tkz-graph}
\usetikzlibrary{patterns}
\usetikzlibrary{shapes}
\usepackage{color}
\usepackage{hyperref}
\usepackage{array,cmll}
\usepackage{enumerate}
\usepackage[all]{xy}
\usepackage{cleveref}
\usepackage{graphics}
\usepackage{lscape}
\usepackage{pdflscape}
\usepackage{graphicx}
\usepackage{multicol}
\newcommand{\commentPROOF}[1]{#1}
\newcommand{\commentPROOFbis}[1]{}

\newenvironment{claimproof}[1]{\par\noindent\underline{Proof of claim.}\space#1}{\hfill $\blacksquare$}
\newtheorem{notation}{Notation}
\newtheorem{remark}{Remark}
\newtheorem{claim}{Claim}

\newcommand{\redbf}[1]{\textcolor{red}{\textbf{#1}}}

\newcommand{\mc}{\mathcal}

\newcommand{\mb}{\mathbb}

\renewcommand{\P}{\mc{P}}
\newcommand{\h}{\mathtt{h}}

\newcommand{\val}[1]{[\![  #1 ]\!]}

\newcommand{\Ag}{\mathsf{Ag}}
\newcommand{\pre}{\mathsf{pre}}

\newcommand{\SetSub}{Sub} % set of substitution maps

\newcommand{\bbA}{\mathbb{A}}

\newcommand{\ls}{[}
\newcommand{\rs}{]}
\newcommand{\sub}{\mathbf{sub}}

\newcommand{\AtProp}{\mathsf{AtProp}}
\newcommand{\support}{\mathsf{support}}

\newcommand{\rmPhi}{\mathrm{\Phi}}

\newcommand{\bfPhi}{\mathbf{\Phi}}

\newcommand{\Profit}{\mathtt{B}_f}
\newcommand{\Gallery}{\mathtt{M}}
\newcommand{\Prob}{\mathtt{Pr}}

\newcommand{\eiKf}{epistemic intuitionistic Kripke frame}
\newcommand{\EiKf}{Epistemic intuitionistic Kripke frame}

\newcommand{\eiKm}{epistemic intuitionistic Kripke model}
\newcommand{\EiKm}{Epistemic intuitionistic Kripke model}

\renewcommand{\blacklozenge}{\lozenge}

\begin{document}
% Title portion
\title{Probabilistic Epistemic Updates on Algebras}
%\titlenote{We can add a note to the title}

%\author{
%Alessandra Palmigiano\inst{3,4}
%\and
%\inst{3}
%\and
%Nachoem Wijnberg\inst{5, 6}
%}
%\and Faculty of Technology, Policy and Management, Delft University of Technology\\
%\email{\{a.palmigiano, a.tzimoulis-1\}@tudelft.nl}
%\and 	Department of Pure and Applied Mathematics, University of Johannesburg\\
%\and
%Amsterdam Business School, University of Amsterdam\\
%\email{N.M.Wijnberg@uva.nl}
%\and College of Business and Economics, University of Johannesburg
%}

\author{Willem Conradie}
\orcid{}
\affiliation{%
  \institution{School of Mathematics, University of the Witwatersrand}
  \streetaddress{}
  \city{Johannesburg}
  \state{}
  \postcode{}
  \country{South Africa}}
\email{willem.conradie@wits.ac.za}
\author{Sabine Frittella}
\affiliation{%
  \institution{INSA Centre Val de Loire,
Univ. Orl\' eans, LIFO EA 4022}
  \city{Bourges}
  \country{France}
}
\email{sabine.frittella@insa-cvl.fr}
\author{Alessandra Palmigiano}
\affiliation{%
 \institution{Faculty of Technology, Policy and Management, Delft University of Technology}
 \streetaddress{}
 \city{Delft}
 \state{}
 \country{the Netherlands}}
\affiliation{%
 \institution{Department of Pure and Applied Mathematics, University of Johannesburg}
 \streetaddress{}
 \city{Johannesburg}
 \state{}
 \country{South Africa}}
\email{a.palmigiano@tudelft.nl}
\author{Apostolos Tzimoulis}
\authornote{This is the corresponding author}
\affiliation{%
 \institution{Faculty of Technology, Policy and Management, Delft University of Technology}
 \streetaddress{}
 \city{Delft}
 \state{}
 \country{the Netherlands}}
\email{a.tzimoulis-1@tudelft.nl}
\author{Nachoem Wijnberg}
\affiliation{%
  \institution{Amsterdam Business School, University of Amsterdam
%\email{N.M.Wijnberg@uva.nl}
%\and College of Business and Economics, University of Johannesburg}
  \city{Amsterdam}
  \country{the Netherlands}}
  \institution{College of Business and Economics, University of Johannesburg}
  \city{Johannesburg}
 \country{South Africa}}
\email{N.M.Wijnberg@uva.nl}

\begin{abstract}
The  present paper contributes to the development of the mathematical theory of epistemic updates using the tools of duality theory. Here we focus on Probabilistic Dynamic Epistemic Logic (PDEL).
We dually characterize the product update construction of PDEL-models as a certain construction transforming the complex algebras associated with the given model into the complex algebra associated with the updated model.
Thanks to this construction, an interpretation of the language of PDEL can be defined on algebraic models based on Heyting algebras. This justifies our proposal for the axiomatization of the intuitionistic counterpart of PDEL.
\end{abstract}




%\keywords{Wireless sensor networks, media access control,
%multi-channel, radio interference, time synchronization}

% DO NOT use this command unless you want to change
% the default behavior
% \authorsaddresses{Authors' addresses: G.~Zhou, Computer Science
%   Department, College of William and Mary, 104 Jameson Rd,
%   Williamsburg, PA 23185, US, \path{gzhou@wm.edu}; V.~B\'eranger,
%   Inria Paris-Rocquencourt, Rocquencourt, France; A.~Patel, Rajiv
%   Gandhi University, Rono-Hills, Doimukh, Arunachal Pradesh, India;
%   H.~Chan, Tsinghua University, 30 Shuangqing Rd, Haidian Qu, Beijing
%   Shi, China; T.~Yan, Eaton Innovation Center, Prague, Czech Republic;
%   T.~He, C.~Huang, J.~A.~Stankovic University of Virginia, School of
%   Engineering Charlottesville, VA 22903, USA; T. F. Abdelzaher,
%   (Current address) NASA Ames Research Center, Moffett Field,
%   California 94035.}

\maketitle

\begin{acks}
%\textbf{Acknowledgement:} 
The research of the second to fourth author has been funded by the \grantsponsor{vidi}{NWO Vidi grant}{} \grantnum{vidi}{016.138.314}, by the \grantsponsor{Aspasia}{NWO Aspasia grant}{} \grantnum{Aspasia}{015.008.054}, and by a Delft Technology Fellowship awarded in 2013.
\end{acks}
\renewcommand{\shortauthors}{W.\ Conradie,
S.\ Frittella,
A.\ Palmigiano,
A.\ Tzimoulis,
N.\ Wijnberg}

\newpage

\tableofcontents

%\newpage

\section{Introduction}
\label{sec:intro}
This paper pertains to a line of research aimed at exploring the notions of agency and information flow in situations in which truth is socially constructed. Such situations are ubiquitous in the real world. A prime example is the validity of contracts. 
Establishing that an agreement constitutes a valid contract  appeals to notions, such as legal competency and bona fide offers, which are inherently socially constructed. The ultimate way in which the validity of a contract can be ascertained is for it to be tested in a court of law. In this last instance, the validity of a contract is thus procedural, and may also admit of situations in which it is indeterminate, such as when the court declares itself incompetent. These are features at odds with standard classical logic. Accommodating these features within classical logic requires additional encoding mechanisms. The alternative is working with logics which are specifically designed to accommodate these characteristics of socially constructed truth.

Examples of situations where truth is socially constructed are certainly not confined to contract law, but are easy to find in many other contexts. These include establishing public opinion in a binding way like referendums, establishing whether a certain item of clothing is fashionable, and determining the value of products in a market.

There is a large literature on logics which very adequately capture  agency and information flow (see \cite{vanBLogicalDyn2011} and references therein), but assume a notion of truth that is classical. There is therefore a need for a uniform methodology for transferring these logics onto nonclassical bases. In \cite{MPS14,KP13},  a uniform methodology is introduced for
defining the nonclassical counterparts of dynamic epistemic logics. This methodology, further pursued in \cite{Ri14,BR15,BDF16}, is grounded on semantics, and is based on the dual characterizations of the transformations of models which interpret epistemic actions.

The present paper expands on \cite{CFPTLori} and applies the methodology of \cite{MPS14,KP13} to obtain nonclassical counterparts of probabilistic dynamic epistemic logic (PDEL) \cite{kooi2003probabilistic,vBGK09}. We will focus specifically on the intuitionistic environment as our case study. This environment allows for a finer-grained analysis when serving as a base for more expressive formalisms such as modal and dynamic logics. Indeed, the fact that the box-type and the diamond-type modalities are no longer interdefinable  makes several mutually independent choices possible which cannot be disentangled  in the classical setting. Moving to the intuitionistic environment also requires the use of intuitionistic probability theory (cf.\ \cite{aguzzoli2008finetti,flaminio2017states}) as the background framework for probabilistic reasoning. From the point of view of applications this generalization is needed to account for  situations in which the probability of a certain proposition $p$ is interpreted as an agent's propensity to bet on $p$ given some evidence for or against $p$. If there is little or no evidence for or against $p$, it should be reasonable to attribute  low probability values to both  $p$ and $\neg p$, which is forbidden by classical probability theory (cf.\  \cite{weatherson2003classical}).

Finally, these mathematical developments appear in tandem with interesting analyses on the philosophical side of formal logic (e.g.\ \cite{artemov2014intuitionistic}), exploring epistemic logic in an evidentialist key, which is congenial to the kind of social situations targeted by our research programme.

Our methodology is based on the dual characterization of the product update construction for standard PDEL-models as a certain construction transforming the complex algebras associated with a given model into the complex algebra associated with the updated model. This dual characterization
naturally generalizes to much wider classes of algebras, which include arbitrary classical S5 algebras and certain monadic Heyting algebras.  As an application of this
dual characterization, we introduce the axiomatization of the intuitionistic analogue of PDEL semantically arising from this construction, and prove its soundness and completeness with respect to the class of so called \emph{algebraic probabilistic epistemic models} (see  Definition \ref{def: alg probab epist model}).

\paragraph{Structure of the paper.}  
In Section \ref{sec:prelim}, we recall the definition of classical PDEL and its relational semantics. We give an alternative presentation of the product update construction which consists in two steps, as done in \cite{KP13}. The two-step construction highlights the elements which will be key in the dualization.  In Section \ref{sec:method}, we expand on the methodology making use of Stone duality.  Section \ref{sec:ha} is the main section, in which the construction of the PDEL-updates on epistemic Heyting algebras is introduced.  
In Section \ref{sec:semantics-IPDEL}, we define axiomatically the intuitionistic version of PDEL (IPDEL) and its interpretation on algebraic probabilistic epistemic models, and discuss the proof of its soundness. 
%In Section \ref{sec:completeness}, we prove the completeness of IPDEL with respect to algebraic probabilistic epistemic models. 
In Section \ref{sec:relsem}, we introduce the relational semantics of IPDEL.
In Section \ref{sec:ArtExample}, we discuss the case study of a decision-making under uncertainty. 
In Section \ref{sec:CCL}, we collect conclusions and further directions. 
Appendix \ref{app:section4} collects some proofs of Section \ref{sec:ha}.
Appendix \ref{Appendix:soundness} contains the proof of soundness of IPDEL with respect to algebraic probabilistic epistemic models.   %\redfootnote{to check once the paper is fully finished.}
Appendix \ref{Appendix:completeness} contains the  proof of completeness of IPDEL with respect to algebraic probabilistic epistemic models.

\section{PDEL language and updates} 
\label{sec:prelim}

%%%%%%%%%%%%%%% Probabilistic epistemic model and update

%\section{Probabilistic epistemic state models  and updates}

%\redbf{this should evolve into one or more sections of preliminaries.
%\\
%~\\
%Here we need to tell the story of the existing literature (van benthem, Baltag, the MoL thesis...) in a way which is both faithful and supports the developments in the next sections.
%\\
%~\\
%A good part of this treatment is there, what needs to be done is:  \\
%(1) if our structures of choice is a mix of different solutions from the various papers, we need to specify which and why, see Sabine's comment in red. \\
%(2) the boolean based logic and its interpretation is missing; \\
%(3) we probably want to start speaking of pre-probabilistic structures (that is, those for which the $P_i$ do not sum up to 1 on eq.\ cells) already in this section and develop the complex algebra treatment of these. Apostolos or Sabine, would you give it a try?}

In the present section, we report on the language of PDEL, and give an alternative, two-step account of the product update construction on PDEL-models. This account is similar to the treatment of epistemic updates in \cite{MPS14,KP13}, and as explained in Section \ref{sec:method}, it lays the ground  to the dualization procedure which motivates the construction introduced in Section \ref{sec:ha}. The specific PDEL framework we report on shares common features with  those of  \cite{BCHS13,Achimescu14} and \cite{vBGK09}.

\paragraph{Structure of the section.}
In Section  \ref{ssec:PDEL:event:PESmodels}, we recall basic facts about probability theory, we present the syntax au PDEL, %(see \autoref{def:proba event model Alexandru}), 
the classical models % (see PES-models, \autoref{def:Prob epis state model Alexandru}) 
and the classical event structures. % (see \autoref{def:proba-epist-event-struct}). 
In Section  \ref{ssec:epist:update:classic}, we present the
alternative construction for epistemic update of a PES-model by a probabilistic event structure.
In Section  \ref{ssec:classic:semantics} and \ref{ssec:classic:axiomatization} respectively, we present the semantics and the axiomatisation of classical PDEL.

\subsection{PDEL-formulas, event structures, and PES-models}
\label{ssec:PDEL:event:PESmodels}

In this section, we first recall basic facts about probability distributions and probability measures, then we introduce the syntax and semantics of Probabilistic Dynamic Epistemic Logic (PDEL).
\begin{remark}	%[\redbf{Reminder on probability theory}]
%\redfootnote{is there a better way to do this than a remark?}
Given a finite set $X$,
 a \emph{probability distribution} $P$ over  $X$ 
is a map 
$$P \ : \ X \rightarrow [0,1]$$ 
such that
$$\sum_{x\in X} P(x) = 1.$$
Recall that a \emph{probability measure} on $\mathcal{P} X$ can be defined as a map
$$\mu  \ : \ \mathcal{P} X \rightarrow [0,1]$$ satisfying the following properties: 
\begin{enumerate}
\item $\mu (\varnothing) = 0 $,
\item $\mu (X) = 1$,
\item for any $A,B \subseteq X$, we have
%\begin{equation*}
$\mu (A \cup B) = \mu (A) + \mu(B) - \mu(A \cap B).$
%\end{equation*}
\end{enumerate}
The probability measure 
$\mu \ : \ \mathcal{P} X \rightarrow [0,1]$ 
determined by the probability distribution $P$ over $X$ is defined as follows: for any $S \subseteq X$,
$$ \mu(S) := \sum_{x\in S} P(x). $$
\end{remark}

In the remainder of the paper, we fix a countable set  $\mathsf{AtProp}$  of proposition letters $p, q$ and a non-empty finite set $\Ag$ of agents $i$. We let $\alpha_1, ..., \alpha_n, \beta $ denote rational numbers.
%\redfootnote{to find out and explain why we work with rational numbers and not with real numbers.}

\begin{definition}[PDEL syntax] \label{def:proba event model Alexandru} The set $\mathcal{L}$ of  {\em PDEL-formulas} $\varphi$ and the class 
of {\em probabilistic event structures} $\mathcal{E}$ over $\mathcal{L}$
(see \autoref{def:proba-epist-event-struct}) are built by simultaneous recursion  as follows:
%\marginred{we don't give the simultaneous recursion is in fact Definition \ref{def:proba event model Alexandru} and Definition \ref{def:proba-epist-event-struct}}
	\[\varphi ::= p %\in \Prop
	\mid \bot \mid \varphi \wedge \varphi \mid \varphi\vee \varphi \mid \varphi\rightarrow  \varphi \mid \lozenge_i\varphi \mid \Box_i \varphi \mid \langle\mathcal{E}, e\rangle \varphi \mid [\mathcal{E}, e] \varphi \mid (\sum_{k = 1}^n\alpha_k · \mu_i(\varphi)) \geq \beta,\]
	where $p \in \mathsf{AtProp}$, $i\in \Ag$, $\alpha_1,..., \alpha_n, \beta $ are rational numbers, and 
	the event structures $\mathcal{E}$ are such as in  \autoref{def:proba-epist-event-struct}.\\
	The connectives $\top$, $\neg$, and $\leftrightarrow$ are defined by the usual abbreviations. 
\end{definition}
	   %be defined as $\bot\to \bot$ and, for all formulas $\phi$ and $\psi$, let $\neg \phi$ be defined as $\phi\to \bot$ and $\phi\leftrightarrow \psi$ be defined as $(\phi\to \psi)\wedge(\psi\to \phi)$.
	   
\begin{definition}[PES-model]%{\cite[Definition 1]{BCHS13}}
\label{def:Prob epis state model Alexandru}
%Given a set of agents $\Ag$ and a set of propositional variables $\Prop$.
A \emph{probabilistic epistemic state model (PES-model)}  is a structure 
$$\mb{M} = \left\langle S, (\sim_i)_{i\in \Ag}, (P_i)_{i\in \Ag}, \val{\cdot}\right\rangle$$
 such that
\begin{itemize}
\item 
$S$ is a finite set,
\item each binary relation $\sim_i$ is an equivalence relation on $S$, %interpreted as agent $i$'s epistemic indistinguishability. This captures the agent's hard information about the actual state of the world;
\item each map  $P_i: S\rightarrow\ ]0, 1]$ assigns a probability distribution
over each $\sim_i$-equivalence class, $($i.e.   $\sum \{P_i (s' ) : s' \sim_i s \} = 1)$, and %for each $i \in \Ag$ and each $s\in S$). This captures the agent's subjective probabilistic information about the state of the world;
%\item  $\Psi$ is a given set of ``atomic propositions'', denoted by $p$, $q$, ... Such atoms $p$ are meant to represent \emph{ontic ``facts''} that might hold in a world.
\item the map $\val{\cdot} : \mathsf{AtProp} \rightarrow \P S$ is a valuation. %assigning to each atomic proposition $p\in \Prop$ some set of states $\val{p} \subseteq S$. Intuitively, the valuation tells us which facts hold in which worlds.
\end{itemize}
\end{definition}

As usual,  the map $\val{\cdot}$ will be identified with its unique extension to $\mathcal{L}$, so that we will be able to write $\val{\varphi}$ for every $\varphi\in \mathcal{L}$.	   

\begin{remark}
\label{rk:proba:def}
%\redbf{Notice that by the definition of $P_i$, for every  $s\in S$ $P_i(s)\neq 0$.
%\\
%Explain why we want only strictly positive probabilities
%\\
%Without this hypothesis we cannot do the math.
%\\
%cases where we would like $P(s)=0$ are (a) unawareness and (b) this agent is mistake in thinking that something is impossible. 
%arguments: (a) S5 cannot handle unawareness and (b) if $P(s)=0$, how do you decide in which an equivalence class you put $s$.}
%
%\redbf{TO EDIT} 
The assumption that the probability of each state is strictly positive is needed for 
 the update defined in \autoref{def: update PES model} to be well-defined.
This is also the convention followed in \cite{ES14} where  subjective probabilities are identified with ``lotteries'' assigned to each agent.
%We notice that the  explanation given there holds also in the intuitionistic setting, since the only difference is that states do not decide on the truth of every atomic proposition and this issue is conceptually independent.
%A more important issue is the relationship between probability distributions and the order between states. We discuss this specific issue in Remark \ref{}.

\end{remark}

\begin{remark} %\redbf{TO EDIT}
In the present treatment, the syntactic $\mu_i$s 
(introduced in \autoref{def:proba event model Alexandru})
are  intended to correspond to probability measures rather than probability distributions, as is more common in the literature. 
%\marginredbf{should we add a footnote explaining \\
%- the difference between a probability distribution and a probability measure\\
%and\\
%- the reason why we use a probability measure?}
%% Rewriting attempt
 Indeed, usually, in the literature formulas talking about probabilities are defined by the following syntax $\alpha P_i(\varphi)\geq\beta$. But the $P_i$ maps are probability distributions defined over the models (i.e.\ in the semantics), hence the notation $P_i(\varphi)$ is ambiguous and neglects the fact that we need to use a probability measure to talk about the probability over the extension of $\varphi$.
%% Using the $\mu_i$s has the purpose of ....
%This has the purpose of creating a more natural match between syntax and semantics under the usual interpretation. Indeed, the usual semantic interpretation of  the formula $\alpha P_i(\varphi)\geq\beta$ yields
%\begin{center}
%	$\mb{M},s\Vdash\alpha P_i(\varphi)\geq\beta$ iff $\alpha\sum_{\begin{smallmatrix}
%		\mb{M},t\Vdash\varphi \\
%		t\sim_is
%		\end{smallmatrix}}P_i(t)\geq\beta$ iff $\alpha P_i^+(\val{\varphi}\cap[s]_{\sim_i})\geq\beta$,
%\end{center}
%from which it is clear that the semantic counterpart of the symbol $P_i$ is $P_i^+$.
%\marginredbf{The formula $\alpha P_i(\varphi)\geq\beta$ is not defined in our syntax.}
%\marginredbf{What is $P^+$? refer to the definition in section 3, Definition \ref{def: complex algebra}.}
\end{remark}

\begin{definition}[Substitution function] 
\label{def:substitution:function}
A \emph{substitution function}
$$\sigma \ : \ \mathsf{AtProp} \rightarrow \mathcal{L} $$ 
is a function that maps all but a finite
\footnote{This assumption guarantees that events affect only a finite number of facts.}
number of proposition letters to themselves.

We will call 
the set 
$$\{p\in \mathsf{AtProp} \mid \sigma(p)\neq p \}$$ 
the domain of $\sigma$ and denote it $dom(\sigma)$.

Let $\SetSub_\mathcal{L}$ denote the set of all  substitution functions and $\epsilon$ the identity substitution.
\end{definition}
	   
\begin{definition}[Probabilistic event structure over a language]	
\label{def:proba-epist-event-struct}   
A \emph{probabilistic event structure  over} $\mathcal{L}$ is a tuple %\marginnote{Ale: I erased the substitution}
$$\mathcal{E} = (E, (\sim_i)_{i\in\Ag}, (P_i)_{i\in \Ag}, \rmPhi, \pre, \sub),$$
such that
\begin{itemize}
\item $E$ is a non-empty finite set,
\item each $\sim_i$ is an equivalence relation on $E$, 
%\end{itemize}
  
%\begin{enumerate}
%\item $E$ is a finite set of possible events;
%\item
 %interpreted as agent $i$'s epistemic indistinguishability between possible events, capturing $i$'s hard information about the event that is currently happening;
\item
each $P_i:E\to\ ]0,1]$ assigns a probability distribution over each $\sim_i$-equivalence class, i.e.\   $$\sum \left\{P_i (e' ) \mid e' \sim_i e \right\} = 1,$$
%\redfootnote{Explain the reasoning behind non-zero}  %This captures some new, independent subjective probabilistic information gained by the agent during the event: when observing the current event (without using any prior information), agent $i$ assigns probability $P_i(e)$ to the possibility that in fact $e$ is the actual event that is currently occurring.
\item
$\rmPhi$ is a finite set of pairwise inconsistent $\mathcal{L}$-formulas, and %(in the above probabilistic epistemic language $L$), called preconditions;
\item
$\pre$ assigns a probability distribution $\pre(\bullet | \phi)$ over $E$ for every  $\phi \in \rmPhi$. %This is an \emph{occurrence probability}: $\pre(e|\phi)$ expresses the prior probability that event $e\in E$ might occur in a(ny) state satisfying precondition $\phi$;
\item $\sub : E \rightarrow \SetSub_\mathcal{L}$ assigns a substitution function to each event in $E$.
%\end{enumerate}
\end{itemize}
\end{definition}

\begin{remark}
The assumption that the probability of each event is strictly positive is needed for the update defined in \autoref{def: update PES model} to be well-defined.
This is also the convention followed in \cite{ES14,ABS16}.
\end{remark}

Informally, elements of  $E$ encode possible events, the relations $\sim_i$ encode as usual the epistemic uncertainty of the agent $i$, who  assigns probability $P_i(e)$ to  $e$ being the actually occurring event, formulas in $\rmPhi$ are intended as  the preconditions of the event, and $\pre(e|\phi)$ expresses the prior probability that the event $e\in E$ might occur in a(ny) state satisfying precondition $\phi$.
In addition, the substitution map $\sub(e)$ assigned to each event $e \in E$ describes how the event $e$ changes the atomic facts of the world as represented by the proposition letters.
In what follows, we will refer to the structures $\mathcal{E}$ defined above as {\em event structures over $\mathcal{L}$}.

%In the present section, we report on the semantics of PDEL.

%\begin{definition}%{\cite{BEK06}}
%A \emph{substitution function} $\sigma : \Prop \rightarrow \Lpdel$
%is a function that maps all but a finite number of propositional atoms into themselves. Call the set $\{ p\in \Prop \mid \sigma(p)\neq p \}$ the domain of $\sigma$ and denote it by $dom(\sigma)$. Let $\SetSub(\Lpdel)$ denote the set of all such possible substitution functions and $\epsilon$ the identity substitution.
%\end{definition}

\begin{notation}
\label{note:pre(e/s)}
For any 
probabilistic epistemic state model
$\mb{M} = \left\langle S, (\sim_i)_{i\in \Ag}, (P_i)_{i\in \Ag}, \val{\cdot}\right\rangle$, any probabilistic event structure $\mathcal{E}$, any $s\in S$ and $e\in E$, we let $\pre(e\mid s)$ denote the value $\pre(e\mid \phi)$, for the unique $\phi\in\rmPhi$ such that $\mb{M},s\Vdash\phi$ (recall that the formulas in $\rmPhi$ are pairwise inconsistent). If no such $\phi$ exists then we let $\pre(e\mid s)=0$.
\end{notation}

\subsection{Epistemic updates}
\label{ssec:epist:update:classic}
In this subsection, we introduce an alternative and equivalent presentation of the  update construction on PES-models.
This presentation is a variant of those introduced in \cite{MPS14,KP13} for models of public announcement logic and dynamic epistemic logic, and consists in  a two-step process, namely, 
a co-product-type construction followed by a suboject-type construction. 
This two-step presentation makes it possible to dualize the two  steps separately, and thus obtain the construction of (probabilistic) epistemic updates on algebras as the composition of the two dualized constructions.
%
%\redbf{Throughout the present subsection, we fix a PES-model 
%$$\mb{M}= (S, (\sim_i)_{i\in \Ag}, (P_i)_{i\in \Ag}, \val{\cdot})$$ 
%and a probabilistic event structure 
%$$\mathcal{E}= (E, (\sim_i)_{i\in\Ag}, (P_i)_{i\in \Ag}, \rmPhi, \pre, \sub)$$  
%over $\mathcal{L}$.}\redfootnote{is this necessary?}
The  two steps are given in \Cref{def: coproduct PES model} and \Cref{def: update PES model}, and \Cref{lem:prelim:PESmodel} 
proves that the updated model of a PES-model is a PES-model too.

\begin{definition}[Intermediate structure]
\label{def: coproduct PES model}
For any PES-model $\mb{M}= \left\langle S, (\sim_i)_{i\in \Ag}, (P_i)_{i\in \Ag}, \val{\cdot}\right\rangle$ and any probabilistic event structure $\mathcal{E}= (E, (\sim_i)_{i\in\Ag}, (P_i)_{i\in \Ag}, \rmPhi, \pre, \sub)$ over $\mathcal{L}$,
let the {\em intermediate structure} of $\mb{M}$ and $\mathcal{E}$ be the tuple
\begin{center}
	$\coprod_{\mathcal{E}}\mb{M}: = \left\langle \coprod_{|E|} S, (\sim_i^{\coprod} )_{i\in \Ag}, (P_i^{\coprod})_{i\in \Ag}, \val{\cdot}_{\coprod}\right\rangle$
\end{center}

where
\begin{itemize}
\item
$\coprod_{|E|} S\cong S\times E$ is the $|E|$-fold coproduct of $S$, %clearly, $\coprod_{|E|} S$ can be identified with $$. %S' = \{ (s,e) \in S \times E \mid pre(e\mid s)\neq 0 \} $
\item
 each binary relation $\sim_i^{\coprod}$ on $\coprod_{|E|} S$ is defined  as follows:
 $$(s,e) \sim^{\coprod}_i (s',e')\ \quad \mbox{ iff }
\quad 
 s \sim_i s'\mbox{ and }e \sim_i e',$$
\item each  map $P^{\coprod}_i: \coprod_{|E|} S\to [0, 1]$ is defined by  
$$(s,e) \mapsto P_i(s)\cdot P_i(e)\cdot \pre(e\mid s),$$ %\frac{P_i(s) \cdot P_i(e) \cdot pre(e \mid s)}{\sum \{P_i(t) \cdot P_i(f) \cdot pre(f \mid t) \mid (s,e) \sim_i (t,f) \}} $$
\item and  the valuation $\val{\cdot}_{\coprod} : \mathsf{AtProp} \rightarrow \P S$ is defined by
$$ \val{p}_{\coprod} : = \left\{ (s,e)  \mid s\in \val{p}_{\mb{M}} \right\} = \val{p}_{\mb{M}}\times E $$ for every $p\in \mathsf{AtProp}$.
\end{itemize}
\end{definition}

\begin{remark}
In general, $P^{\coprod}_i$ does not induce probability distributions over the $\sim_i^{\coprod}$-equivalence classes. Hence, $\coprod_{\mathcal{E}}\mb{M}$ is not a PES-model.
%\footnote{Definition \ref{def:measures} will be introduced in Section \ref{sec:ha} precisely with the purpose of capturing the dual of $P^{\coprod}_i$.\redbf{Sabine : I am not sure that this note is relevant at that time in the paper.}} 
However, the second step of the construction will yield a PES-model.
\end{remark}

%\begin{fact}\marginnote{edit this}
%For every $i\in \Ag$, the map $P^{\coprod}_i: \coprod_{|E|} S\to [0, 1]$ defined as above is a probability distribution restricted to $\sim^{\coprod}_i $-cells.
%\end{fact}
%\commentPROOF{
%\begin{proof} 
%Indeed, for each $(s, e)\in \coprod_{|E|} S$,
%\redbf{This computation is not correct}
%\begin{align*}
%\sum_{(s', e')\sim^{\coprod}_i (s, e)}P^{\coprod}_i(s', e') &  = \sum_{(s', e')\sim^{\coprod}_i (s, e)}P_i(s)\cdot P_i(e)\tag{\redbf{$ P^{\coprod}_i(s', e') := P_i(s)\cdot P_i(e)\cdot \pre(e\mid s) $}}\\
% & = \sum_{s'\sim_i s}\sum_{e'\sim_i  e}P_i(s)\cdot P_i(e)\\
%  &=  \sum_{s'\sim_i s}P_i(s)\cdot(\sum_{e'\sim_i  e} P_i(e))\\
%   &=  \sum_{s'\sim_i s}P_i(s)\cdot 1\\
%    &=  \sum_{s'\sim_i s}P_i(s)\\
%     &=   1.\\
%\end{align*}
%\begin{align*}
%\sum_{(s, e)\sim^{\coprod}_i (s', e')}P^{\coprod}_i(s', e') 
%&  = \sum_{(s', e')\sim^{\coprod}_i (s, e)}P_i(s)\cdot P_i(e) \cdot \pre(e\mid s)
%\tag{definition of $P^{\coprod}_i$}
%\\
% & = \sum_{s'\sim_i s}\sum_{e'\sim_i  e}P_i(s)\cdot P_i(e)\cdot \pre(e\mid s)
%\tag{definition of $ \sim^{\coprod}_i$}
%\\
%  &=  \sum_{s'\sim_i s}P_i(s)\cdot \left( 
%  \sum_{e'\sim_i  e} P_i(e) \cdot \pre(e\mid s) \right)
%\end{align*}
%\end{proof}
%}
Finally, in order to define the updated model, observe that the map $\pre : E \times \rmPhi \rightarrow [0,1]$ in $\mathcal{E}$ 
induces the  map $pre : E  \rightarrow \mathcal{L}$ defined below.
\begin{definition}
\label{def:pre-bis}
For any probabilistic event structure 
$\mathcal{E}= (E, (\sim_i)_{i\in\Ag}, (P_i)_{i\in \Ag}, \rmPhi, \pre, \sub)$ 
over $\mathcal{L}$, let the map $pre$ be defined as follows:
\begin{align*}
pre : E  &\rightarrow \mathcal{L} \\
e & \mapsto \bigvee 
\left\{ \phi \in \rmPhi \mid \pre(e\mid\phi)\neq 0 \right\}.
\end{align*}
\end{definition}

%\marginnote{this is kinda ugly \\
%\redbf{Sabine: I put 2 possible alternative notations.\\ 
%The first is the same but displayed (it does not modifies the %number of pages, so we can use it. \\
%I think that the second one is easier to read.}}
%$
%e  \mapsto \bigvee_{
%	\begin{smallmatrix}
%	\phi \in \rmPhi \\
%	 \pre(e\mid\phi)\neq 0
%	\end{smallmatrix}
%} \phi.$
%$$
%e  \mapsto \bigvee_{
%	\begin{smallmatrix}
%	\phi \in \rmPhi \\
%	 \pre(e\mid\phi)\neq 0
%	\end{smallmatrix}
%}
%\phi. \quad \text{\redbf{OR}} \quad  e  \mapsto \bigvee 
%\{ \phi \in \rmPhi \mid \pre(e\mid\phi)\neq 0 \}
%.
%$$

%and for every model $\mb{M}$

%\begin{align*}
%\overline{pre}_{\mb{M}} : E & \rightarrow \mathcal{P}(S)\\
%e & \mapsto \val{\bigvee_{
%	\begin{smallmatrix}
%		\phi \in \rmPhi \\
%		 \pre(e\mid\phi)\neq 0
%	\end{smallmatrix}
%}
%\phi}_{\mb{M}}
%\end{align*}
 %Notice that the assumption that the formulas in $\rmPhi$ are mutually inconsistent implies that, whenever $\mb{M}, s\Vdash \overline{pre}(e)$, there exists exactly one $\phi_s\in \rmPhi$ such that $\mb{M}, s\Vdash \phi_s$. Hence

 %\begin{align*}
%pre_{\mb{M}} : S\times E & \rightarrow [0, 1]\\
%(s, e) & \mapsto pre(\phi_s \mid e) \tag{if $\mb{M}, s\Vdash \overline{pre}(e)$} \\
%(s, e) & \mapsto 0 \tag{otherwise} \\
%\end{align*}
\begin{definition}[Updated model]
\label{def: update PES model}
For any PES-model $\mb{M}= \left\langle S, (\sim_i)_{i\in \Ag}, (P_i)_{i\in \Ag}, \val{\cdot}\right\rangle$ and any probabilistic event structure $\mathcal{E}= (E, (\sim_i)_{i\in\Ag}, (P_i)_{i\in \Ag}, \rmPhi, \pre, \sub)$ over $\mathcal{L}$, let
the {\em epistemic update} $\mb{M}^{\mathcal{E}}$ of the model $\mb{M} $ by the probabilistic event structure $\mathcal{E}$ be as follows:
$$\mb{M}^{\mathcal{E}}: = \left\langle  S^{\mathcal{E}}, (\sim_i^{\mathcal{E}} )_{i\in \Ag}, (P_i^{\mathcal{E}})_{i\in \Ag}, \val{\cdot}_{\mb{M}^{\mathcal{E}}}\right\rangle$$
with
\begin{enumerate}
\item $S^{\mathcal{E}} := \left\{ (s,e) \in \coprod_{|E|} S \;\; \middle| \;\; \mb{M}, s\Vdash pre(e)\right\} $;
%\marginred{Typo $pre(e/s)$}
\item $\sim_i^{\mathcal{E}}  \ =\  \sim_i^{\coprod}\cap \ (S^{\mathcal{E}}\times S^{\mathcal{E}})$ for any $i\in \Ag$; %$(s,e) \sim_i (t,f)$ iff $s \sim_i t$ and $(e \sim_i f)$
\item each map $P_i^{\mathcal{E}}:S^{\mathcal{E}}\to [0, 1]$ is defined by the assignment  \[(s,e) \mapsto \frac{P^{\coprod}_i(s,e)}{\sum \left\{P^{\coprod}_i(s',e') \; \middle| \; (s,e) \sim_i (s',e') \right\}}; \]
\item 
%\redfootnote{to check any definition with the substitution map!}
the map 
$\val{\cdot}_{\mb{M}^{\mathcal{E}}} : \mathsf{AtProp} \rightarrow \P( S^\mathcal{E})$ is defined as follows:
%If $p \notin dom(\sub(e))$, then $ \val{p}_{\mb{M}^{\mathcal{E}}} : =\val{p}_{\coprod}\cap S^{\mathcal{E}}$;

$$
\val{p}_{\mb{M}^{\mathcal{E}}} : = \val{sub(p)}_{\coprod} \cap S^{\mathcal{E}}
%\left\{
%    \begin{array}{lll}
%        \val{sub(p)}_{\coprod} \cap S^{\mathcal{E}} & \quad& \mbox{if } p\in dom(\sub(e)) \text{for some } e\in E
%        \\
%        \val{p}_{\coprod}\cap S^{\mathcal{E}} && \mbox{otherwise }  
%    \end{array}
%\right.
$$
where the map $sub(p) : E  \rightarrow \mathcal{L}$ is given by: 
\begin{align*}
sub(p)(e) :=  
\left\{
    \begin{array}{lll}
    	\sub(e)(p) & \quad\quad& \mbox{if } p\in dom(\sub(e))
        \\
        p && \mbox{otherwise. }  
    \end{array}
\right.
\end{align*}

\end{enumerate}

\end{definition}

\begin{lemma}
\label{lem:prelim:PESmodel}
For any PES-model $\mb{M}$ and any probabilistic event structure $\mathcal{E}$ over $\mathcal{L}$,
the {\em epistemic update} $\mb{M}^{\mathcal{E}}$ of the model $\mb{M} $ by the probabilistic event structure $\mathcal{E}$ is a PES-model.
\end{lemma}
\begin{proof}
To prove that $\mb{M}^{\mathcal{E}}$ is a PES-model (\autoref{def:Prob epis state model Alexandru}), we need to show that it satisfies the following properties:
\begin{enumerate}
\item \label{proof:update:PES:model:1}
the set $S^\mathcal{E}$ is finite,
\item \label{proof:update:PES:model:2}
each relation $\sim^\mathcal{E}_i$ is an equivalence relation on $S$,
\item \label{proof:update:PES:model:3}
each map  $P^\mathcal{E}_i: S^\mathcal{E}\rightarrow\ ]0, 1]$ assigns a probability distribution
over each $\sim^\mathcal{E}_i$-equivalence class,
\item  \label{proof:update:PES:model:4}
the map $\val{\cdot} : \mathsf{AtProp} \rightarrow \P S^\mathcal{E}$ is a valuation map.
\end{enumerate}
\texttt{Proof of item \ref{proof:update:PES:model:1}.} The product of finite sets is finite.

\texttt{Proof of items \ref{proof:update:PES:model:2} and \ref{proof:update:PES:model:4}.} Trivial.

\texttt{Proof of item \ref{proof:update:PES:model:3}.} The fact that  $P^{\mathcal{E}}_i(s,e)>0$ for every $(s,e)\in S^\mathcal{E}$ follows from $P_i(s)>0$ for every $s\in S$ and Definition \ref{def: coproduct PES model}.  Hence, by construction, $P^\mathcal{E}_i$ is a probability distribution over $\sim^\mathcal{E}_i$-equivalence classes.
\end{proof}

\subsection{Semantics}
\label{ssec:classic:semantics}
In this subsection, we provide the semantics of PDEL
%\redfootnote{either \textit{the semantics of PDEL} or \textit{the semantics of probabilistic dynamic epistemic logic} because L stand sfor logic.}
over PES-models.
\begin{definition}[Probability measure]
%\redbf{TO IMPROVE}
\label{def:M-probability-measure}
Given a PES-model
$$\mb{M} = \left\langle S, (\sim_i)_{i\in \Ag}, (P_i)_{i\in \Ag}, \val{\cdot} \right\rangle,$$
let the \emph{probability measure} $\mu^{\mathbb{M}}_i \ : \ S \times \mathcal{L} \rightarrow [0,1]$ be defined as follows: for any $\phi \in \mathcal{L}$,
%$$ \mu^{\mb{M}}_i (s,\phi) \quad := \quad \sum_{s \sim_i s' \& s'\in \val{\phi}} P_i(s').$$
$$ \mu^{\mb{M}}_i (s,\phi) \quad := \quad \sum_{
\substack{
s \sim_i s'\\
s'\in  \val{\phi}
}} P_i(s')
 .$$
Notice that $\mu_i$ defines a probability measure on each $\sim_i$-equivalence class.
\end{definition}
\begin{definition}[Semantics of PDEL]
Given a PES-model
$$\mb{M} = \left\langle S, (\sim_i)_{i\in \Ag}, (P_i)_{i\in \Ag}, \val{\cdot} \right\rangle,$$
and the probability measures $\mu_i^{\mb{M}}$ defined as in \Cref{def:M-probability-measure},
the formulas of the language $\mathcal{L}$ are interpreted as follows:
\begin{align*}
\mb{M},s \models \bot & \textcolor{white}{\quad \text{iff} \quad} never
%\\
%\mb{M},s \models \top & \textcolor{shite}{\quad \text{iff} \quad} alsays
\\
\mb{M},s \models p & \quad \text{iff} \quad 
s \in \val{p}
\\
\mb{M},s \models \phi \wedge \psi & \quad \text{iff} \quad 
\mb{M},s \models \phi\quad \text{ and }\quad \mb{M},s \models \psi 
\\
\mb{M},s \models \phi \vee \psi & \quad \text{iff} \quad 
\mb{M},s \models \phi \quad \text{ or }\quad \mb{M},s \models \psi 
\\
\mb{M},s \models \phi \rightarrow \psi & \quad \text{iff} \quad 
\mb{M},s \models \phi \quad\text{ implies }\quad \mb{M},s \models \psi 
\\
\mb{M},s \models \lozenge_i \phi  & \quad \text{iff} \quad \text{there exists } s' \sim_i s \text{ such that } 
\mb{M},s' \models \phi 
\\
\mb{M},s \models \Box_i \phi  & \quad \text{iff} \quad \mb{M},s' \models \phi \quad \text{ for all } s' \sim_i s 
\\
\mb{M},s \models \langle\mathcal{E}, e\rangle \phi  
& \quad \text{iff} \quad \mb{M},s \models pre(e)
 \quad  
\text{ and } \quad \mb{M}^{\mathcal{E}}, (s,e) \models \phi 
\\
\mb{M},s \models [ \mathcal{E}, e ] \phi  
& \quad \text{iff} \quad 
 \mb{M},s \models pre(e)
 \quad  
\text{ implies } \quad \mb{M}^{\mathcal{E}}, (s,e) \models \phi
\\
\mb{M},s \models \left(\sum_{k = 1}^n\alpha_k · \mu_i(\varphi)\right) \geq \beta  
& \quad \text{iff} \quad \sum_{k=1}^{n} \alpha_k · \mu^{\mb{M}}_i(s,\varphi) \geq \beta  
\end{align*} 
\end{definition}

%\redbf{the preliminaries have been updated up to this point. We need to add the substitution map everywhere.}

\subsection{Axiomatization}
\label{ssec:classic:axiomatization}
%\redbf{to check the axiomatization\\}
%\paragraph*{Axiomatization.}

PDEL is a logical framework bringing together epistemics, dynamics,  and probabilities. Hence its axiomatization describes the behaviour of each of these components 
as well as their interactions. The full axiomatization of PDEL is given in  \Cref{table:PDEL} on page \pageref{table:PDEL}
and includes the axioms of classical multi-modal logic S5, understood as the basic epistemic logic, axioms capturing the theory of linear inequalities with rational coefficients (cf.~\cite[Theorem 4.3]{fagin1990logic}),  axioms capturing basic classical probability theory (cf.\ \cite{ABS16,vBGK09,fagin1990logic,ES14,FH94}), and axioms encoding the interaction between the dynamic modalities and the other logical connectives \cite{vBGK09,ABS16}, as well as the following inference rules: modus ponens, uniform substitution (see \cite{WanCao13}), necessitation for the static and dynamic modalities, and a substitution rule for the probabilistic operators $\mu_i$ (cf.~\cite{ABS16,vBGK09,ES14,FH94}).

%\begin{remark}\label{rk:axiom:classic:proba} \redbf{ APOSTOLOS: DO WE REALLY NEED THIS REMARK? we already give the references for all the axioms above}
%Explain why the axioms imply that $P_i(s)>0$ ALSO explain that P3  of the intuitionistic logic follows as well as $\mu(\psi)+\mu(\lnot\psi)=1$
%\end{remark}

\renewcommand{\arraystretch}{1.3}

\begin{table}
\caption{\textsc{Axioms of PDEL} %\red{to put the table in the axiomatization section.}
}
\label{table:PDEL}
\begin{tabular}{|ll|ll|}
\hline
\multicolumn{2}{|c|}{\textbf{Axioms of classical modal logic S5}}
\\
\hline
%\hline
& Tautologies of classical propositional logic
\\
k. & $ \Box_i ( \varphi \rightarrow \psi ) \rightarrow ( \Box_i \varphi \rightarrow \Box_i \psi)$
\\
dual. & $\Box_i \varphi \leftrightarrow \neg \lozenge_i \neg \varphi$
\\
t. & $\Box_i \varphi \rightarrow \varphi$ 
\\
iv.  & $\Box_i \varphi \rightarrow \Box_i \Box_i \varphi$
\\
v.  & $\neg \Box_i \varphi \rightarrow \Box_i \neg \Box_i \varphi$
\\
\hline
\multicolumn{2}{|c|}{\textbf{Axioms capturing the theory of linear inequalities with rational coefficients}}                                                                       \\
\hline
n0. & $t\geq t$\\
n1. & $(t\geq \beta) \leftrightarrow (t + 0 \cdot \mu_i(\varphi) \geq \beta)$ 
\\
n2. & $\left( \sum_{k=1}^n \alpha_k \cdot \mu_i(\varphi_k) \geq \beta \right) \rightarrow
\left( \sum_{k=1}^n \alpha_{\sigma(k)} \cdot \mu_i(\varphi_{\sigma(k)}) \geq \beta \right)
 $ for any permutation $\sigma$ over $\{ 1, ..., n\}$
\\
n3. & $\left(\left( \sum_{k=1}^n \alpha_k \cdot \mu_i(\varphi_k) \geq \beta \right) \wedge \left( \sum_{k=1}^n \alpha'_k \cdot \mu_i(\varphi_k) \geq \beta' \right)\right) \rightarrow 
\left( \sum_{k=1}^n (\alpha_k + \alpha'_k) \cdot \mu_i(\varphi_k) \geq (\beta + \beta') \right)$
\\
n4. & $(( t \geq \beta) \wedge (d \geq 0)) \rightarrow
( d \cdot t \geq d \cdot \beta  )$
\\
n5. & $ ( t \geq \beta ) \vee (\beta \geq t)$
\\
n6. & $ ((t \geq \beta) \wedge (\beta \geq \gamma))
\rightarrow (t \geq \gamma)   $
\\
\hline
\multicolumn{2}{|c|}{\textbf{Axioms capturing basic classical probability theory}}                                                                       
\\
\hline
p1. & $\mu_i(\bot) = 0$
\hspace{4.5cm}
p2. \quad $\mu_i(\top) = 1$
\\
p3. & $\mu_i(\varphi\land\psi)+\mu_i(\varphi\land\lnot\psi)=\mu_i(\varphi)$
%p3. & $(\varphi \rightarrow \psi) \rightarrow (\mu_i(\varphi) \leq \mu_i (\psi))$ 
\\
%p4. & $\mu_i(\varphi) + \mu_i(\psi) = \mu_i(\varphi \vee \psi) + \mu_i(\varphi \wedge \psi) $
%\\
p4. & $\Box_i \varphi \leftrightarrow ( \mu_i(\varphi)=1 ) $
\\
p5. &
$\left( \sum_{k=1}^n \alpha_k \cdot \mu_i(\varphi_k) \geq \beta \right) \rightarrow \Box_i \left( \sum_{k=1}^n \alpha_k \cdot \mu_i(\varphi_k) \geq \beta \right)$
\\
\hline
\multicolumn{2}{|c|}{\textbf{Reduction Axioms}}                                                                       \\
\hline
i1. & $\ls \mathcal{E},e \rs p  \leftrightarrow 
\left( pre(e) \rightarrow  \sub(e)(p) \right)$  
\\
i2. & $\ls \mathcal{E}, e \rs \neg \varphi \leftrightarrow
\left(  pre(e) \rightarrow \neg \ls  \mathcal{E},e \rs \varphi \right)$ \\
%\hline
i4. & $\ls \mathcal{E},e \rs (\varphi\wedge \psi) \leftrightarrow 
\left( \ls \mathcal{E},e \rs A \wedge \ls \mathcal{E},e \rs B \right)$                                  \\
%\hline
i5.  & $\ls \mathcal{E},e \rs\Box_i A \leftrightarrow 
\left( pre(e)\rightarrow  \bigwedge\{\Box_i\ls\mathcal{E}, f\rs A \mid e\sim_i f\}\right)$ 
\\
i6. & 
$\ls \mathcal{E},e\rs
\left( \sum_{k=1}^n \alpha_k \cdot \mu_i(\psi_k) \geq \beta \right) 
\leftrightarrow ( pre(e) \rightarrow C \geq D)$ with
\\
& $C = \sum_{\phi \in \Phi} \sum_{e \sim_i f} \sum_{k=1}^n \alpha_k \cdot \pre(f\mid \phi) \cdot \mu_i(\phi \wedge [\mathcal{E},f ]\psi_k )$ and
\\
& $D = \sum_{\phi \in \Phi} \sum_{e \sim_i f} \beta \cdot \pre(f \mid \phi) \cdot \mu_i(\phi) $\\
   \hline
\multicolumn{2}{|c|}{\textbf{Inference Rules}}                                                                          \\
\hline
MP & if $\vdash A\rightarrow B$ and $\vdash A$, then $\vdash B$ 
\\
%\hline
Nec$_i$        & if $\vdash A$, 
then $\vdash \Box_i  A$\\
Nec$_\alpha$   & if $\vdash A$, then $\vdash \ls \mathcal{E},e \rs A$\\
Sub$_\mu$ & if $\vdash A \rightarrow B$, then $\vdash \mu_i (A) \leq \mu_i(B)$\\
SubEq & if $\vdash A \leftrightarrow B$, then $\vdash \phi\leftrightarrow\phi[A/B]$\\
\hline
\end{tabular}

\end{table}

%\autoref{table:PDEL}\redfootnote{we already use $\sigma$ for the substitution map. Maybe we should use an other letter like $\eta$ for the permutation} 

\renewcommand{\arraystretch}{1}

\begin{lemma}[Soundness and Completeness]
\label{lem:sound-complete-PDEL}
PDEL is sound and complete w.r.t. the axiomatization given in  \Cref{table:PDEL}.
\end{lemma}
\begin{proof}
The statement follows from the general proof in Appendix \ref{Appendix:completeness} and Stone type duality.
\end{proof}

\section{Methodology}
\label{sec:method}
In the present section, we expand on the methodology of the paper. In the previous section, we gave a two-step account of the  {\em product update} construction which, for any  PES-model $\mb{M}$ and any event model $\mathcal{E}$ over $\mathcal{L}$, yields the updated model  $\mb{M}^\mathcal{E}$  as a certain {\em submodel} of a certain {\em intermediate model} $\coprod_{\mathcal{E}} \mb{M}$. This account is analogous to those given in \cite{MPS14,KP13} of the product updates of models of PAL and Baltag-Moss-Solecki's dynamic epistemic logic EAK. In each instance, the original product update construction can be % understood as the concatenation of a certain coproduct-type construction, followed by a subobject-type construction, as
illustrated by the following diagram (which uses the notation  introduced in the instance treated in the previous section):
\[
\mb{M} \hookrightarrow \coprod_{\mathcal{E}}\mb{M} \hookleftarrow \mb{M}^\mathcal{E}.
 \]
As is well known (see e.g.\ \cite{DaveyPriestley2002}) in duality theory, coproducts can be dually characterized as products, and subobjects as quotients. In the light of this fact, the construction of {\em product update}, regarded as a ``subobject after coproduct'' concatenation, can be dually characterized on the algebras dual to the relational structures of PES-models by means of a ``quotient after product'' concatenation, as illustrated in the following diagram:
\[
\bbA \twoheadleftarrow \prod_{\mathcal{E}}\bbA \twoheadrightarrow \bbA^\mathcal{E},
\]
resulting in the following two-step process. First, the coproduct $\coprod_{\mathcal{E}}M$ is dually characterized as a certain {\em product} $\prod_{\mathcal{E}}\bbA$, indexed as well by the states of $\mathcal{E}$, and such that $\bbA$ is the algebraic dual of $\mb{M}$; second, an appropriate {\em quotient} of $\prod_{\mathcal{E}}\bbA$ is then taken, which dually characterizes the submodel step.
On which algebras are we going to apply the ``quotient after product'' construction? The prime candidates are the algebras  associated with the PES-models via standard Stone-type duality:

\begin{definition}[Complex algebra]
	\label{def: complex algebra}
	For any PES-model $\mb{M} = \left\langle S, (\sim_i)_{i\in \Ag}, (P_i)_{i\in \Ag}, \val{\cdot}\right\rangle $, its \emph{complex algebra}  is the tuple
	$$\mb{M}^+ :=  \left( \P S, (\lozenge_i)_{i\in \Ag}, (\Box_i)_{i\in \Ag}, (P^+_i)_{i\in \Ag} \right) $$
where for each $i\in \Ag$ and $X\in \P S$,
\begin{align*}  
\lozenge_i X  & = \left\{ s\in S \mid \exists x \: ( s\sim_i x \text{ and } x\in X) \right\},
\\  
\Box_i X  & =    \left\{ s\in S \mid \forall x \: (s \sim_i x \implies x\in X) \right\},
\\
\mathsf{dom}(P^+_i) & =  \left\{X\in\P S \mid \exists y \ \forall x \: (x\in X\implies x\sim_i y) \right\},
\\	
P^+_i X &=\sum_{x\in X} P_i(x).
\end{align*}
Notice that the domain of $P_i^+$ consists of all the subsets of the equivalence classes of $\sim_i$.
%	where for each $i\in \Ag$ and $X\in \P S$, \begin{center}  
%	$\lozenge_i X  = \{ s\in S \mid \exists x( s\sim_i x$ and $x\in X) \}$,\\  $\Box_i X  =   \{ s\in S \mid \forall x(s \sim_i x\implies x\in X)\}$,
%\\
%	$\mathsf{dom}(P^+_i)=\{X\in\P S \mid \exists y\forall x (x\in X\implies x\sim_i y) \}$,\footnote{i.e.\ the domain of $P_i^+$ consists of all the subsets of the equivalence classes of $\sim_i$.}
%\\	
%$P^+_i X =\sum_{x\in X} P_i(x)$.
%	\end{center}
	%\begin{enumerate}
		%\item
%\begin{align*}
%P^+_i X & =  \left\{ \begin{array}{ll}
%\sum_{x\in X} P_i(x) & \textrm{if $\exists y\forall x (x\in X\implies x\sim_i y)$}\\
%1 & \textrm{otherwise}
%\end{array} \right.
%\end{align*}
		
	%\end{enumerate}
	%Notice that since in the present setting $\sim_i$ is an equivalence relation, $\lozenge_i$ and $\blacklozenge_i$ coincide.
\end{definition}
%\redfootnote{add the substitution map?}
In this setting, the ``quotient after product'' construction behaves exactly in the desired way, in the sense that one can check {\em a posteriori} that the following holds:
%\footnote{\redbf{Caveat: we are abusing notation here. Proposition \ref{prop:compatibility with duality} should be formulated using Definition \ref{def: es over L induce es over A} and Fact \ref{fact: es over L induce es over A}.}}
\begin{proposition}
\label{prop:compatibility with duality}
For every PES-model $\mb{M}$ and any  event structure $\mathcal{E}$ over $\mathcal{L}$, the algebraic structures $(\mb{M}^+)^{\mathcal{E}}$ and $(\mb{M}^{\mathcal{E}})^+$ can be identified.
\end{proposition}
%\redfootnote{Can we find a short title that discribe proposition \ref{prop:compatibility with duality}?}

%\redbf{Add an explanation about the proof. The part on the boolean algebra comes from the paper \cite{KP13} and the part for the probabilities is given in Lemma \ref{lem:duality:probabilities}}

\begin{proof}
%\redfootnote{Is this explanation enough? }
This results follows from:
(1)
Fact 12 in \cite{KP13} that states that for any (non probabilistic) Kripke model $\mb{N}$, the structures 
$(\mb{N}^+)^{\mathcal{E}}$ and $(\mb{N}^{\mathcal{E}})^+$ can be identified,
and (2) Lemma \ref{lem:duality:probabilities} on page \pageref{lem:duality:probabilities} that states that the probability measures 
on the complex algebras $(\mb{M}^+)^{\mathcal{E}}$ and $(\mb{M}^{\mathcal{E}})^+$ are the same.
\end{proof}

Moreover, the ``quotient after product'' construction holds in much greater generality than the  class of complex algebras of PES-models, which is exactly its added value over the update on relational structures. In the following section, we are going to define it in detail in the setting of epistemic Heyting algebras. % probabilistic dynamic update construction we are going to introduce next  holds in much greater generality , but restricted to the complex algebras,

%$$\Diamond (p \to q)\to(\Box p \to \Diamond q),$$
%$$(\Diamond p\to \Box q)\to \Box(p\to q),$$
%and is closed under substitution, modus ponens and necessitation $(\vdash \varphi/\vdash\Box \varphi)$.
%The logic MIPC is
%the smallest set of formulas in the language above which contains all the axioms of intuitionistic propositional
%logic, the following modal axioms
%\begin{itemize}
%	\item[]$\Box p \to p,\ p \to \Diamond p,$
%	\item[]$\Box (p \to q)\to (\Box p \to\Box q),\ \Diamond(p\vee q)\to(\Diamond p\vee \Diamond q),$
%	\item[]$\Diamond p \to\Box \Diamond p,\  \Diamond\Box p\to \Box p,$
%	\item[]$\Box (p \to q)\to(\Diamond p \to \Diamond q),$
	%(A5)~&(\Diamond p \to \Box q)\to \Box (p \to q)
%\end{itemize}
%sand is closed under 

\section{Probabilistic dynamic epistemic updates on finite Heyting algebras}
\label{sec:ha}

%\newpage

%\Large{\redbf{Intuitionistic case}}

%Intuitionistic APE-structures
%\marginnote{This needs to become the new section 3, with a new title, see also comments there}

%For the static probability free fragment we use the work presented in \cite{KP13}.
%\redbf{Should we work with finite Heyting algebras? %distributive lattices, g\"odel algebras
%\\}

The present section aims at introducing the algebraic counterpart of the event update construction presented in Section \ref{sec:prelim}. %We begin by defining algebraic probabilistic epistemic structures, which generalise the complex algebras of the probabilistic epistemic frames.
For the sake of enforcing a neat separation between syntax and semantics, throughout the present section,  we will disregard the logical language $\mathcal{L}$, and work  on {\em algebraic probabilistic epistemic structures} (APE-structures, see Definition \ref{def: alg probab epist structure}) rather than on APE-models (i.e.\ APE-structures endowed with valuations). To be able to define the update construction, we will need to  base our treatment on  a modified definition of  event structure over an algebra, rather than over $\mathcal{L}$.

\paragraph{Structure of the section.}In Section \ref{ssec:eha}, we introduce epistemic Heyting algebras. In Section \ref{ssec:measuresAPE}, we recall the definition of intuitionistic probability from \cite{weatherson2003classical} and endow epistemic Heyting algebras with measures to define algebraic probabilistic epistemic structures. In Section \ref{ssec:probevents}, we define probabilistic event structures over epistemic algebras, as the intuitionistic algebraic counterparts of classical probabilistic event structures. In Section \ref{ssec:intermediate}, we introduce the construction of intermediate pre-probabilistic event structure as the first step of the algebraic event update construction. Finally, in Section \ref{ssec: abstract charact i minimal els pseudo quotient}, we introduce the pseudo-quotient update construction and define the event update on algebraic probabilistic epistemic structures. 

%\begin{enumerate}
%\item We recall the definition of an epistemic Heyting algebra
%\item We define algebraic probabilistic epistemic structures
%\item We define probabilistic events
%\item We define the intermediate algebra
%\item We define the update algebra
%\end{enumerate}

\subsection{Epistemic Heyting algebras% and their event structures}
}
\label{ssec:eha}

In this section we introduce epistemic Heyting algebras. We start by recalling the definition of monadic Heyting algebras, which provide algebraic semantics for the logic MIPC, the intuitionistic analogue of the classical modal logic S5 (cf.\ \cite{bezhanishvili1998varieties,bezhanishvili1999varieties,KP13}). Then, we introduce the concept of $i$-minimal elements of monadic Heyting algebras. Finally, we define epistemic Heyting algebras as those monadic Heyting algebras whose $i$-minimal elements are enough to describe certain subalgebras of interest for the developments of the next sections.

%\redbf{explain why we chose monadic Heyting algebras.}

\begin{definition}[Monadic Heyting algebra (cf.\ \cite{bezhanishvili1998varieties})]
%\marginnote{see margin note in proposition after definition \ref{def:pseudo quotient general}}
	\label{def: epist algebra}
	A {\em monadic Heyting algebra} is a tuple 
	$$  \mb{A}:=\left( \mb{L}, (\lozenge_i)_{i\in \Ag} , (\Box_i)_{i\in \Ag}\right) $$ such that
	 $\mb{L}$ is a Heyting algebra, %($\sigma$-algebra);
		and each $\lozenge_i$ and $\Box_i$ is a monotone unary operation on $\mb{L}$
		such that for all $a, b\in \mb{L}$,
		%\begin{equation}
		%\label{eq:conjugation}
\begin{multicols}{2}
\begin{align}
& a\leq\lozenge_ia
\tag{M1}
\label{axiom:epist-alg:refl}
\\
& \Box_ia\leq a
\tag{M2}
\label{axiom:epist-alg:refl2}
\\
& \lozenge_i(a\lor b)\leq\lozenge_ia\lor\lozenge_ib
\tag{M3}
\label{axiom:epist-alg:distribd1}
\\
& \Box_i(a\to b)\leq\Box_ia\to\Box_ib
\tag{M4}
\label{axiom:epist-alg:distribb1}
\end{align}
\begin{align}
& \lozenge_i a \leq \Box_i\lozenge_i a
\tag{M5}
\label{axiom:epist-alg:sym}
\\
& \lozenge_i\Box_ia\leq\Box_ia
\tag{M6}
\label{axiom:epist-alg:trans}
\\
& \Box_i(a\to b)\leq\lozenge_ia\to\lozenge_ib
\tag{M7}
\label{axiom:epist-alg:subalgebra}
\\
& \lozenge_i\bot\leq\bot
\tag{M8}
\label{axiom:epist-alg:distribd2}
\\
& \top\leq\Box_i\top
\tag{M9}
\label{axiom:epist-alg:distribb2}
\end{align}
\end{multicols}

%\begin{align}
%& a\leq\lozenge_ia
%\tag{M1}
%\label{axiom:epist-alg:refl}
%\\
%& \Box_ia\leq a
%\tag{M2}
%\label{axiom:epist-alg:refl2}
%\\
%& \lozenge_i(a\lor b)\leq\lozenge_ia\lor\lozenge_ib
%\tag{M3}
%\label{axiom:epist-alg:distribd1}
%\\
%& \Box_i(a\to b)\leq\Box_ia\to\Box_ib
%\tag{M4}
%\label{axiom:epist-alg:distribb1}
%\\
%& \lozenge_i a \leq \Box_i\lozenge_i a
%\tag{M5}
%\label{axiom:epist-alg:sym}
%\\
%& \lozenge_i\Box_ia\leq\Box_ia
%\tag{M6}
%\label{axiom:epist-alg:trans}
%\\
%& \Box_i(a\to b)\leq\lozenge_ia\to\lozenge_ib
%\tag{M7}
%\label{axiom:epist-alg:subalgebra}
%\\
%& \lozenge_i\bot\leq\bot
%\tag{M8}
%\label{axiom:epist-alg:distribd2}
%\\
%& \top\leq\Box_i\top
%\tag{M9}
%\label{axiom:epist-alg:distribb2}
%\end{align}
%\bluebf{Add the necessary axioms for the completeness proof}

%\begin{center}
%	\begin{tabular}{l c l}
%		$\lozenge_i (a \to b)\leq\Box_i a \to \lozenge_i b$ & $\quad$ & $\lozenge_i a\to \Box_i b\leq\Box_i(a\to b)$
%		%fisher servi
%		\\
%		$\lozenge_i a \wedge b \leq \lozenge_i ( a \wedge  \lozenge_i b)$ 
%		% symetry
%		&& 
%		$\lozenge_i\lozenge_ia\leq\lozenge_ia$
%		%$\Box_i(a\vee\Box_i b)\leq\Box_i a\vee b$
%		\\
%		$a\leq\lozenge_ia$ && %$\Box_i a\leq a$
%%		\\
%%		 && %$\Box_i a\leq \Box_i\Box_i a.$
%	\end{tabular}
%\end{center}
\end{definition}
%

%M1. & $\lozenge_i (a \to b)  \leq\Box_i a \to \lozenge_i b$ & M2.
%& $\lozenge_i a\to \Box_i b  \leq\Box_i(a\to b)$
%\\
%M3. & $\lozenge_i a \wedge b \leq \lozenge_i ( a \wedge  \lozenge_i b)$ & M4 
%& $\lozenge_i\lozenge_ia \leq\lozenge_ia$
%\\
%M5. & $a \leq\lozenge_ia$ &&

\begin{remark}
	The algebraic and duality theoretic treatment of monadic Heyting algebras has been developed in \cite{bezhanishvili1998varieties} and \cite{bezhanishvili1999varieties}. 
	In particular, as mentioned in \cite[Lemma 2]{bezhanishvili1998varieties}, in the presence of \eqref{axiom:epist-alg:distribb2}, 
	axiom \eqref{axiom:epist-alg:distribb1} is equivalent to $\Box_ia\land\Box_ib\leq\Box_i(a\land b)$, so all modalities are normal, and  $\lozenge_i\lozenge_ia\leq\lozenge_ia$ and $\Box_ia\leq\Box_i\Box_ia$ are derivable from the axioms. These conditions correspond also in the best known intuitionistic settings to the transitivity of the associated accessibility relations (cf.\ \cite{ALBA}). This implies in particular that $\lozenge_i$ is a closure operator for each $i\in\Ag$.
\end{remark}

The next definition intends to capture algebraically the notion of equivalence cell in the epistemic space of agents. Notice that for any equivalence relation $R$ on a set $X$ and any $x\in X$, the equivalence cell $R[x]=R^{-1}[x]=\langle R \rangle\{x\}$ is a minimal nonempty fixed point of  $\langle R \rangle$.\footnote{%\redbf{should we put this definition in the text.}\\
Recall that, for any binary relation $R\subseteq X \times X$, we define the maps $R$, $R^{-1}$ and $\langle R \rangle$ as follows:
\begin{align*}
R \ : \ X & \rightarrow \P X
& R^{-1} \ : \ X & \rightarrow \P X 
\\
x & \mapsto \{ x' \in X \mid (x,x') \in R \}
& x & \mapsto \{ x' \in X \mid (x',x)\in R \}
\\
&&&\\
\langle R \rangle \ : \ \P X & \rightarrow \P X
\\
S & \mapsto \{ x' \in X \mid \exists x \in S,\ (x',x)\in R \}.
\end{align*}
} 
This justifies the following definition.

\begin{definition}[$i$-minimal elements]
	\label{def:i-minimal}
	Let $\mb{A}$ be a monadic Heyting algebra.
	An element $a\in \mb{A}$ is  \emph{$i$-minimal} if
	\begin{enumerate}
		\item \label{def:i-minimal:item:bot} $a\neq \bot$,
		\item \label{def:i-minimal:item:fixpoint} $\lozenge_i a = a $ and
		\item \label{def:i-minimal:item:minimal} if $b \in \mb{A}$, $b < a$ and $\lozenge_i b = b $, then $b = \bot$.
	\end{enumerate}
Let $\mathsf{Min}_i(\mb{A})$ denote the set of the $i$-minimal elements of $\mb{A}$.
\end{definition}

\begin{remark}
Notice that, for any $b\in\mb{A}\setminus\{\bot\}$, 
there exists at most one $a\in\mathsf{Min}_i(\mb{A})$ such that $b\leq a$. Indeed every such $a$ must coincide with $\lozenge_ib$.\label{page:uniqueimin}
\end{remark}

\begin{definition}[Epistemic Heyting algebra]
%\marginnote{see margin note in proposition after definition \ref{def:pseudo quotient general}}
	\label{def:epist-Heyting-algebra}
	An  {\em epistemic Heyting algebra} is a finite monadic Heyting algebra 
	$$  \mb{A}:=\left( \mb{L}, (\lozenge_i)_{i\in \Ag} , (\Box_i)_{i\in \Ag}\right) $$ 
	such that for every $i\in\Ag$ and every $a\in\mb{A}$ the following holds:
\begin{align}
& \lozenge_ia\lor\lnot\lozenge_ia=\top.
\tag{E}
\label{axiom:epist-alg:boolean}
\end{align}	

%\bluebf{Add the necessary axioms for the completeness proof}
\end{definition}

\begin{remark}
\label{rk:eHA:axiomE}
The axiom above captures algebraically the requirement that $i$-minimal elements, re\-pre\-senting cells in the partition, cover the whole space. 
\end{remark}
%\bluebf{Alessandra or Willem: Could you explain the reasoning behind the axiom \ref{axiom:epist-alg:boolean}? \\

%Some intuition : 
%It implies that the subalgebra of the diamonds is a boolean algebra. Some ideas are that probability formulas will be interpreted as diamond formulas, but they are inherently classical (since $x>m\lor\lnot x>m$ is true). Maybe an epistemic explanation as well: essentially this implies that in the dual esakia space the diamond formulas are downsets, i.e. the agent cannot distinguish the point in time that he is at?\\
%this additional condition intuitively encode the fact that the alternatives represented by the various i-minimal elements are exhaustive}

%\begin{notation}
In the remainder of the present section, $\mb{A}$ will denote  an  epistemic Heyting algebra.
%\end{notation}

\begin{lemma}
\label{lem:epist-Hey-alg:diamond-A}
	If $\mathbb{A}$ is an Epistemic Heyting algebra, then, for every agent $i$,
	$$\lozenge_i\mathbb{A}:=\{\lozenge_i a\in\mathbb{A}\mid a\in\mathbb{A}\}$$ 
	is a Boolean sub-algebra of $\mathbb{A}$. 
	Furthermore, if 
	$$\Box_i\mathbb{A}:=\{\Box_i a\in\mathbb{A}\mid a\in\mathbb{A}\},$$ 
	then  $\lozenge_i\mathbb{A}=\Box_i\mathbb{A}$.
\end{lemma}
\begin{proof}
That $\lozenge_i\mathbb{A}$ is a subalgebra of $\mathbb{A}$ follows from the fact that 
the equalities 
$$\lozenge_i(\lozenge_ia\land b)=\lozenge_ia\land\lozenge_ib
\quad\quad \text{ and } \quad\quad 
\lozenge_i(\lozenge_i a\to \lozenge_ib)=\lozenge_ia\to\lozenge_ib$$ hold in every monadic Heything algebra  (see for example \cite[Lemma 2]{bezhanishvili1998varieties}). 
That $\lozenge_i\mathbb{A}$ is a Boolean algebra follows from the axiom \eqref{axiom:epist-alg:boolean} : $\lozenge_i a\lor\lnot\lozenge_i a=\top$.
	
Finally, we can easily prove that $\lozenge_i\mathbb{A}=\Box_i\mathbb{A}$ using the axioms 
\eqref{axiom:epist-alg:refl},
\eqref{axiom:epist-alg:refl2},
\eqref{axiom:epist-alg:sym} and
\eqref{axiom:epist-alg:trans}.
\end{proof}

\begin{remark}\label{rk:basicEHA}Given the fact that Epistemic Heyting algebras are finite and since $\lozenge_i\mathbb{A}$ is a Boolean algebra, it is not hard to see that $i$-minimal elements are the atoms of $\lozenge_i\mathbb{A}$ and hence $\bigvee\mathsf{Min}_i(\mb{A})=\top$.
\end{remark}
\begin{notation}
\label{notation:downset}
For any poset (partially ordered set) $\mb{P}=(P,\leq)$, we let  
\begin{align*}
\downarrow_{\mb{P}} \ : \ \P \mb{P} & \rightarrow \P \mb{P} \\
X & \mapsto X\!\!\downarrow_{\mb{P}} \ := 
\{ x' \in \mb{P} \mid x' \leq x \text{ for some } x\in X \}.
\end{align*}
For the sake of readability, we drop the subscript and let  $X\!\! \downarrow$ denote the downset generated by $X$. In addition, if $X = \{x\}$, we let  $x\! \downarrow$ denote the downset generated by $\{x\}$.
\end{notation}

%\redfootnote{We need to explain that when taking the pseudo-product and quotient the outcome remains an epistemic heyting algebra, that is that the union of $i$-minimal elements is $\top$. It's easy to show this using the characterisation of $i$-minimal elements that we already have.}

\subsection{Algebraic probabilistic epistemic structures 
\label{ssec:measuresAPE}
%\redbf{Title: Probabilities on Heyting algebras?}
}
%The following definition introduces the algebraic counterparts of PES-models
In this Section, we introduce $i$-premeasures and $i$-measures and define algebraic pre-probabilistic and probabilistic epistemic structures which will serve as the underlying structures of intuitionistic probabilistic epistemic logic.\\

The following definition is an adaptation of a proposal of Weatherson's (see \cite[page 2]{weatherson2003classical}) in which the notion of probability is generalised and made parametric in a given  consequence relation.
Even though there is no consensus on what an intuitionistic probability function should be%
%\footnote{For instance, Marra argues that an intuitionistic probability should have specific conditions accounting for the interaction between the probability measure and the intuitionistic implication.}%
, Weatherson's proposal captures necessary conditions for such a function and establishes a systematic link between logic and probability. 
The definition below has also been adopted by \cite{aguzzoli2008finetti,flaminio2017states}. %\redbf{Tomasso et al use it and refer to weatherson. Can someone check the dutch book one?}

\begin{definition}[Intuitionistic probability measures]
Let $\mathbb{H}$ be a Heyting algebra. A function $\Prob:\mathbb{H}\to[0,1]$ is \emph{an intuitionistic probability measure} if the following conditions are satisfied: for all $a,b\in\mathbb{H}$,
\begin{align*}
& (1) \; \Prob(\bot)=0, &&
 (3) \;  \text{if } a \leq_{\mb{H}}b, \text{ then } \Prob(a)\leq \Prob(b),
\\
&  (2) \;  \Prob(\top)=1,
&& (4) \;  \Prob(a) + \Prob(b) = \Prob(a\vee b) + \Prob(a\wedge b).
\end{align*}
\end{definition}

Notice that, for  intuitionistic probability measures, it does no longer hold that $\Pr(p \vee \neg p)=1$.

Given that, in classical PDEL, the probability functions range over equivalence classes instead of the whole model, we  need to mirror that fact by defining probability functions that are probability measures on the quotient algebras generated by $i$-minimal elements. 
%\redbf{to comment on the definition and explain that we are adapting weatherson definition to  the setting in which measures make sense only on the counterparts of equivalence classes, namely i-minimal elements.}
%\begin{definition}
%Let $\mb{A} :=  \langle \P S, (\lozenge_i)_{i\in \Ag}, (\mu_i)_{i\in \Ag} \rangle $ be the complex algebra of a Probabilistic Epistemic State Model. Let $i\in \Ag$.
%An element $a\in \mb{A}$ is  \textit{$i$-minimal} if
%\begin{itemize}
%\item $a\neq \bot$,
%\item $\lozenge_i a = a $ and
%\item if $b \in \mb{A}$, $b < a$, and $\lozenge_i b = b $, then $b = \bot$.
%\end{itemize}
%\end{definition}

%\red{the following definition generalizes the main features of classical probability measures to the intuitionistic setting. Among other things, this definition forces  probability measures to assign 1 to every tautology, 0 to every contradiction, be additive and order preserving.}

\begin{definition}[$i$-premeasure \& $i$-measure]
\label{def:measures}
%	Let $\mb{A} =  \langle \mb{L}, (\lozenge_i)_{i\in \Ag} , (\blacklozenge_i)_{i\in \Ag} \rangle$  be an epistemic algebra. Given $i\in\Ag$,
 A partial function $\mu:\mb{A}\to\mb{R}^+$ is an \emph{ $i$-premeasure} on $\mb{A}$, if it satisfies the following properties:
	\begin{enumerate}
	    \item \label{def:epAlg:two:domain}
	    $\mathsf{dom}(\mu)=\mathsf{Min}_i(\mb{A}){\downarrow}$;
		\item \label{def:epAlg:two:monotone}
$\mu$ is order-preserving; %$\mu(\bot)=0$  and
		\item \label{def:epAlg:two:join}
for every $a\in \mathsf{Min}_i(\mb{A})$ and all $b, c\in a{\downarrow}$, we have $\mu(b\vee c) = \mu(b)+ \mu(c) - \mu(b\wedge c)$;
		\item \label{def:epAlg:two:bot}  $\mu(\bot)=0$ if $\mathsf{dom}(\mu)\neq\varnothing$.
	\end{enumerate}
	An $i$-premeasure on $\mb{A}$ is an {\em $i$-measure}, if it satisfies the following properties:
	\begin{enumerate}
	\setcounter{enumi}{4}
		\item \label{def:epAlg:two:fixedpoints}
$\mu(a) = 1$ for every $a\in \mathsf{Min}_i(\mb{A})$.
	\item \label{def:epAlg:two:nonzero}for every $a\in \mathsf{Min}_i(\mb{A})$ and all $b, c\in a{\downarrow}$ such that $b<c$, it holds that $\mu(b)<\mu(c)$;%\redfootnote{to explain/motivate this condition.}
%		\item \label{def:epAlg:two:top} $\mu(\top)=1$.
	\end{enumerate}
	\end{definition}
Condition \eqref{def:epAlg:two:domain} ensures that the probability measures are defined on the quotient algebras generated by $i$-minimal elements. 
Conditions \eqref{def:epAlg:two:monotone} to \eqref{def:epAlg:two:fixedpoints}  are imported from Wheatherson's definition of in\-tui\-tionistic probabilistic functions.
Condition \eqref{def:epAlg:two:nonzero} corresponds to the fact that in the classical case, the probability distributions over the elements of the equivalence classes do not take value $0$ (see \autoref{def:Prob epis state model Alexandru}, page \pageref{def:Prob epis state model Alexandru})

\begin{remark}
In the case when $\mathsf{Min}_i(\mb{A}){\downarrow}=\varnothing$, there exists a unique $i$-(pre)measure, the empty function. Throughout this section, all the results regarding $i$-minimal elements and $i$-(pre)measure hold vacuously in the case when $\mathsf{Min}_i(\mb{A}){\downarrow}=\varnothing$. 
\end{remark}

\begin{definition}[ApPE-structure \& APE-structure]
	\label{def: alg probab epist structure}
	An {\em algebraic pre-probabilistic epistemic structure (ApPE-structure)} is a tuple 
	$$ \mathcal{F}:=  \left( \mb{A}, (\mu_i)_{i\in \Ag} \right) $$ 
	such that 
\begin{enumerate}
\item $\mb{A}$ is an  epistemic Heyting algebra (see Definition \ref{def:epist-Heyting-algebra}), and
\item each  $\mu_i$ is an $i$-premeasure on $\mb{A}$.
\end{enumerate}	
	  An ApPE-structure $\mathcal{F}$ is an {\em algebraic probabilistic epistemic structure (APE-structure)} if  each $\mu_i$ is an $i$-measure on $\mb{A}$. 
	  
	  We refer to $\mb{A}$ as the {\em support} of $ \mathcal{F}$ and we denote it $\support (\mc{F})$.
	  
%	  \redbf{do we have that $\mu_i(a)>0$ if $a\neq \bot$???}
\end{definition}

\subsubsection{The algebraic epistemic structure associated to a classical model}
\label{sssec:complexe:algebra:calssical}

\begin{lemma}
	\label{lem:charact i minimal elements complex algebra}
	For any PES-model $\mb{M}$, the $i$-minimal elements of its  complex algebra $\mb{M}^+$ are exactly the equivalence classes of $\sim_i$.
\end{lemma}

\begin{proof}
See Appendix \ref{app:section4:0} page \pageref{app:section4:0}.
\end{proof}

\begin{proposition}
	\label{prop: complex algebras are APE structures}
	For any PES-model $\mb{M}$, its complex algebra $\mb{M}^+$ (see  \Cref{def: complex algebra}) is an APE-structure (see \Cref{def: alg probab epist structure}).
\end{proposition}
%\bluebf{Adapt the proof with the new axioms for
% the completeness proof}

\begin{proof}
See Appendix \ref{app:section4:1} page \pageref{app:section4:1}.
\end{proof}

\subsection{Probabilistic event structures over epistemic Heyting algebras}
\label{ssec:probevents}
%We will find it useful to introduce the following auxiliary definition.
%For any set $X$, a \emph{pre-ordered multiset} on $X$ is a multiset $\rmPhi$ of elements of $X$ such that for any $x\in X$ having multiplicity $n$ in $\rmPhi$, the copies $x_1,\ldots,x_n$ of $x$ are endowed with an irreflexive linear order which we assume to be $x_1\prec \cdots\prec x_n$. In what follows, we will use the membership symbol $\in$ in the context of multisets on $X$ always referring to the copies of a given element on $X$. Hence, even when we abuse notation and write $y\in\rmPhi$, this symbol refers to some copy of the element $y\in X$. 
%\redfootnote{orders and preorders are reflexive relations. We should fine a different name : strict order \href{http://mathworld.wolfram.com/StrictOrder.html}{http://mathworld.wolfram.com/StrictOrder.html}}

 In this section, we introduce intuitionistic event structures, which are needed to correctly generalise probabilistic epistemic updates to an intuitionistic metatheory.\\

%\redbf{alternative version 1.}
We will find it useful to introduce the following auxiliary definitions.
Recall that a \emph{multiset} is a generalisation of the concept of set that allows multiple instances of the same element. Hence, $\{a,a,b\}$ and $\{a,b\}$ are the same set, but different multisets. However, order does not matter, so $\{a,a,b\}$ and $\{a,b,a\}$ are the same multiset.
Let $\rmPhi$ be a multiset on the set $X$ and $a,b\in \rmPhi$. We say that $a$ and $b$ \emph{arise from the same element} if $a$ and $b$ are copies of the same element from $X$. We denote it $a =_X b$. %}

%\redbf{TO EDIT}
%

\begin{definition}[Ordered multiset on a lattice]
\label{def:ordered:multiset}
Let $\mb{L}=(L,\leq)$ be a finite lattice.
An \emph{ordered multiset} $\bfPhi = (\rmPhi,\prec)$ on $\mb{L}$ is a multiset $\rmPhi$ of elements of $L$ equipped with a strict order $\prec$  such that, for all pairwise distinct elements $x,y,z\in \rmPhi$,
\begin{enumerate}
\item \label{def:it:ord:multi:order} if $x \prec y$, then $x \leq_{\mb{L}} y$;
\item \label{def:it:ord:multi:not:bot}  if $x\neq\bot$ and 
$ x \leq_{\mb{L}} y$, then  $x \prec y$ or $y \prec x$;
\item \label{def:it:ord:multi:tree}  if $x\prec y$ and $x\prec z$, then 
$y \prec z$ or $z \prec y$.
\end{enumerate}
%Furthermore if $x\neq\bot$ and $y\neq\bot$ then  $x<_L y$ implies that $x\prec y$ and  if $x$ and $y$ arise from the same element then $x\prec y$ or $y\prec x$.
%\redfootnote{a strict order is a irreflexive, transitive, asymmetric binary relation. A strict order is total if for all a,b, either $a<b$ or $b<a$ or $a=b$} 
%which we assume to be $x_1\prec \cdots\prec x_n$. 

In the present paper, we  use the membership symbol $\in$ in the context of multisets on $\mb{L}$ always referring to the copies of a given element of $\mb{L}$. For instance, the variable $y$ in the symbol $y\in\rmPhi$ refers to one specific copy of some element of $\mathbb{L}$. 
\end{definition}

\begin{remark} In Section \ref{sec:semantics-IPDEL}, we will be working with event structures over logical languages rather than with event structures over algebras (see  Definition \ref{def:algebraic event}). Event structures over languages (see Definition \ref{def:intuitionistic-proba-epist-event-struct}) are tuples where $\Phi$ is a set of formulas each pair of which is made either of incompatible formulas or of formulas one of which implies the other. However, some of these formulas might be identified with each other under some valuations. In order to define updates on algebras independently from logic, in Definition \ref{def:algebraic event}  the ordered multisets above will play the same role played by the sets $\Phi$ in event structures over languages. Specifically, the multiset structure serves to keep track of the fact that some elements of the lattice might be the interpretation of more than one formula in the set $\Phi$, and the order on the multiset $\bfPhi$ helps to keep track of the logical structure of the set $\Phi$. Finally, condition 3 makes sure that the order structure of $\bfPhi$ is an upward forest, and conditions 1 and 2 together guarantee that, with the exception of formulas which are mapped to $\bot$, the logical structure of the set $\Phi$ is preserved and reflected by the order $\prec$.

\end{remark}

Now let us introduce probabilistic event structures in the intuitionistic setting: 
\begin{definition}[Probabilistic event structure over
an epistemic Heyting algebra]
	\label{def:algebraic event}
	For any epistemic Heyting algebra $\mb{A}$ (see Definition \ref{def:epist-Heyting-algebra}), a \emph{probabilistic event structure over} $\mb{A}$  is a tuple
	$$\mb{E} = (E, (\sim_i)_{i\in\Ag}, (P_i)_{i\in \Ag}, \bfPhi, \overline{\pre})$$
	such that 
\begin{enumerate}
\item $E$ is a non-empty finite set;
\item each $\sim_i$ is an equivalence relation on $E$;
\item
each $P_i:E\to\ ]0,1]$ assigns a probability distribution over each $\sim_i$-equivalence class, i.e.\   $$\sum \left\{P_i (e' ) \mid e' \sim_i e \right\} = 1;$$
%\item $E$, $\sim_i$, $P_i$ are as in Definition \ref{def:proba-epist-event-struct};
\item $\bfPhi = (\rmPhi,\prec) $ is a finite ordered multiset on $\mb{A}$ such that,  
for all $a, b\in \rmPhi$ which arise from distinct elements in $\mb{A}$, either
\begin{center}
$a\wedge_{\text{{$ \mb{A}$}}} b = \bot \quad $ or $ \quad a <_{\text{$ \mb{A}$}} b \quad $ or $ \quad b <_{\text{{$ \mb{A}$}}} a$;
\end{center} 
\item the map $\overline{\pre}$ $ :  E \times \rmPhi \rightarrow [0,1]$  assigns a probability distribution $\overline{\pre}(\bullet | a)$ over $E$ for every  $a \in \rmPhi$;
\item \label{def:alg:event:item:6} for all $a\in\rmPhi$ and  $e\in E$, if $\overline{\pre}(e | a)=0$ then $\overline{\pre}(e | b)=0$  for all  $b\in\rmPhi$ such that $a\prec b$.
%\redfootnote{for all $a\in\rmPhi$ and  $e\in E$, if $\pre(e | a)=0$ then $\pre(e | b)=0$  for all  $b\in\rmPhi$ such that $a<_{\mb{A}} b$  or ( $a =_{\mb{A}} b$ and $a \prec b$).}
\end{enumerate}	
%\redbf{$\bf{\Phi}$	for the list $(\rmPhi,\preceq)$ ?}
	\end{definition}
The definition above is a proper generalization of the analogous definition given in the classical setting (\autoref{def:proba-epist-event-struct}). The main generalization concerns the fact that the elements in  $\rmPhi$ (which are the potential interpretants of formulas) are no longer required to be mutually inconsistent but may also be   `logically dependent'. In this latter case, the precondition function is required to satisfy an additional compatibility condition which is similar to the one adopted in \cite{aguzzoli2008finetti}. 
For sake of readability, in what follows, we will simply refer to probabilistic event structures over epistemic Heyting algebras as
{\em event structures}.
%\redfootnote{Discuss here why we use a different $\rmPhi$; explain why we don't want mutually inconsistent elements but we also allow for $\leq$-comparable elements.}

\begin{remark}[The substitution map] 
\label{rk:substitution:map:algebra}
  Clearly, a purely algebraic counterpart of  
the substitution map which was part of the definition of \textit{probabilistic event structures over a language} (see Definition \ref{def:proba-epist-event-struct}) cannot be given.
\end{remark}

\begin{remark}[The order $\leq_{\mb{A}}$ on the set $\rmPhi$] 
\label{rk:Phi:forest}
%\redbf{Apostolos: Can you give me your opininon on this remark, please?}\\
The classical and the intuitionistic setting are distinguished by the fact that states are pairwise incomparable in the classical setting and (non-trivially) ordered in the intuitionistic setting. Thus,  in probabilistic event structures over a language (see Definition \ref{def:proba-epist-event-struct}) it is enough to require the set $\rmPhi$ to contain mutually inconsistent formulas in order to tell apart  states of the Kripke model. However, due to the order between states of   intuitionistic Kripke frames, mutually incompatible formulas are not enough to separate distinct but comparable states. To overcome this hurdle we require $\rmPhi$ to  satisfy the following condition:
for all $a_k, a_j\in \rmPhi$,
\begin{center}
$a_j\wedge a_k = \bot \quad $ or $ \quad a_j < a_k \quad $ or $ \quad a_k < a_j$.
\end{center} 
This condition makes it possible to compute the probabilities of a given  non-maximal state, even if there is no proposition uniquely identifying this state (cf.~ \Cref{def:mua-mba}).
\end{remark}

%\begin{remark}[The order $\prec$ on the multiset $\rmPhi$]
%\label{rk:Phi:list}
%	Another technical requirement concerns the fact that $\bfPhi = (\rmPhi,\prec)$ is an ordered multiset and not a set any more. 
%	This requirement is intended to cater for situations where $\rmPhi$ arises from a set of formulas $\rmPsi$ and different formulas in $\rmPsi$ receive the same interpretation on the given model. In this case, the multiplicity of each member of $\rmPhi$ corresponds to how many formulas in $\rmPsi$ are mapped to that member. 
%	\redbf{Moreover the order on the various copies of the same member reflects the relation of logical entailment between the original formulas mapped to the same member.}
%\end{remark}

%In the next subsection, we introduce APE-structures based on  epistemic Heyting algebras. In Subsection \ref{ssec:intermediate ha} we introduce the first step of the two-step update, namely, the  `product' construction. In Subsection \ref{ssec: abstract charact i minimal els pseudo quotient}, we introduce the second and final step, the `quotient' construction. %and  Next, we introduce the two-step construction of the probabilistic event update: first moving to an intermediate structure, the product, and then taking a quotient of that structure to obtain the updated algebraic probabilistic epistemic structure.

\subsection{The intermediate (pre-)probabilistic epistemic structure}\label{ssec:intermediate ha}
\label{ssec:intermediate}
%\marginnote{this nees to become the new section 4, see comments there}
In the present subsection, we define the intermediate ApPE-structure $\prod_{\mb{E}}\mathcal{F}$ associated with  any APE-structure $\mathcal{F}$ and any event structure $\mb{E}$ over the support of $\mathcal{F}$ (see Definition \ref{def: alg probab epist structure} for the definition of support):

\begin{equation}
\label{eq:def intermediate model}
\prod_{\mb{E}}\mathcal{F} := \left( \prod_{ \mb{E}} \mb{A} , (\mu'_i)_{i\in \Ag} \right) .
\end{equation}
\paragraph{Structure of the subsection.}
First, we define the intermediate algebra 
$\prod_{\mid E \mid} \mb{A}$ which will become the support of the intermediate ApPE-structure $\prod_{\mb{E}}\mathcal{F}$ (see Definition \ref{def:support intermediate APE} 
and Proposition \ref{prop:intermediate-ha:epist-alg}) and we identify its $i$-minimal elements (see Proposition \ref{prop: charact i minimal}).
Then, we introduce the $i$-premeasures on the intermediate algebra 
(see Definition \ref{def: prod F over E} and Proposition \ref{prop:intermediate-ha:i-premeasure}). 
Finally, we show that the definition ApPE-structure is coherent with the relational semantics in the classical case (see Proposition \ref{prop: plus of coprod same as prod of plusses}).
%\marginred{to improve the text introducing the section.}

\subsubsection{The intermediate algebra and its $i$-minimal elements}
\label{sssec:intermediate-ha:inter-ha}
%Let us define the algebra which will become the support of the intermediate APE-structure above:
\begin{definition}[Intermediate algebra]
\label{def:support intermediate APE}
For every  epistemic Heyting algebra $\mb{A} = (\mb{L}, (\lozenge_i)_{i\in \Ag}, (\Box_i)_{i\in \Ag})$ and 
every event structure $\mb{E} = (E, (\sim_i)_{i\in\Ag}, (P_i)_{i\in \Ag}, \bfPhi, \overline{\pre})$ over $\mb{A}$, let {\em the intermediate algebra} be
\[\prod_{\mb{E}} \mb{A}: = (\prod_{\mid E\mid}\mb{L}, \{\lozenge'_i, \Box'_i\mid i\in \Ag\}),\]
where
\begin{enumerate}
\item $\prod_{\mid E\mid}\mb{L}$ is the $ |E|$-fold power of $\mb{L}$, the elements of which can be seen
 either as $| E |$-tuples of elements in $\mb{A}$, or as maps  $f: E \to \mb{A}$;

\item for any $f : E \rightarrow \mb{A} $, 
let us define $\lozenge'_i (f)$ as follows:
\begin{align*}
\lozenge'_i (f) : E  &\rightarrow \mb{A}\\
e &\mapsto \bigvee \{ \lozenge_i f(e') \mid e'\sim_i e \};
\end{align*}

\item for any $f : E \rightarrow \mb{A} $, %the map 
let us define $\Box'_i (f)$ as follows:
\begin{align*}
\Box'_i (f) : E  &\rightarrow \mb{A}\\
e & \mapsto \bigwedge \{ \Box_i f(e') \mid e'\sim_i e \}.
\end{align*}
\end{enumerate}
Below, the algebra $\prod_{\mb{E}} \mb{A}$ will be sometimes abbreviated as $\mb{A}'$.
%\marginredbf{Should we give the definition of the order, $\bot$ and $\top$ in $\mb{A}'$ to help the reader to understand the proof below.}
\end{definition}
We refer to \cite[Section 3.1]{KP13} for an extensive justification of the definition of the operations $\lozenge'_i$ and $\Box'_i$.
%\paragraph*{Problem:} How should we define the maps $(\mu'_i)_{i\in \Ag}$?

\begin{proposition} 
\label{prop:intermediate-ha:epist-alg}
For every  epistemic Heyting algebra $\mb{A}$ and every event structure $\mb{E}$ over $\mb{A}$, the algebra $\mb{A}'$ is an epistemic Heyting algebra.
\end{proposition}
%\bluebf{Adapt the proof with the new axioms for
% the completeness proof}

\begin{proof}
To prove that $\mb{A}'$ is an epistemic Heyting algebra (\autoref{def:epist-Heyting-algebra}), we need to show that $\mb{A}'$ is a monadic Heyting algebra such that  for every $i\in\Ag$ and every $f\in\mb{A}'$, we have: 
$\lozenge_if\lor\lnot\lozenge_if=\top.$
%\label{axiom:epist-alg:boolean} 
%\redbf{to check}

%\redbf{WARNING: we keep changing the definition of epistemic Heyting algebra!!!}

The proof that $\mb{A}'$ is a monadic Heyting algebra can be found in \cite[Proposition 8.1]{KP13}. Let $i\in\Ag$, $f\in\mb{A}'$, and $e\in E$. We have:
\begin{align*}
(\lozenge'_if\lor\lnot\lozenge'_if)(e) & = (\lozenge'_if)(e)\lor\lnot(\lozenge'_if)(e)\\
& =\bigvee\{\lozenge_i(f(e'))\mid e'\sim e\}\lor\lnot\bigvee\{\lozenge_i(f(e'))\mid e'\sim e\} \tag{by definition of $\lozenge'_i$}\\
& =\lozenge_i\bigvee\{f(e')\mid e'\sim e\}\lor\lnot\lozenge_i\bigvee\{f(e')\mid e'\sim e\} \tag{by the normality of $\lozenge_i$} \\
& = \top \tag{since $\lozenge_ia\lor\lnot\lozenge_ia=\top$}.
\end{align*}
Hence,  $(\lozenge'_if\lor\lnot\lozenge'_if)(e)=\top$ for all $e\in E$, which by definition yields that $\lozenge'_if\lor\lnot\lozenge'_if=\top$.

 \end{proof}

\begin{proposition}
\label{prop: charact i minimal}
For every epistemic Heyting algebra  $\mb{A}$ and every agent   $i \in \Ag$, 
$$\mathsf{Min}_i(\mb{A}') = \{f_{e, a}\mid e\in E \mbox{ and } a\in \mathsf{Min}_i(\mb{A})\},$$
%The $i$-minimal elements of $\mb{A}'$ are exactly the maps %in the set
%\begin{equation}
%\label{eq:i-minimal maps}
%M := \{ f_{e,a} \mid e\in E \tand a\in \mb{A} i \text{-minimal} \},
%\end{equation}
%The $i$-minimal elements of $\mb{A}'$ are exactly the maps
where for any $e\in E$ and  $a\in \mathsf{Min}_i(\mb{A})$, the map $f_{e,a}$ is defined as follows:
\begin{align*}
f_{e,a} : E &\rightarrow \mb{A}
\\
e'  &\mapsto  \left\{ \begin{array}{ll}
 a & \textrm{if $e' \sim_i e$}\\
 \bot & \textrm{otherwise.}
  \end{array} \right.
\end{align*}
\end{proposition}

\begin{proof}
See Appendix \ref{app:section4:2} page \pageref{app:section4:2}.
\end{proof}

\subsubsection{The $i$-premeasures on the intermediate algebra}
\label{sssec:intermediate-ha:i-premeasure.}

Before providing $i$-premeasures for the product epistemic algebra (Definition \ref{def: prod F over E} and Proposition \ref{prop:intermediate-ha:i-premeasure}), we present an auxiliary definition.

\begin{definition}
\label{def:mua-mba}
Let $\mathcal{F}=\left( \mb{A}, (\mu_i)_{i\in \Ag} \right) $ be an APE-structure 
and let $\mb{E} = (E, (\sim_i)_{i\in\Ag}, (P_i)_{i\in \Ag}, \bfPhi, \overline{\pre})$ 
be an event structure over $\mb{A}$. 
For all $a\in\rmPhi$ and $i\in\Ag$,
we define the partial function $\mu^a_i: \mb{A} \to\mb{R}^+$ by
\begin{align}
\label{eq:def:mua}
\mu_i^a(x):=\mu_i(x\land a)-\sum_{b\in\mathrm{mb}(a)}\mu_i(x\land b)
\end{align}
where  $\mathrm{mb}(a)$ denotes the multiset of the $\prec$-maximal elements of $\bfPhi$ $\prec$-below $a$.
\end{definition}

We make the following observations regarding $\mu^a_i$:

\begin{proposition}\label{prop: muai properties}
For every APE-structure  $\mathcal{F}=\left( \mb{A}, (\mu_i)_{i\in \Ag} \right) $ and every event structure $\mb{E}$ over $\mb{A}$, 
%it is the case that 
$\mu^a_i$ is an $i$-premeasure over $\mb{A}$. Furthermore,  if $a\leq y$ then $\mu^a_i(x)=\mu^a_i(x\land y)$.
\end{proposition}
%\commentPROOF{

\begin{proof}
See Appendix \ref{app:section4:3} page \pageref{app:section4:3}.
\end{proof}

\begin{remark}
Notice that if $a\leq y$, then for every $b\in\mathrm{mb}(a)$ we have $b\leq y$, thus $\mu_i(x\land y\land a)=\mu_i(x\land a)$ and $\mu_i(x\land y\land b)=\mu_i(x\land b)$, which implies that $\mu^a_i(x)=\mu^a_i(x\land y)$.
\end{remark}

\begin{definition}[Intermediate structure]
\label{def: prod F over E}
For any APE-structure $\mathcal{F}=\left( \mb{A}, (\mu_i)_{i\in \Ag} \right) $ and any event structure $\mb{E} = (E, (\sim_i)_{i\in\Ag}, (P_i)_{i\in \Ag}, \rmPhi, \overline{\pre})$  over $\mb{A}$, let
the {\em intermediate structure} be
\[
\prod_{\mb{E}}\mathcal{F} := \left( \prod_{\mb{E}} \mb{A} , (\mu'_i)_{i\in \Ag} \right)
\]
where
\begin{enumerate}
\item $\prod_{\mb{E}} \mb{A} = \mb{A}'$ is defined as in Definition \ref{def:support intermediate APE};
\item  each  $\mu'_i
%: \mb{A}'\to [0, 1]
$ is defined as follows:
\begin{align}
\mu'_i: \mathsf{Min}_i(\mb{A}'){\downarrow} 
& \to \mb{R}^+
\label{eq:def:mui'}
\\
f & \mapsto \sum_{e\in E}\sum_{a\in \rmPhi}P_i(e)\cdot\mu_i^a(f(e))\cdot \overline{\pre}(e\mid a).
\notag
\end{align}
\end{enumerate}
\end{definition}

\begin{proposition}
\label{prop:intermediate-ha:i-premeasure}
For every  APE-structure $\mathcal{F}$  and every event structure $\mb{E}$ over the support of $\mathcal{F}$, the intermediate structure $\prod_{ \mb{E}} \mathcal{F}$ is an ApPE-structure (see Definition \ref{def: alg probab epist structure}). % $\mu'_i$ defined as in \eqref{eq: def mu prime} verifies the properties of  Definition
Furthermore,  if $\bigvee_{a \in \rmPhi} a\leq y$ then $\mu'_i(x)=\mu'_i(x\land y)$.
\end{proposition}
%\bluebf{Adapt the proof with the new axioms for
 %the completeness proof}
%\commentPROOF{	
\begin{proof}
 Proposition \ref{prop:intermediate-ha:epist-alg} states that $\prod_{\mb{E}}\mb{A}$ is an epistemic Heyting algebra. %Boolean (resp.\ $\sigma$-)algebra, and it can be easily verified that  the operations $\lozenge_i$ are S5-type normal modal\color[rgb]{0,0,0} operators for each $i\in \Ag$
%is entirely analogous to the proof of \cite[Proposition 8]{KP13}, and is omitted. 
To prove that $\prod_{ \mb{E}} \mathcal{F}$ is an ApPE-structure, it remains to show that for every $i\in \Ag$, the map $\mu_i'$ is an $i$-premeasure (see items (\ref{def:epAlg:two:domain} - \ref{def:epAlg:two:bot}) of Definition \ref{def:measures}). Fix $i\in \Ag$.
The map $\mu_i'$ is clearly well-defined. Since the maps $\{ \mu_i^a \}_{a\in \rmPhi}$ are $i$-premeasures,  the items 
\ref{def:epAlg:two:domain}, \ref{def:epAlg:two:monotone}, and \ref{def:epAlg:two:bot}
are trivially true.

\medskip

\texttt{Proof of item \ref{def:epAlg:two:join}.}
%As to  condition \eqref{def:epAlg:two:fixedpoints}, let $g: E\to \mb{A}$ such that $\bot\neq g = \blacklozenge'_i g$, and let us show that $\mu'_i(g) = 1$. By \eqref{eq: def mu prime}, it is enough to show that \[\sum_{e\in E} P_i(e)\cdot \mu_i(g(e))\geq 1.\]
%
%Indeed, the assumption $g\neq \bot$ implies that $g(e^*)\neq \bot$ for some $e^*\in E$. Then, reasoning as above using \eqref{eq:two diamonds}, it is easy to see that $\blacklozenge_i(g(e^*)) = g(e^*)$. Hence,
%
%\begin{align*}
%\sum_{e\in E} P_i(e)\cdot \mu_i(g(e)) & \geq  \sum_{e\sim_i e^*} \mu_i(g(e)) \cdot P_i(e)  \\
%& = \mu_i(g(e^*))\cdot \sum_{e\sim_i e^*}  P_i(e) \tag{$g(e) = g(e^*)$ whenever $e\sim_i e^*$}\\
%& =  \mu_i(g(e^*)) \tag{definition of $P_i$}  \\
%& = 1. \tag{$\blacklozenge_ig(e^*) = g(e^*)\neq \bot$ and Definition  \ref{def: alg probab epist model}.\eqref{def:epAlg:two:fixedpoints}}\\
%\end{align*}
%As to condition \eqref{def:epAlg:two:join},
%
%\redbf{for every $a\in \mathsf{Min}_i(\mb{A})$ and all $b, c\in a{\downarrow}$ it holds that $\mu(b\vee c) = \mu(b)+ \mu(c) - \mu(b\wedge c)$;}
%
By Proposition \ref{prop: charact i minimal}, $i$-minimal elements of $\mb{A}'$ are of the form $f_{e, b}: E\to \mb{A}$  for some $e\in E$ and some $i$-minimal element $b\in \mathsf{Min}_i(\mb{A})$. Fix one such element $f_{e, b} \in \mathsf{Min}_i(\mb{A}')$, and
let $g, h: E\to \mb{A}$ such that $g, h\leq f_{e, b}$.
% and $g\wedge h = \bot$. %Since $\mu'_i$ is order-preserving, and condition 2(c) of Definition \ref{def: alg probab epist model} holds for $\mu'_i$, these assumptions imply that $\mu'_i(g\vee h)\leq \mu'_i(f_{e, a}) = 1$. Hence, \[\mu'_i(g\vee h) = \]
%Observe preliminarily that, by the monotonicity of $\mu_i$,
%\begin{equation}
%\label{eq: rid of min}
%\sum_{e'\sim_i e} P_i(e')\cdot \mu_i(g(e')\vee h(e'))\leq \sum_{e'\sim_i e} P_i(e')\cdot \mu_i(f_{e, a}(e')) = \sum_{e'\sim_i e} P_i(e')\cdot \mu_i(a) = 1.
%\end{equation}
By definition, $f \leq f_{e,b}$ can be rewritten as $f(e')\leq f_{e,b}(e')$
for any $e'\in E$.
Since $ f_{e,b}(e') = \bot $  for any $e' \nsim_i e$, we can deduce that $g(e') = h(e') = \bot$ for any $e'\nsim_i e$. Similarly, we can deduce that 
$g(e') \leq b$ and $ h(e') \leq b$ for any $e'\sim_i e$.
Hence,
\begin{align*}
%& \qquad
\mu'_i(g\vee h) 
&  = \sum_{e'\in E}\sum_{a\in \rmPhi}P_i(e')\cdot\mu^a_i(g(e')\lor h(e'))\cdot \overline{\pre}(e'\mid a)  \tag{by definition}
\\
%&  = \sum_{e'\sim_i e}\sum_{a\in \rmPhi}P_i(e')\cdot\mu^a_i(g(e')\lor h(e'))\cdot \pre(e'\mid a) \tag{$g(e') = h(e') = \bot$ for any $e'\nsim_i e$ and $\mu^a_i(\bot)=0$}\\
%&  =
%\sum_{e'\sim_i e}\sum_{a\in \rmPhi}P_i(e')\cdot(\mu^a_i(g(e'))+\mu^a_i(h(e'))-\mu^a_i(g(e')\land h(e')))\cdot \pre(e'\mid a)
%\tag{$\mu_i^a$ is an $i$-premeasure}
%\\
&  =
\sum_{e'\in E}\sum_{a\in \rmPhi}P_i(e')\cdot
\left( \mu^a_i(g(e'))+\mu^a_i(h(e'))-\mu^a_i(g(e')\land h(e')) \right) \cdot \overline{\pre}(e'\mid a)
\tag{$\mu_i^a$ is an $i$-premeasure, $b\in  \mathsf{Min}_i(\mb{A})$, 
 and $g(e')\leq b $ and $h(e') \leq b $ for any $e'\in E$}
\\
& = \mu'_i(g)  + \mu'_i(h) - \mu'_i(g\wedge h).
\tag{by definition}
\end{align*}

\medskip

Finally, the fact that  if $\left( \bigvee_{a \in \rmPhi} a  \right) \leq y$ then $\mu'_i(x)=\mu'_i(x\land y)$ follows from 
Proposition \ref{prop: muai properties}.
\end{proof}

%\begin{remark}
%\label{remark:setPhi}
%In the intuitionistic case the event structure 
%$$\mathcal{E} = (E, (\sim_i)_{i\in\Ag}, (P_i)_{i\in \Ag}, \rmPhi, \pre)$$ over $\mathcal{L}$, is such that $\rmPhi$ is a non-empty finite set of $\mathcal{L}$-formulas and for any $\varphi,\psi\in\rmPhi$, 
%\begin{itemize}
%\item $\varphi$ and $\psi$ are NOT equivalent, and
%\item either  $\varphi$ and $\psi$  contradict each other
% or one implies the other.
%\end{itemize}
%\end{remark}

\subsubsection{The intermediate algebra for the classical case}
\label{sssec:intermediate-ha:duality}
%\redfootnote{Is this title ok?}

Here, we show that the construction described above, applied to the complex algebras of classical models, dualizes the construction of the intermediate model of Section \ref{ssec:epist:update:classic}. This is the first step towards the result stated in Proposition \ref{prop:compatibility with duality}.

\begin{definition}\label{def: es over L induce es over A}

For any PES-model $\mb{M} = \left\langle S, (\sim_i)_{i\in \Ag}, (P_i)_{i\in \Ag}, \val{\cdot}\right\rangle$
(see Definition \ref{def:Prob epis state model Alexandru}) 
and any probabilistic event structure
$\mathcal{E} = (E, (\sim_i)_{i\in\Ag}, (P_i)_{i\in \Ag}, \rmPhi, \pre)$ over $\mathcal{L}$ (see Definition \ref{def:proba-epist-event-struct}), 
let  the \emph{probabilistic event structure over $\mb{M}^+$} (see Definitions \ref{def: complex algebra} and \ref{def:algebraic event}) be	
$$\mb{E}_{\mathcal{E}}: = (E, (\sim_i)_{i\in\Ag}, (P_i)_{i\in \Ag}, \bfPhi_\mb{M}, \overline{\pre}_\mb{M}),$$
where 
\begin{itemize}
\item 
$\bfPhi_\mb{M} = (\rmPhi_\mb{M},\prec_\mb{M})$ is the ordered multiset such that $\rmPhi_\mb{M} := \{\val{\phi}_\mb{M}\mid \phi\in \rmPhi\}$ and the strict order
$\prec_\mb{M}$ is the empty relation;
\item  the map $\overline{\pre}_\mb{M} : \rmPhi_\mb{M} \rightarrow ( E \rightarrow [0,1])$ assigns a probability distribution $\overline{\pre}_\mb{M}(\bullet | \val{\phi}) : E \rightarrow [0,1]$ over $E$ for every  $\phi \in \rmPhi$ such that: 
\begin{align}
\overline{\pre}_\mb{M}(\bullet | \val{\phi}) : E & \rightarrow [0,1]
\label{eq:def:over:pre}
\\
e & \mapsto \pre(e \mid \phi).
\notag
\end{align}
\end{itemize}
\end{definition}

\begin{proposition}
\label{fact: es over L induce es over A}
For any PES-model $\mb{M}$ (see Definition \ref{def:Prob epis state model Alexandru}) and any event structure $\mathcal{E}$ over $\mathcal{L}$  (see Definition \ref{def:proba-epist-event-struct}), the tuple $\mb{E}_{\mathcal{E}}$ is an event structure over the epistemic Heyting algebra underlying $\mb{M}^+$.
\end{proposition}
\begin{proof}
We need to verify that the tuple $\mb{E}_{\mathcal{E}}$ satisfies 
\autoref{def:algebraic event}. Items 1 to 3 are trivially satisfied. 
Hence, we only need to prove that 
\begin{itemize}
\item[4.] $\bfPhi_\mb{M} = (\rmPhi_\mb{M},\prec_\mb{M}) $ is a finite ordered multiset on $\mb{M}^+$ such that,  for all $a, b\in \rmPhi_\mb{M}$ which arise from distinct elements in $\mb{M}^+$, either
\begin{center}
$a\wedge_{\text{{$\mb{M}^+$}}} b = \bot \quad $ or $ \quad a <_{\text{$ \mb{M}^+$}} b \quad $ or $ \quad b <_{\text{{$ \mb{M}^+$}}} a$;
\end{center} 
\item[5.] the map $\overline{\pre}_\mb{M}$ $ :  E \times \rmPhi_\mb{M} \rightarrow [0,1]$  assigns a probability distribution $\overline{\pre}_\mb{M}(\bullet | a)$ over $E$ for every  $a \in \rmPhi_\mb{M}$;
\item[6.] for all $a\in\rmPhi$ and  $e\in E$, if $\overline{\pre}_\mb{M}(e | a)=0$ then $\overline{\pre}_\mb{M}(e | b)=0$  for all  $b\in\rmPhi$ such that $a\prec b$.
%\redfootnote{for all $a\in\rmPhi$ and  $e\in E$, if $\pre(e | a)=0$ then $\pre(e | b)=0$  for all  $b\in\rmPhi$ such that $a<_{\mb{A}} b$  or ( $a =_{\mb{A}} b$ and $a \prec b$).}
\end{itemize}	
%\redbf{$\bf{\Phi}$	for the list $(\rmPhi,\preceq)$ ?}
\texttt{Proof of 4.} 
First, we need to prove that $\bfPhi_\mb{M}$ is an ordered multiset (\autoref{def:ordered:multiset}). $\rmPhi_\mb{M}$ is clearly a multiset, hence we only need to prove that the empty relation $\prec_\mb{M}$ satisfies the following conditions: for all pairwise distinct elements $x,y,z\in \rmPhi_\mb{M}$,
\begin{enumerate}[(i)]
\item %\label{def:it:ord:multi:order} 
if $x \prec_\mb{M} y$, then $x \leq_{\mb{M}^+} y$;
\item %\label{def:it:ord:multi:not:bot} 
 if $x\neq \bot_{\mb{M}^+}$ and 
$ x \leq_{\mb{M}^+} y$, then  $x \prec_\mb{M} y$ or $y \prec_\mb{M} x$;
\item %\label{def:it:ord:multi:tree}  
if $x\prec_\mb{M} y$ and $x\prec_\mb{M} z$, then 
$y \prec_\mb{M} z$ or $z \prec_\mb{M} y$.
\end{enumerate}
Conditions (i) and (iii) are trivially satisfied.
Notice that, since $\mathcal{E}$ is a (classical) probabilistic event structure, $\rmPhi$ is a finite set of pairwise inconsistent $\mathcal{L}$-formulas. Assume that $\val{\phi},\val{\psi}\in \rmPhi_{\mb{M}}$ are pairwise distinct (i.e.\ $\phi \neq \psi$ in the language $\mathcal{L}$) and such that $\val{\phi} \leq_{\mb{M}^+} \val{\psi}$.
One can easily verify that  $\phi \wedge \psi = \bot$ implies that 
$\val{\phi}= \bot_{\mb{M}^+}$. Hence, $\prec_\mb{M}$ satisfies condition (ii).
This finishes the proof that the ordered multiset $\bfPhi_{\mb{M}}$ is well-defined.

\smallskip

Let $\val{\phi}, \val{\psi}\in \rmPhi_\mb{M}$  arise from distinct elements in $\mb{M}^+$. By definition, $\phi \wedge \psi = \bot$. Hence, 
$a\wedge_{{\mb{M}^+}} b = \bot $, which proves item 4.

\bigskip

\texttt{Proof of 5.} Since $\mathcal{E}$ is a (classical) probabilistic event structure, 
$\pre$ assigns a probability distribution $\pre(\bullet | \phi)$ over $E$ for every  $\phi \in \rmPhi$. Hence, the map $\overline{\pre}_\mb{M}$ is well-defined.

\bigskip

\texttt{Proof of 6.} Since $\prec_{\mb{M}}$ is the empty relation, this condition is trivially true.
\end{proof}

\begin{remark} 
\label{rk:mb:classic:empty}
Notice that, in the classical case, 
$\mathrm{mb}(a) = \varnothing$ for all $a\in \bfPhi$. 
Indeed, $\mathrm{mb}(a)$ denotes the multiset of the $\prec$-maximal elements of $\bfPhi$ $\prec$-below $a$. But, since in the classical case $\prec_{\mb{M}}$ is the empty relation, there is no element below $a$ in $\bfPhi$.
\end{remark}

\begin{proposition}
%\marginnote{this proof is not correct; in particular if we insist that the co-product is a probabilistic model then this result shouldn't hold}
\label{prop: plus of coprod same as prod of plusses}
For every PES-model $\mb{M}$ and any event structure $\mathcal{E}$ over $\mathcal{L}$, \[(\coprod_{\mathcal{E}}\mb{M})^+\cong \prod_{\mb{E}_{\mathcal{E}}}\mb{M}^+.\]
\end{proposition}

\begin{proof}
See Appendix \ref{app:section4:4} page \pageref{app:section4:4}.
\end{proof}

\begin{corollary}
\label{cor:complex:alg:of:prod}
For every PES-model $\mb{M}$ and any event structure $\mathcal{E}$ over $\mathcal{L}$, the complex algebra $(\coprod_{\mathcal{E}}\mb{M})^+$ of 
the intermediate structure $\coprod_{\mathcal{E}}\mb{M}$ is an ApPE-structure.
\end{corollary}

%\section{Probabilistic epistemic updates on algebras}
%\marginnote{this needs to become the new section 5, see comments there}

\subsection{The pseudo-quotient and the updated APE-structure}
\label{ssec: abstract charact i minimal els pseudo quotient}

In the present subsection, we define the APE-structure $\mathcal{F}^{\mb{E}}$, 
resulting from the update of the APE-structure $\mathcal{F}$ with  the event structure $\mb{E}$ over the support of $\mathcal{F}$, 
by taking a suitable  pseudo-quotient of the intermediate APE-structure $\prod_{\mb{E}}\mathcal{F}$. 
Some of the results which are relevant for the ensuing treatment (such as the characterization of the $i$-minimal elements in the pseudo-quotient) 
are independent of the fact that we will be  working with the intermediate algebra. 
Therefore, in what follows, we will discuss them in the more general setting of arbitrary epistemic Heyting algebras $\mb{A}$.

\paragraph{Structure of the subsection.}
First, we define the pseudo-quotient algebra (Definition \ref{def:pseudo quotient general}) and prove that it is 
an epistemic Heyting algebra (Proposition \ref{prop:pseudoisep}).
Then, we characterize the $i$-minimal elements of 
the pseudo-quotient algebra (Proposition \ref{prop: charact i minimal elements of pseudo quotient gen}).
Finally,  we define the APE-structure $\mathcal{F}^{\mb{E}}$, 
resulting from the update of the APE-structure $\mathcal{F}$ with  the event structure $\mb{E}$
(Definition \ref{def: updated APE structure F of E} and Proposition \ref{prop:APEstruct-EoverF}) and show that this definition is compatible with the update on PES-models (Lemma \ref{lem:duality:probabilities}).

\subsubsection*{Pseudo-quotient algebra.}

\begin{definition}[Pseudo-quotient algebra]
\label{def:pseudo quotient general} (cf.\ \cite[Sections 3.2, 3.3]{MPS14})
 For any epistemic Heyting algebra 
$\mb{A} := (\mb{L}, (\lozenge_i)_{i\in \Ag}, (\Box_i)_{i\in \Ag}),$ 
 and any $a\in \mb{A}$, let the {\em pseudo-quotient algebra} be
 $$\mb{A}^a := (\mb{L}/{\cong_a}, (\lozenge_i^a)_{i\in \Ag}, (\Box_i^a)_{i\in \Ag}),$$ 
where 
\begin{itemize}
\item $\cong_a$ is defined as follows:  
for all $b, c\in \mb{L},$
$$b\cong_a c \quad \text{iff} \quad b\wedge a = c\wedge a,$$
\item for every $i\in\Ag$ the operations $\lozenge_i^a$
and $\Box_i^a$ are defined as follows:
\begin{align*}
\lozenge_i^a: \mb{A}^a & \to \mb{A}^a &
\text{and} &&
\Box_i^a : \mb{A}^a & \to \mb{A}^a
\\
b & \mapsto [\lozenge_i(b\wedge a)]
&&& b & \mapsto [\Box_i(a\to b)],
\end{align*}
where  $[c]$ denotes  the $\cong_a$-equivalence class of $c\in \mb{A}$.
%\item each operation $\lozenge_i^a$ is defined by the assignment $\lozenge_i^a[b]: = [\lozenge_i(b\wedge a)]$ and
%\item each operation $\Box_i^a$ is defined by the assignment $\Box_i^a[b]: =[\Box_i(a\to b)]$,
\end{itemize}
 
\end{definition}
%I  (cf.\ \cite[Lemma ???]{MPS14}) that $\cong_a$ is a congruence on $\mb{B}$, and that if $\mb{A}$ is an epistemic algebra then so is $\mb{A}^a$.
\begin{proposition}\label{prop:pseudoisep} (cf.\ \cite[Fact 12]{MPS14}) 
For any epistemic Heyting algebra $\mb{A}$,
the  pseudo-quotient algebra $\mb{A}^a$ (see Definition \ref{def:pseudo quotient general}) is an epistemic Heyting algebra.
\end{proposition}
%\bluebf{Adapt the proof with the new axioms for
% the completeness proof}
\commentPROOF{
\begin{proof}
%\redbf{to check}
%\redbf{WARNING: we keep changing the definition of epistemic Heyting algebra!!!}
The proof that $\mb{A}^a$ is a monadic Heyting algebra can be found in \cite[Fact 12]{MPS14}. To show that $\mb{A}^a$ is an epistemic Heyting algebra (see \autoref{def:epist-Heyting-algebra}), it remains to prove  that $\lozenge^a_i [b]\lor\lnot\lozenge^a_i [b]=[\top]$ for all $i\in\Ag$ and $b\in \mb{A}^a$. We have that $\lozenge^a_i[b]=[\lozenge_i(b\land a)]$ and that $\lnot\lozenge^a_i [b]=\lnot[\lozenge_i(b\land a)]=[\lnot\lozenge_i(b\land a)]$.  Hence, 
$$\lozenge^a_i [b]\lor\lnot\lozenge^a_i [b]=[\lozenge_i(b\land a)\lor\lnot\lozenge_i(b\land a)]=[\top],$$ 
since $\mb{A}$ is an epistemic Heyting algebra.
%\begin{align*}
%\lozenge^a_i\lozenge^a_i [b] &= \lozenge^a_i[\lozenge_i(b\land a)] \tag{by definition}\\
% & = [\lozenge_i(\lozenge_i(b\land a)\land a)] \tag{by definition}\\
%& \leq[\lozenge_i(\lozenge_i(b\land a))] \tag{by the monotonicity of $\lozenge_i$}\\
%& \leq[\lozenge_i(b\land a)] \tag{since $\lozenge_i\lozenge_i b\leq\lozenge_i b$}\\
%& = \lozenge^a_i[b]. 
%\end{align*}
%We also have:
%\begin{align*}
%&  \lozenge_i^a[b]\land[c]\leq\lozenge_i^a([b]\land\blacklozenge_i^a[c])
%\\  
%\Leftrightarrow\quad & [\lozenge_i(b\land a)]\land [c]\leq\lozenge_i^a([b]\land[\blacklozenge_i (c\land a)])\\
%\Leftrightarrow\quad & [\lozenge_i(b\land a)\land c]\leq
%\lozenge_i^a([b \land\blacklozenge_i (c\land a)])\\
%\Leftrightarrow\quad & [\lozenge_i(b\land a)\land c]\leq[\lozenge_i(b\land\blacklozenge_i(c\land a)\land a)]\\
%\Leftrightarrow\quad & \lozenge_i(b\land a)\land c\land a \leq\lozenge_i(b\land\blacklozenge_i(c\land a)\land a)\land a \\
%\Leftrightarrow\quad & \lozenge_i(b\land a)\land (c\land a)\land a \leq\lozenge_i((b\land a)\land\blacklozenge_i(c\land a))\land a
%\end{align*}
%The last inequality is true since $\mb{A}$ is an epistemic algebra.
\end{proof}
}

\subsubsection*{The $i$-minimal elements of the pseudo-quotient algebra.}

\begin{lemma}
\label{lem:epistHA:i-minimal-1}
For any epistemic Heyting algebra $\mb{A}$ and any $a\in \mb{A}$, if $b\in \mathsf{Min}_i(\mb{A})$ and $b\wedge a\neq \bot$, then $[b]\in \mathsf{Min}_i(\mb{A}^a)$.
\end{lemma}

\begin{proof}%\redbf{to check}
Fix some $b \in \mathsf{Min}_i(\mb{A})$ such that 
$b\wedge a\neq \bot$. We need to prove that $[b] \in \mb{A}^a$ satisfies items \ref{def:i-minimal:item:bot}, \ref{def:i-minimal:item:fixpoint}, and \ref{def:i-minimal:item:minimal} of 
Definition \ref{def:i-minimal}.\\

\texttt{Proof of item \ref{def:i-minimal:item:bot}.}
By assumption, $[b]\neq \bot$, hence $[b]$ satisfies item \ref{def:i-minimal:item:bot}.\\

\texttt{Proof of item \ref{def:i-minimal:item:fixpoint}.}
To show that $\blacklozenge_i^a[b] = [b]$, it is enough to show that $\blacklozenge_i(b\wedge a)\wedge a = b\wedge a.$
Clearly, $b\wedge a\leq b$ implies that $\blacklozenge_i(b\wedge a)\wedge a\leq \blacklozenge_i b \wedge a = b\wedge a$, making use that $\blacklozenge_i b = b$. Conversely, recalling that $\blacklozenge_i$ is reflexive (Definition \ref{def: epist algebra}, axiom \eqref{axiom:epist-alg:refl}), we have $b\wedge a = (b\wedge a)\wedge a\leq \blacklozenge_i(b\wedge a)\wedge a$. 
Hence, $\blacklozenge_i^a[b] = [b]$.\\

\texttt{Proof of item \ref{def:i-minimal:item:minimal}.}
We need to prove that  
$[b]$ is a minimal fixed point of $\lozenge_i^a$.
Let $[\bot]\neq [c]\leq [b]$ such that $\blacklozenge_i^a[c] = [c]$, and let us show that $[c] = [b]$. It is enough to show that $c\wedge a  = b\wedge a$.  The assumption that $[c]\leq [b]$ implies that $c\wedge a\leq b\wedge a\leq b$. 
Hence, $\blacklozenge_i(c\wedge a)\leq \blacklozenge_i b = b$. Notice that the assumption that $\blacklozenge_i$ is transitive (Definition \ref{def: epist algebra}, axiom \eqref{axiom:epist-alg:trans}) implies that $\blacklozenge_i\blacklozenge_i(c\wedge a) = \blacklozenge_i(c\wedge a)$, that is $\blacklozenge_i(c\wedge a)$ is a fixed point of $\blacklozenge_i$. 
Moreover, $\bot\neq c\wedge a\leq\blacklozenge_i(c\wedge a)$ implies that $\blacklozenge_i(c\wedge a)\neq \bot$. 
Hence, by the $i$-minimality of $b$ in $\mb{A}$, we conclude that $\blacklozenge_i(c\wedge a) = b$, and hence $\blacklozenge_i(c\wedge a)\wedge a = b\wedge a$. Moreover, the assumption that $\blacklozenge_i^a[c] = [c]$ implies  that $\blacklozenge_i(c\wedge a)\wedge a = c\wedge a.$ Thus, the following chain of identities holds
$c\wedge a = \blacklozenge_i(c\wedge a)\wedge a = b\wedge a$ as required.
\end{proof}

\begin{lemma}
\label{lem:epistHA:i-minimal-2}
For any  epistemic Heyting algebra  $\mb{A}$ and any $a\in \mb{A}$, if $[b]\in \mathsf{Min}_i(\mb{A}^a)$, then  $\blacklozenge_i(b\wedge a)$ is the unique $i$-minimal element of $\mb{A}$ which belongs to $[b]$.
\end{lemma}

\begin{proof}
See Appendix \ref{app:section4:4.2} page \pageref{app:section4:4.2}.
\end{proof}

\noindent Combining the two lemmas above, we obtain the following result.

\begin{proposition}
\label{prop: charact i minimal elements of pseudo quotient gen}
The following are equivalent for any $\mb{A}$ and any $a\in \mb{A}$:
\begin{enumerate}
\item $[b]\in \mathsf{Min}_i(\mb{A}^a)$;
\item $[b] = [b']$ for a unique  $b'\in \mathsf{Min}_i(\mb{A})$ such that $b'\wedge a\neq \bot$.
\end{enumerate}
\end{proposition}
\begin{notation}
\label{notation:i:min:Aa}
In what follows, whenever $[b]\in \mathsf{Min}_i(\mb{A}^a)$, we will assume w.l.o.g.\ that $b\in \mathsf{Min}_i(\mb{A})$ is the ``canonical'' (in the sense of Proposition \ref{prop: charact i minimal elements of pseudo quotient gen})  representant of $[b]$.
\end{notation}

\subsubsection*{The updated APE-structure.}

 For any APE-structure $\mathcal{F}$ and any event structure $\mb{E}$ over the support $\mb{A}$ of $\mathcal{F}$,  the map $\overline{\pre}$ in $\mb{E}$ induces the map $\overline{pre}$ defined as follows:
\begin{align}
\overline{pre}: E & \to \mb{A} \notag\\
e & \mapsto \bigvee_{
	\begin{smallmatrix}
		a \in \rmPhi \\
		 \overline{\pre}(e\mid a)\neq 0
	\end{smallmatrix}
}
a
\label{def:overlinePRE}
\end{align} 

It immediately follows from Propositions \ref{prop: charact i minimal} and \ref{prop: charact i minimal elements of pseudo quotient gen} that the $i$-minimal elements of $\mb{A}^{\mb{E}}$ are exactly the elements $[f_{e, b}]$ for $e\in E$ and $b\in \mathsf{Min}_i(\mb{A})$ such that $b\wedge \overline{pre}(e')\neq \bot$ for some $e'\sim_i e$.

\begin{definition}[Updated APE-structure]
\label{def: updated APE structure F of E}
For any APE-structure $\mathcal{F}$ and any  event structure $\mb{E}$ over the support of $\mathcal{F}$, the {\em updated APE-structure}  is the tuple
$$\mathcal{F}^{\mb{E}}: = (\mb{A}^{\mb{E}}, (\mu_i^{\mb{E}})_{i\in \Ag})$$
such that:
\begin{enumerate}
\item $\mb{A}^{\mb{E}}$ is obtained by instantiating Definition  \ref{def:pseudo quotient general} to $\prod_{\mb{E}}\mb{A}$ and $\overline{pre}\in \prod_{\mb{E}}\mb{A}$, i.e.\
$$\mb{A}^{\mb{E}}: = (\prod_{\mb{E}}\mb{A})^{\overline{pre}};$$
\item 
The maps $\mu_i^\mb{E}$ are defined as follows:
\begin{align*}
\mu_i^\mb{E} : 
\mathsf{Min}_i(\mb{A}^\mb{E}){\downarrow}
& \to [0,1] \\
[g] & \mapsto
\left\{ \begin{array}{ll}
0 & \text{if } [g]=\bot,\\
\frac{\mu'_i(g)}{\mu'_i(f)} \quad\quad & \text{otherwise,}
\end{array} 
 \right. 
\end{align*}
where [f] is the only element in $\mathsf{Min}_i(\mb{A}^\mb{E})$ such that $[g]\leq[f]$.\footnote{See Definition \ref{def: prod F over E} for the definition of the maps $(\mu_i')_{i\in \Ag}$.}
\end{enumerate}
\end{definition}

\begin{lemma}\label{lemma:pseudomuwelldefined} For any APE-structure $\mathcal{F}$ and any  event structure $\mb{E}$ over the support of $\mathcal{F}$, 
the maps $(\mu_i^\mb{E} )_{i\in\Ag}$ of
the updated APE-structure
$\mathcal{F}^{\mb{E}}: = (\mb{A}^{\mb{E}}, (\mu_i^{\mb{E}})_{i\in \Ag})$ are well-defined.
\end{lemma}
\begin{proof}
Let us first prove the following claim.
\begin{claim}
For each $[h]\in\mathsf{Min}_i(\mb{A}^\mb{E}){\downarrow}$ such that $[h]\neq\bot$, we have $\mu_i'([h]) \neq 0 $. 
\end{claim}
\begin{claimproof}
Let $e\in E$ be such that $(h\land \overline{pre})(e)\neq\bot$. 
Notice that 
\begin{align*}
& (h\land \overline{pre})(e)\neq\bot \quad \quad
\text{iff} \quad \quad  h(e) \wedge 
 \bigvee_{
	\begin{smallmatrix}
		a \in \rmPhi \\
		 \overline{\pre}(e\mid a)\neq 0
	\end{smallmatrix}} a \quad \neq \quad \bot.
\end{align*}
This implies that there is $a \in \bfPhi$ such that 
$$\overline{\pre}(e \mid a)   >  0 \quad \quad \text{ and } \quad \quad h(e) \wedge a \neq \bot.$$
Since $\mu_i$ is an $i$-measure (see \autoref{def:measures}), 
we have $\mu_i((h\land a)(e))>0$. Then, the following set is non-empty
$$\{a\in\rmPhi  \ \mid \  \mu_i((h\land a)(e))>0\textrm{ and }\overline{\pre}(e | a)>0\}.$$ 

Since $\bfPhi$ is finite, it is well-founded with respect to the order of the multiset $\prec$, hence it contains at least one minimal element. Let $a_0$ be such a minimal element. 
From item \eqref{def:alg:event:item:6} of Definition \ref{def:algebraic event}, we deduce that,  for every $b\in\rmPhi$ such that  $b \prec a_0$, it is the case that $\overline{\pre}(e | b)>0$. The minimality of $a_0$ implies that, for every $b\in \rmPhi$  such that $b\prec a_0$,   we have
$\mu_i((h\land b)(e))=0$. This implies that, for all $b \in \mathrm{mb}(a)$, we have $\mu_i((h\land b)(e))=0$.
Hence,  
\begin{align*}
\mu_i^{a_0}(h(e)) & =\mu_i( g(e) \land a_0)-\sum_{b\in\mathrm{mb}(a_0)}\mu_i(h(e)\land b) 
\tag{see \autoref{def:mua-mba}}
\\
& =\mu_i( h(e) \land a_0).
\end{align*}

Therefore $\mu_i^{a_0}(h(e))>0$ and $P_i(e)\cdot\mu_i^{a_0}(h(e))\cdot\overline{\pre}(e | a_0)>0$. This guarantees that $\mu_i'([h]) \neq 0 $. 
This finishes the proof of the claim.
\end{claimproof}
\bigskip

Now, let us prove that the map $\mu_i^\mb{E}$ is well-defined.
Recall that, if $[g]\neq\bot$, then $[f]$ is unique (see 
Remark \ref{page:uniqueimin}).  
From the claim above, it follows that  the division $\frac{\mu'_i(g)}{\mu'_i(f)}$ is defined.
Finally, let us verify that
$\mu_i^\mb{E}$ assigns exactly one value to every $[g] \in \mathsf{Min}_i(\mb{A}^\mb{E}){\downarrow}$. Let $g_1,g_2 \in [g]$. Then we have
$\mu_i'(g_1) = \mu_i'(g_1\wedge \overline{pre}) =
\mu_i'(g_2\wedge \overline{pre}) = \mu_i'(g_2) $ (see Proposition \ref{prop:intermediate-ha:i-premeasure}). 
Since $\mu_i'$ is order-preserving, strictly positive for $[f] \neq \bot$  and $\mu_i^\mb{E}([g])=\frac{\mu'_i(g)}{\mu'_i(f)}$ with $0 < [g] \leq [f]$, we have that the division $\frac{\mu'_i(g)}{\mu'_i(f)}$ is defined and $\frac{\mu'_i(g)}{\mu'_i(f)}\leq 1$.
Hence, $\mu_i^\mb{E}$ is well-defined for any $i\in\Ag$.
%If $[g]=\bot$ then $\mu'(g)=0$. Hence the above is well-defined.
\end{proof}

\begin{proposition} 
\label{prop:APEstruct-EoverF}
For any APE-structure $\mathcal{F}$ and any event structure $\mb{E}$ over the support of $\mathcal{F}$, the tuple
$\mathcal{F}^{\mb{E}}= (\mb{A}^{\mb{E}}, (\mu_i^{\mb{E}})_{i\in \Ag})$ is an APE-structure.
\end{proposition}

\begin{proof}
See Appendix \ref{app:section4:5} page \pageref{app:section4:5}.
\end{proof}

\subsubsection*{The updated algebra for the classical case.}

In this section, we conclude the proof of Proposition \ref{prop:compatibility with duality} by showing that the pseudo-quotient construction described above, applied to the complex algebras of the intermediate classical models, dualizes the submodel construction in Section \ref{ssec:epist:update:classic}.\\

The definition of the complex algebra of a PES-model (Definition \ref{def: complex algebra}) can be equivalently reformulated as follows.
%\redbf{copy paste of \autoref{def: complex algebra}}
\begin{definition}[Complex algebra]
	\label{def:Intuitionist complex algebra}
	For any PES-model $\mb{M} = \left\langle S, (\sim_i)_{i\in \Ag}, (P_i)_{i\in \Ag}, \val{\cdot}\right\rangle$, its  complex algebra  is the tuple
	$$\mb{M}^+ :=  \left(  \P S, (\lozenge_i)_{i\in \Ag}, (\Box_i)_{i\in \Ag}, (P^+_i)_{i\in \Ag} \right) $$
where 
\begin{enumerate}
\item 
for each $i\in \Ag$ and $X\in \P S$,
\begin{align*}  
\lozenge_i X  & = \{ s\in S \mid \exists x \: ( s\sim_i x \text{ and } x\in X) \},
\\  
\Box_i X  & =   \{ s\in S \mid \forall x \: (s \sim_i x \implies x\in X)\},
\end{align*}
\item $\mb{A}:=\langle  \P S, (\lozenge_i)_{i\in \Ag}, (\Box_i)_{i\in \Ag} \rangle  $ is an epistemic Heyting algebra,
\item for each $i\in \Ag$ and $X\in \P S$,
\begin{align*}
P^+_i : \mathsf{Min}_i(\mb{A}){\downarrow} & \to \mb{A}
\\	
X  & \mapsto \sum_{x\in X} P_i(x).
\end{align*}
\end{enumerate}
Notice that the domain of $P_i^+$ consists of all the subsets of the equivalence classes of $\sim_i$.
%	where for each $i\in \Ag$ and $X\in \P S$, \begin{center}  
%	$\lozenge_i X  = \{ s\in S \mid \exists x( s\sim_i x$ and $x\in X) \}$,\\  $\Box_i X  =   \{ s\in S \mid \forall x(s \sim_i x\implies x\in X)\}$,
%\\
%	$\mathsf{dom}(P^+_i)=\{X\in\P S \mid \exists y\forall x (x\in X\implies x\sim_i y) \}$,\footnote{i.e.\ the domain of $P_i^+$ consists of all the subsets of the equivalence classes of $\sim_i$.}
%\\	
%$P^+_i X =\sum_{x\in X} P_i(x)$.
%	\end{center}
	%\begin{enumerate}
		%\item
%\begin{align*}
%P^+_i X & =  \left\{ \begin{array}{ll}
%\sum_{x\in X} P_i(x) & \textrm{if $\exists y\forall x (x\in X\implies x\sim_i y)$}\\
%1 & \textrm{otherwise}
%\end{array} \right.
%\end{align*}
		
	%\end{enumerate}
	%Notice that since in the present setting $\sim_i$ is an equivalence relation, $\lozenge_i$ and $\blacklozenge_i$ coincide.
\end{definition}

\begin{lemma}
\label{lem:duality:probabilities}
For any PES-model $\mb{M}$ and any event structure $\mathcal{E}$ over $\mathcal{L}$, $$(P_i^+)^{\mb{E}_\mathcal{E}} \cong (P_i^\mathcal{E})^+.$$
\end{lemma}
%\commentPROOF{
%\begin{proof}\redbf{to fix this proof}
%By  Lemma \ref{lem:prelim:PESmodel} and Proposition \ref{prop: complex algebras are APE structures}, we know that
%$(\mb{M}^\mb{E})^+$ is an APE-structure (cf.\ Definition 
%\ref{def: alg probab epist structure})
%and 
%$P^+_i(X)=\sum_{x\in X}P_i^\mb{E}(x)$ \redbf{WHY???}.
%For any $X \in $
% Furthermore, we have that for every $(s,e),(t,f)\in X$,  $(t,f) \sim_i (s,e)$  if and only if $t\sim_i s$ and $f\sim_i e$. So we have:
%
%\begin{align*}
%P_i^+(X) & =\frac{\sum_{(t,f)\in X}P_i(t)\cdot P_i(f)\cdot\pre( f\mid t)}{\sum\{P_i(s')\cdot P_i(e')\cdot \pre(e'\mid s')\mid (s',e')\sim_i(s,e)\}}\\
%&\\
%& =\frac{\sum_{f\sim_i e}\sum_{t\in X_f}P_i(t)\cdot P_i(f)\cdot \pre( f\mid t)}{\sum_{e'\sim_ie}\sum_{s'\sim_i s}P_i(s')\cdot P_i(e')\cdot \pre(e'\mid s')}\\
%&\\
%& = \frac{\sum_{f\sim_i e}\sum_{\varphi\in\rmPhi}P_i(f)\cdot \pre( f\mid \varphi)\cdot(\sum_{t\in X_f\cap\val{\varphi}}P_i(t))}{\sum_{e'\sim_ie}\sum_{\varphi\in\rmPhi}P_i(e')\cdot \pre(e'\mid \varphi)\cdot(\sum_{s'\sim_i s, s'\in\val{\varphi}}P_i(s'))}\\
%&\\
%& = \frac{\sum_{f\sim_i e}\sum_{\varphi\in\rmPhi}P_i(f)\cdot \pre( f\mid \varphi)\cdot\mu_i(X_f\cap\val{\varphi})}{\sum_{e'\sim_ie}\sum_{\varphi\in\rmPhi}P_i(e')\cdot \pre(e'\mid \varphi)\cdot\mu_i([s]\cap\val{\varphi})}\\
%\end{align*}
%\end{proof}
%}

\begin{proof}
See Appendix \ref{app:section4:6} page \pageref{app:section4:6}.
\end{proof}

\section{Intuitionistic PDEL}
\label{sec:semantics-IPDEL}

In this section, we introduce the Intuitionistic Probabilistic Dynamic Epistemic Logic (IPDEL). We define its syntax in Section \ref{ssec:syntaxe:IPDEL}, and its algebraic semantics (\autoref{def:semantics:IPDEL}) in Section \ref{ssec:IPDEL:semantics}. Then,   in Section \ref{ssec:IPDEL:axiomatisation}, we introduce the axiomatisation of IPDEL (\autoref{table:IPDEL}) and state its soundness and completeness. For the  proofs, see Appendices \ref{Appendix:soundness} and  \ref{Appendix:completeness}. %without proofs, the axiomatization obtained by using the   for the following interpretation of the PDEL-formulas (cf.\ Definition \ref{}) on  constructiochecking for sound interaction axioms

\subsection{The language  of IPDEL}
\label{ssec:syntaxe:IPDEL} 
\begin{definition}[IPDEL language] \label{def:IPDEL:syntax}
The set $\mathcal{L}$ of  {\em IPDEL-formulas} $\varphi$ and the class %$\mathsf{IPEM}_{\mathcal{L}}$ 
of {\em intuitionistic probabilistic event structures} $\mathbb{E}$ over $\mathcal{L}$  are built by simultaneous recursion  as follows:
	\[\varphi ::= p %\in \Prop
	\mid \bot \mid \varphi \wedge \varphi \mid \varphi\vee \varphi \mid \varphi\rightarrow  \varphi \mid \lozenge_i\varphi \mid \Box_i \varphi \mid \langle\mathcal{E}, e\rangle \varphi \mid [\mathcal{E}, e] \varphi \mid (\sum_{k = 1}^n\alpha_k · \mu_i(\varphi)) \geq \beta,\]
	where $i\in\Ag$, and, following \cite{FH94}, we let $\alpha_k, \beta$ be rational numbers,  and the event structures $\mathcal{E}$ are as in Definition \ref{def:intuitionistic-proba-epist-event-struct}.\\
	The connectives $\top$, $\neg$, and $\leftrightarrow$ are defined by the usual abbreviations. 
\end{definition}

\begin{definition}[Intuitionistic probabilistic event structure]	
\label{def:intuitionistic-proba-epist-event-struct}   
An \emph{intuitionistic probabilistic event structure  over} $\mathcal{L}$ is a tuple %\marginnote{Ale: I erased the substitution}
$$\mathcal{E} = (E, (\sim_i)_{i\in\Ag}, (P_i)_{i\in \Ag}, \rmPhi, \pre, \sub),$$
such that
\begin{itemize}
\item $E$ is a non-empty finite set;
\item each $\sim_i$ is an equivalence relation on $E$; 
%\end{itemize}
  
%\begin{enumerate}
%\item $E$ is a finite set of possible events;
%\item
 %interpreted as agent $i$'s epistemic indistinguishability between possible events, capturing $i$'s hard information about the event that is currently happening;
\item
each $P_i:E\to\ ]0,1]$ assigns a probability distribution over each $\sim_i$-equivalence class, i.e.\   $$\sum \left\{P_i (e' ) \mid e' \sim_i e \right\} = 1;$$ %This captures some new, independent subjective probabilistic information gained by the agent during the event: when observing the current event (without using any prior information), agent $i$ assigns probability $P_i(e)$ to the possibility that in fact $e$ is the actual event that is currently occurring.
\item
$\rmPhi $ is a finite set of formulas in $\mathcal{L}$  such that,  for all $\phi_k, \phi_j\in \rmPhi$, one and only one of the following conditions is true: 
\begin{itemize}
\item $\vdash (\phi_j \wedge \phi_k) \rightarrow \bot$,
\item $\vdash \phi_k\to\phi_j$,
\item $\vdash \phi_j\to\phi_k$;
\end{itemize}

\item the map $\pre$ $ :  E \times \rmPhi \rightarrow [0,1]$  assigns a probability distribution  $\pre(\bullet | \phi)$ over $E$ for every  $\phi \in \rmPhi$; %This is an \emph{occurrence probability}: $\pre(e|\phi)$ expresses the prior probability that event $e\in E$ might occur in a(ny) state satisfying precondition $\phi$;
\item the map $\sub : E \rightarrow \SetSub_\mathcal{L}$ assigns a substitution function (see \autoref{def:substitution:function}) to each event in $E$;
\item  for all $\phi_j \in\rmPhi$ and  $e\in E$, if $\pre(e | \phi_j)=0$ then $\pre(e | \phi_k)=0$  for all  $\phi_k \in\rmPhi$ such that $\vdash \phi_j\to\phi_k$.
%\item \bluebf{for all $\phi\in\rmPhi$ and  $e\in E$, if $\pre(e | \phi)>0$ and there exists $b\in\rmPhi$ such that $b<a$, then there exists $b\in\rmPhi$ such that $b<a$ and $\pre(e | b)>0$;} 
%\item \bluebf{for all $a,b\in\rmPhi$ and  $e\in E$,
%if  $b<a$ and $\pre(e | a)=0$, then $\pre(e | b)=0$.}
%\end{enumerate}
\end{itemize}
\end{definition}
\begin{remark} The conditions on $\rmPhi$ match the conditions of $\bfPhi$ given in  \autoref{def:algebraic event} (cf.~\Cref{def:induced:event:structure}). The requirement in \autoref{def:algebraic event} that $\bfPhi$ is a multiset stems from the fact that the interpretation of distinct formulas $\phi_k,\phi_j$ such that $\phi_k\to\phi_j$ might coincide in a model.
\end{remark}
\begin{remark} The conditions on the preconditions are given using $\vdash$. One should refer to Section \ref{ssec:IPDEL:axiomatisation} and Table \ref{table:IPDEL} for the axiomatisation of IPDEL.
\end{remark}

\subsection{Algebraic semantics}
\label{ssec:IPDEL:semantics}
In what follows, we define the models, the event structures on the language, the event structures on the model, the updated models and the semantics. Notice that the definition of the event structure on the model relies on the definition of the event structure on the language, and that the definitions of the event structure on the model, the updated models and the semantics are given by simultaneous induction.

\begin{definition}[APE-models]
	\label{def: alg probab epist model}
	{\em Algebraic probabilistic epistemic models (APE-models)} are tuples $ \mathcal{M}=  \langle \mc{F}, v\rangle $ such that
	%\begin{enumerate}
		%\item 
$\mc{F} =  \left( \mb{A}, (\mu_i)_{i\in \Ag} \right)$ is an  APE-structure
		%\item
and $v:\mathsf{AtProp}\to\mb{A}$.
	%\end{enumerate}
\end{definition}
The update construction of Section \ref{sec:ha} extends from APE-structures to APE-models. Indeed, for any APE-model $\mathcal{M} =  \langle \mb{A}, (\mu_i)_{i\in \Ag} , v\rangle$ and any event structure  
$\mathcal{E}$ (see Definition \ref{def:intuitionistic-proba-epist-event-struct}), the event structure $\mathcal{E}$ induces an event structure over the algebra $\mb{A}$ (see Definition \ref{def:algebraic event}) as follows.

\begin{definition}
\label{def:induced:event:structure}
For any APE-model $\mathcal{M}=  \langle \mb{A}, (\mu_i)_{i\in \Ag} , v\rangle$ and any event structure  
$$\mathcal{E} = (E, (\sim_i)_{i\in\Ag}, (P_i)_{i\in \Ag}, \rmPhi, \pre, \sub),$$ 
over $\mathcal{L}$, the following tuple is an event structure over $\mb{A}$ %$$\prod_{\mathcal{E}}\mathcal{M}: = \langle\prod_{\mb{E}_{\mathcal{E}}}\mathcal{F}, v^{\prod}\rangle$$
%where $\prod_{\mb{E}_{\mathcal{E}}}\mathcal{F}$ is the ApPE-structure defined as in Definition \ref{def: prod F over E}, $\mb{E}_{\mathcal{E}}$ is the event structure over the support $\mb{A}$ of $\mathcal{F}$ specified as follows:
$$\mb{E}_{\mathcal{E}}: = (E, (\sim_i)_{i\in\Ag}, (P_i)_{i\in \Ag}, \bfPhi_\mathcal{M}, \overline{\pre}_\mathcal{M}),$$  
where 
\begin{itemize}
\item $\bfPhi := (\rmPhi_\mathcal{M}, \prec_{\mc{M}})$ with
$\rmPhi_\mathcal{M}: = \{\val{\phi}_\mathcal{M}\mid \phi\in \rmPhi\}$
and 
$\prec_{\mathcal{M}} := \{ (\val{\phi_j},\val{\phi_k})\: \mid \quad  \vdash \phi_j \rightarrow \phi_k\}$, and
\item $\overline{\pre}_\mathcal{M}$ assigns a probability distribution $\overline{\pre}_\mathcal{M}(\bullet | a)$ over $E$ for every  $a \in \rmPhi_\mathcal{M}$.
\end{itemize}
\end{definition}
It is straightfoward to verify that $\mb{E}_{\mathcal{E}}$ defined above is an event structure.

\begin{definition}[Updated model] The update of the APE-model  $ \mathcal{M}=  \langle \mc{F}, v\rangle $ 
by the intuitionistic probabilistic event structure 
$\mathcal{E} = (E, (\sim_i)_{i\in\Ag}, (P_i)_{i\in \Ag}, \rmPhi, \pre, \sub)$ is given by the APE-model
$$\mathcal{M}^\mathcal{E}: = \langle\mathcal{F}^{\mathcal{E}}, v^{\mathcal{E}}\rangle, $$
where
\begin{itemize}
\item $\mathcal{F}^\mathcal{E}: = \mathcal{F}^{\mb{E}_{\mathcal{E}}}$  as in Definition \ref{def: updated APE structure F of E}, 
\item and the map $v^{\mathcal{E}}$ is defined as follows: 
\begin{align*}
v^{\mathcal{E}} : \AtProp & \rightarrow \mb{A}^{\mb{E}_{\mathcal{E}}} 
\\
p & \mapsto  
\left\{
    \begin{array}{lll}
        [v^{\prod}(\sub(e)(p))] 
        & \qquad \quad
        & \mbox{if } p\in dom(\sub(e))
        \\
        {[v^{\prod}(p) ]}
        && \mbox{otherwise }  
    \end{array}
\right.
\end{align*}
where
\begin{align*}
v^{\prod}(p): E & \rightarrow \prod_{\mb{E}_{\mathcal{E}}}\mb{A}
&& \text{and} 
& v^{\prod}(\sub(e)(p)): E & \rightarrow \prod_{\mb{E}_{\mathcal{E}}}\mb{A}
\\
e & \mapsto v(p)
&&
& e & \mapsto v(\sub(e)(p)).
\end{align*}
\end{itemize}
\end{definition}
\begin{notation} We define the $e$-th projection $\pi_e$ 
for every $e\in E$, 
the quotient map $\pi$ and the map $\iota$ as follows: 
\begin{align*}
\pi_e: \prod_{\mb{E}_{\mathcal{E}}}\mb{A}
& \rightarrow \mb{A}
&& \text{and} 
& \pi: \prod_{\mb{E}_{\mathcal{E}}}\mb{A}
& \rightarrow \mb{A}^{\mb{E}_{\mathcal{E}}}
&& \text{and} 
& \iota : \mb{A}^{\mb{E}_{\mathcal{E}}}
& \rightarrow \prod_{\mb{E}_{\mathcal{E}}}\mb{A}
\\
g & \mapsto g(e)
&&
& g & \mapsto [g]
&&
&[g] & \mapsto g\wedge \overline{pre}.
\end{align*}
As explained in \cite[Section 3.2]{MPS14}, the map
$\iota$
is well-defined.
\end{notation}
% updated APE-model $\mathcal{M}^\mathcal{E}$ i
%\paragraph*{Algebraic semantics.} \red{To move after the definition of AIPE-structures.}
\begin{definition}[Semantics]
\label{def:semantics:IPDEL}
The interpretation of $\mathcal{L}$-formulas  on any APE-model $\mathcal{M}$ is defined recursively as follows:

		\begin{align*}
\val{\bot}_{\mathcal{M}}   & =   \bot^\mb{A}   
&
\val{\top}_{\mathcal{M}}   & =   \top^\mb{A}
\\
\val{p}_{\mathcal{M}}  & =   v(p)
&
\val{\varphi\rightarrow\psi}_{\mathcal{M}}   & =  
\val{\varphi}_{\mathcal{M}}\rightarrow^\mb{A}\val{\psi}_{\mathcal{M}}
\\
\val{\varphi\wedge\psi}_{\mathcal{M}}  & =   
\val{\varphi}_{\mathcal{M}}\wedge^\mb{A}\val{\psi}_{\mathcal{M}}  
&
\val{\varphi\vee\psi}_{\mathcal{M}}   & =   
\val{\varphi}_{\mathcal{M}}\vee^\mb{A}\val{\psi}_{\mathcal{M}}
\\
\val{\lozenge_i\varphi}_{\mathcal{M}} &  =   
\lozenge_i\val{\varphi}_{\mathcal{M}}  
&
\val{\Box_i\varphi}_{\mathcal{M}} &  =   \Box_i\val{\varphi}_{\mathcal{M}}
\\
\val{\langle\mathcal{E}, e\rangle\varphi}_{\mathcal{M}} & = 
\val{pre(e)}_{\mathcal{M}}\wedge^\mb{A} \pi_e\circ \iota(\val{\varphi}_{\mathcal{M}^{\mb{E}_\mathcal{E}}})  
&
\val{[\mathcal{E}, e]\varphi}_{\mathcal{M}}  & =   \val{pre(e)}_{\mathcal{M}}\rightarrow^\mb{A} \pi_e\circ \iota(\val{\varphi}_{\mathcal{M}^{\mb{E}_\mathcal{E}}})
		\end{align*}
%\vspace{-1.5cm}
		\begin{align*}
\left[\!\!\left[\left(\sum_{k = 1}^n\alpha_k · \mu_i(\varphi_k)\right)  \geq \beta \right]\!\!\right]_{\mathcal{M}} & = \bigvee
\left\{a\in \mb{A} \; \middle\vert \; a\in \mathsf{Min}_i(\mb{A}) \mbox{ and }
 \left(\sum_{k = 1}^n\alpha_k\mu_i(\val{\varphi_k}_{\mathcal{M}}\wedge a) \right)\geq \beta \right\}
		\end{align*}
\end{definition}

\subsection{Axiomatisation}
\label{ssec:IPDEL:axiomatisation} 
IPDEL is intended as the intuitionistic counterpart of classical PDEL.  The full axiomatisation of IPDEL is given in Table \ref{table:IPDEL} (see page \pageref{table:IPDEL}). This axiomatisation differs from the one of classical PDEL (cf.\ \Cref{table:PDEL}) in that the axioms for S5 are replaced by the axioms of intuitionistic modal logic MIPC and axiom E (see \Cref{def:epist-Heyting-algebra}), and  the axioms capturing classical probability theory are replaced by axioms capturing intuitionistic probability theory. In particular, axioms p3 and p4 in \Cref{table:PDEL} are different from the axioms P3 and P4 in \Cref{table:IPDEL}. It is not hard to see that axiom p3 implies P3 and $\mu_i(\varphi)+\mu_i(\lnot\varphi)=1$ in the presence of p1 and p2. Axiom P3 is strictly weaker that p3, since the aforementioned equality is generally false for intuitionistic probabilities. Axioms p4 and P4 are classically equivalent. In intuitionistic logic,  P4 is strictly stronger than p4. Indeed, as \Cref{lemma:axiomreplacement}  shows, p4 is not strong enough to express the strict monotonicity of $i$-measures.
Finally, notice that axioms \ref{proof:axiom:epist-alg:distribd2} and \ref{proof:axiom:epist-alg:distribb2} from \autoref{def: epist algebra} are not in 
\autoref{table:IPDEL}. Indeed, they follow from the remaining axioms and the necessitation rules (see \Cref{lem:m8m9} and also compare with \cite{bezhanishvili1998varieties}).

\renewcommand{\arraystretch}{1.07}

\begin{table}
\caption{\textsc{Axioms of IPDEL}}
\label{table:IPDEL}
\begin{tabular}{|ll|ll|}
\hline
\multicolumn{4}{|c|}{\textbf{Axioms of IPL}}
\\
\hline
%\hline
H1.    & $A\rightarrow (B\rightarrow A)$& 
H2. &
      $(A\rightarrow (B\rightarrow C))\rightarrow ((A\rightarrow B)\rightarrow (A\rightarrow C))$                   \\

H3.     & $A\rightarrow (B\rightarrow A \wedge B)$                                           
     &  H4. & $(A\rightarrow C)\rightarrow ((B\rightarrow C)\rightarrow (A\vee B\rightarrow C))$                     \\
H5.    & $A \wedge B \rightarrow A$                                                   & H6. & $A \wedge B \rightarrow B$                                                   \\
H7. & $A\rightarrow A \vee B$                                                    & H8.
    & $B\rightarrow A \vee B$                                                    \\
H9.
      & $\bot\rightarrow A$                                               && \\
      
%\hline
%\multicolumn{4}{|c|}{\textbf{Inference rule of IPL}}
%\\
%\hline
%MP. &     
%\multicolumn{3}{l|}{ if $\vdash A\rightarrow B$ and $\vdash A$, then $\vdash B$ }
%\\
\hline
\multicolumn{4}{|c|}{\textbf{Axioms for static modalities} 
%\redbf{Axioms of Definition \ref{def: epist algebra}}
}
\\
\hline
M1. & $p\to\lozenge_i p$ & M2.
& $\Box_i p\to p$
\\
M3. & $\lozenge_i(p\lor q)\to\lozenge_ip\lor\lozenge_iq$ & M4. 
& $\Box_i(p\to q)\to (\Box_ip\to\Box_iq )$
\\
M5. & $\lozenge_i p \to \Box_i\lozenge_i p$ &M6.& $\lozenge_i\Box_ip\to\Box_ip$
\\
M7. & $\Box_i(p\to q)\to(\lozenge_ip\to\lozenge_iq)$ &
E. &$\lozenge_ip\lor\lnot\lozenge_ip$\\
\hline
%\multicolumn{4}{|c|}{\textbf{Axioms for static modalities} \redbf{Axioms of MIPC}}
%\\
%\hline
%M1. & 
%$\Box_i A \rightarrow A $ 
%& M2. & $A \rightarrow \Diamond_i A$
%\\
%M3.  & $\Box_i (A\rightarrow B)
%\rightarrow (\Box_i A\rightarrow \Box_i B)$ &
%M4.    & $\Diamond_i (A\vee B)\rightarrow \Diamond_i A\vee \Diamond_i B$                   
%\\
%M5. & $\Diamond_i \Box_i A \rightarrow \Box_i A  $ &
%M6. & 
%$\Diamond_i A \rightarrow \Box_i\Diamond_i A$
%\\
%M7. & 
%$\Box_i (A \rightarrow B) 
%\rightarrow ( \Diamond_i A \rightarrow \Diamond_i B )$&&
%\\
%\hline
%\hline
%\multicolumn{4}{|c|}{\bluebf{Which axiomatisation do we take for the static modalities? I think it is the first one.}}
%\\
%\hline
%\hline
%\multicolumn{4}{|c|}{\textbf{Inference Rule for $\Box_i$}}                                                \\
%\hline
%%\hline
%Nec$_i$        & \multicolumn{3}{l|}{ if $\vdash A$, 
%then $\vdash \Box_i  A$  }                        \\
%\hline
\multicolumn{4}{|c|}{\textbf{Axioms for inequalities}}                                                                       \\
\hline
N0. & $t\geq t$
&
N1. & $(t\geq \beta) \leftrightarrow (t + 0 \cdot \mu_i(\varphi) \geq \beta)$
\\
N2. & \multicolumn{3}{l|}{$\left( \sum_{k=1}^n \alpha_k \cdot \mu_i(\varphi_k) \geq \beta \right) \rightarrow
\left( \sum_{k=1}^n \alpha_{\sigma(k)} \cdot \mu_i(\varphi_{\sigma(k)}) \geq \beta \right)
 $ for any permutation $\sigma$ over $\{ 1, ..., n\}$}
\\
N3. & \multicolumn{3}{l|}{$\left( \sum_{k=1}^n \alpha_k \cdot \mu_i(\varphi_k) \geq \beta \right) \wedge \left( \sum_{k=1}^n \alpha'_k \cdot \mu_i(\varphi_k) \geq \beta' \right) \rightarrow 
\left( \sum_{k=1}^n (\alpha_k + \alpha'_k) \cdot \mu_i(\varphi_k) \geq (\beta + \beta') \right)$}
\\
N4. & $(( t \geq \beta) \wedge (d \geq 0)) \rightarrow
( d \cdot t \geq d \cdot \beta  )$
&
N5. & $ ( t \geq \beta ) \vee (\beta \geq t)$
\\
N6. & $ (t \geq \beta)
\rightarrow (t > \gamma)  \text{ for all }\gamma<\beta $
&&
\\
\hline
\multicolumn{4}{|c|}{\textbf{Axioms for Intuitionistic Probabilities}}                                                                       \\
\hline
P1. & $\mu_i(\bot) = 0$
&
P2. & $\mu_i(\top) = 1$
\\
P3. &  $\mu_i(\varphi) + \mu_i(\psi) = \mu_i(\varphi \vee \psi) + \mu_i(\varphi \wedge \psi) $
&
P4. & $((\Box_i(\varphi\to\psi))\land(\mu_i(\varphi)=\mu_i(\psi)))\leftrightarrow\Box_i(\psi\leftrightarrow\varphi)$
\\
P5. & \multicolumn{3}{l|}{
$\left( \sum_{k=1}^n \alpha_k \cdot \mu_i(\varphi_k) \geq \beta \right) \rightarrow \Box_i \left( \sum_{k=1}^n \alpha_k \cdot \mu_i(\varphi_k) \geq \beta \right)$}
\\
\hline
\multicolumn{4}{|c|}{\textbf{Reduction Axioms}}                                                                       \\
\hline
I1. & $\left[\mathcal{E},e\right] p  \leftrightarrow 
 pre(e)
\rightarrow \sub(e,p)$
& I2. & $\langle\mathcal{E},e\rangle p  \leftrightarrow pre(e)
\land \sub(e,p)$
\\
I3. & $\left[\mathcal{E},e\right] \top \leftrightarrow \top$
& I4. &$\langle\mathcal{E},e\rangle \top  \leftrightarrow pre(e)$ \\
I5. & $\left[\mathcal{E},e\right] \bot  
\leftrightarrow \lnot pre(e)$
& I6. & $\langle\mathcal{E},e\rangle \bot \leftrightarrow \bot$ \\
%\hline
I7. & $\left[\mathcal{E},e\right](\psi_1\land\psi_2)
\leftrightarrow \left[\mathcal{E},e\right]\psi_1
\land \left[\mathcal{E},e\right]\psi_2$                                  & I8.
 & $\langle\mathcal{E},e\rangle
(\psi_1\land\psi_2) \leftrightarrow \langle\mathcal{E},e\rangle\psi_1
\land \langle\mathcal{E},e\rangle\psi_2$                                  \\
I9.   & $\left[\mathcal{E},e\right](\psi_1\lor\psi_2)
 \leftrightarrow
pre(e)\rightarrow
\langle\mathcal{E},e\rangle\psi_1
\lor \langle\mathcal{E},e\rangle\psi_2$
& I10. & $\langle\mathcal{E},e\rangle(\psi_1\lor\psi_2) \leftrightarrow \langle\mathcal{E},e\rangle\psi_1
\lor \langle\mathcal{E},e\rangle\psi_2$\\
I11.  & $\left[\mathcal{E},e\right](\psi_1\rightarrow\psi_2) \leftrightarrow \langle\mathcal{E},e\rangle\psi_1
\rightarrow \langle\mathcal{E},e\rangle\psi_2$                                    
& I12. & 
$\langle\mathcal{E},e\rangle
(\psi_1\rightarrow\psi_2) \leftrightarrow pre(e)
\land( \langle\mathcal{E},e\rangle\psi_1
\rightarrow
\langle\mathcal{E},e\rangle\psi_2)$
\\
%\hline
I13.   & $\left[\mathcal{E},e\right]\lozenge_i\psi \leftrightarrow pre(e)\rightarrow \bigvee_{e'\sim_i e}\lozenge_i(\langle\mathcal{E},e'\rangle\psi)
$
& I14.
  & $\langle\mathcal{E},e\rangle\lozenge_i\psi \leftrightarrow pre(e)
\land \bigvee_{e'\sim_i e}\lozenge_i(\langle\mathcal{E},e'
\rangle\psi)
$   \\
%\hline
I15.  & 
$\left[\mathcal{E},e\right]\Box_i\psi \leftrightarrow pre(e)\rightarrow \bigwedge_{e'\sim_i e}\Box_i([\mathcal{E},e']\psi)$ 
& I16.
   & $\langle\mathcal{E},e\rangle\Box_i\psi \leftrightarrow pre(e)
\land \bigwedge_{e'\sim_i e}\Box_i([\mathcal{E},e'
]\psi)
$\\
I17. & 
$\left[\mathcal{E},e\right](\alpha\mu_i(\psi)\geq\beta)
\leftrightarrow pre(e)\rightarrow (C + D\geq 0)$
& I18. 
& $\langle\mathcal{E},e\rangle(\alpha\mu_i(\psi)
\geq\beta)
\leftrightarrow pre(e) \land (C' + D \geq 0)$
\\
& \multicolumn{3}{l|}{where \quad}
\\
 & 
\multicolumn{3}{c|}{$C := \sum_{
	\begin{smallmatrix}
        e'\sim_i e\\
		\phi \in \rmPhi
		%\val{\phi}_{\mathcal{M}}\wedge (\val{\langle \mb{E}, e'\rangle p}_{\mathcal{M}}\wedge a)\neq \bot
	\end{smallmatrix}
}
\alpha P_i(e') \pre(e'\mid\phi)\mu^\phi_i(\left[ \mathcal{E}, e'\right] \psi)$ \quad \quad \quad
$C' := \sum_{
	\begin{smallmatrix}
	e'\sim_i e\\
	\phi \in \rmPhi
	%\val{\phi}_{\mathcal{M}}\wedge (\val{\langle \mb{E}, e'\rangle p}_{\mathcal{M}}\wedge a)\neq \bot
	\end{smallmatrix}
}
\alpha P_i(e') \pre(e'\mid\phi)\mu^\phi_i(\langle \mathcal{E}, e'\rangle \psi)$ }
\\

& 
\multicolumn{3}{c|}{
$D := \sum_{
	\begin{smallmatrix}
        e'\sim_i e\\
		\phi \in \rmPhi
		%\val{\phi}_{\mathcal{M}}\wedge  a\neq \bot
	\end{smallmatrix}
}
-\beta P_i(e') \pre(e'\mid\phi) \mu^\phi_i(\top)$,}
\\
& \multicolumn{3}{l|}{with}
\\
& \multicolumn{3}{c|}{ $\mu^\phi_i(\psi) := \mu_i(\psi\land\phi)-\sum_{\sigma\in\mathrm{mb}(\phi)}\mu_i(\psi\land\sigma)$ \quad
and \quad
$\mathrm{mb}(\phi) := \max_{\rightarrow} \rmPhi\cap\!\!\downarrow\!\! \phi$.}
\\
\hline
%\multicolumn{4}{|c|}{\textbf{Inference Rules}}                                                                          \\
%\hline
%vNec   & \multicolumn{3}{l|}{ if $\vdash A$, then $\vdash \ls\alpha\rs A$                                                  } \\
%\hline
\multicolumn{4}{|c|}{\textbf{Inference Rules}}                                                                          \\
\hline
MP. &  
if $\vdash A\rightarrow B$ and $\vdash A$, then $\vdash B$  &&
\\
%\hline
Nec$_i$        & 
if $\vdash A$, 
then $\vdash \Box_i  A$ 
&
Nec$_\alpha$   & 
if $\vdash A$, then $\vdash \ls \mathcal{E},e \rs A$ \\
Sub$_\mu$ & 
if $\vdash A \rightarrow B$, then $\vdash \mu_i (A) \leq \mu_i(B)$
&
SubEq & 
if $\vdash A \leftrightarrow B$, then $\vdash \phi\leftrightarrow\phi[A/B]$\\
\hline
\multicolumn{4}{c}{\redbf{}}
\end{tabular}

\end{table}

%\marginnote{We need a C' in the reductioin axioms}
\renewcommand{\arraystretch}{1}

%\redbf{It contains the axioms of inuitionistic modal logic MIPC, the axioms for linear inequalities on rational numbers \cite[Theorem 4.3]{fagin1990logic}, the axioms for intuitionistic probabilities (cf.\ \cite{ABS16,vBGK09,fagin1990logic,ES14,FH94}), the reduction axioms for dynamic modalities and the following inference rules: modus ponens, necessitation for the static and dynamic modalities, a substitution rule for the probabilistic orperators $\mu_i$ \cite{ABS16,vBGK09,ES14,FH94}, and a uniform substitution  rule.}

%\redbf{
%Notice that axioms \ref{proof:axiom:epist-alg:distribd2} and \ref{proof:axiom:epist-alg:distribb2} from \autoref{def: epist algebra} are not in 
% \autoref{table:IPDEL}. Indeed, one can easily prove that they are derivable.
\begin{lemma}\label{lem:m8m9} Axioms  \ref{proof:axiom:epist-alg:distribd2} and \ref{proof:axiom:epist-alg:distribb2} from \autoref{def: epist algebra} are derivable from 
rules and axioms in  \autoref{table:IPDEL}.
\end{lemma}
\begin{proof}
Axiom  \ref{proof:axiom:epist-alg:distribb2} (i.e.\ $\top\leq\Box_i\top$)   is a direct consequence of the necessitation rule.
Axiom \ref{proof:axiom:epist-alg:distribd2} (i.e.\ $\lozenge_i\bot\leq\bot$)  can be derived as follows: 
by instantiating axiom \ref{proof:axiom:epist-alg:trans} with $\bot$, one gets 
 $\lozenge_i \Box_i \bot \rightarrow \Box_i \bot$;
by instantiating axiom 
 \ref{proof:axiom:epist-alg:refl2} with $\bot$, one gets 
 $\Box_i \bot \rightarrow \bot$; since, in addition, $\bot \rightarrow \Box_i \bot$ (axiom H9), one gets that
 $\Box_i \bot \leftrightarrow \bot$;
by substitution of logical equivalence (rule  SubEq) in 
$\lozenge_i \Box_i \bot \rightarrow \Box_i \bot$, one gets 
$\lozenge_i  \bot \rightarrow \bot$ as required.
\end{proof}

\begin{lemma}\label{lemma:axiomreplacement}
%\redbf{TO FIX THE PROOF}
Axiom P4 in Table \ref{table:IPDEL} implies axiom p4 in Table \ref{table:PDEL}. In classical logic the two formulas are equivalent in the context of the rest of the axioms. Finally, there exists an ApPE-structure that validates axiom p4 but doesn't validate axiom P4.
\end{lemma}
\begin{proof} Recall that 
\begin{align*}
\text{(P4)} \quad & ((\Box_i(\phi\to\psi))\land(\mu_i(\phi)=\mu_i(\psi))])\leftrightarrow\Box_i(\psi\leftrightarrow\phi), \\
\text{(p4)} \quad & \Box_i \varphi \leftrightarrow ( \mu_i(\varphi)=1 ).
\end{align*}
	That P4 implies  p4 follows immediately by replacing $\psi$ with $\top$. Now, let us prove  that p4 implies P4 in classical logic.
	We first show that p4 implies 
	$\Box_i(\psi\leftrightarrow\phi) \rightarrow ((\Box_i(\phi\to\psi))\land(\mu_i(\phi)=\mu_i(\psi))])$ as follows.

\begin{align*}
 \Box_i(\psi\leftrightarrow\phi) 
 \Leftrightarrow \quad &
 \mu_i(\psi\leftrightarrow\phi)=1 
\tag{Axiom p4}
\\
 \Leftrightarrow \quad  & 
\mu_i((\lnot\psi\lor\phi)\land(\lnot\phi\lor\psi))=1 
\tag{classical logic equivalence}
\end{align*}
Notice that 
\begin{align}
(\lnot\psi\lor\phi)\land(\lnot\phi\lor\psi) \quad \rightarrow \quad (\lnot\psi\lor\phi)
\label{proof:axiom:p4:1}
\\
(\lnot\psi\lor\phi)\land(\lnot\phi\lor\psi) \quad \rightarrow \quad (\lnot\phi\lor\psi)
\label{proof:axiom:p4:2}
\end{align}
Hence, using the rule 
Sub$_\mu$ : if $\vdash A \rightarrow B$, then $\vdash \mu_i (A) \leq \mu_i(B)$, the equality $\mu_i((\lnot\psi\lor\phi)\land(\lnot\phi\lor\psi))=1$ and the equations  \eqref{proof:axiom:p4:1} and \eqref{proof:axiom:p4:2}, one can prove that
$$\left(\mu_i(\lnot\psi\lor\phi)=1\right)\land\left(\mu_i(\lnot\phi\lor\psi)=1\right) $$
Using p4, we can derive that $\Box_i (\phi \rightarrow \phi)$.
It remains to derive that $\mu_i(\psi)=\mu_i(\phi)$ as follows.
\begin{align*}
 \quad & 
\left(\mu_i(\lnot\psi\lor\phi)=1\right)\land\left(\mu_i(\lnot\phi\lor\psi)=1\right) 
%\tag{Rule $Sub_\mu$}
\\
\Rightarrow \quad & 
\left(\mu_i(\lnot(\lnot\psi\lor\phi))=0\right)\land\left(\mu_i(\lnot\phi\lor\psi)=1\right) 
\tag{$\mu_i(\varphi)=1-\mu_i(\lnot \varphi)$ in PDEL, see \autoref{table:PDEL}}
\\
\Rightarrow \quad  & 
\left(\mu_i(\psi\land\lnot\phi)=0\right)\land\left(\mu_i(\lnot\phi\lor\psi)=1\right) 
\tag{De Morgan laws}
\\
\Rightarrow \quad  & 
\left(\mu_i(\psi\land\lnot\phi)=0\right)\land\left(\mu_i(\lnot\phi)+\mu_i(\psi)-\mu_i(\psi\land\lnot\phi)=1\right) 
\tag{$\mu_i(\varphi) + \mu_i(\psi) = \mu_i(\varphi \vee \psi) + \mu_i(\varphi \wedge \psi) $ in PDEL, see \autoref{table:PDEL}}
\\
\Rightarrow \quad  & 
\mu_i(\lnot\phi)+\mu_i(\psi)=1
\\
\Rightarrow \quad  & 
\mu_i(\lnot\phi)+\mu_i(\psi)=\mu_i(\phi)+\mu_i(\lnot\phi) 
\tag{$\mu_i(\phi)+\mu_i(\lnot \phi)=1$ in PDEL, by axioms p2 and p3}
\\
\Rightarrow \quad  & \mu_i(\psi)=\mu_i(\phi).
\end{align*}

Now, we show that p4 implies 
	$ ((\Box_i(\phi\to\psi))\land(\mu_i(\phi)=\mu_i(\psi))])\rightarrow 
	 \Box_i(\psi\leftrightarrow\phi) $ as follows.
\begin{align*}
 &  \Box_i(\phi\to\psi)\land\left(\mu_i(\phi)=\mu_i(\psi)\right)\\ 
\Rightarrow \quad  & \left(\mu_i(\lnot\phi\lor\psi)=1\right)\land\left(\mu_i(\phi)=\mu_i(\psi)\right)
\tag{Axiom p4}
\\
\Rightarrow \quad   & \left(\mu_i(\lnot\phi)+\mu_i(\psi)-\mu_i(\lnot\phi\land\psi)=1\right)\land\left(\mu_i(\phi)=\mu_i(\psi)\right)
\tag{$\mu_i(\varphi) + \mu_i(\psi) = \mu_i(\varphi \vee \psi) + \mu_i(\varphi \wedge \psi) $ in PDEL}
\\
\Rightarrow \quad   & \left(\mu_i(\lnot\phi)+\mu_i(\psi)-\mu_i(\lnot\phi\land\psi)=1\right)\land\left(\mu_i(\lnot\phi)=\mu_i(\lnot\psi)\right)
\tag{$\mu_i(\phi)+\mu_i(\lnot \phi)=1$ in PDEL}
\\
\Rightarrow \quad   & \left(\mu_i(\lnot\psi)+\mu_i(\psi)-\mu_i(\lnot\phi\land\psi)=1\right)
\\
\Rightarrow \quad   & \left(1-\mu_i(\lnot\phi\land\psi)=1\right)
\tag{$\mu_i(\phi)+\mu_i(\lnot \phi)=1$ in PDEL}
\\
\Rightarrow \quad   &
\left(\mu_i(\lnot\phi\land\psi)=0\right)
\tag{$\mu_i(\phi)+\mu_i(\lnot \phi)=1$ in PDEL}
\\
\Rightarrow \quad   &
\left(\mu_i(\phi\lor\lnot\psi)=1\right)
\tag{$\mu_i(\phi)+\mu_i(\lnot \phi)=1$ in PDEL}
\\
\Rightarrow \quad   &
\left(\mu_i(\psi\to\phi)=1\right)
\\
\Rightarrow \quad   &
\Box_i(\psi\to\phi)
\tag{Axiom p4}
\end{align*}

This concludes the proof that in classical logic p4 and P4 are equivalent. Finally, consider the  Heyting algebra $\mathbb{H}$ in Figure \ref{fig:countermodelforequivalence}
with 
\begin{align*}
\lozenge x :=
\left\{ \begin{array}{ll}
	\top & \text{if } x\neq\bot,\\
	\bot \quad\quad & \text{if } x=\bot
\end{array} 
\right.  &\qquad &\Box x :=
\left\{ \begin{array}{ll}
	\bot & \text{if } x\neq\top,\\
	\top \quad\quad & \text{if } x=\top
\end{array} 
\right.
\end{align*}
and $\mu(\bot)=0$, $\mu(a)=0.5$, $\mu(b)=0.5$ and $\mu(\top)=1$. 

\begin{figure}[h]
\begin{center}
	\begin{tikzpicture}[-,>=stealth',shorten >=1pt,thick]
	\SetGraphUnit{3} 
	\tikzset{VertexStyle/.style = {draw=white,circle,thick,
			minimum size=.1cm,
			font=\Large\bfseries},thick}
	\Vertex[L={\footnotesize $\top$},x=0,y=0]{A};
	\Vertex[L={\footnotesize $b$},x=0,y=-1]{B};
	\Vertex[L={\footnotesize $a$},x=0,y=-2]{D};
	\Vertex[L={\footnotesize $\bot$},x=0,y=-3]{F};
	\Edge(A)(B);
	\Edge(B)(D);
	\Edge(D)(F);
	\end{tikzpicture} 
\end{center}
\caption{Heyting algebra $\mathbb{H}$}
\label{fig:countermodelforequivalence}
\end{figure}
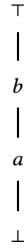

It is easy to see that the  Heyting algebra in \autoref{fig:countermodelforequivalence} satisfies all axioms of IPDEL except for $P4$ and it satisfies p4. It falsifies P4 because $(\Box(a\to b))\land(\mu(a)=\mu(b))=\top$, while $\Box(a\leftrightarrow b)=\bot$. 
\end{proof}

%\subsection{Soundness and completeness}
%\label{ssec:IPDEL:soundness}
%\label{ssec:IPDEL:syntax}
\begin{theorem}[Soundness]
\label{lem:soundness-IPDEL}
The axiomatization for IPDEL given in Table \ref{table:IPDEL} is sound w.r.t.\ APE-models.
\end{theorem}

%\subsection{Completeness}

\begin{theorem}[Completeness]
\label{th:completeness}
	The axiomatisation for IPDEL given in Table \ref{table:IPDEL} is weakly complete w.r.t.\ APE-models.
\end{theorem}

The proof of soundness is given in Appendix \ref{Appendix:soundness} and the proof of completeness is given in Appendix \ref{Appendix:completeness}.

\section{Relational semantics}
\label{sec:relsem}

In this section, we introduce the finite relational semantics of IPDEL, as the dual structures of epistemic Heyting algebras within the duality between monadic Heyting algebras and MIPC-frames (cf.\ \cite{bezhanishvili1999varieties,KP13}). Specifically, we specialize this duality\footnote{Because we consider only finite algebras and finite relational structures we can dispense with the topology.} by identifying the condition correspon\-ding to axiom E. Moreover, we present a dual correspondence between the probability distributions on intuitionistic Kripke frames and measures on epistemic Heyting algebras. This correspondence appears in \cite{flaminio2017states} in the context of finite GBL-algebras. %Here we elaborate on this result in the special case of epistemic Heyting algebras. 
Furthermore, we generalize the model-theoretic constructions presented in \Cref{ssec:epist:update:classic} for the Boolean setting and show that they dually correspond to the constructions presented in \Cref{sec:ha}.
%\redbf{First, we define MIPC-frames that are the counter part of S5-Kripke frames for S5 
%intuistionistic modal logic \cite{}., Then ...}
Finally, notice that these results readily imply the completeness and the finite model property of IPDEL with respect to this class of relational structures via the algebraic completeness presented in  \Cref{Appendix:completeness}.

\paragraph{Structure of this section.} In Section \ref{ssec:dualityframes}, we introduce the \emph{epistemic intuitionistic Kripke frames} as the class of relational structures dually corresponding to epistemic Heyting algebras. In Section \ref{ssec:dualityprob}, we introduce the probability distributions associated with any agent $i$ and prove that each dually corresponds to an $i$-measure. In Section \ref{ssec:dualityinterm}, we introduce the construction of intermediate epistemic intuitionistic Kripke frames and prove that it dually corresponds to the construction of intermediate epistemic Heyting algebras presented in Section \ref{ssec:intermediate}. In Section \ref{ssec:dualityupdate}, we define the dual construction to the pseudo-quotient defined in \ref{ssec: abstract charact i minimal els pseudo quotient}. Finally, in Section \ref{ssec:relsemipdel} we use this construction to define the interpretation of IPDEL-formulas on IPDEL-models.

\subsection{Epistemic Heyting algebras and epistemic intuitionistic Kripke frames}
\label{ssec:dualityframes}
We first recall the definition on the objects of the duality between finite monadic Heyting algebras and MIPC-frames\footnote{A complete exposition can be found in \cite{bezhanishvili1999varieties}.}. We then identify the MIPC-frames corresponding to epistemic intuitionistic Kripke frames and show that their dual algebras exactly correspond to epistemic Heyting algebras.

%\redbf{What does MIPC mean ?
%Monadic Intuitionistic Propositional Calculus \redfootnote{alessandra's paper: Prior’s MIPC \cite{Prior1957}  monadic Heyting algebras}???}
\begin{definition}[Finite MIPC-frames] A \emph{finite MIPC-frame} is a tuple
$$\mb{F}= \langle S,\leq,(R_i)_{i\in\Ag} \rangle $$
such that 
 $(S,\leq)$ is a finite poset and  each $R_i$ is an equivalence relation on $S$ such that
 $$(R_i\circ\geq) \ \subseteq \ ({\geq}\circ  R_i)\qquad\qquad R_i \ =\ ({\geq}\circ  R_i)\cap(R_i\circ\leq).$$ 
\end{definition}

\begin{notation} For any poset $(S,\leq)$ and any set $X\subseteq S$, we define the downset and the upset generated by $X$ as
$$X \!\! \downarrow \ =\ \{w\in S \mid \exists v\in X, w\leq v\}\qquad\text{ and }\qquad 
X \!\! \uparrow \ =\ \{w\in S\mid \exists v\in X w\geq v\}$$ respectively. 
We let $\mathcal{P}^\downarrow(S)=\{X\!\! \downarrow \ \mid X\subseteq S\}$ be the set of all downsets of $S$.
\end{notation}

\begin{definition}[Complex algebra of a finite MIPC-frame]
\label{def:MIPC:comp:alg}
For any finite MIPC-frame $\mb{F}= \langle S,\leq,(R_i)_{i\in\Ag} \rangle $, let its \emph{complex algebra} be:  $$\mb{F}^+=(\mathcal{P}^\downarrow(S),\land,\lor,\to,(\lozenge_i)_{i\in\Ag},(\Box_i)_{i\in\Ag},\bot)$$ where
\begin{align}
	X\land Y  & :=X\cap Y,
\label{MIPC:and}
\\
	X\lor Y & := X\cup Y,
\label{MIPC:or}
\\
	X\to Y  & := S\setminus ((X\cap (S\setminus Y))\uparrow),
\label{MIPC:imply}
\\
	\lozenge_iX & :=R_i^{-1}[X],
\label{MIPC:diamond}
\\
	\Box_i X &  :=S\setminus ({\geq}\circ  R_i)^{-1}[S\setminus X],
\label{MIPC:box}
\\
\bot &  := \emptyset.
\label{MIPC:bot}
\end{align}
We also use the standard notation
\begin{align}
\top & := S ,
\label{MIPC:top}
\\
\neg X & := X \rightarrow \bot =  S \setminus X\!\! \uparrow.
\label{MIPC:neg}
\end{align}
\end{definition}

%\begin{lemma}[cf.\ Lemma ??? in \cite{KP13}] Let $\mb{F}=(W,\leq,(R_i)_{i\in\Ag})$ be a finite MIPC-frame. Its complex algebra $\mb{F}^+$ is a finite monadic Heyting algebra.
%\end{lemma}
 
\begin{definition}[MIPC frame associated to a finite monadic Heyting algebra]
\label{def:MIPC:ass:frame}
For any finite monadic Heyting algebra\footnote{see \autoref{def: epist algebra}, page \pageref{def: epist algebra}.} $\mathbb{A}= ( \mb{L}, (\lozenge_i)_{i\in \Ag} , (\Box_i)_{i\in \Ag} )$, let its \emph{associated frame} be: 
$$\mathbb{A}_+= \langle \mathcal{J}(\mathbb{A}),\leq,(R_i)_{i\in\Ag} \rangle $$ 
where
\begin{itemize}
	\item $\mathcal{J}(\mathbb{A})$ is the set of join-irreducible elements of $\mathbb{A}$;
	\item $\leq \ \subseteq  \mathcal{J}(\mathbb{A}) \times \mathcal{J}(\mathbb{A})$ is the order inherited from $\mathbb{A}$, i.e.~$j \leq j'$ iff $j \leq_\mathbb{A} j'$ for all $j,j' \in \mathcal{J}(\mathbb{A})$;
	\item  $R_i \ \subseteq  \mathcal{J}(\mathbb{A}) \times \mathcal{J}(\mathbb{A})$
	is defined as follows: 
	$jR_i j'$ if and only if $\lozenge_i j=\lozenge_i j'$
	for all $j,j'\in \mathcal{J}(\mathbb{A})$ and every $i\in \Ag$.
\end{itemize}
\end{definition}

The following lemma is  stated in \cite[Fact 20,Proposition 21]{KP13} and \cite{bezhanishvili1999varieties}:
\begin{lemma}\label{duality:basiclemma1}
	If $\mb{F}$ is a finite MIPC-frame, then $\mb{F}^+$ is a finite monadic Heyting algebra. If $\mathbb{A}$ is a finite monadic Heyting algebra then $\mathbb{A}_+$ is a finite MIPC-frame. Furthermore $(\mb{F}^+)_+\cong \mb{F}$ and $(\mathbb{A}_+)^+\cong\mathbb{A}$.
\end{lemma}

\begin{notation} 
\label{notation:eta}
Let  $\eta:\bbA\rightarrow (\bbA_{+})^+$ and $\epsilon: \mb{F} \rightarrow ( \mb{F}^+)_{+}$ denote the natural isomorphisms inherited from the object dualities $(\mathbb{A}_+)^+\cong\mathbb{A}$ and $(\mb{F}^+)_+\cong \mb{F}$. (see \cite{DaveyPriestley2002} for more details on $\eta$ and $\epsilon$.)
\end{notation}
%\begin{proof}
%	In \cite{bezhanishvili1999varieties}, the proof presented is more general and $R_i$ is defined as: $$aR_i b\iff (\forall c\in\mbA)(a\leq\lozenge_ic \iff b\leq\lozenge_ic).$$ Let us show that these two definitions are equivalent. Recall that in any monadic Heyting algebra $\mbA$, $\lozenge_i$ is monotone and for all $d\in\mbA$ $d\leq\lozenge_i d$ and $\lozenge_i\lozenge_id=\lozenge_id$. 
	
%	Assume that  $\lozenge_ia=\lozenge_ib$. If $a\leq\lozenge_ic$ then $\lozenge_ia\leq\lozenge_i\lozenge_ic$, i.e.\ $\lozenge_ia\leq\lozenge_i c$. This implies that $\lozenge_ib\leq\lozenge_ic$. Since $b\leq\lozenge_ib$ we obtain that $b\leq\lozenge_ic$.
	
%	On the other hand assume that $a\leq\lozenge_ic$ if and only if $b\leq\lozenge_ic$, for all $c\in\mbA$. Since $a\leq\lozenge_ia$ we have that $b\leq\lozenge_ia$ which implies that $\lozenge_ib\leq\lozenge_ia$. Symmetrically we obtain $\lozenge_ia\leq\lozenge_ib$, which implies that $\lozenge_ia=\lozenge_ib$.
%\end{proof}

\begin{definition}[\EiKf] 
\label{def:epist:intui;kripke:frame}
An \emph{\eiKf} is a finite $MIPC$-frame $\mb{F}= \langle S,\leq,(R_i)_{i\in\Ag} \rangle $
such that, for every $i \in \Ag$, the equivalence relation $R_i$ is upwards and downwards closed w.r.t.\ the order relation $\leq$. 
\end{definition}

The following lemma characterises the dual spaces of epistemic Heyting algebras\footnote{see \autoref{def:epist-Heyting-algebra}, page \pageref{def:epist-Heyting-algebra}.}:

\begin{lemma}\label{duality:keylemma}
	If $\mb{A}$ is an epistemic Heyting algebra, then $\mb{A}_+$ is an \eiKf . If $\mb{F}$ is an \eiKf,  then $\mb{F}^+$ is an epistemic Heyting algebra.
\end{lemma}  
\begin{proof}
	Since, by definition, all epistemic Heyting algebras are finite monadic Heyting algebras, it follows from \autoref{duality:basiclemma1} that their dual spaces are  finite MIPC-frames. 
	
	\medskip
Let $\mathbb{A}=( \mb{L}, (\lozenge_i)_{i\in \Ag} , (\Box_i)_{i\in \Ag})$  and $\mb{A}_+=\langle S,\leq,(R_i)_{i\in\Ag} \rangle$. By \autoref{duality:basiclemma1}, it is enough to show that the equivalence relations $R_i$ are upwards and downwards closed. Since $R_i$ is symmetric it is enough to show that $R_i$ is upwards closed.
%Let us prove that, for every $i\in\Ag$, the 	axiom \ref{axiom:epist-alg:boolean} (i.e.\ $\lozenge_ia\lor\lnot\lozenge_ia=\top$) is equivalent to the fact that the equivalence relation $R_i$ are upwards closed.
%

\smallskip

	Assume, for contradiction, that the equivalence relation $R_i$ is not upwards closed for some $i\in \Ag$. Hence, there is at least one equivalence class defined by the relation $R_i$ that is not upwards closed.
	Since the empty set is upwards and downwards closed, this equivalence class is non-empty.
	Let $w\in S$ be an element of that class, 
	let $v\in S$ be such that $v\geq w$ and $v\notin R_i[w]$, and let $a$ be the element of the dual algebra corresponding to the downset generated by $w$. Then $\lozenge_ia=R_i^{-1}[w{\downarrow}]$. 
	
	First, let us show that $v\notin R_i^{-1}[w{\downarrow}]$. Heading towards a contradiction, let us assume that  $v\in R_i^{-1}[w{\downarrow}]$. This means that there exists $z\in S$ such that $z\leq w$ and $(v,z)\in R_i$, therefore $(v,w)\in (R_i\circ \leq)$. Furthermore, we have that $(v,w)\in({\geq}\circ  R_i)$, because $(w,w)\in R_i$ and $v\geq w$. Since $R_i=({\geq}\circ  R_i)\cap(R_i\circ\leq)$, we deduce that $(v,w)\in R_i$, which is a contradiction. 
	This proves that $v\notin R_i^{-1}[w{\downarrow}]$. 
	
	From \eqref{MIPC:neg}, we have that $\lnot\lozenge_ia=S\setminus((R^{-1}[w{\downarrow}]){\uparrow})$. By assumption, $w \leq v$, hence $v\in (R^{-1}[w{\downarrow}]){\uparrow}$ and  $v\notin\lnot\lozenge_ia$. Hence $v\notin\lozenge_ia\lor\lnot\lozenge_ia$, and therefore	 axiom \ref{proof:axiom:epist-alg:boolean} does not hold,  contradicting the assumption that $\mb{A}$ is an epistemic Heyting algebra. Hence, $R_i$ is upwards closed.

\medskip

	As to the second part of the statement, let 
	 $\mb{F}= \langle S,\leq,(R_i)_{i\in\Ag} \rangle $  and  $\mathbb{F}^+=( \mb{L}, (\lozenge_i)_{i\in \Ag} , (\Box_i)_{i\in \Ag})$. By \autoref{duality:basiclemma1}, it remains to prove that $\mb{F}^+$ satisfies axiom \ref{axiom:epist-alg:boolean} (i.e.\ $\lozenge_ia\lor\lnot\lozenge_ia=\top$) for every $i\in\Ag$. Since $R_i$ is upwards closed for every $i\in\Ag$, it follows that $(R_i^{-1}[X{\downarrow}]){\uparrow}=R_i^{-1}[X{\downarrow}]$. Therefore $R_i^{-1}[X{\downarrow}]\cup (S\setminus((R_i^{-1}[X{\downarrow}]){\uparrow})=S$, i.e.\ axiom $E$ holds in $\mathbb{F}^+$, as required. 	
\end{proof}

\begin{definition}[\EiKm]
\label{def:epit:int:kripke:model}
 An \emph{\eiKm} is a tuple $\mb{M}= \langle \mb{F}, V \rangle$ such that $\mb{F}$ is an 
\eiKf \
 and $V:\mathsf{AtProp}\to\mathcal{P}^\downarrow(S)$.
 \end{definition}

\begin{corollary}
\label{cor:eiKf:i:minimal:eq:cell}
	For any \eiKf \ $\mb{F}= \langle S,\leq,(R_i)_{i\in\Ag} \rangle $, the $i$-minimal elements of $\mb{F}^+$ are exactly the equivalence cells of $R_i$.
\end{corollary}

\begin{proof}
Recall (cf.\ \autoref{def:i-minimal}) that an element  $a\in \mb{F}^+$ is  $i$-minimal if
	\begin{enumerate}
		\item \label{proof:i-minimal:item:bot} $a\neq \bot$,
		\item \label{proof:i-minimal:item:fixpoint} $\lozenge_i a = a $ and
		\item \label{proof:i-minimal:item:minimal} if $b \in \mb{F}^+$, $b < a$ and $\lozenge_i b = b $, then $b = \bot$.
	\end{enumerate}

Let $X\subseteq S$ be an $R_i$-equivalence cell of $\mb{F}$. Hence,
$X$ is a non-empty set, which proves item \eqref{proof:i-minimal:item:bot}. Moreover, by definition of $\lozenge$ (see \eqref{MIPC:diamond}), we have $\lozenge_i X := R_i^{-1} [X] = X$, which proves item \eqref{proof:i-minimal:item:fixpoint}. Finally, if $\varnothing\neq Y\subseteq X$ then $\lozenge_i Y= R_i^{-1}[Y]=X$, which proves item \eqref{proof:i-minimal:item:minimal}. 

\medskip

  Let $a\in\mb{F}^+=\mathcal{P}^\downarrow(S)$ be an $i$-minimal element. To prove that $a$ is an equivalence cell of $R_i$, we need to show that  $a = R_i^{-1}[w]$ for some $w\in S$.
  By item \ref{proof:i-minimal:item:bot}, $a \neq \varnothing$; hence, there exists $w\in a$.
  Recall that $\lozenge_iX  :=R_i^{-1}[X]$ (see
\eqref{MIPC:diamond}). By item \ref{proof:i-minimal:item:fixpoint}, $a = \lozenge_i a  =R_i^{-1}[a]$; hence, $a$ is the union of equivalence cells. By item \ref{proof:i-minimal:item:minimal}, the only equivalence cell or union of equivalence cells  smaller than $a$ is the empty set; hence, $a$ contains exactly one equivalence cell.
\end{proof}

%\begin{remark}
%\label{rk:join:prime:eiKf}
%Notice that the join prime elements of the complex algebra of a finite MIPC-frame are the downsets generated by a singleton. \redbf{(cf.\ ...)}
%\end{remark}

\begin{corollary}
\label{cor:i:minimal:below}
	For every \eiKf \ $\mb{F}= \langle S,\leq,(R_i)_{i\in\Ag} \rangle $, and
	every join-prime element $j$ of $\mb{F}^+$, there exists some $i$-minimal element $a$ such that $j\leq a$.
\end{corollary}
\begin{proof}
If $j$ is a join-prime element of $\mb{F}^+$, then $j=w{\downarrow}$ for some $w \in S$. Let $a=R_i^{-1}[w]$, which is an $i$-minimal element by \Cref{cor:eiKf:i:minimal:eq:cell}. Since the equivalence relation $R_i$ is upwards and downwards closed  for every $i\in \Ag$, we have $w{\downarrow} \subseteq R_i^{-1}[w]$, as required.
\end{proof}

\subsection{Epistemic intuitionistic Kripke frames and probabilities}
\label{ssec:dualityprob}
In this section, we define $i$-probability distributions. Applying ideas of \cite{flaminio2017states} to the setting of epistemic Heyting algebras, we define a correspondence between maps from epistemic intuitionistic Kripke frames to non-negative reals and premeasures on epistemic Heyting algebras (see \Cref{def:measures}).

\begin{definition}[$i$-probability distribution]
\label{def:i:prob:distr}
Let $\mb{F}=\langle S,\leq, (R_i)_{i\in\Ag} \rangle$ be an \eiKf. An \emph{$i$-probability distribution} over $S$ is a map $P_i:S\to\ ]0,1]$ such that  $\sum_{w\in X}P_i(w)=1$ for each equivalence cell $X$ of $R_i$.
\end{definition}

\begin{lemma}
\label{lem:rel:sem:mass:measure}
For any \eiKf \  $\mb{F}=\langle S,\leq,(R_i)_{i\in\Ag}\rangle$, any
map $f:S\to\mathbb{R}^+$ defines the $i$-premeasure  $f^+$ on $\mb{F}^+$ as follows:
\begin{align}
\label{eq:masstomeasure}
f^+ \ : \ \mathsf{Min}_i(\mb{A}){\downarrow} & \rightarrow \mb{R}^+\\
 a & \mapsto \sum_{x\in a}P_i(x).
 \notag
\end{align}
Moreover, if $f$ is an $i$-probability distribution, then the map $f^+$ is an $i$-measure (see \autoref{def:measures}) on $\mb{F}^+$.
\end{lemma}

\begin{proof}
This result directly follows from the definition of $f^+$ and \autoref{cor:eiKf:i:minimal:eq:cell}.
\end{proof}

\begin{definition}
\label{def:mui:minus}
For any finite monadic Heyting algebra $\mathbb{A}= ( \mb{L}, (\lozenge_i)_{i\in \Ag} , (\Box_i)_{i\in \Ag} )$
and any $i$-premeasure $\mu_i$ on $\mb{A}$, let
\begin{align}
(\mu_i)_{+} \ : \ \mathcal{J}(\mb{A}) & \rightarrow \mb{R}^+
\\
b & \mapsto  \mu_i(b)-\mu_i\left(\bigvee_{c<b}c\right) .
\notag
\end{align}

\end{definition}

It follows from the monotonicity of $\mu_i$ that $(\mu_i)_{+}$  is well-defined.

\begin{lemma}\label{theo:finiteduality1}
	Let $\bbA$ be an epistemic Heyting algebra equipped with an $i$-premeasure $\mu_i$. Let the map $\eta:\bbA\rightarrow (\bbA_{+})^+$ be the natural isomorphism (see \autoref{notation:eta}). 
	Then, $((\mu_i)_{+})^+(\eta(a))=\mu_i(a)$ for every $a\in\bbA$.
\end{lemma}
\begin{proof}
Notice that, by definition, 
\begin{align}
((\mu_i)_{+})^+ \circ \eta \ : \ \mathsf{Min}_i(\mb{A}){\downarrow} & \rightarrow \mb{R}^+
\notag
\\
b & \mapsto \sum_{x\in b}(\mu_i)_{+}(x) 
= \sum_{x\in b} \left( \mu_i(x)-\mu_i\left(\bigvee_{c<x}c\right) \right).
\notag
\end{align}
	Since $\bbA$ is a finite poset, we can define the \emph{height}  of its elements as follows: for every $a \in \bbA$,
	$$ \h(a) := 
	\begin{cases} 
	0 & \text{if } a = \bot\\
	\max\{\h(b)\mid b<a\}+1 & \text{otherwise.}
	\end{cases}$$ 
Notice that the only element of height 0 is $\bot$.
%\redbf{\textit{The only element of height $0$ is $\bot$  (i.e.\ $\h(\bot)=0$). For $a\neq\bot$, we define $\h(a):=\max\{\h(b)\mid b<a\}+1$.}}
 The proof will proceed by induction on the height of the elements of $\bbA$ below the $i$-minimal elements.

As to the base case, it is immediate to see that $((\mu_i)_{+})^+(\eta(\bot))=((\mu_i)_{+})^+(\varnothing)=\mu_i(\bot)=0$.
	
As to the induction step, assume that $\mu_i(a)=((\mu_i)_{+})^+(\eta(a))$ for all $a \in \mathsf{Min}_i(\mb{A})$ such that $\h(a)\leq n$. Now, let $b$ be such that $\h(b)=n+1$. 
	If $b$ is a join prime element of $\bbA$, 
	then $\eta(b)=b{\downarrow}$ and by definition $\left(\bigvee_{c<b}c\right) <b$. 
	This implies that $\h\left(\bigvee_{c<b}c\right)<\h(b)$.
	Hence, by  induction hypothesis, 
	$$\mu_i\left(\bigvee_{c<b}c\right)=((\mu_i)_{+})^+\left(\eta\left(\bigvee_{c<b}c\right)\right)=((\mu_i)_{+})^+\left(b{\downarrow}\setminus\{b\}\right).$$ 
	Therefore, 
	\begin{align*}
	((\mu_i)_{+})^+(b{\downarrow})
	&=\sum_{x\in b\downarrow}((\mu_i)_{+})(x)
	\\
	&=((\mu_i)_{+})(b)+\sum_{x\in b\downarrow\setminus\{b\}}((\mu_i)_{+})(x)
	\\
	& =((\mu_i)_{+})(b)+((\mu_i)_{+})^+(b\downarrow\setminus\{b\})\\
	& =\mu_i(b)-\mu_i\left(\bigvee_{c<b}c\right)+((\mu_i)_{+})^+\left(\eta\left(\bigvee_{c<b}c\right)\right)
	\\
	& =\mu_i(b). \tag{by induction hypothesis}
	\end{align*}		
	
	If $b$ is not a join prime element then it can be written as the union of elements strictly below it. Since both $\mu_i$ and $((\mu_i)_{+})^+$ satisfy condition 3 of Definition \ref{def:measures} and have the same values on elements of height strictly smaller than $n+1$, it follows that $\mu_i(b)=((\mu_i)_{+})^+(\eta(b))$.\end{proof}

\begin{corollary}
 Let $\bbA$ be an epistemic Heyting algebra equiped with an $i$-measure $\mu_i : \mathsf{Min}_i(\mb{A}){\downarrow}  \rightarrow \mb{R}^+$.
Then the map 
\begin{align}\label{eq:measuretomass}
(\mu_i)_{+} \ : \ \mathcal{J}(\mathbb{A}) & \rightarrow \ ]0,1] \\
 a & \mapsto \mu_i(a)-\mu_i \left(\bigvee_{b<a}b \right)
 \notag
\end{align}
is an  $i$-probability distribution over $\bbA_{+}$.
\end{corollary}
\begin{proof}
The map $(\mu_i)_{+}$ is well-defined.
Indeed, $(\mu_i)_{+}(b)$ is strictly positive for any $b \in \mathcal{J}(\mathbb{A}) $, because $\mu_i$ is strictly monotone (see \autoref{def:measures} \autoref{def:epAlg:two:nonzero}) and
$(\mu_i)_{+}(b) \leq 1$, because there exists an $i$-minimal element $a$ such that $b \leq a $ (see \autoref{cor:i:minimal:below}) and because $\mu_i(a)=1$ (see \autoref{def:measures} \autoref{def:epAlg:two:fixedpoints}).
	  \autoref{theo:finiteduality1} implies that $1=\mu_i(a)=((\mu_i)_{+})^+(a)=\sum_{x\in X}(\mu_i)_{+}(x)$ for every $i$-minimal element $a$, which shows that  $(\mu_i)_{+}$ is an $i$-probability distribution over $\mb{A}_{+}$, as required.
\end{proof}

\begin{lemma}\label{theo:finiteduality2}
	Let $\mb{F}$ be an \eiKf \ equipped with a probability distribution $P_i$. Let the map $\epsilon: \mb{F} \rightarrow ( \mb{F}^+)_{+}$ be the natural isomorphism (see \autoref{notation:eta}). Then $((P_i)^+)_{+}(\epsilon(w))=P_i(w)$  for every $w\in \mb{F}$.
\end{lemma}
\begin{proof}
	For every join prime element $w{\downarrow}$ of $\mb{F}^+$, we have that $v\in w{\downarrow}$ if and only if $v\leq w$. Thus we obtain: 
	$$((P_i)^+)_{+}(\epsilon (w))=(P_i)^+(w{\downarrow})-(P_i)^+\left(\bigvee_{b<w{\downarrow}}b\right)=\sum_{v\leq w}P_i(v)-\sum_{v<w}P_i(v)=P_i(w).$$
\end{proof}

\subsection{Dualizing the product updates of APE structures}
\label{ssec:dualityinterm}
In this section, we introduce the generalization of the construction of the intermediate structure presented in Section \ref{ssec:epist:update:classic}, and show that it dualizes  the intermediate construction on algebras presented in Section \ref{ssec:intermediate}.

\begin{definition}[Intermediate intuitionistic structure]
\label{def:inter:intui:struct}
For any   \eiKm \ $\mb{M}=\langle S,\leq,(R_i)_{i\in\Ag},\val{\cdot} \rangle$ and any
intuitionistic probabilistic event structure  
$\mathcal{E}=(E, (\sim_i)_{i\in\Ag}, (P_i)_{i\in \Ag}, \rmPhi, \pre, \sub)$
over $\mathcal{L}$ 
 (see \autoref{def:intuitionistic-proba-epist-event-struct}), let the intermediate intuitionistic structure of $\mb{M}$ and $\mc{E}$ be the tuple:

$${\coprod}_\mathcal{E}\mb{M}:=\langle \coprod_{|E|} S,\leq^{\coprod},(R^{\coprod}_i)_{i\in\Ag},\val{\cdot}_{\coprod}\rangle$$ 
where 

\begin{itemize}
\item
$\coprod_{|E|} S\cong S\times E$ is the $|E|$-fold coproduct of $S$, %clearly, $\coprod_{|E|} S$ can be identified with $$. %S' = \{ (s,e) \in S \times E \mid pre(e\mid s)\neq 0 \} $
\item the order relation $\leq^{\coprod}$ on $\coprod_{|E|} S$ is defined  as follows:
$$(s,e) \leq^{\coprod}_i (s',e')\ \quad \mbox{ iff }
\quad 
 s \leq_i s'\mbox{ and }e = e',$$
\item 
 each binary relation $R_i^{\coprod}$ on $\coprod_{|E|} S$ is defined  as follows:
 $$(s,e) R^{\coprod}_i (s',e')\ \quad \mbox{ iff }
\quad 
 s R_i s'\mbox{ and }e \sim_i e',$$
\item and  the valuation $\val{\cdot}_{\coprod} : \mathsf{AtProp} \rightarrow \P S$ is defined by
$$ \val{p}_{\coprod} : = \left\{ (s,e)  \mid s\in \val{p}_{\mb{M}} \right\} = \val{p}_{\mb{M}}\times E $$ for every $p\in \mathsf{AtProp}$.
\end{itemize}
For any \eiKm \ $\mb{M}=\langle\mb{F},\val{\cdot}\rangle$,
let $${\coprod}_\mathcal{E}\mb{F}:=\langle\coprod_{|E|} S,\leq^{\coprod},(R^{\coprod}_i)_{i\in\Ag}\rangle.$$

%\begin{itemize}
%	\item $(w_1,e_1)\leq^{\coprod}(w_2,e_2)$ if and only if $w_1\leq w_2$ and $e_1=e_2$;
%	\item $(w_1,e_1)R^{\coprod}_i(w_2,e_2)$ if and only if $w_1R_iw_2$ and $e_1\sim_ie_2$;
%\end{itemize}
%and define $V_{\coprod}(p)=\{(w,e)\in W\times E\mid M,w\models\sub(e)\}$.
\end{definition}

\begin{lemma}\label{lem:duality1}
	Let $\mb{M}=\langle \mb{F},\val{\cdot}\rangle$ be an \eiKm. Then $({\coprod}_\mathcal{E}\mb{F},\val{\cdot}_{\coprod})$ is also an \eiKm. Moreover, $({\coprod}_\mathcal{E}\mb{F})^+={\prod}_{\mathbb{E}_{\mathcal{E}}}(\mb{F}^+)$.
\end{lemma}
\begin{proof}
	Given \cite[Fact 23]{KP13},  \autoref{duality:basiclemma1} and  \autoref{duality:keylemma}, it remains to show that each $R^{\coprod}_i$ is upwards closed. This follows from each $R_i$ being upwards closed and the definition of $\leq^{\coprod}$.
\end{proof}

\begin{definition}
\label{def:eiKm:i:prob:distr:inter:struct}
For any \eiKf\ $\mb{F}=\langle S,\leq,(R_i)_{i\in\Ag}\rangle$, 
 any \eiKm\ $\mb{M}=\langle\mb{F},\val{\cdot}\rangle$, any $i$-probability distribution $P_i$ on $\mb{F}$ (see \autoref{def:i:prob:distr}), and any intuitionistic probabilistic event structure $\mathcal{E}=(E, (\sim_i)_{i\in\Ag}, (P_i)_{i\in \Ag}, \rmPhi, \pre, \sub)$ over $\mathcal{L}$, let us define the function $P^{\coprod}_i \ : \ S \times E \rightarrow \ \mathbb{R}^+$ by recursion on the order $\leq^{\coprod}$ as follows:
\begin{align}
P^{\coprod}_i(w,e) 
= 
\left(
\sum_{\varphi\in\rmPhi}P_i(e)\cdot P^\varphi_i(w)\cdot \pre(e\mid \varphi)
\right) 
- \sum_{v<w}
P_i^{\coprod}(v,e)
\end{align}
where 
\begin{align}
P_i^\varphi(w)=\sum_{
\begin{smallmatrix}
%w\in S\\
v\leq w
\end{smallmatrix}
}
\Big\{
P_i(v)
\ \Big | \
 \mb{M},v\models\varphi\text{ and }\ \mb{M},v\nvDash\psi 
\ \text{ for all } \psi\in\mathrm{mb}(\val{\varphi})\Big\}.
\end{align}
Recall that $\mathrm{mb}(a)$ denotes the multiset of the $\prec$-maximal elements of $\bfPhi$ $\prec$-below $a$ (see \autoref{def:mua-mba}).
\end{definition}

\begin{lemma}\label{lem:duality:product}
For every $\mb{M}$, $P_i$ and $\mathcal{E}$ as in \autoref{def:eiKm:i:prob:distr:inter:struct} and for every $w\in S$,
	\begin{align}
	P^\varphi_i(w)=((P_i)^+)^{\val{\varphi}}(w\!\downarrow).
	\end{align}
\end{lemma}
\begin{proof}                                                                                                                                                                                                                                                                                                                                                                                                                                                                                                                                                                                                                                                                                                                                                                                                                                                                                                                                                                                                                                                                                                                                                                                                                                                                                                                                                                                                                                                                                                                                                                                                                                                                                                                                                                                                                                                                                                                                                                                                                                                                                                                                                                                                                                                                                                                                                                                                                                                                                                                                                                                                                                                                                                                                                                                                                                                                                                                                                                                                                                                                                                                                                                                                                                                                                                                                                                                                                                                                                                                                                                                                                                                                                                                                                                                                                                                                                                                                                                                                                                                                                                                                                                                                                                                                                                                                                                                                                                                                                                             
	\begin{align*}
		 P^\varphi_i(w) 
		& = \sum_{
\begin{smallmatrix}
%w\in S\\
v\leq w
\end{smallmatrix}
}
\Big\{
P_i(v)
\ \Big | \
 \mb{M},v\models\varphi\text{ and }\ \mb{M},v\nvDash\psi 
\ \text{ for all } \psi\in\mathrm{mb}(\val{\varphi})\Big\}\\
		& = 
		\sum_{\begin{smallmatrix}
%w\in S\\
v\leq w
\end{smallmatrix}} \Big\{
		P_i(v)
		\ \Big| \ M,v\models \varphi
		\Big\}
		- 
		\sum_{\begin{smallmatrix}
%w\in S\\
v\leq w
\end{smallmatrix}}
		\Big\{
		P_i(v)
		\ \Big| \ M,v\models \bigvee_{\val{\psi}\in\mathrm{mb}(\val{\varphi})}\psi
		\Big\} 
		\\
		& = (P_i)^+(w\!\downarrow\land \val{\varphi}) - (P_i)^+(w\!\downarrow\land \bigvee_{\val{\psi}\in\mathrm{mb}(\val{\varphi})}\val{\psi})
		\tag{see \autoref{lem:rel:sem:mass:measure} and equation \eqref{eq:masstomeasure}}
		\\
		& = ((P_i)^+)^{\val{\varphi}}(w\!\downarrow).
		\tag{see \autoref{def:mua-mba} and equation \eqref{eq:def:mua}}
	\end{align*}
\end{proof}

\begin{lemma}\label{cor:probforprod}
For every $\mb{M}$, $P_i$ and $\mathcal{E}$ as in \autoref{def:eiKm:i:prob:distr:inter:struct},	$$(P_i^{\coprod})^+=((P_i)^+)'.$$
\end{lemma}
\begin{proof}
Recall that 
\begin{align}
((P_i)^+)': \mathsf{Min}_i(\prod_{\mb{E}} \mb{A}){\downarrow} 
& \to \mb{R}^+
\tag{see \autoref{def: prod F over E}}
\\
f & \mapsto \sum_{e\in E}\sum_{a\in \rmPhi}P_i(e)\cdot\mu_i^a(f(e))\cdot \overline{\pre}(e\mid a).
\notag
\end{align}
 By  \autoref{theo:finiteduality1} and \autoref{theo:finiteduality2}, it is enough to show that $P_i^{\coprod}=(((P_i)^+)')_{+}$. We show this by induction on the well-founded order $\leq^{\coprod}$. 

The induction hypothesis, for an element $f \in \mathsf{Min}_i(\prod_{\mb{E}} \mb{A}){\downarrow}  $ is 
\begin{align*}
((P_i)^+)'(f)=
 (P_i^{\coprod})^+(f)=
 \sum_{v<w}P_i^{\coprod}(v,e).\tag{\texttt{IH}$_f$}
\end{align*}
The case where $f=\bot$ is trivially true.
%\redbf{to give the induction hypothesis and the base case.}
Notice that the element $(w,e){\downarrow}$ corresponds to the map $g_{(w,e)} \ : \ E \rightarrow S $ such that $g_{(w,e)}(e)=w\!\downarrow$ and $g_{(w,e)}(e')=\varnothing$ for every $e'\neq e$. Hence, we have:
\begin{align*}
 ((P_i)^+)'(g_{(w,e)})
 & =\sum_{\varphi\in\rmPhi}P_i(e)\cdot (P_i^+)^{\val{\varphi}}(w\!\downarrow)\cdot \overline{\pre}(e\mid \val{\varphi})\\ 
 \tag{\autoref{lem:duality:product} and \eqref{eq:def:over:pre}} 
 & =\sum_{\varphi\in\rmPhi}P_i(e)\cdot P_i^\varphi(w)\cdot \pre(e\mid \varphi)
\end{align*}
 
Notice that 
\begin{align}
(((P_i)^+)')_{+}((w,e))=((P_i)^+)'(g_{(w,e)})-((P_i)^+)'(f)
\tag{see \autoref{def:mui:minus}}
\end{align}
with
 $f(e)=w{\downarrow}\setminus\{w\}$ and $f(e')=\varnothing$ for $e'\neq e$. 
 
Notice that $f < g$. Hence, by the induction hypothesis on $f$, we have 
 $$((P_i)^+)'(f)=
 (P_i^{\coprod})^+(f)=
 \sum_{v<w}P_i^{\coprod}(v,e).$$ 
 Hence, we get
 \begin{align*}
 (((P_i)^+)')_{+}((w,e))&=((P_i)^+)'(g_{(w,e)})-\sum_{v<w}P_i^{\coprod}(v,e)
 \\
 &=\sum_{\varphi\in\rmPhi}P_i(e)\cdot P_i^\varphi(w)\cdot \pre(e\mid \varphi)-\sum_{v<w}P_i^{\coprod}(v,e)
 \\
 &=P_i^{\coprod}((w,e)).
 \tag{see \autoref{def:eiKm:i:prob:distr:inter:struct}}
 \end{align*}
 
\end{proof}

\subsection{Dualizing the updated APE structures}
\label{ssec:dualityupdate}
In the present section, we introduce the generalization of the construction of the update model presented in Section \ref{ssec:epist:update:classic} and show that it dualizes the construction of the updated APE structure presented in Section \ref{ssec: abstract charact i minimal els pseudo quotient}.

\begin{definition}
For any \eiKf\ $\mb{F}=\langle S,\leq,(R_i)_{i\in\Ag}\rangle$, any \eiKm\ 
$\mb{M}=\langle\mb{F},\val{\cdot}\rangle$ and any intuitionistic probabilistic event structure
 $\mathcal{E}=(E, (\sim_i)_{i\in\Ag}, (P_i)_{i\in \Ag}, \rmPhi, \pre, \sub)$  over $\mathcal{L}$, let
\begin{align*}
pre : E  &\rightarrow \mathcal{L} \\
e & \mapsto \bigvee 
\left\{ \phi \in \rmPhi \mid \pre(e\mid\phi)\neq 0 \right\}.
\end{align*}
\end{definition}

\begin{definition}[Updated intuitionistic structure]
	\label{def:upd:intui:struct}
	For any \eiKm \ $\mb{M}=\langle S,\leq,(R_i)_{i\in\Ag},\val{\cdot} \rangle$ and any
	intuitionistic probabilistic event structure  
	$\mathcal{E}=(E, (\sim_i)_{i\in\Ag}, (P_i)_{i\in \Ag}, \rmPhi, \pre, \sub)$
	over $\mathcal{L}$ 
	(see \autoref{def:i:prob:distr}), let the updated intuitionistic structure of $\mb{M}$ and $\mc{E}$ be the tuple:
	
	$$\mb{M}^\mathcal{E}:=\langle S^\mathcal{E},\leq^{\mathcal{E}},(R^{\mathcal{E}}_i)_{i\in\Ag},\val{\cdot}^{\mathcal{E}}\rangle$$ 
	where 
	
	\begin{itemize}
		\item
		$S^\mathcal{E}=\{(w,e)\in \coprod_{|E|}S\mid \mathbb{M},w\models pre(e)\}$, %clearly, $\coprod_{|E|} S$ can be identified with $$. %S' = \{ (s,e) \in S \times E \mid pre(e\mid s)\neq 0 \} $
		\item $\leq^\mathcal{E}=\leq^{\coprod}\ \cap\ (S^\mathcal{E}\times S^\mathcal{E})$,
		\item 
		  $R^\mathcal{E}_i=R^{\coprod}_i\ \cap\ (S^\mathcal{E}\times S^\mathcal{E})$ for each $i\in\Ag$,
		\item $\val{\cdot}_{\mathcal{E}} : \mathsf{AtProp} \rightarrow \P S$ is defined by
		$$ \val{p}^{\mathcal{E}} : = \left\{(w,e)\in S^\mathcal{E}\mid \mathbb{M},w\models\sub(e)(p)\right\}$$ for every $p\in \mathsf{AtProp}$.
	\end{itemize}
	For any \eiKm \ $\mb{M}=\langle\mb{F},\val{\cdot}\rangle$,
	let $$\mb{F}^\mathcal{E}:=\langle S^\mathcal{E},\leq^{\mathcal{E}},(R^{\mathcal{E}}_i)_{i\in\Ag}\rangle.$$
	
	%\begin{itemize}
	%	\item $(w_1,e_1)\leq^{\coprod}(w_2,e_2)$ if and only if $w_1\leq w_2$ and $e_1=e_2$;
	%	\item $(w_1,e_1)R^{\coprod}_i(w_2,e_2)$ if and only if $w_1R_iw_2$ and $e_1\sim_ie_2$;
	%\end{itemize}
	%and define $V_{\coprod}(p)=\{(w,e)\in W\times E\mid M,w\models\sub(e)\}$.
\end{definition}

\begin{lemma}
	 If $\mb{M}=\langle\mb{F},\val{\cdot}\rangle$ is an  \eiKm, then so is  $\mathbb{M}_{\mathcal{E}}$. Moreover, $(F^\mathcal{E})^+=(F^+)^{\mathbb{E}^{\mathcal{E}}}$.
\end{lemma}
\begin{proof}
	It follows from \cite[Definition 22,Fact 23]{KP13} and Lemma \ref{lem:duality1}.
\end{proof}

\begin{definition}
	\label{def:eiKm:i:prob:distr:upda:struct}
	For any \eiKf\ $\mb{F}=\langle S,\leq,(R_i)_{i\in\Ag}\rangle$, 
	any \eiKm\ $\mb{M}=\langle\mb{F},\val{\cdot}\rangle$, any $i$-probability distribution $P_i$ on $\mb{F}$ (see \autoref{def:i:prob:distr}), and any intuitionistic probabilistic event structure $\mathcal{E}=(E, (\sim_i)_{i\in\Ag}, (P_i)_{i\in \Ag}, \rmPhi, \pre, \sub)$ over $\mathcal{L}$, the updated $i$-probability distribution $P^{\mathcal{E}}_i \ : \ S^\mathcal{E} \rightarrow \ ]0,1]$  is defined as follows:
	\begin{align}
	P_i^\mathcal{E}(w,e):=\frac{P^{\coprod}_i(w,e)}{\sum\{P^{\coprod}_i(w',e')\mid (w',e')R^\mathcal{E}_i(w,e)\}}\end{align}
	where $P^{\coprod}_i$ is as for \autoref{def:eiKm:i:prob:distr:inter:struct}.
\end{definition}

\begin{lemma}
For every $\mb{M}$, $P_i$ and $\mathcal{E}$ as in \autoref{def:eiKm:i:prob:distr:upda:struct},	$$(P_i^{\mathcal{E}})^+=((P_i)^+)^{\mathbb{E}_\mathcal{E}}.$$
\end{lemma}
\begin{proof}
By \autoref{cor:eiKf:i:minimal:eq:cell} and \autoref{lem:duality1} the $i$-minimal elements of $(\mathbb{M}^\mathcal{E})^+$ are the  equivalence cells of $R_i$. Now, let $g\in(\mathbb{M}^\mathcal{E})^+$, $f$ the $i$-minimal element above $g$ and  $(w,e)\in g$. By \autoref{cor:probforprod} $\sum\{P^{\coprod}_i(w',e')\mid (w',e')R^\mathcal{E}_i(w,e)\}=(P_i^{\coprod})^+(f)$ and $\sum_{(w',e')\in g}P^{\coprod}_i(w',e')=(P_i^{\coprod})^+(g)$. Therefore:
\begin{align*}
((P_i)^+)^{\mathbb{E}_\mathcal{E}}(g) &=\frac{(P_i^{\coprod})^+(g)}{(P_i^{\coprod})^+(f)}\\&=\frac{\sum_{(w',e')\in g}P^{\coprod}_i(w',e')}{\sum\{P^{\coprod}_i(w',e')\mid (w',e')R^\mathcal{E}_i(w,e)\}}\\&=\sum_{(w',e')\in g}\frac{P^{\coprod}_i(w',e')}{\sum\{P^{\coprod}_i(w',e')\mid (w',e')R^\mathcal{E}_i(w,e)\}}\\&= \sum_{(w',e')\in g}P_i^{\mathcal{E}}(w,e)\\&=(P_i^{\mathcal{E}})^+(g).
\end{align*}
\end{proof}

\subsection{Relational semantics for IPDEL}
\label{ssec:relsemipdel}

\begin{definition}
	An IPDEL-model is a structure 
	$\mb{N} = \left\langle \mb{M}, (P_i)_{i\in \Ag}\right\rangle$ such that $\mb{M} = \left\langle S,\leq,(R_i)_{i\in\Ag}, \val{\cdot}\right\rangle$ is an epistemic intuitionistc Kripke model, and $P_i$ is a probability distribution over $S$ for every $i\in \mathsf{Ag}$. For every IPDEL-model $\mathbb{N}$ and every event structure $\mathcal{E}$, we let $\mathbb{N}^\mathcal{E}=\left\langle \mb{M}^\mathcal{E}, (P^\mathcal{E}_i)_{i\in \Ag}\right\rangle$ (cf.\ Definitions \ref{def:upd:intui:struct} and \ref{def:eiKm:i:prob:distr:upda:struct}).
\end{definition}

It can be verified straightforwardly that for every IPDEL-model $\mathbb{N}$ and every event structure $\mathcal{E}$, the structure $\mathbb{N}^\mathcal{E}$ is an IPDEL-model.

\begin{definition}[Semantics of IPDEL]
For every IPDEL-model
$\mb{N} = \left\langle \mb{M}, (P_i)_{i\in \Ag}\right\rangle$ where $\mb{M} = \left\langle S,\leq,(R_i)_{i\in\Ag}, \val{\cdot}\right\rangle$ the IPDEL-formulas are interpreted on $\mathbb{N}$ as follows:
\begin{align*}
\mb{N},s \models \bot & \textcolor{white}{\quad \text{iff} \quad} never
%\\
%\mb{M},s \models \top & \textcolor{shite}{\quad \text{iff} \quad} alsays
\\
\mb{N},s \models p & \quad \text{iff} \quad 
s \in \val{p}
\\
\mb{N},s \models \phi \wedge \psi & \quad \text{iff} \quad 
\mb{N},s \models \phi\quad \text{ and }\quad \mb{N},s \models \psi 
\\
\mb{N},s \models \phi \vee \psi & \quad \text{iff} \quad 
\mb{N},s \models \phi \quad \text{ or }\quad \mb{N},s \models \psi 
\\
\mb{N},s \models \phi \rightarrow \psi & \quad \text{iff} \quad 
\mb{N},s' \models \phi \quad\text{ implies }\quad \mb{N},s' \models \psi\text{ for every } s'\leq s 
\\
\mb{N},s \models \lozenge_i \phi  & \quad \text{iff} \quad \text{there exists } s' R_i s \text{ such that } 
\mb{N},s' \models \phi 
\\
\mb{N},s \models \Box_i \phi  & \quad \text{iff} \quad \mb{N},s' \models \phi \quad \text{ for all } s' ({\geq}\circ R_i) s 
\\
\mb{N},s \models \langle\mathcal{E}, e\rangle \phi  
& \quad \text{iff} \quad \mb{N},s \models pre(e)
\quad  
\text{ and } \quad \mb{N}^{\mathcal{E}}, (s,e) \models \phi 
\\
\mb{N},s \models [ \mathcal{E}, e ] \phi  
& \quad \text{iff} \quad 
\mb{N},s \models pre(e)
\quad  
\text{ implies } \quad \mb{N}^{\mathcal{E}}, (s,e) \models \phi
\\
\mb{N},s \models \left(\sum_{k = 1}^n\alpha_k · \mu_i(\varphi)\right) \geq \beta  
& \quad \text{iff} \quad \sum_{k=1}^{n} \alpha_k · (P_i)^+(\val{\varphi}\cap R_i[s]) \geq \beta.  
\end{align*} 
\end{definition}

Recalling that in epistemic intuitionistic Kripke frames, and hence on IPDEL-models, the relations $R_i$ are both upwards and downwards closed, this implies that the seventh clause in the definition above can be simplified as follows:

$$\mb{N},s \models \Box_i \phi   \quad \text{iff} \quad \mb{N},s' \models \phi \quad \text{ for all } s' R_i s. $$

\section{Case study: Decision-making under uncertainty}
\label{sec:ArtExample}
In the present section, we illustrate the relational semantic update process described in Section \ref{sec:relsem} by means of a case study that involves the assessment of the likelihood of a socially constructed event (a bankruptcy), taking place at some point in the future.

The focal feature of the case study is that this assessment depends to a greater extent  on  the actions, beliefs and expectations of the agents than on factual information.

%The relational semantics of IPDEL is going to be introduced in an other article \cite{CFPT-relsemPDEL}. 
%Indeed it necessitates extensive discussions over intuitionistic probabilities and their definition over intuitionistic models. However, in this section, we present an example that illustrate 
%our motivations to develop probabilistic dynamic epistemic logic over an intuitionistic base. 

In what follows, we first present the case study informally, and then we introduce a simplified formalization of the problem using probabilistic epistemic intuitionistic Kripke models and probabilistic intuitionistic epistemic event structures. 

\subsection{Informal presentation}
%\paragraph*{Case study: `Is this painting a good investement?'}
%As discussed in the introduction, intuitionistic epistemic logic provides a suitable framework to model procedural truth. 
%In this section, we consider the art industry from the point of view of an investor. 
Around 1950, there was a small businessman $w$ in Amsterdam whose main business was to sell the products of foreign textile manufacturers to Dutch clothing firms. Like most small businessmen in Amsterdam at the time, he banked with the Amsterdamsche Bank (which later became the present ABN AMRO). 

One day, $w$ received an invitation to lunch with one of the directors of that bank. This invitation puzzled him a great deal, because he did not know this director personally, and a small businessman like him usually only dealt with bank employees at much lower levels.
However, he accepted the invitation and showed up for the lunch at the top floor of the bank's headquarters, in the city centre.

During the copious lunch, the bank director talked about all kinds of general subjects and asked $w$'s opinion about the economic climate in Amsterdam. Rather than being flattered, $w$ found it hard to imagine he was invited to provide opinions about matters the bank knew better than he. When the dessert was served, the banker mentioned aside some other matter the name of a certain Amsterdam firm $f$, which was an important client of $w$. This firm, the bank director said, was doing very well under the present solid leadership.

The small businessman realised that this must have been the point of the whole lunch. And if this large bank went to so much effort to increase the confidence of one small businessman in this firm, it must have been very important to the bank that $w$ believed that $f$ was doing well.

The small businessman said he wanted to wash his hands, although coffee still needed to be served, but instead of walking to the bathroom he ran down the stairs and on the street to find a telephone booth and call to the office to stop all deliveries to $f$ and also claim back any supplies that had already been delivered.

Two weeks later, $f$ went bankrupt and it turned out that the bank not only was its major creditor but also had preferential right to sell off any stocks in the possession of $f$ to pay back the debt to the bank before other creditors would be satisfied.

\subsection{Analysis of the situation}

Let $\Profit$ be the following proposition:
\begin{center}
	\textit{`Firm $f$ will bankrupt within a month.'}
\end{center}
Notice that, while being two-valued, intuitionistic logic allows for   $\Profit$ to be either true, or false, or undecided in a model, and the availability of the third option seems to  adequately reflect real-life situations. Indeed, there is a strict judicial procedure which establishes the truth of  $\Profit$, and when this procedure is not (yet) in place it seems reasonable to not assign it a truth value.

%and hence after the event $i$ is more confident than before about the success of her investment in five years.
Accordingly, the sum of the probability attributed to $\Profit$ by $w$ and the probability attributed to $\lnot\Profit$ by $w$ does not need to be $1$.

For simplicity we regard everything which happened from the invitation to the banker's utterance about firm $f$ as one single event. We also propose that the uncertainty of $w$ concerns how to interpret this event, and very much simplifying this story, the two mutually inconsistent interpretations of this event are  
\begin{center}
	\textit{$e_1$:`The banker is trying to manipulate my opinions.'}
\end{center}
\begin{center}
	\textit{$e_2$:`The banker only wants to exchange information.'}
\end{center}
The uncertainty of $w$ about how to interpret the event is encoded in the shape of the event structure, which consists of two states, corresponding to $e_1$ and $e_2$ above respectively, to each of which $w$ assigns his (subjective) probability.

For the sake of illustrating how the substitution map works and to simplify the subsequent treatment we also include the following atomic proposition $\Gallery$ in our language, the intended meaning of which is:
\begin{center}
	\textit{`The banker is manipulative.'}
\end{center}

\subsection{Formalization: initial model and event structure}
%\marginnote{Do we need the proposition exhibit and apply? Do we need the other agents? They are irrelevant to the example.}

Let the set of atomic propositions be $\AtProp := \{\Profit, \Gallery\}$ as discussed above.

\paragraph*{Initial model.}
In the formalization discussed below, we only consider the viewpoint of agent $w$; hence, in the model and the event structure we specify only the subjective probabilities of agent $w$. 
The initial model is
$$\mb{M} :=\left\langle S, \leq, \sim_w, P_w, \val{\cdot}\right\rangle$$ 
with:
\begin{itemize}
\item $S := \{ s_0,s_1,s_2\}$,
\item $\leq \ := \{(s,s) \mid s\in S\} \cup \{ (s_1,s_0), (s_2,s_0)\}$,
\item $\sim_w \ := S \times S$,
\item $P_w : S \rightarrow \ ]0,1]$ with 

$$P_w(s_0) := 0.1,\qquad P_w(s_1):= 0.1, \qquad P_w(s_2):= 0.8,$$

\item $\val{\cdot} : \AtProp \rightarrow \P S $ is such that $\val{\Profit} := \{ s_1\}$  and $\val{\Gallery} := \bot$.
\end{itemize}

This model represents a situation in which $w$ has no additional information about the financial health of firm $f$. Hence, we assume that the probability assigned by $w$ to each state of the model reflects the average risk of bankruptcy of firms in that industry during that period. For $w$ to be willing to do business with $f$ it is not just enough that $f$ does not have a higher probability of bankruptcy than the average firm, but also the probability of being in an uncertain state should be low. The model $\mb{M}$ is drawn in Figure \ref{fig:ex:initial-model-MArt}. 

%\begin{figure}
%\xymatrix{
%&&&&s_1,\ 0.3:\ p && s_2, 0.1:\ \neg p
%\\
%&&&&&s_0,\ 0.6 \ar[ul]\ar[ur]
%}

%\caption{Initial model $\mb{M}$: the agent is sceptical about $p$}
%\label{fig:ex:initial-model-MArt}
%\end{figure}

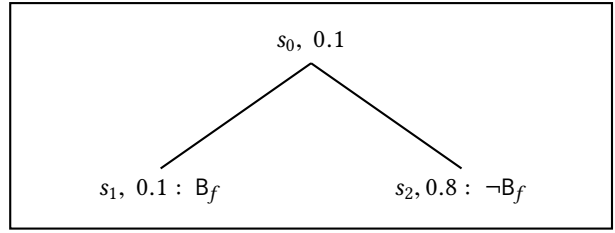
\begin{figure}
\begin{center}
	\begin{tikzpicture}
	[minimum size=6mm, inner sep=0mm,
	place/.style={rectangle,thick},
	transition/.style={rectangle,draw=black!50,fill=black!20,thick}]
	\node at (8,-2) [place]{$s_2, 0.8:\ \neg \Profit$};
	\node at (4,-2) [place]{$s_1,\ 0.1:\ \Profit$};
	\node at (6, 0)[place]{$s_0,\ 0.1 $};
	\draw [-,thick] (6,-0.3) -- (4,-1.7);
	\draw [-, thick] (6,-0.3) -- (8,-1.7);
	\draw [thick] (2,-2.5) rectangle (10,0.5);
	\end{tikzpicture}
\end{center}
\caption{Initial model $\mb{M}$}
\label{fig:ex:initial-model-MArt}
\end{figure}

\paragraph*{Event structure.}
We consider the following pointed event structure:
$$(\mathcal{E},e_1) := (E, \sim_w, P_w, \rmPhi, \pre, \sub) $$
where 
\begin{itemize}
\item $E := \{e_1, e_2\}$,
\item $\sim_w \ := E\times E$,
\item $P_w(e_1)=0.95$ and $P_w(e_2)=0.05$,

\item $\rmPhi = \{ \top, \Profit, \neg \Profit \}$,
\item $\pre : E \times \rmPhi \rightarrow [0,1]$ is given in Figure \ref{array:mapPRE}.

%\begin{figure}
% \begin{minipage}[b]{.46\linewidth}
%\centering 
%\xymatrix{
%&& e_1,\ 1 & e_2,\ 0.9
%\\
%&&& e_3,\ 0.1 \ar@{-}[u]_i
%}
%\caption{Event structure $\mb{E}$}
%\label{fig:ex:event-struct-MArt}
% \end{minipage} \hfill
\item the definition of the map $\sub : E \times \{\Gallery\} \rightarrow \mathcal{L}$ is given in Figure \ref{array:mapSUB},
\end{itemize}
where $e_1$ and $e_2$ correspond to the two interpretations of the event discussed in the previous section.  The event structure $\mb{E}$ is partially represented in Figure \ref{fig:ex:event-struct-MArt}.

\begin{figure}
\begin{center}
	\begin{tikzpicture}
	[minimum size=6mm, inner sep=0mm,
	place/.style={rectangle,thick},
	transition/.style={rectangle,draw=black!50,fill=black!20,thick}]
	\node at (8,0) [place]{$e_1$};
	\filldraw  (8, 0.3) circle (2pt);
	\filldraw  (10, 0.3) circle (2pt);
	\node at (8,0.6) [place]{$0.95$};
	%\node at (6,0) [place]{$e_1,\  1$};
	\node at (10, 0.6)[place]{$0.05 $};
	\node at (10,0) [place]{$e_2$};
	%\draw [->,thick] (6,-1.7) -- (4,-0.3);
	%\draw [-, thick] (8,-1.7) -- (8,-0.3);
	%\draw [thick] (5.5,-0.5) rectangle (6.5,0.5);
	\draw [thick] (7,-0.5) rectangle (11,1);
	\end{tikzpicture}
\end{center}
\caption{Event structure $\mb{E}$}
\label{fig:ex:event-struct-MArt}
%\begin{center}
\end{figure}
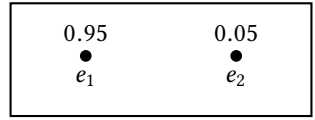

By stipulating that $P_w(e_1)=0.95$ and $P_w(e_2)=0.05$, we indicate that $w$ believes that it is far more likely that the banker is trying to manipulate his opinion on $f$.

The map $\pre$ provides the objective probability $\pre(e\mid \phi)$ of each event $e\in E$ happening when one assumes that  the formula $\phi\in \rmPhi$ holds. Each line of Figure \ref{array:mapPRE} gives the probability distribution $\pre(\bullet \mid \phi) : E \times [0,1]$ for each $\phi\in\rmPhi$.
The values in Figure \ref{array:mapPRE} are based on the following assumptions:
\begin{itemize}
\item If we consider the row where $\phi = \top$, which corresponds to the state in which the bankruptcy of $f$ is undetermined, it is reasonable to assume that the probability of $e_1$, namely the banker trying to manipulate $w$'s opinion on $f$, is significantly higher than that of $e_2$.
\item If we consider the row where  $\phi = \Profit$, which corresponds to the state in which $f$ is going to be bankrupt within a month, it is reasonable to regard $e_1$ as almost certain.
\item If we consider the row where $\phi = \neg \Profit$, which corresponds to the state in which $f$ is financially healthy then it is reasonable to assign a very low probability to the event in which the banker wants to manipulate $w$'s opinion about $f$, since the banker has nothing to gain from it. 
\end{itemize}

 \begin{figure}
 \begin{minipage}[b]{.46\linewidth}
	\centering 
	%\begin{table}
	\renewcommand{\arraystretch}{1.3}
\begin{tabular}{|p{1.5cm}<{\centering}|p{1cm}<{\centering}|p{1cm}<{\centering}|}
\cline{2-3}

\multicolumn{1}{c|}{}
& $e_1$ & $e_2$ 
\\
\hline
$\top$ & $0.8$ & $0.2$ 
\\  
\hline
$\Profit$ & $0.99$ & $0.01$
\\ 
\hline 
$\neg \Profit$ & $0.05$ & $0.95$ 
\\  
\hline
\end{tabular}
\caption{The map $\pre$}
\label{array:mapPRE}
%\end{center}
%\end{table}
%\renewcommand{\arraystretch}{1}
\end{minipage} \hfill
 \begin{minipage}[b]{.46\linewidth}
\centering 
\renewcommand{\arraystretch}{1.3}
\begin{tabular}{|p{1.5cm}<{\centering}|p{1cm}<{\centering}|p{1cm}<{\centering}|}
\cline{2-3}

\multicolumn{1}{c|}{}
& $e_1$ & $e_2$ 
\\
\hline
$\Gallery$ & $\top$ & $\bot$
%\\ 
%\hline 
%$\mathtt{Exhibit}$ & $\top$ & $\bot$ & $\bot$
\\  
\hline
\end{tabular}

\caption{The map $\sub$}
\label{array:mapSUB}
\renewcommand{\arraystretch}{1}
\end{minipage}
%\end{table}
%\renewcommand{\arraystretch}{1}
\end{figure}

\begin{remark}
The poset $\rmPhi$ ordered by logical implication is a tree and is drawn in Figure \ref{fig:Phi:exampleArt}.

%\begin{figure}
%\xymatrix{
%&&&&&&&\top
%\\
%&&&&&&p \ar[ur] && \neg p \ar[ul]
%}
%\caption{The partial order given by $(\rmPhi, \rightarrow)$}
%\label{fig:Phi:exampleArt}
%\end{figure}

\begin{figure}
	\begin{center}
		\begin{tikzpicture}
		[minimum size=6mm, inner sep=0mm,
		place/.style={rectangle,thick},
		transition/.style={rectangle,draw=black!50,fill=black!20,thick}]
		\node at (6.5,-2) [place]{$\neg \Profit$};
		\node at (4,-2) [place]{$\Profit$};
		\node at (5.25, 0)[place]{$\top$};
		\draw [->, thick] (4,-1.7) -- (5.15,-0.3);
		\draw [->, thick] (6.5,-1.7) -- (5.35,-0.3);
		\draw [thick] (2.5,-2.5) rectangle (8,0.5);
		\end{tikzpicture}
	\end{center}
	\caption{The partial order given by $(\rmPhi, \rightarrow)$}
	\label{fig:Phi:exampleArt}
\end{figure}
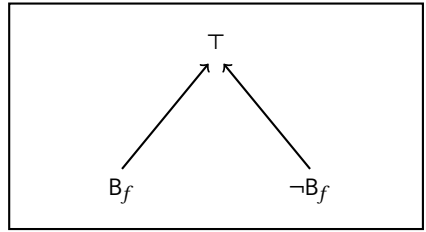

\end{remark}

%\subsection{\redbf{To do the update on the algebra}}

\subsection{Updated model}
\label{ssec:updated:model}

In this section, we show how the  initial model described in the section above is updated with the event structure. The updated  model 
$$\mb{M}^{(\mathcal{E},e_1)} :=\left\langle S', \leq', \sim'_w, P'_w, \val{\cdot}'\right\rangle$$  
is defined as follows:
\begin{itemize}
\item $S' := S\times E$,
\item  $(s,e) \leq' (s',e')$ iff $s \leq s' $ and $e=e'$ for all $(s,e),(s',e')\in S'$,
\item  $(s,e) \sim_w' (s',e')$ iff $s \sim_w s' $ and $e\sim_w e'$ for all $(s,e),(s',e')\in S'$,
\item the map
$P'_w $ is shown in \Cref{fig:ex:update-model-MArt}, where the actual values are rounded off,
%\begin{align*}
%P'_j : \quad \; S' &\rightarrow\ ]0,1]
%\\
%(s,e)&  \mapsto  \frac{P^{\coprod}_i(s,e)}{\sum \left\{P^{\coprod}_i(s',e') \ \middle| \ 
%\mid
%(s,e) \sim_i (s',e') \right\}}. 
%\end{align*}
%where $P^{\coprod}_i$ is defined by recursion on the inverse order of $S'$ as follows:
%$$P^{\coprod}_i(s,e) = (\sum_{\alpha\in\Phi}P_i(e)\cdot P^\alpha_i(s)\cdot \pre(e\mid \alpha)) -\sum\{P_i^{\coprod}(t,e)\mid s< t\}$$ and $$P_i^\alpha(s)=\sum
%\left\{P_i(t)
%\ \middle| \
% t\in S,\ s\leq t,\ \mb{M},t\models\alpha\text{ and, for all } \beta\in\rmPhi, %\text{ if } \beta\neq\alpha \text{ and }
% \beta\to\alpha \text{ then } \mb{M},t\nvDash\beta \right\}$$
\item the map $\val{\cdot}' : \AtProp \rightarrow \P S'$ is defined as follows:
\begin{align*}
\val{\Profit}'  :\! &=  \val{\Profit} \times E;
\\
\val{\Gallery}' :\! &=   
\left( \val{\sub (e_1,\Gallery)}\times \{ e_1 \}\right) 
\cup \left( \val{\sub (e_2,\Gallery)}\times \{e_2\}\right)\\
& = \{ (s_0,e_1), (s_1,e_1), (s_2,e_1) \}.
\end{align*}
%\begin{itemize}
%\item $\val{\Invest}':= \val{\Invest} \times E = \{ (s_1,e_1), (s_1,e_2), (s_1,e_3) \}$, 
%\item $\val{\Gallery':=   \val{\sub (e_1,\mathtt{Apply})}\times \{ e_1 \} \cup \val{\sub (e_2,\mathtt{Apply})}\times \{e_2\} \cup \val{\sub (e_3,\mathtt{Apply})} )\times \{e_3\} = (\val{\top}, \val{\top}, \val{\bot} ) = $, 
%\item $\val{\mathtt{Exhibit}}':= (\val{\sub (e_1,\mathtt{Exhibit})}, \val{\sub (e_2,\mathtt{Exhibit})}, \val{\sub (e_3,\mathtt{Exhibit})} ) = (\val{\top}, \val{\bot}, \val{\bot} )$.
%\end{itemize}

\end{itemize}

The updated model $\mb{M}^{(\mb{E},\mathtt{e_1})}$ is drawn in Figure \ref{fig:ex:update-model-MArt}.

%One can see that the probability assigned to $p$ by the agent ($P(p)=P(t_p,s_1)+P(t_f,s_1)=0.35$) is greater than the probability he assigned to it before the event ($P(p)=P(s_1)=0.3$).

%\begin{figure}
%\def\g#1{\save
%	[].[dr]!C="g#1"*[F]\frm{}\restore}%

%\xymatrix{
%\g1 (e_1, s_1), 0.596 : p & (e_1, s_2), 0.007 : \neg p
%& \g2 (e_2, s_1), 0.324 : p & (e_2, s_2), 0.07 : \neg p
%&
%(e_3, s_1), 0.006 : p & (e_3, s_2), 0.012 : \neg p
%\\
% (e_1, s_0), 0.397 \ar[u]\ar[ur] && (e_2, s_0), 0.54 \ar[u]\ar[ur] && (e_3, s_0), 0.048\ar[u]\ar[ur] &
%}

%\caption{Updated model $\mb{M}^\mb{E} $}
%\label{fig:ex:update-model-MArt}
%\end{figure}

\begin{figure}
\begin{center}
	\begin{tikzpicture}
	[minimum size=6mm, inner sep=0mm,
	place/.style={rectangle,thick},
	transition/.style={rectangle,draw=black!50,fill=black!20,thick}]
	%\node at ( 1.5,0) [place]{\scriptsize$(e_1, s_2), 0.007 : \neg \Profit$};
	%\node at (-1,0) [place]{\scriptsize$(s_1, e_1), 0.596 : \Profit$};
	%\node at (0.25, -2)[place]{\scriptsize$(s_0, e_1), 0.397$};
	\node at ( 8,-2) [place]{\scriptsize$(s_2, e_1), 0.15 : \neg \Profit,\Gallery$};
	\node at (4,-2) [place]{\scriptsize$(s_1, e_1), 0.3848 : \Profit,\Gallery$};
	\node at (6, 0)[place]{\scriptsize$( s_0, e_1), 0.31 :\Gallery$};
	\node at (15,-2) [place]{\scriptsize$(s_2, e_2), 0.15 : \neg \Profit,\lnot\Gallery$};
	\node at (11,-2) [place]{\scriptsize$( s_1, e_2), 0.0002 : \Profit,\lnot\Gallery$};
	\node at (13, 0)[place]{\scriptsize$( s_0, e_2), 0.005 :\lnot\Gallery$};
	%\draw [->, thick] (0.25,-1.7) -- (-1,-0.3);
	%\draw [->, thick] (0.25,-1.7) -- (1.5,-0.3);
	\draw [-, thick] (6,-0.3) -- (4,-1.7);
	\draw [-, thick] (6,-0.3) -- (8,-1.7);
	\draw [-, thick] (13,-0.3) -- (11,-1.7);
	\draw [-, thick] (13,-0.3) -- (15,-1.7);
	%\draw [thick] (-2.5,-2.5) rectangle (2.9,0.5);
	\draw [thick] (2,-2.5) rectangle (17,0.5);
	\end{tikzpicture}
\end{center}
\caption{Updated model $\mb{M}^\mb{E} $ }
\label{fig:ex:update-model-MArt}
\end{figure}

As expected, the fact that $w$ assigns a much greater probability to $e_1$ than $e_2$ implies that the probabilistic weight of the model above is concentrated among the three leftmost states. Of these three states, the weight is shared almost equally between the two in which $\Profit$ is either true or undecided, which reverses the subjective probability assigned in the initial model. This reversal captures $w$'s decision to abruptly stop all deliveries to $f$.

\subsection{Syntactic inference of a property of the afternoon event}

In the present section, we will use the Hilbert style presentation of IPDEL to derive the formula \eqref{eq:thres}. This formula  gives the threshold of reasonable optimism  which enables $w$ to revise his subjective probability about $\Profit$ after the afternoon event $(\mathcal{E},e_1)$ takes place. Specifically, the probability $w$ assigns to $\Profit$ should not be less than $19.8$ times that he assigns to $\lnot\Profit$ in order for the event $(\mathcal{E},e_1)$ as specified in the sections above to be enough for $w$ to revert his judgment about $\Profit$.

\begin{proposition}\label{prop:proofex} The formula
\begin{equation}\label{eq:thres}
(19.8\mu_w(\Profit)>\mu_w(\lnot\Profit))\leftrightarrow[\mathcal{E},e_1](\mu_w(\Gallery\land\Profit)>\mu_w(\lnot\Gallery\land\lnot\Profit)),
\end{equation}
where $\alpha\mu_i(\varphi)>\beta\mu_i(\psi)$ is shorthand for $(\beta\mu_i(\psi)\geq\alpha\mu_i(\varphi))\to\bot$, is derivable in IPDEL.
\end{proposition}
\begin{proof}
In order to show the equivalence \eqref{eq:thres}, we will use the IPDEL axioms to equivalently rewrite its right-hand side into its left-hand side.

\begin{align*}
	& [\mathcal{E},e_1](\mu_w(\Gallery\land\Profit)>\mu_w(\lnot\Gallery\land\lnot\Profit))
	\\
	\text{iff } \quad
	& [\mathcal{E},e_1]\left((\mu_w(\lnot\Gallery\land\lnot\Profit)\geq\mu_w(\Gallery\land\Profit))\to\bot\right)  \tag{notation for $>$}\\
	\text{iff } \quad
	& \langle \mathcal{E},e_1\rangle(\mu_w(\lnot\Gallery\land\lnot\Profit)\geq\mu_w(\Gallery\land\Profit))\to\langle \mathcal{E},e_1\rangle\bot  \tag{I11 in Table \ref{table:IPDEL}}
	\\
	\text{iff } \quad
	& \langle \mathcal{E},e_1\rangle(\mu_w(\lnot\Gallery\land\lnot\Profit)\geq\mu_w(\Gallery\land\Profit))\to\bot 
	\tag{I6 in Table \ref{table:IPDEL}}\\
\end{align*}

%\begin{center}
%\begin{tabular}{cll}
%	&$[E,e_1](\mu_w(\Gallery\land\Profit)>\mu_w(\lnot\Gallery\land\lnot\Profit))$ &\\
%	iff & $[E,e_1]\left((\mu_w(\lnot\Gallery\land\lnot\Profit)\geq\mu_w(\Gallery\land\Profit))\to\bot\right)$ & (notation for $>$)\\
%	iff & $\langle E,e_1\rangle(\mu_w(\lnot\Gallery\land\lnot\Profit)\geq\mu_w(\Gallery\land\Profit))\to\langle E,e_1\rangle\bot$ & (I11 in Table \ref{table:IPDEL})\\
%	iff & $\langle E,e_1\rangle(\mu_w(\lnot\Gallery\land\lnot\Profit)\geq\mu_w(\Gallery\land\Profit))\to\bot$ & (I6 in Table \ref{table:IPDEL})\\
%\end{tabular}
%\end{center}

In what follows we focus on equivalently rewriting the antecedent of the implication above.

\begin{align*}
&\langle \mathcal{E},e_1\rangle(\mu_w(\lnot\Gallery\land\lnot\Profit)\geq\mu_w(\Gallery\land\Profit))
\\
\text{iff } \quad
&  \sum_{\begin{smallmatrix}
				e'\in E\\
				\phi \in \rmPhi
		\end{smallmatrix}}P_w(e')\cdot\pre(e'\mid\phi)\cdot\mu_w^\phi(\langle \mathcal{E},e'\rangle(\lnot\Gallery\land\lnot\Profit))
  \geq\sum_{\begin{smallmatrix}
		e'\in E\\
		\phi \in \rmPhi
		\end{smallmatrix}}P_w(e')\cdot\pre(e'\mid\phi)\cdot\mu_w^\phi(\langle \mathcal{E},e'\rangle(\Gallery\land\Profit))  \tag{I18 in Table \ref{table:IPDEL}}
\\
\text{iff } \quad
&   P_w(e_2)\cdot\pre(e_2\mid\lnot\Profit)\cdot\mu_w(\lnot\Profit)\geq P_w(e_1)\cdot\pre(e_1\mid\Profit)\cdot\mu_w(\Profit) \tag{by Lemma \ref{lem:proof:example}}
\\
\text{iff } \quad
&  0.05\cdot0.95\cdot\mu_w(\lnot\Profit)\geq 0.95\cdot 0.99\cdot \mu_w(\Profit) \tag{Definition of $(\mathcal{E},e_1)$}
\\
\text{iff } \quad
&   0.05\cdot\mu_w(\lnot\Profit)\geq0.99\cdot \mu_w(\Profit)
\tag{by Lemma \ref{lemma:measures:basic}}
\\
\text{iff } \quad
&  \mu_w(\lnot\Profit)\geq 19.8\mu_w(\Profit). 
\tag{by Lemma \ref{lemma:measures:basic}}
\end{align*}

%\begin{center}
%	\begin{tabular}{cll}
%		&$\langle E,e_1\rangle(\mu_w(\lnot\Gallery\land\lnot\Profit)\geq\mu_w(\Gallery\land\Profit))$ &\\
%		iff & $\sum_{\begin{smallmatrix}
%				e'\in E\\
%				\phi \in \rmPhi
%		\end{smallmatrix}}P_w(e')\cdot\pre(e'\mid\phi)\cdot\mu_w^\phi(\langle E,e'\rangle(\lnot\Gallery\land\lnot\Profit))$&\\
%	& $\geq\sum_{\begin{smallmatrix}
%		e'\in E\\
%		\phi \in \rmPhi
%		\end{smallmatrix}}P_w(e')\cdot\pre(e'\mid\phi)\cdot\mu_w^\phi(\langle E,e'\rangle(\Gallery\land\Profit))$ & (I18 in Table \ref{table:IPDEL})\\
%	iff & $P_w(e_2)\cdot\pre(e_2\mid\lnot\Profit)\cdot\mu_w(\lnot\Profit)\geq P_w(e_1)\cdot\pre(e_1\mid\Profit)\cdot\mu_w(\Profit)$ & ($\ast$)\\
%	iff &$0.05\cdot0.95\cdot\mu_w(\lnot\Profit)\geq 0.95\cdot 0.99\cdot \mu_w(\Profit)$ & (definition of $(E,e_1)$)\\
%	iff & $0.05\cdot\mu_w(\lnot\Profit)\geq0.99\cdot \mu_w(\Profit)$ & (by Lemma \ref{lemma:measures:basic})\\
%	iff &$\mu_w(\lnot\Profit)\geq 19.8\mu_w(\Profit)$. &  (by Lemma \ref{lemma:measures:basic})
%	\end{tabular}
%\end{center}

Hence,
\begin{align*}
		 & \langle \mathcal{E},e_1\rangle(\mu_w(\lnot\Gallery\land\lnot\Profit)\geq\mu_w(\Gallery\land\Profit))\to\bot
\\
\text{iff } \quad
&   (\mu_w(\lnot\Profit)\geq 19.8\mu_w(\Profit))\to\bot
\\
\text{iff } \quad
&   19.8\mu_w(\Profit)>\mu_w(\lnot\Profit),
\end{align*}
%\begin{center}
%	\begin{tabular}{cll}
%		 & $\langle E,e_1\rangle(\mu_w(\lnot\Gallery\land\lnot\Profit)\geq\mu_w(\Gallery\land\Profit))\to\bot$ & \\
%		 iff & $(\mu_w(\lnot\Profit)\geq 19.8\mu_w(\Profit))\to\bot$ &\\
%		 iff & $19.8\mu_w(\Profit)>\mu_w(\lnot\Profit)$, &
%	\end{tabular}
%\end{center}
as required. 
\end{proof}

\begin{lemma}
\label{lem:proof:example}
The following propositions are provable in IPDEL.
	\begin{enumerate}
	\item \label{lem:item:1}
	$\langle \mathcal{E},e_1\rangle(\Gallery
	\land\Profit)\leftrightarrow\Profit$ and $\langle \mathcal{E},e_1\rangle(\lnot\Gallery
	\land\lnot\Profit)\leftrightarrow\bot$;
	\item \label{lem:item:2}
	$\langle \mathcal{E},e_2\rangle(\lnot\Gallery
	\land\lnot\Profit)\leftrightarrow\lnot\Profit$ and $\langle \mathcal{E},e_2\rangle(\Gallery
	\land\Profit)\leftrightarrow\bot$;
	\item \label{lem:item:3}
	$\mu_w^\top(\Profit)=0$ and $\mu_w^\top(\lnot\Profit)=0$;
	\item \label{lem:item:4}
	$\mu_w^{\Profit}(\lnot\Profit)=0$ and $\mu_w^{\lnot\Profit}(\Profit)=0$;
	\item \label{lem:item:5}
	$\mu_w^{\Profit}(\Profit)=\mu_w(\Profit)$ and $\mu_w^{\lnot\Profit}(\lnot\Profit)=\mu_w(\lnot\Profit)$.
\end{enumerate}
\end{lemma}
\begin{proof}
\texttt{Proof of item \eqref{lem:item:1}.}
\begin{align*}
 & 	\langle \mathcal{E},e_1\rangle(\Gallery
 \land\Profit) 
\\
\text{iff } \quad
& \langle \mathcal{E},e_1\rangle\Gallery
 \land\langle \mathcal{E},e_1\rangle\Profit \tag{I8 in Table \ref{table:IPDEL}}
\\
\text{iff } \quad
& pre(e_1)\land\sub(e_1,\Gallery)\land pre(e_1)\land\sub(e_1,\Profit) \tag{I2 in Table \ref{table:IPDEL}}
\\
\text{iff } \quad
& \sub(e_1,\Gallery)\land\sub(e_1,\Profit) 
\tag{$pre(e_1)$ is $\top\lor\Profit\lor\lnot\Profit$}
\\
\text{iff } \quad
& \top\land\Profit 
\tag{Definition of $\sub$}
\\
\text{iff } \quad
& \Profit
\end{align*}
and
\begin{align*}
		& 	\langle \mathcal{E},e_1\rangle(\lnot\Gallery
		\land\lnot\Profit)
\\
\text{iff } \quad
& \langle \mathcal{E},e_1\rangle\lnot\Gallery
		\land\langle \mathcal{E},e_1\rangle\lnot\Profit 
		\tag{I8 in Table \ref{table:IPDEL}}
\\
\text{iff } \quad
& (pre(e_1)\land(\lnot\langle \mathcal{E},e_1\rangle\Gallery))
		\land(pre(e_1)\land(\lnot\langle \mathcal{E},e_1\rangle\Profit))
		\tag{I12 and I6 in Table \ref{table:IPDEL}}
\\
\text{iff } \quad
& pre(e_1)\land\lnot(pre(e_1)\land\sub(e_1,\Gallery))\land pre(e_1)\land\lnot(pre(e_1)\land\sub(e_1,\Profit)) 
\tag{I2 in Table \ref{table:IPDEL}}
\\
\text{iff } \quad
& \lnot\sub(e_1,\Gallery)\land\lnot\sub(e_1,\Profit)
\tag{$pre(e_1)$ is $\top\lor\Profit\lor\lnot\Profit$}
\\
\text{iff } \quad
& \lnot\top\land\lnot\Profit 
\tag{Definition of $\sub$}
\\
\text{iff } \quad
& \bot & \tag{$\lnot\top\leftrightarrow\bot$}
\end{align*}

\texttt{Proof of item \eqref{lem:item:2}.} The proof is similar to that of item 1.\\

\texttt{Proof of item \eqref{lem:item:3}.} Notice that $\mu_w^\top(\Profit)$ is shorthand for $\mu_w(\top\land\Profit)-(\mu_w(\Profit\land\Profit)+\mu_w(\lnot\Profit\land\Profit))$ (cf.\ \autoref{def:mua-mba}). 
Therefore:
\begin{align*}
		& \mu_w^\top(\Profit)=0
\\
\text{iff } \quad
& \mu_w(\top\land\Profit)-(\mu_w(\Profit\land\Profit)+\mu_w(\lnot\Profit\land\Profit))=0
\\
\text{iff } \quad
& \mu_w(\top\land\Profit)-\mu_w(\Profit\land\Profit)=0
\tag{P1 in Table \ref{table:IPDEL} and Lemma \ref{lemma:measures:basic}}
\\
\text{iff } \quad
& \mu_w(\Profit)-\mu_w(\Profit)=0.
\end{align*}
The last equality follows by N0 in Table \ref{table:IPDEL}. The proof of the second inequality is similar.\\

\texttt{Proof of item \eqref{lem:item:4}.} Notice that $\mu_w^{\Profit}(\lnot\Profit)$ is shorthand for $\mu_w(\Profit\land\lnot\Profit)$ and $\mu_w^{\lnot\Profit}(\Profit)$ is shorthand for $\mu_w(\lnot\Profit\land\Profit)$. Hence, the equality follows from Axiom P1 in Table \ref{table:IPDEL}. \\

\texttt{Proof of item \eqref{lem:item:5}.} Notice that $\mu_w^{\Profit}(\Profit)$ is shorthand for $\mu_w(\Profit\land\Profit)$ and $\mu_w^{\lnot\Profit}(\lnot\Profit)$ is shorthand for $\mu_w(\lnot\Profit\land\lnot\Profit)$. Hence, the required equality is straightforwardly true.\\
\end{proof}

Since, as discussed in Section \ref{sec:relsem}, IPDEL is sound and complete with respect to the class of relational models, Proposition \ref{prop:proofex} implies that  every IPDEL model $\mathcal{M}$ which supports the left-hand side of the equivalence \eqref{eq:thres} will be updated by the event $(E,e_1)$ to a model that satisfies $\mu_w(\Gallery\land\Profit)>\mu_w(\lnot\Gallery\land\lnot\Profit)$. Hence, in each such model agent $w$ will update his subjective probabilities concerning $\Profit$ analogously to the model in the example above (see Section \ref{ssec:updated:model} and Figure \ref{fig:ex:update-model-MArt}).

\section{Conclusion}
\label{sec:CCL}
\paragraph{Present contributions.} In this paper, we have introduced the logic IPDEL, the intuitionistic counterpart of classical PDEL, as an instance of a general methodology, based on the mathematical construction of updates on algebras, which makes it possible to define non-classical counterparts of DEL-type logics on different propositional bases. This methodology makes it possible to also obtain the update construction on relational and topological models via appropriate (extended) dualities, and hence define relational semantics for the defined logics. In this way we have shown that IPDEL, which is sound by construction with respect to the class of algebraic probabilistic epistemic  models (cf.\ Definition \ref{def: alg probab epist model}), is also complete with respect to  APE-models and hence also with respect to their dual relational structures. Since these structures are finite by definition, this result immediately implies that IPDEL has the finite model property. The logic IPDEL is intended as a tool to analyze decision-making under uncertainty in situations in which truth is socially constructed and hence decisions are taken in  contexts in which the truth value of certain states of affair might be undetermined. To show IPDEL at work, we partially formalize one such situation.

\paragraph{Generalizing APE-structures.} APE-structures are based on epistemic Heyting algebras (cf.\ Definition \ref{def:epist-Heyting-algebra}), the definition of which requires the image of each diamond operator to have a Boolean algebra structure. Thus, epistemic Heyting algebras are a proper subclass of monadic Heyting algebras. This additional condition guarantees that the $i$-minimal elements induce a partition on the dual structure of each epistemic Heyting algebra, and hence that axioms such as $\mu(\top)=1$ or $(\mu(\varphi)\geq\alpha)\lor(\mu(\varphi)<\alpha)$  are valid. 
One natural question that presents itself is whether this condition can be dropped and hence base APE-structures on general monadic Heyting algebras. 
Addressing this question requires solving issues of technical and conceptual nature. 
On the technical side, the additional requirement plays a role in the completeness theorem, and specifically makes sure that, in the finite lattice that we extract from the Lindenbaum-Tarski algebra, a sublattice can be defined out of the image of each diamond (cf.\ Lemma \ref{lemma2:completeness}). 
This issue would partially be addressed by relaxing the condition that APE-structures be finite (see paragraph below). On the conceptual side, we would need to restructure the definition of probabilistic measure. 
The axiom $(\mu(\varphi)\geq\alpha)\lor(\mu(\varphi)<\alpha)$ is tightly linked to the metatheory of the real numbers and in particular to the validity of trichotomy. Hence, in the context of a different metatheory in which trichotomy does not hold such as the constructive metatheory of real numbers, it seems reasonable that this axiom might be dropped. However the condition $\mu(\top)=1$ expresses the link between probability and the underlying logic. For this reason this axiom should arguably be kept.

\paragraph{Finite to infinite models.} Another natural question is whether we can drop the condition that APE-structures be finite. A first step would be to investigate the case of APE-structures based on perfect Heyting algebras, i.e.\ those Heyting algebras which are isomorphic to algebras of upsets or downsets of given posets. Does every probability measure on such a Heyting algebra correspond to a discrete probability distribution on the corresponding dual poset? More generally, possibly infinite APE-structures would dually correspond to relational Esakia spaces endowed with probability distributions. Are there purely algebraic conditions on probability measures guaranteeing that the corresponding probability distribution be discrete?

%Regarding weatheson's axioms and how they relate with classical probabilities and the role of implication.

 \paragraph{Proof theory for probabilistic logics.} As mentioned in the introduction, the present paper pertains to a line of research aimed at studying the phenomenon of dynamic (probabilistic epistemic) updates in contexts at odds with classical truth. The language and semantics of the formal settings previously studied (i.e.\ those of the nonclassical versions of PAL and EAK) have served as a basis for a research program in structural proof theory aimed at developing a uniform methodology for endowing dynamic logics with so-called {\em analytic calculi} (see \cite{CiRa14,GMPTZ}). This research program has successfully addressed PAL and DEL \cite{GKPLori,TrendsXIII,FGKPS14b,FGKPS14a}, and PDL \cite{PDL}, and has been further generalized into the proof-theoretic framework of {\em multi-type calculi} \cite{TrendsXIII}. This methodology has been successfully deployed to introduce analytic calculi for logics particularly impervious to the standard treatment 
 \cite{Inquisitive,GP:lattice,GP:linear,LoRC}, and is now ready to be applied to the issue of endowing PDEL and its non-classical versions with analytic calculi.
 
 %\paragraph{Relational semantics for IPDEL.} As mentioned earlier, Esakia duality  for Heyting algebras \cite{Esa74} and Bezhanishvili's duality for monadic Heyting algebras \cite{bezhanishvili1999varieties} can serve as a basis for the definition of a relational  semantics for IPDEL. Developing this semantics requires an extensive duality-theoretic treatment of probability measures and probability distributions  in the intuitionistic setting, analogous to the one developed in \cite{flaminio2016states}, which is currently work in progress.
 
% \redbf{to expand the conclusion}

\bibliographystyle{alpha}
\bibliography{Ref}

\newpage
\appendix
\addcontentsline{toc}{section}{Appendices}

\section{Proofs of Section \ref{sec:ha}}
\label{app:section4}
\subsection*{Proof of Lemma \ref{lem:charact i minimal elements complex algebra}} 
\label{app:section4:0}

\paragraph{Lemma \ref{lem:charact i minimal elements complex algebra}.}
For any PES-model $\mb{M}$, the $i$-minimal elements of its  complex algebra $\mb{M}^+$ are exactly the equivalence classes of $\sim_i$.

	\begin{proof} 
	%\redbf{Apostolos can you check the proof please?}
		Let $\mb{M} = \left\langle S, (\sim_i)_{i\in \Ag}, (P_i)_{i\in \Ag}, \val{\cdot}\right\rangle $ be a PES-model
		and  
		$\mb{M}^+ =  \left( \P S, (\lozenge_i)_{i\in \Ag}, (\Box_i)_{i\in \Ag}, (P^+_i)_{i\in \Ag} \right) $ be its complex algebra.
		For any $i\in\Ag$ and any $s\in S$, let $[s]_i$ be the $\sim_i$-equivalence  cell of $s$. Fix $i\in \Ag$.

\medskip

First, let us prove that any 	$\sim_i$-equivalence cell corresponds to an $i$-minimal element of $\mb{M}^+$.	
		Since $\sim_i$ is reflexive, $[ s ]_i \neq \varnothing$.
		Since $\sim_i$ is symmetric and transitive, $[s]_i = \lozenge_i \{ s\} = \lozenge_i \lozenge_i \{ s\} = \lozenge_i [s]_i$.
		This shows that $[s]_i$ is a fixed-point of $\lozenge_i$. It remains to show that $[s]_i$ is a minimal fixed-point $\lozenge_i$.
		Let $X \subseteq S$ be an $i$-minimal element of $\mb{M}^+$. By definition, we have that $X \subseteq [s]_i$, $X \neq \varnothing$ and $\lozenge_i X = X$.
		The assumption that $\lozenge_i X = X$ implies that $X =
		\bigcup_{x\in X}  \lozenge_i \{ x\} = \bigcup_{x\in X}  [ x ]_i $. The assumption that $X \subseteq [s]_i$ implies that all $x\in X$ must be $\sim_i$-equivalent to $s$, and hence to each other. Therefore, $X$ cannot be the union of more than one equivalence cell.
		Moreover, the assumption that $X \neq \varnothing$ implies that there exists at least one equivalence cell in $\bigcup_{x\in X}  [ x ]_i$. This concludes the proof that, for any $s\in S$, its $\sim_i$-equivalence cell $[s]_i$ corresponds to an $i$-minimal element of $\mb{M}^+$, as required.
	
	\medskip
		
Now, 
let us prove that any $i$-minimal  element of $\mb{M}^+$ correspond to the  $\sim_i$-equivalence cell of an element $s \in S$.
		Let $X$ be an $i$-minimal element of $\mb{M}^+$.
		The assumption that $X = \lozenge_i X$ implies that $X = \bigcup_{x\in X}  [ x ]_i $. The assumption that $X \neq \varnothing$ implies that there exists at least one equivalence cell $[s]_i$ in $\bigcup_{x\in X}  [ x ]_i$. Since $[s]_i$ is an $i$-minimal element of $\mb{M}^+$ and $[s]_i \subseteq X$, we have $X = [s]_i$ by minimality of $X$.
	\end{proof}

\subsection*{Proof of Proposition \ref{prop: complex algebras are APE structures}} 
\label{app:section4:1}

\paragraph{Proposition \ref{prop: complex algebras are APE structures}.}
For any PES-model $\mb{M}$, its complex algebra $\mb{M}^+$ (see  \Cref{def: complex algebra}) is an APE-structure (see \Cref{def: alg probab epist structure}).

\begin{proof}
		Let $\mb{M} = \left\langle S, (\sim_i)_{i\in \Ag}, (P_i)_{i\in \Ag}, \val{\cdot}\right\rangle $ be a PES-model (see \Cref{def:Prob epis state model Alexandru}) and let   
		$\mb{M}^+ =  \left( \P S, (\lozenge_i)_{i\in \Ag}, (\Box_i)_{i\in \Ag}, (P^+_i)_{i\in \Ag} \right) $ be its complex algebra.
		$\mb{M}^+$ is an APE-structure if its support is an epistemic Heyting algebra and if  each $P^+_i$ is an $i$-measure over $\left\langle S, (\sim_i)_{i\in \Ag}, (P_i)_{i\in \Ag}\right\rangle$.
		Clearly, $\left( \P S ,  (\lozenge_i)_{i\in \Ag}, (\Box_i)_{i\in \Ag} \right)$ is an epistemic Heyting algebra (see \Cref{def:epist-Heyting-algebra}), since $\sim_i$ is an equivalence relation and $ \P S$ is a boolean algebra.
To finish the proof we need to show that each $P^+_i$ is an $i$-measure on $\support (\mb{M}^+)$. Hence, for every $i\in \Ag$, we need to prove the following properties:
\begin{enumerate}[(a)]
	    \item 
	    $\mathsf{dom}(P_i^+)=\mathsf{Min}_i(\support (\mb{M}^+)){\downarrow}$;
		\item 
$P_i^+$ is order-preserving; %$\mu(\bot)=0$  and
		\item 
for every $i$-minimal element $X\in \P S$ and all $Y_1, Y_2\in X{\downarrow}$,  we have $$P_i^+(Y_1\cup Y_2) = P_i^+(Y_1)+ P_i^+(Y_2)-P_i^+(Y_1\cap Y_2);$$
		\item   $P_i^+(\varnothing)=0$ if $\mathsf{dom}(P_i^+)\neq\varnothing$;
	\item for every $i$-minimal element $X\in \P S$, we have $P_i^+(X) = 1$.
		\item for every $i$-minimal element $X\in \P S$ and all $Y_1, Y_2\in X{\downarrow}$ such that $Y_1 \subset Y_2$, it holds that $$P_i^+(b)<P_i^+(c).$$

\end{enumerate}		
	
	Fix $i \in \Ag$.

\medskip
		
		\texttt{Proof of (a).} 
		By definition,
		$\mathsf{dom}(P^+_i)  =  \left\{X\in\P S \mid \exists y \ \forall x \: (x\in X\implies x\sim_i y) \right\}$.
		Notice that 
		$$ \left\{X\in\P S \mid \exists y \ \forall x \: (x\in X\implies x\sim_i y) \right\} = \left\{X \mid X \subseteq [s] \text{ and } s \in S \right\}.$$
		By \Cref{lem:charact i minimal elements complex algebra}, we deduce that $\mathsf{dom}(P^+_i)  = \mathsf{Min}_i(\support (\mb{M}^+)){\downarrow}$.

\medskip
		
		\texttt{Proof of (b).} 	Since $P_i(s) \geq 0$ for all $s\in S$, the maps $P_i^+$ are monotone.

\medskip
		
		\texttt{Proof of (c).} By Lemma \ref{lem:charact i minimal elements complex algebra}, if $X$ is an $i$-minimal element of $\mb{M}^+$, then $X = [s]$ for some $s\in S$. If $Y_1, Y_2\in X{\downarrow}$, then $Y_1\cup Y_2\subseteq [s]$. Hence, 
%		for $1\leq j\leq 2$, we have that  
%		$$\sum_{x\in Y_j} P_i(x)\leq \sum_{x\in Y_1\cup Y_2} P_i(x)\leq \sum_{x\in [s]} P_i(x) = 1.$$ Therefore,
		\begin{align*}
		P^+_i(Y_1\cup Y_2) & = \sum_{x\in Y_1\cup Y_2} P_i(x) 
		\tag{Definition of $P_i^+$}
		\\
		& = \sum_{x\in Y_1} P_i(x) + \sum_{x\in Y_2} P_i(x) -\sum_{x\in Y_1\cap Y_2}P_i(x)
		%\tag{$P_i$ is a probability distribution on $[s]$}
		\\
		& = P^+_i(Y_1) + P^+_i(Y_2) - P^+_i(Y_1\cap Y_2).
		\tag{Definition of $P_i^+$}
		\end{align*}
		
\medskip

		\texttt{Proof of (d).} 
		By definition, $P_i^+(\varnothing)=0$.  %, and because $P_i$ are probability mass functions over equivalence classes $\mu_i(S)=1$.
		
\medskip

		\texttt{Proof of (e).} Let $X\in \P S$ be an $i$-minimal element. By Lemma \ref{lem:charact i minimal elements complex algebra}, there exists an $s\in S$  such that $[s]= X$. Hence, using the definition  of $P_i$ (see Definition \ref{def: complex algebra}), we have:
		\[
		P_i^+(X) =\sum_{x\in [s]} P_i(x) = 1.
		\]

\medskip
		
		\texttt{Proof of (f).} Let $X\in \P S$ be $i$-minimal element
		and  $Y_1, Y_2\in X{\downarrow}$ such that $Y_1 \subset Y_2$. 
		By definition, we have that 
		\begin{align*}
		P_i^+(Y_2) & = \sum_{x\in Y_2} P_i(x)
		 = \sum_{x\in Y_1} P_i(x) + \sum_{x\in Y_2 \smallsetminus Y_1} P_i(x)
		 = P_i^+(Y_1) + \sum_{x\in Y_2 \smallsetminus Y_1} P_i(x).
		\end{align*}
		Since $Y_1 \subset Y_2$, there exists $s \in Y_2 \smallsetminus Y_1$. Since $P_i : S \rightarrow \ ]0,1]$, we have $P_i(s)>0$ for all $s \in Y_2 \smallsetminus Y_1$.
		Hence $\sum_{x\in Y_2 \smallsetminus Y_1} P_i(x)>0$ and 
		$P_i^+(Y_1)<P_i^+(Y_2)$.
\end{proof}
		
\subsection*{Proof of Proposition \ref{prop: charact i minimal}}
\label{app:section4:2}

\paragraph{Proposition \ref{prop: charact i minimal}}
for every epistemic Heyting algebra  $\mb{A}$ and every agent   $i \in \Ag$, 
$$\mathsf{Min}_i(\mb{A}') = \{f_{e, a}\mid e\in E \mbox{ and } a\in \mathsf{Min}_i(\mb{A})\},$$
%The $i$-minimal elements of $\mb{A}'$ are exactly the maps %in the set
%\begin{equation}
%\label{eq:i-minimal maps}
%M := \{ f_{e,a} \mid e\in E \tand a\in \mb{A} i \text{-minimal} \},
%\end{equation}
%The $i$-minimal elements of $\mb{A}'$ are exactly the maps
where for any $e\in E$ and  $a\in \mathsf{Min}_i(\mb{A})$, the map $f_{e,a}$ is defined as follows:
\begin{align*}
f_{e,a} : E &\rightarrow \mb{A}
\\
e'  &\mapsto  \left\{ \begin{array}{ll}
 a & \textrm{if $e' \sim_i e$}\\
 \bot & \textrm{otherwise.}
  \end{array} \right.
\end{align*}

\begin{proof}
Recall that $f\in \mb{A}'$ is an $i$-minimal element (see \autoref{def:i-minimal}) if it satisfies the following conditions:
\eqref{def:i-minimal:item:bot} $f\neq \bot$, 
\eqref{def:i-minimal:item:fixpoint} $\lozenge_i f = f $ and
\eqref{def:i-minimal:item:minimal} if $g \in \mb{A}$, $g < f$ and $\lozenge_i g = g $, then $g = \bot$.

Let us first  prove that any map $f_{e,a}$ as above is an $i$-minimal element of $\mb{A}'$. By definition, $f_{e, a}(e) = a\neq \bot_\mb{A}$. Hence $f_{e, a}\neq \bot_{{\mb{A}'}}$. As to showing that $\blacklozenge'_i f_{e,a} = f_{e,a}$, fix $e'\in E$, and let us show that $(\blacklozenge'_i f_{e,a}) (e') = f_{e,a}(e')$. By definition,
$$\blacklozenge'_i f_{e,a}(e') =  \bigvee \{ \blacklozenge_i f_{e,a}(e'') \mid e'' \sim_i e' \}.$$

We proceed by cases: (a) If $e' \sim_i e$, then:
%then $f_{e,a}(e') = a$.
%Since $e' \sim_i e$,
%
\begin{align*}
\blacklozenge'_i f_{e,a}(e') & =  \bigvee \{ \blacklozenge_i f_{e,a}(e'') \mid e'' \sim_i e' \} \tag{by definition} \\
& = \bigvee \{ \blacklozenge_i a \mid e'' \sim_i e' \} \tag{ $f_{e,a}(e'')=a$, since  $e\sim_i e'$ and $\sim_i$ symmetric and transitive} \\
& =  \blacklozenge_i a  \tag{the join is nonempty since $\sim_i$ is reflexive}\\
& = a \tag{$a$ is $i$-minimal, hence is a fixed point of $\lozenge_i$}\\
& = f_{e,a}(e'). \tag{definition of $f_{e,a}$ and $e' \sim_i e$}
\end{align*}
(b) If $e'\nsim e$, then:
\begin{align*}
\blacklozenge'_i f_{e,a}(e') & =  \bigvee \{ \blacklozenge_i f_{e,a}(e'') \mid e'' \sim_i e' \} \tag{by definition}\\
& = \bigvee \{ \blacklozenge_i \bot \mid e'' \sim_i e' \} \tag{$e\nsim_i e'$}\\
& =  \blacklozenge_i \bot  \\
& = \bot \tag{$\lozenge_i \bot = \bot$}\\
& = f_{e,a}(e').
\end{align*}
Finally, we need to show that $f_{e, a}$ is a minimal non-bottom fixed-point of $\blacklozenge'_i$. Notice preliminarily that if $g: E\to \mb{A}$  is a fixed point for $\blacklozenge'_i$ then
\begin{equation}
\label{eq:fixed points property}
 g(e) = g(e') \mbox{ whenever } e\sim_i e'.
 \end{equation}
 Indeed,
\[g(e) =(\blacklozenge'_i g)(e) = \bigvee \{ \blacklozenge_i g(e'')  \mid e'' \sim_i e \} = \bigvee \{ \blacklozenge_i g(e'')  \mid e'' \sim_i e' \} = (\blacklozenge'_i g)(e') = g(e').\]
Given that $\sim_i$ is reflexive, this  implies in particular that, for every $e'\in E$, \begin{equation}\label{eq:two diamonds}(\blacklozenge'_i g)(e') = \blacklozenge_i g(e').\end{equation}

Let $g$ be as above, assume that  $\bot \neq g\leq f_{e, a}$,  and let us show that $g = f_{e, a}$.  Clearly, the assumption $g\leq f_{e, a}$ implies that $g(e') = \bot$ for every $e'\in E$ such that $e'\not\sim_i e$. Let $e'\in E$ such that $g(e')\neq \bot$. Together with the assumption that $g\leq f_{e, a}$, this implies that $f_{e, a}(e')\neq \bot$, hence $e'\sim_i e$ and $\bot \neq g(e')\leq a$.
To prove that $g(e) = a$, by the $i$-minimality of $a$ it suffices to show that $g(e')$ is a fixed point of $\blacklozenge_i$.
%Moreover, by assumption, and given that $\sim_i$ is reflexive,
Indeed, by \eqref{eq:two diamonds}:
$$\blacklozenge_i g(e') = (\blacklozenge'_i g)(e') = g(e'),$$
%which shows that $\bot\neq g(e')\leq a$ is a fixed point of $\blacklozenge_i$,
as required. Finally, the fact above and the preliminary observation \eqref{eq:fixed points property} imply that $g(e') = a$ for every $e'\in E$ such that $e'\sim_i e$.

This finishes the proof that $f_{e,a}$ is $i$-minimal.

\bigskip

Conversely, let $g: E\to \mb{A}$ be $i$-minimal in $\mb{A}'$, and let us show that $g = f_{e,a}$ for some $e\in E$ and some $i$-minimal element $a\in \mb{A}$. The assumption that $g\neq \bot$ implies that $g(e)\neq \bot$ for some $e\in E$. Let $g(e) = a\in \mb{A}$. Then, the assumption that $g = \blacklozenge'_i g$ and the observation \eqref{eq:fixed points property} imply that $g(e') = a$ for every $e'\in E$ such that $e'\sim_i e$. Then, the proof is finished if we show that $a$ is $i$-minimal in $\mb{A}$. Indeed, then, by construction we would have $\bot\neq f_{e, a}\leq g$, hence the minimality of $g$ would yield  $f_{e, a} =  g$.

By definition, we have that $a = g(e')\neq \bot$. By observation \eqref{eq:two diamonds}, \[\blacklozenge_i a = \blacklozenge_i g(e) = (\blacklozenge'_i g)(e) = g(e) = a,\]
which shows that $a$ is a fixed point of $\blacklozenge_i$. Finally, let $\bot\neq b\leq a$ such that $\blacklozenge_i b = b$. Then, with an argument analogous to the one given above,  the map $f_{e, b}: E\to \mb{A}$ would be proven to be a non-bottom fixed-point of $\blacklozenge'_i$. Moreover, $f_{e, b}\leq g$, and hence  the $i$-minimality of $g$ would yield $f_{e, b} = g$, hence $a = b$.
\end{proof}

\subsection*{Proof of Proposition \ref{prop: muai properties}}
\label{app:section4:3}

\paragraph{Proposition \ref{prop: muai properties}.}
For every APE-structure  $\mathcal{F}=\left( \mb{A}, (\mu_i)_{i\in \Ag} \right) $ and every event structure $\mb{E}$ over $\mb{A}$, 
%it is the case that 
$\mu^a_i$ is an $i$-premeasure over $\mb{A}$. Furthermore,  if $a\leq y$ then $\mu^a_i(x)=\mu^a_i(x\land y)$.

\begin{proof}
%\redbf{APOSTOLOS: check this proof}
%\redbf{to update w.r.t.\ the new definitions for $P_i(s)$, $P_i(e)$ and $\pre$.}
For every $a\in\rmPhi$ and every $i\in \Ag$, we want to prove that $\mu_i^a$ is an $i$-premeasure over $\mb{A}$, 
hence we need to prove that $\mu_i^a$ is a partial function $\mb{A}\rightarrow \mb{R}^+$ that satisfies items (\ref{def:epAlg:two:domain} - \ref{def:epAlg:two:bot}) of Definition \ref{def:measures}. 
Fix $a\in\rmPhi$ and  $i\in \Ag$.

\medskip

\texttt{Proof of item \ref{def:epAlg:two:domain}.}
We want to prove that $\mathsf{dom}(\mu)=\mathsf{Min}_i(\mb{A}){\downarrow}$.
The map $\mu_i$ is an $i$-premeasure,
hence $\mathsf{dom}(\mu_i)=\mathsf{Min}_i(\mb{A}){\downarrow}$. Therefore the map $\mu_i^a$ is defined on every $x\in \mathsf{Min}_i(\mb{A}){\downarrow}$ 
and we can restrict its domain as follows: $\mathsf{dom}(\mu_i^a):=\mathsf{Min}_i(\mb{A}){\downarrow}$.

\medskip

\texttt{Proof that $\mu_i^a$ is well-defined.}
We need to prove that
$\mu^a_i(x)\geq 0$ for all $x\in \mathsf{Min}_i(\mb{A}){\downarrow}$. 
Recall that
$\bfPhi $ is a finite ordered multiset of elements of $\mb{A}$ such that,  for all distinct $b,c\in \rmPhi$, either
%\begin{center}
$b\wedge c = \bot  $ or $  b < c  $ or $  c < b$
%\end{center} 
(see Definition \ref{def:algebraic event} and Remark
\ref{rk:Phi:forest}).
Hence, for every  $b,c\in\mathrm{mb}(a)$ we have $b\land c=\bot$. Indeed, by item 2 of Definition \ref{def:ordered:multiset} and what was mentioned above, if $b\land c\neq\bot$, then either $b\prec c$ or $c\prec b$. Hence, they cannot both be maximal.

\smallskip

Fix $x\in \mathsf{Min}_i(\mb{A}){\downarrow}$.  
Let us prove by induction on the size of $S$ that 
for any  $S\subseteq \mathrm{mb}(a)$,
\begin{align}
\mu_i \left( \bigvee_{b\in S}x\land b \right) = \sum_{b\in S}\mu_i(x\land b).
 %\tag{$IH$}
\label{eq:proof:mu-i-a:1}
\end{align}

\texttt{Base case : $|S|=0$.}
Assume that $S=\varnothing$. Then, we trivially have that 
\begin{align*}
\mu_i(\bigvee_{b\in S}x\land b)=\mu_i(\bot)=0=\sum_{b\in S}\mu_i(x\land b).
\tag{$\mathtt{ IH_0}$}
\end{align*}

\texttt{Induction step : $\mathtt{IH_n} \Rightarrow \mathtt{IH_{n+1}}$.}
Assume that, for any set $S'$ that contains exactly $n$ elements, we have
\begin{align*}
\mu_i(\bigvee_{b'\in S'}x\land b') = \sum_{b'\in S'}\mu_i(x\land b').
\tag{$\mathtt{ IH_n}$}
\label{align:proof:muia:IHn}
\end{align*} 
 
Let $S$ contain exactly $n+1$ elements, 
$S' \subset S$  contain exactly $n$ elements, and
$S = S' \cup \{c\}$.
Let us prove $\mathtt{IH}_{n+1}$:
 \begin{align*}
 & \quad \; \mu_i \left( \bigvee_{b\in S}x\land b \right)
 \\
 & = \mu_i \left( (x \land c) \vee \bigvee_{b'\in S'}(x\land b') \right)
\tag{$S = S' \cup \{c\}$} 
 \\
 & = \mu_i(x \land c) + \mu_i \left( \bigvee_{b'\in S'}x\land b' \right) - \mu_i \left( (x \land c) \wedge \bigvee_{b'\in S'}(x\land b') \right)
 \tag{$\mu_i$ is an $i$-premeasure}
\\
 & = \mu_i(x \land c) + \mu_i \left( \bigvee_{b'\in S'}x\land b' \right) - \mu_i\left( \bigvee_{b'\in S'} x \land c \wedge x\land b' \right)
 \tag{$\wedge$ distributes over $\vee$}
\\
 & = \mu_i(x \land c) + \mu_i \left( \bigvee_{b'\in S'}x\land b' \right) - \mu_i( \bot )
 \tag{$c\neq b'$ implies $c\wedge b'= \bot$}
\\
 & = \mu_i(x \land c) + \sum_{b'\in S'}\mu_i(x\land b')
 \tag{$\mu_i( \bot )=0$ and \eqref{align:proof:muia:IHn}}
 \\
 & = \sum_{b\in S}\mu_i(x\land b)
 \tag{$S = S' \cup \{c\}$}
 \end{align*}
By induction, for any $x\in \mathsf{Min}_i(\mb{A}){\downarrow}$, 
we have
$
\mu_i\left(\bigvee_{b\in\mathrm{mb}(a)}x\land b \right)  =  \sum_{b\in\mathrm{mb}(a)}\mu_i(x\land b).
$

\medskip

Since $\mathrm{mb}(a)$ denotes the set of the $\prec$-maximal elements of $(\bfPhi\ \cap \downarrow\!\! a)\setminus\{a\}$,
we have that $\bigvee_{b\in\mathrm{mb}(a)}x\land b\leq x\land a$. 
By monotonicity of $\mu_i$, we get that 
$$\sum_{b\in\mathrm{mb}(a)}\mu_i(x\land b) =  \mu_i\left(\bigvee_{b\in\mathrm{mb}(a)}x\land b \right) \leq \mu_i(x\land a).$$ 
Hence, $\mu^a_i(x)\geq 0$ for any $x\in \mathsf{Min}_i(\mb{A}){\downarrow}$ as required.

\medskip

\texttt{Proof of item \ref{def:epAlg:two:monotone}.}
We want to  show that $\mu^a_i$ is order-preserving. 
Using $\eqref{eq:proof:mu-i-a:1}$ and the fact that $\wedge$ distributes over $\vee$, we get that: for any $x\in \mathsf{Min}_i(\mb{A}){\downarrow}$,
\begin{align}
\sum_{b\in\mathrm{mb}(a)}\mu_i \left( x\land b \right) =
\mu_i\left(\bigvee_{b\in\mathrm{mb}(a)}x\land b \right) =
\mu_i\left(x\land\bigvee_{b\in\mathrm{mb}(a)}b \right).
\label{eq:proof:mu-i-a:2}
\end{align}
Fix $x,y  \in \mathsf{Min}_i(\mb{A}){\downarrow}$
such that $x\leq y$. 
Notice that 
$\bigvee_{b\in\mathrm{mb}(a)}b \leq a$ and
$x\land a\land y =x$. 
Furthermore, $x\land a\leq y\land a$ and $y\land(\bigvee_{b\in\mathrm{mb}(a)} b)\leq y\land a$. Hence $(x\land a)\lor(y\land(\bigvee_{b\in\mathrm{mb}(a)} b))\leq y\land a$.  
From this we can deduce that: 
\begin{align*}
& \quad \quad \;\; (x\land a)\lor \left(y\land\left(\bigvee_{b\in\mathrm{mb}(a)} b\right) \right) \leq y\land a
\\
& \Rightarrow \quad \mu_i \left((x\land a)\lor \left( y\land\bigvee_{b\in\mathrm{mb}(a)} b \right) \right) \leq\mu_i(y\land a) 
\tag{$\mu_i$ is order-preserving}
\\
& \Leftrightarrow \quad
\mu_i(x\land a)+
\mu_i \left( y\land\bigvee_{b\in\mathrm{mb}(a)} b \right)
-\mu_i \left(x\land a\land y\land\bigvee_{b\in\mathrm{mb}(a)}  b\right)\leq\mu_i(y\land a) 
\tag{$\mu_i$ is an $i$-premeasure}
\\
& \Leftrightarrow\quad
\mu_i(x\land a)+
\mu_i \left( y\land\bigvee_{b\in\mathrm{mb}(a)} b \right)
-\mu_i \left( x \land\bigvee_{b\in\mathrm{mb}(a)} b \right)
\leq\mu_i(y\land a) 
\tag{$x\land a\land y=x$}
\\
& \Leftrightarrow\quad
\mu_i(x\land a)-
\mu_i \left( x \land\bigvee_{b\in\mathrm{mb}(a)} b \right)
\leq\mu_i(y\land a)-
\mu_i \left( y\land\bigvee_{b\in\mathrm{mb}(a)} b \right)
\\
& \Leftrightarrow\quad
\mu_i(x\land a)-
\sum_{b\in\mathrm{mb}(a)}\mu_i(x\land b) 
\leq\mu_i(y\land a)-\sum_{b\in\mathrm{mb}(a)}\mu_i(y\land b) 
\tag{by \eqref{eq:proof:mu-i-a:2}}
\\
& \Leftrightarrow\quad
\mu^a_i(x)\leq\mu^a_i(y) .
\end{align*}

\medskip

\texttt{Proof of item \ref{def:epAlg:two:join}.}
%for every $a\in \mathsf{Min}_i(\mb{A})$ and all $b, c\in a{\downarrow}$ it holds that $\mu(b\vee c) = \mu(b)+ \mu(c) - \mu(b\wedge c)$;
We need to show that  $\mu^a_i(x\lor y)=\mu^a_i(x)+\mu^a_i(y)-\mu^a_i(x\land y)$
for all $x,y\in \mathsf{Min}_i(\mb{A}){\downarrow}$. We have:

\begin{align*}
\mu^a_i(x\lor y) & =  \mu_i((x\lor y)\land a)-\sum_{b\in\mathrm{mb}(a)}\mu_i((x\lor y)\land b) \\
& =  \mu_i((x\land a)\lor(y\land a))-\sum_{b\in\mathrm{mb}(a)}\mu_i((x\land b)\lor(y\land b)) \tag{distributivity}\\
& =  (\mu_i(x\land a)+\mu_i(y\land a)-\mu_i(x\land y\land a))-\sum_{b\in\mathrm{mb}(a)}(\mu_i(x\land b)+\mu_i(y\land b)-\mu_i(x\land y\land b)) \tag{$\mu_i$ is an $i$-measure}\\
& = \mu^a_i(x)+\mu^a_i(y)-\mu^a_i(x\land y).
\end{align*}

\medskip

\texttt{Proof of item \ref{def:epAlg:two:bot}.}
If $\mathsf{Min}_i(\mb{A}){\downarrow} \neq \varnothing$, it follows from $\mu_i(\bot)=0$ (because $\mu_i$ is a $i$-premeasure) that
$\mu_i^a(\bot)=0$.
\end{proof}

\subsection*{Proof of Proposition \ref{prop: plus of coprod same as prod of plusses}}
\label{app:section4:4}

\paragraph{Proposition \ref{prop: plus of coprod same as prod of plusses}.}
For every PES-model $\mb{M}$ and any event structure $\mathcal{E}$ over $\mathcal{L}$, \[(\coprod_{\mathcal{E}}\mb{M})^+\cong \prod_{\mb{E}_{\mathcal{E}}}\mb{M}^+.\]

\begin{proof}
The proof that the supports of the two APE-structures (\autoref{def: alg probab epist structure}) can be identified is essentially the same as that of \cite[Fact 23.3]{KP13}, and is omitted. 
Recall that the basic identification between  $\mathcal{P}(\coprod_{|E|}S)$ and $\prod_{|E|}\mathcal{P}(S)$ associates every subset $X\subseteq \coprod_{|E|}S$ with the map 
\begin{align*}
g: E & \to \mathcal{P}(S)\\
e & \mapsto X_e: = \{s\in S\mid (s, e)\in X\}.
\end{align*}
Let us prove that this identification induces an identification between the maps\footnote{Refer to Definitions \ref{def: coproduct PES model} and \ref{def: complex algebra} for the definitions of the intermediate structure $\coprod_{\mathcal{E}}\mb{M}$ and of the complex algebra associated to a model.} 
\begin{align*}
(P^+_i)': \prod_{|E|} & \mathcal{P}(S)\to [0, 1] & \text{and} &&
(P_i^{\coprod})^+: \mathcal{P}(\coprod_{|E|}S)\to [0, 1].
\end{align*}

In what follows, we fix a subset $X\subseteq \coprod_{|E|}S$ in the domain of $P_i^{\coprod}$ and let $g\in \prod_{|E|}\mathcal{P}(S)$ be defined as its counterpart as discussed above. 
%Notice preliminarily that when $\rmPhi$ consists of mutually exclusive propositions $P^+_i(X)=(P^+)^a_i(X)$ for any $a\in\rmPhi$ (cf.\ Definition \ref{def:mua-mba}). \redbf{to explain}
Recall that 
for any $s\in S$ and $e\in E$,
$\pre(e\mid s)$ denotes the value $\pre(e\mid \phi)$ for the unique $\phi\in\rmPhi$ such that $\mb{M},s\Vdash\phi$ (see Notation \ref{note:pre(e/s)}).
Then, we have:
%\redbf{Notation problem: here the map $\pre$ represent 3 things: $\pre(e\mid s)$, $\pre(e\mid \phi) $ and $\pre(e \mid \val{\phi})$ !!!}
\begin{align*}
(P_i^{\coprod})^+(X) &  = \sum_{(s, e)\in X} P_i^{\coprod}((s, e)) 
\tag{Definition \ref{def: complex algebra} on $P_i^{\coprod}$} 
\\
& = \sum_{(s, e)\in X} P_i(s)\cdot P_i( e)\cdot\pre(e\mid s) 
\tag{Definition \ref{def: coproduct PES model} } 
\\
& = \sum_{ e\in E} \sum_{s\in X_e} P_i(s)\cdot P_i( e)\cdot\pre(e\mid s) 
\tag{$X_e: = \{s\in S\mid (s, e)\in X\}$}  
\\
& = \sum_{ e\in E} P_i( e)\cdot \left( 
\sum_{s\in X_e} P_i(s)\cdot \pre(e\mid s) 
\right)
\\
& = \sum_{ e\in E} P_i( e)\cdot 
\sum_{\phi\in \rmPhi}
\left( 
\sum_{s\in X_e \cap \val{\phi}} P_i(s)\cdot \pre(e\mid s) 
\right)
\tag{$\rmPhi$ provides a partition of $\{ s\in S \mid \pre(e\mid s) \neq 0\}$}
\\
& = \sum_{ e\in E} P_i( e)\cdot \left(
\sum_{\phi\in \rmPhi}
\left( 
\sum_{s\in X_e \cap \val{\phi}} P_i(s)
\right) \cdot \pre(e\mid \phi)  
\right)
\tag{Notation \ref{note:pre(e/s)}}
\\
& = \sum_{ e\in E} P_i( e)\cdot \left(
\sum_{\phi\in \rmPhi}
 P^+_i( X_e \cap \val{\phi})
\cdot \overline{\pre}_\mb{M}(e\mid \val{\phi})  
\right)
\tag{Definition \ref{def: complex algebra}} 
\\
& = \sum_{ e\in E} P_i( e)\cdot \left(
\sum_{\phi\in \rmPhi}
 (P_i^+)^{\val{\phi}}( X_e )
\cdot \overline{\pre}_\mb{M}(e\mid \val{\phi})  
\right)
\tag{Remark \ref{rk:mb:classic:empty} : $\mathrm{mb}(\val{\phi})=\varnothing$}
\\
%& = \sum_{ e\in E} P_i( e)\cdot \left(
%\sum_{\phi\in \rmPhi}
%\mu_i^\phi( X_e )
%\cdot \pre_\mb{M}(e\mid \val{\phi})  
%\right)
%\tag{notation, see  \autoref{rk:classic:alg:model}
%\redfootnote{Apostolos: I have questions about this step}}
%\\
& = \sum_{ e\in E} \sum_{\phi\in \rmPhi} P_i( e)\cdot   (P_i^+)^{\val{\phi}} ( X_e )
\cdot \overline{\pre}_\mb{M}(e\mid \val{\phi})  
\\
& = \sum_{ e\in E} \sum_{\phi\in \rmPhi} P_i( e)\cdot   (P_i^+)^{\val{\phi}} ( g(e) )
\cdot \overline{\pre}_\mb{M}(e\mid \val{\phi})  
\\
& =  (P_i^+)' (g) \tag{Definition \ref{def: prod F over E} on $\mb{M}^+$}
\end{align*}
\end{proof}

\subsection*{Proof of Lemma \ref{lem:epistHA:i-minimal-2}} 
\label{app:section4:4.2}

\paragraph{Lemma \ref{lem:epistHA:i-minimal-2}.}
For any  epistemic Heyting algebra  $\mb{A}$ and any $a\in \mb{A}$, if $[b]\in \mathsf{Min}_i(\mb{A}^a)$, then  $\blacklozenge_i(b\wedge a)$ is the unique $i$-minimal element of $\mb{A}$ which belongs to $[b]$.

\begin{proof}
Let us first prove that $\blacklozenge_i(b\wedge a) \in [b]$. 
By assumption, $[b] \in \mathsf{Min}_i(\mb{A}^a)$,
hence
$ [b] = \blacklozenge_i^a[b] = b \wedge a = 
 \blacklozenge_i (b \wedge a ) \wedge a$.
This implies that $\blacklozenge_i (b \wedge a ) \in [b]$.

\medskip

Now, we need to show  that $\blacklozenge_i(b\wedge a)$ is an $i$-minimal element of $\mb{A}$. 
Hence, we need to prove that $\blacklozenge_i(b\wedge a)$ satisfies items \ref{def:i-minimal:item:bot}, \ref{def:i-minimal:item:fixpoint}, and \ref{def:i-minimal:item:minimal} of 
Definition \ref{def:i-minimal}.\\

%%% item 1
\texttt{Proof of item \ref{def:i-minimal:item:bot}.}
By assumption, $[b] \in \mathsf{Min}_i(\mb{A}^a)$, hence $[b]\neq \bot$ and $b\wedge a \neq \bot$.
Since $\blacklozenge_i$ is reflexive (Definition \ref{def: epist algebra}, axiom \eqref{axiom:epist-alg:refl}),
$ \bot \neq b\wedge a\leq \blacklozenge_i(b\wedge a)$, which shows that $\blacklozenge_i(b\wedge a)\neq \bot$ as required.\\

%%% item 2
\texttt{Proof of item \ref{def:i-minimal:item:fixpoint}.}
Since
$\blacklozenge_i$ is transitive (Definition \ref{def: epist algebra}, axiom \eqref{axiom:epist-alg:trans}), we have that
$\blacklozenge_i(b\wedge a) = \blacklozenge_i \blacklozenge_i(b\wedge a)$ as required. \\

%%% item 3
\texttt{Proof of item \ref{def:i-minimal:item:minimal}.}
Let $c\in \mathsf{Min}_i(\mb{A})$ and $c\leq \blacklozenge_i(b\wedge a)$. 
We need to  prove that $c= \blacklozenge_i(b\wedge a)$.
To do so, we follow the following steps: 
\begin{enumerate}[(i)]
\item we prove that $[b]=[c]$,
\item we show that $c\wedge a \neq \bot$, 
\item we prove that $\blacklozenge_i(b\wedge a)$.
\end{enumerate}
%By assumption, $\bot \neq  c$ and $\blacklozenge_i c = c$.
%$\bot\neq c\leq \blacklozenge_i(b\wedge a)$ such that $\blacklozenge_i c = c$, and let us show that $\blacklozenge_i(b\wedge a) = c$. 

\texttt{Step (i).} From the assumptions that $c\leq \blacklozenge_i(b\wedge a)$ and that $[b] = \blacklozenge_i^a[b]$, we get that $c\wedge a\leq \blacklozenge_i(b\wedge a)\wedge a = b\wedge a$, which proves that $[c]\leq [b]$. 

\smallskip

\texttt{Step (ii).}  Since $c\leq\lozenge_i(b\land a)$, we have that $c\leq\lozenge_i a$, that is $c=c\land\lozenge_i a$. This gives the following chain of equalities: 
$$c= c\land\lozenge_i a=\lozenge_ic\land\lozenge_ia=\lozenge_i(\lozenge_i c\land a).$$ 
The last equality is true in all monadic Heyting algebras (see e.g.\ \cite[Definition 1]{bezhanishvili1998varieties}). Now, since $\lozenge_i c=c$, we get that $c=\lozenge_i(c\land a)$, which implies
$\lozenge_i(c\land a)\neq\bot$ and $c\land a\neq\bot$. 

\smallskip

\texttt{Step (iii).}  By Lemma \ref{lem:epistHA:i-minimal-1},
$[c]\in \mathsf{Min}_i(\mb{A}^a)$.
By the $i$-minimality of $[b]$, we get $[b] = [c]$, that is $b\wedge a = c\wedge a$. Hence $\blacklozenge_i(b\wedge a) = \blacklozenge_i(c\wedge a)\leq \blacklozenge_i(c) = c$, which, together with the assumption that $c\leq \blacklozenge_i(b\wedge a)$, proves that $\blacklozenge_i(b\wedge a) = c$, as required. This finishes the proof that $\blacklozenge_i(b\wedge a)$ is an $i$-minimal element of $\mb{A}$. 

\bigskip

To show the uniqueness, let $c_1, c_2\in [b]$ and assume that both $c_1$ and $c_2$ are $i$-minimal elements of $\mb{A}$. Then $c_1\wedge a = c_2\wedge a$, and hence $\blacklozenge_i(c_1\wedge a) = \blacklozenge_i(c_2\wedge a)$. Reasoning as above, one can show that $\bot\neq \blacklozenge_i(c_j\wedge a)\leq c_j$ and $\blacklozenge_i(c_j\wedge a)$ is a fixed point of $\blacklozenge_i$ for $1\leq j\leq 2$. Hence,  the $i$-minimality of $c_j$ implies that $\blacklozenge_i(c_j\wedge a) =  c_j$. Thus, the following chain of identities holds:\[c_1 = \blacklozenge_i(c_1\wedge a) = \blacklozenge_i(c_2\wedge a) = c_2.\]
\end{proof}

\subsection*{Proof of Proposition \ref{prop:APEstruct-EoverF}} 
\label{app:section4:5}

\paragraph{Proposition \ref{prop:APEstruct-EoverF}.}
For any APE-structure $\mathcal{F}$ and any event structure $\mb{E}$ over the support of $\mathcal{F}$, the tuple
$\mathcal{F}^{\mb{E}}= (\mb{A}^{\mb{E}}, (\mu_i^{\mb{E}})_{i\in \Ag})$ is an APE-structure.

\begin{proof}%\redbf{to check}
Let $\mb{E} = (E, (\sim_i)_{i\in\Ag}, (P_i)_{i\in \Ag}, \bfPhi, \overline{\pre})$ be an event structure and 
$ \mathcal{F}:=  \left( \mb{A}, (\mu_i)_{i\in \Ag} \right) $ be an APE-structure.
To prove that $\mathcal{F}^{\mb{E}}$ is an APE-structure (see Definition \ref{def: alg probab epist structure}), we need to prove that 
%\begin{enumerate}[a)]
%\item 
$\mb{A}^{\mb{E}}$ is an  epistemic Heyting algebra (see Definition
\ref{def:epist-Heyting-algebra}), and that
%\item 
each map $\mu_i^\mb{E}$ is an $i$-measure on $\mb{A}^\mb{E}$. 
%\end{enumerate}
By Proposition \ref{prop:pseudoisep}, $\mb{A}^{\mb{E}}$ is an epistemic Heyting algebra.
Hence, it remains to prove that, for each $i\in \Ag$, the map $\mu_i^\mb{E}$ is an $i$-measure (see Definition \ref{def:measures}), i.e.\ we need to prove that:
\begin{enumerate}
	    \item \label{proof:APEstruct-EoverF:epAlg:two:domain}
	    $\mathsf{dom}(\mu_i^\mb{E})=\mathsf{Min}_i(\mb{A}^\mb{E}){\downarrow}$;
		\item \label{proof:APEstruct-EoverF:epAlg:two:monotone}
$\mu_i^\mb{E}$ is order-preserving; %$\mu(\bot)=0$  and
		\item \label{proof:APEstruct-EoverF:epAlg:two:join}
for every $a\in \mathsf{Min}_i(\mb{A}^\mb{E})$ and all $b, c\in a{\downarrow}$, it holds that $\mu_i^\mb{E}(b\vee c) = \mu_i^\mb{E}(b)+ \mu_i^\mb{E}(c) - \mu_i^\mb{E}(b\wedge c)$;
		\item \label{proof:APEstruct-EoverF:epAlg:two:bot}  $\mu_i^\mb{E}(\bot)=0$ if $\mathsf{dom}(\mu_i^\mb{E})\neq\varnothing$;
		\item \label{proof:APEstruct-EoverF:epAlg:two:fixedpoints}
$\mu_i^\mb{E}(a) = 1$ for every $a\in \mathsf{Min}_i(\mb{A}^\mb{E})$;
		\item \label{proof:APEstruct-EoverF:epAlg:two:nonzero} for every $a\in \mathsf{Min}_i(\mb{A}^\mb{E})$ and all $b, c\in a{\downarrow}$ such that $b<c$, it holds that $\mu_i^\mb{E}(b)<\mu_i^\mb{E}(c)$.
%		\item \label{def:epAlg:two:top} $\mu(\top)=1$.
	\end{enumerate}

\texttt{Proof of \eqref{proof:APEstruct-EoverF:epAlg:two:domain}.}
This condition is satisfied by definition.

\medskip

The remaining items, are trivially satisfied if 
the domain of $\mu_i^\mb{E}$ is empty. For the remaining of the proof,
let us assume that the domain of $\mu_i^\mb{E}$ is non-empty.

\medskip

\texttt{Proof of item \eqref{proof:APEstruct-EoverF:epAlg:two:monotone}.}
The definition of $\mu'_i$ (see Definition \ref{def: prod F over E}), the Proposition \ref{prop: muai properties} and the fact that, if $\overline{\pre}(e\mid a)\neq 0$, then $a\leq\overline{pre}(e)$ (see Definition of $\overline{pre}$ \eqref{def:overlinePRE}), imply that  $\mu'_i(g)=\mu'_i(g\land\overline{pre})$.
Assume that $[g_1]\leq[g_2]\leq[f_{e,a}]$. This means  that $g_1\land\overline{pre}\leq g_2\land \overline{pre}$. Since $\mu'_i$ is
an $i$-premeasure (\autoref{prop:intermediate-ha:i-premeasure}), it 
is monotone. Hence,   $\mu'_i(g_1)=\mu'_i(g_1\land\overline{pre})\leq\mu'_i(g_2\land \overline{pre})=\mu'_i(g_2)$. This implies that $$\frac{\mu'_i(g_1)}{\mu'_i(f_{e,a})}\leq\frac{\mu'_i(g_2)}{\mu'_i(f_{e,a})}$$ that is, $\mu_i^{\mb{E}}([g_1])\leq\mu_i^{\mb{E}}([g_2])$. 

\medskip

\texttt{Proof of item \eqref{proof:APEstruct-EoverF:epAlg:two:join}.}
Let $[g_1]$ and $[g_2]$ in $\mathcal{F}^{\mb{E}}$ such that $[g_1] \leq [f_{e,a}]$ and $[g_2] \leq [f_{e,a}]$. We have:
\begin{align*}
& \quad \;\: \mu_i^{\mb{E}}([g_1]\lor[g_2]) \\
&=\frac{\mu'_i((g_1\land \overline{pre})\lor(g_2\land \overline{pre}))}{\mu'_i(f_{e,a})}\\
& = \frac{\mu'_i(g_1\land \overline{pre})+\mu'_i(g_2\land \overline{pre})-\mu'_i((g_1\land g_2)\land\overline{pre})}{\mu'_i(f_{e,a})}
\tag{Proposition \ref{prop:intermediate-ha:i-premeasure}. $\mu_i'$ is an $i$-premeasure}
\\
& = \frac{\mu'_i(g_1\land \overline{pre})}{\mu'_i(f_{e,a})}+\frac{\mu'_i(g_2\land \overline{pre})}{\mu'_i(f_{e,a})}-\frac{\mu'_i((g_1\land g_2)\land \overline{pre})}{\mu'_i(f_{e,a})}\\
& = \frac{\mu'_i(g_1)}{\mu'_i(f_{e,a})}+\frac{\mu'_i(g_2)}{\mu'_i(f_{e,a})}-\frac{\mu'_i(g_1\land g_2)}{\mu'_i(f_{e,a})}\\
& = \mu_i^{\mb{E}}([g_1])+\mu_i^{\mb{E}}([g_2])-\mu_i^{\mb{E}}([g_1\land g_2]).
\end{align*}

\medskip

\texttt{Proof of Items \eqref{proof:APEstruct-EoverF:epAlg:two:bot} and
\eqref{proof:APEstruct-EoverF:epAlg:two:fixedpoints}.}
Trivial.

\medskip

\texttt{Proof of item \eqref{proof:APEstruct-EoverF:epAlg:two:nonzero}.}
Recall that,  if $[g]\neq\bot$, then $\mu_i^\mb{E}([g])>0$ (see Claim in Lemma \ref{lemma:pseudomuwelldefined}). Let $\bot\neq[g]<[h]$. 
The monotonicity of  the $\mu^a_i$ guarantees that, for all $e\in E$ and $a\in\rmPhi$, we have  
$$P_i(e)\cdot\mu^a_i(g(e))\cdot\overline{\pre}(e | a) \quad \leq \quad  P_i(e)\cdot\mu^a_i(h(e))\cdot\overline{\pre}(e | a).$$ 
Furthermore, since $[g]<[h]$, there  exists an $e\in E$ such that the set $$\{\ a\in\rmPhi\quad \mid\quad \overline{\pre}(e | a)>0\text{ and }g(e)\land a< h(e)\land a\ \}$$ is non-empty. 
Since $\rmPhi$ is finite, the order $\prec$ is well-founded and the aforementioned set contains at least one minimal element. Let $a_0$ be such a minimal element. From Definition \ref{def:algebraic event}, we have that, $\overline{\pre}(e | b)>0$ for all $b\in\rmPhi$ with  $b \prec a_0$. 
By the minimality of $a_0$, we have that $g(e)\land b=h(e)\land b$  for all such $b\prec a_0$. 
Hence, 
$$\sum_{b\in\mathrm{mb}(a_0)}\mu_i(g(e)\land b)  = \sum_{b\in\mathrm{mb}(a_0)}\mu_i(h(e)\land b) $$
where  $\mathrm{mb}(a)$ denotes the multiset of the $\prec$-maximal elements of $\bfPhi$ $\prec$-below $a$ (see \autoref{def:mua-mba}).
Since  $\mathcal{F}$ is an APE-structure, $\mu_i$ is strictly monotone. 
Hence, 
$g(e)\land a_0 < h(e)\land a_0$ implies that
\begin{align*}
\mu_i^{a_0}(g(e))  & =  \mu_i( g(e) \land a_0)-\sum_{b\in\mathrm{mb}(a_0)}\mu_i(g(e)\land b)
\\
& < \mu_i( h(e) \land a_0)-\sum_{b\in\mathrm{mb}(a_0)}\mu_i(h(e)\land b)\\
 &  = 
\mu_i^{a_0}(h(e)) .
\end{align*}
Hence, for some $e\in E$ and $a\in\rmPhi$, we have  
$$P_i(e)\cdot\mu^a_i(g(e))\cdot\overline{\pre}(e | a)< P_i(e)\cdot\mu^a_i(h(e))\cdot\overline{\pre}(e | a).$$ 
The inequality above, the definition of $\mu_i'$ (see \autoref{def: prod F over E}) and the monotonicity of $\mu_i'$ (see \autoref{prop:intermediate-ha:i-premeasure}) imply that 
$\mu'_i([g])<\mu'_i([h])$, which in turn implies that $\mu_i^\mb{E}([g])<\mu_i^\mb{E}([h])$.
\end{proof}

\subsection*{Proof of Lemma \ref{lem:duality:probabilities}}
\label{app:section4:6}

\paragraph{Lemma \ref{lem:duality:probabilities}.}
For any PES-model $\mb{M}$ and any event structure $\mathcal{E}$ over $\mathcal{L}$, $$(P_i^+)^{\mb{E}_\mathcal{E}} \cong (P_i^\mathcal{E})^+.$$

\begin{proof}
Using Definitions \ref{def: update PES model}
and \ref{def: complex algebra}, we get that:
for any $X \in \mathsf{Min}_i((\mb{M}^\mathcal{E})^+){\downarrow} $,
$$(P_i^\mathcal{E})^+ (X) =
\sum_{(s,e)\in X}
\frac{P_i(e) \cdot P_i(s) \cdot \pre(e\mid s)}{
\sum_{(s',e') \sim_i (s,e)} 
P_i(e') \cdot P_i(s') \cdot \pre(e'\mid s')}.
$$
By using Definitions \ref{def: complex algebra} and 
\ref{def: updated APE structure F of E}, we get that:
for any $[g]\in \mathsf{Min}_i((\mb{M}^+)^\mb{E}){\downarrow}$ ,
$$(P_i^+)^{\mb{E}_\mathcal{E}} ( [g] )
= \frac{\sum_{e\in E} \sum_{\phi \in \rmPhi} P_i(e) \cdot 
(P_i^+)^{\val{\phi}}(g(e)) \cdot \overline{\pre}(e\mid \val{\phi})}{
\sum_{e\in E} \sum_{\phi\in \rmPhi} P_i(e) \cdot 
(P_i^+)^{\val{\phi}}(f(e)) \cdot \overline{\pre}(e\mid \val{\phi})}.
$$
Let $X \in \mathsf{Min}_i((\mb{M}^\mathcal{E})^+){\downarrow} $.
Following the notation introduced in the proof of \autoref{prop: plus of coprod same as prod of plusses}, let 
$[g]\in \mathsf{Min}_i((\mb{M}^+)^\mb{E}){\downarrow}$ be the map  such that
\begin{align*}
g: E & \to \mathcal{P}(S)\\
e & \mapsto X_e: = \{s\in S\mid (s, e)\in X\}.
\end{align*} 
Notice that $X $ is a subset of one of the $i$-equivalence classes of $(\mb{M}^\mathcal{E})^+$,
hence $g = g \wedge \overline{pre}$ and $[g] \leq [f]$ for some $[f] \in \mathsf{Min}_i((\mb{M}^+)^\mb{E}){\downarrow}$.
Let 
$$[X]_i := \{ (s,e) \mid \exists (s',e')\in X, \ (s,e)\sim_i (s',e')\}.$$ 
We can easily see that $([X]_i)_e = f(e)$ where $f$ is the canonical representative of $[f]$.
We have:
\begin{align*}
& \quad \; \: (P_i^\mathcal{E})^+ (X) 
\\
& =
\sum_{(s,e)\in X}
\frac{P_i(e) \cdot P_i(s) \cdot \pre(e\mid s)}{
\sum_{(s',e') \sim_i (s,e)} 
P_i(e') \cdot P_i(s') \cdot \pre(e'\mid s')}
\\
& =
\frac{\sum_{(s,e)\in X} P_i(e) \cdot P_i(s) \cdot \pre(e\mid s)}{
\sum_{(s',e') \in [X]_i} 
P_i(e') \cdot P_i(s') \cdot \pre(e'\mid s')}
\tag{$X $ is a subset of the equivalence classes $ [X]_i$}
\\
& = 
\frac{\sum_{e\in E}  P_i(e) \cdot \sum_{s\in X_e} P_i(s) \cdot \pre(e\mid s)}{
\sum_{e'\in E} P_i(e') \cdot 
\sum_{s' \in f(e')} P_i(s') 
\cdot \pre(e'\mid s')}
\tag{$([X]_i)_e = f(e)$}
\\
& = \frac{\sum_{e\in E} P_i(e) \cdot
\sum_{\phi \in \rmPhi} \pre(e\mid \phi) \cdot
\sum_{s\in g(e) \cap \val{\phi}}   P_i(s)  }{
\sum_{e'\in E} P_i(e') \cdot 
\sum_{\phi\in\rmPhi} \pre(e'\mid \phi) \cdot
\sum_{s' \in f(e') \cap \val{\phi}} 
 P_i(s')}
 \tag{In the classical case, $\rmPhi$ gives a partition of $S^\mathcal{E}$}
\\
& = \frac{\sum_{e\in E} P_i(e) \cdot
\sum_{\phi \in \rmPhi} \pre(e\mid \phi) \cdot 
(P_i^+)(g(e)\cap\val{\phi})  }{
\sum_{e'\in E} P_i(e') \cdot 
\sum_{\phi\in\rmPhi} \pre(e'\mid \phi) \cdot
(P_i^+)(f(e) \cap\val{\phi})}
\tag{\autoref{def:Intuitionist complex algebra}}
\\
& = \frac{\sum_{e\in E} P_i(e) \cdot
\sum_{\phi \in \rmPhi} \pre(e\mid \phi) \cdot (P_i^+)^{\val{\phi}}(g(e))  }{
\sum_{e'\in E} P_i(e') \cdot 
\sum_{\phi\in\rmPhi} \pre(e'\mid \phi) \cdot
(P_i^+)^{\val{\phi}}(f(e) )}
\tag{Remark \ref{rk:mb:classic:empty} : $\mathrm{mb}(\val{\phi})=\varnothing$}
\\
& = \frac{\sum_{e\in E} \sum_{\phi \in \rmPhi} P_i(e) \cdot 
(P_i^+)^{\val{\phi}}(g(e)) \cdot \overline{\pre}(e\mid \val{\phi})}{
\sum_{e\in E} \sum_{\phi\in \rmPhi} P_i(e) \cdot 
(P_i^+)^{\val{\phi}}(f(e)) \cdot \overline{\pre}(e\mid \val{\phi})}
\\
& = 
(P_i^+)^{\mb{E}_\mathcal{E}} ( [g] ).
\end{align*}

\end{proof}

\section{Proof of the Soundness of IPDEL}
\label{Appendix:soundness}
\begin{proposition}[Soundness]
\label{lem:soundness-IPDEL}
The axiomatization for IPDEL given in Table \ref{table:IPDEL} is sound w.r.t.\ APE-models.
\end{proposition}

 %\redbf{APOSTOLOS: Please check if this proof is detailed enough; I think it is}
By definition, the underlying structure of an APE-structures is an epistemic Heyting algebra. Hence, it satisfies the axioms of intuitionistic propositional logic and the axioms M1 -- M7 and E for static modalities.

\subsection*{Axioms for inequalities.} As discussed in Remark \ref{rk:basicEHA}, it is the case that $\bigvee\mathsf{Min}_i(\mb{A})=\top$ for every epistemic Heyting algebra $\mb{A}$. This implies that axioms N0 and N5 are satisfied in every APE-model. Axioms N1, N2, N3, N4 and N6 are also satisfied because if the valuation of their antecedent is above any $i$-minimal element $a$ then so will be the valuation of their succedent. \\

\subsection*{Axioms for probabilities.} The fact that axioms P1-P3 are satisfied in every  APE-model is shown similarly as axiom N0. Since $\lozenge_i\mb{A}$ is a subalgebra of $\mathbb{A}$ for every epistemic Heyting algebra $\mathbb{A}$, it is the case that $\val{\varphi}_\mathcal{M}\in\lozenge_i\mathbb{A}$ for every $i$-probability formula $\varphi$ and every APE-model based on $\mathbb{A}$. Hence,  \autoref{lem:epist-Hey-alg:diamond-A} implies the satisfiability of P5. 

Finally, let us show that P4 is satisfied in every APE-model based on $\mathbb{A}$. For the right to left direction, as discussed in Remark \ref{rk:basicEHA}, every element of $\lozenge_i\mathbb{A}$ can be written as a union of $i$-minimal elements and therefore $\val{\Box_i(\varphi\leftrightarrow\psi)}=\bigvee\{a\in\mathsf{Min}_i(\mb{A})\mid a\land\val{\varphi}=a\land\val{\psi}\}$. This of course implies that $\bigvee\{a\in\mathsf{Min}_i(\mb{A})\mid a\land\val{\varphi}=a\land\val{\psi}\}\leq\val{\mu_i(\varphi)=\mu_i(\psi)}$. 
As for the left to right direction, we have that $\val{\Box_i(\varphi\to\psi)\land(\mu(\varphi)=\mu(\psi))}=\bigvee\{a\in\mathsf{Min}_i(\mb{A})\mid a\land\val{\varphi}\leq a\land\val{\psi}\textrm{ and }\mu_i(a\land\val{\varphi})=\mu_i(a\land\val{\psi})\}$. 
By the strict monotonicity of the $i$-measure $\mu_i$, the following holds, as required. $$\val{\Box_i(\varphi\to\psi)\land(\mu(\varphi)=\mu(\psi))}\leq\bigvee\{a\in\mathsf{Min}_i(\mb{A})\mid a\land\val{\varphi}=a\land\val{\psi}\}=\val{\Box_i(\varphi\leftrightarrow\psi)}.$$

\subsection*{Reduction axioms}

In this section, we aim at proving the soundness of the reduction axioms as stated in 
Lemma \ref{lem:soundness-IPDEL}. 
To do so we need to define two maps $F$ and $f$ as follows.

\subsubsection*{Preliminary results}

Throughout this section, we let  $\mb{A}$ denote the complex algebra of a model $\mc{M}$ and  $\mathcal{E}$ denote an event structure. Recall the definition of the event structure $\mathbb{E}_\mathcal{E}$ (cf.\  \Cref{def:induced:event:structure}). Then we define
a map $F : \mc{L} \rightarrow \prod_{\mathbb{E}_\mathcal{E}}\mb{A}$ that associates  an element in $\prod_{\mathbb{E}_\mathcal{E}}\mb{A}$ to each formula.
We want $F$ (Definition \ref{def:the map F}) to be the map such that 
\[\val{\psi}_{\mathcal{M}^{\mb{E}_\mathcal{E}}} = [F(\psi)].\]
$\val{\psi}_{\mathcal{M}^{\mb{E}_\mathcal{E}}} $ is the evaluation of the formula $\psi$ in the updated algebra $\mb{A}^{\mb{E}_\mathcal{E}}$ corresponding to the updated model $\mathcal{M}^{\mb{E}_\mathcal{E}}$.
Hence, $F(\psi)$ is a representative of the equivalence class 
$\val{\psi}_{\mathcal{M}^{\mb{E}_\mathcal{E}}} $ in the product algebra $\mb{A}^\Pi$.

Since $F(\psi) \in \mb{A}^\Pi$, $F(\psi)$ is a tuple of elements of the algebra $\mb{A}$. To aid the computation, we define  the map 
$f : \mc{L} \times E \rightarrow \mc{L}$ (see Definition \ref{definition: fonestep}) such that $F(\psi)(e)=\val{f(\psi,e)}_\mathcal{M}$.
This means that $f(\psi,e)$ is a formula such that its evaluation 
$\val{f(\psi,e)}_\mathcal{M}$ in the algebra $\mb{A}$ is equal to the $e^{th}$ coordinate of the tuple $F(\psi)$.
We first prove that the maps $F$ and $f$ have the desired properties in Lemma \ref{lemma: newf(psi)}.
Then we prove the key lemma \ref{lemma: newf(psi) second property} that we will use to prove the reduction axioms (see Section \ref{ssec:Appendix:soundness}).

\begin{definition}\label{definition: fonestep}
The map $f \ : \ \mathcal{L} \times E \rightarrow \mathcal{L}$ is defined by recursion as follows: for every $\psi\in \mathcal{L}$ and $e\in E$, 
%let us define by recursion the formula $f(\psi,e)$:
%
\begin{align*}
f(p,e)  = \;& \sub(e,p),
\\
f(\bot,e)  = \;& \bot,
\\
f(\top,e) = \;& \top,
\\
f(\psi_1\wedge \psi_2,e) = \;& f(\psi_1,e)\wedge f(\psi_2,e),
\\
f(\psi_1\vee \psi_2,e) = \;& f(\psi_1,e)\vee f(\psi_2,e),
\\
f(\psi_1\rightarrow \psi_2,e) = \;& f(\psi_1,e)\rightarrow f(\psi_2,e),
\\
f(\lozenge_i \psi,e) = \;& \bigvee_{e'\sim_i e}\lozenge_i(f(\psi,e')\wedge pre(e')),
\\
f(\Box_i\psi,e) = \;&\bigwedge_{e'\sim_i e}\Box_i(pre(e')\to f(\psi,e')),
\\
f(\langle\mathcal{E}', e'\rangle \psi,e) = \;& f(pre(e')\land f(\psi,e'),e),
\\
f(\left[\mathcal{E}',e'\right]\psi,e) = \;& f(pre(e')\rightarrow f(\psi,e'),e),
\\
f(\alpha\mu_i(\psi)\geq \beta,e) = \;& \alpha\sum_{
	\begin{smallmatrix}
        e'\sim_i e\\
		\phi \in \rmPhi
		%\val{\phi}_{\mathcal{M}}\wedge (\val{p}_{\mathcal{M}}\wedge \overline{\pre}_{\mathcal{M}}(e')\wedge a)\neq \bot
	\end{smallmatrix}
}
\mu^\phi_i(f(\psi,e'))\cdot P_i(e')\cdot \pre(e'\mid \phi)+\sum_{
	\begin{smallmatrix}
        e'\sim_i e\\
		\phi \in \rmPhi
		%\val{\phi}_{\mathcal{M}}\wedge (\val{p}_{\mathcal{M}}\wedge \overline{\pre}_{\mathcal{M}}(e')\wedge a)\neq \bot
	\end{smallmatrix}
}
-\beta\mu^\phi_i(\top)P_i(e') \pre(e'\mid \phi) \geq 0.
\end{align*}
\end{definition}

\begin{definition}
\label{def:the map F}
Let us define the map 
$F_{\mb{E}_\mathcal{E}} : \mc{L} \rightarrow \mb{A}^\Pi$ such that 
for every $e\in E$, the $e^{th}$ coordinate of $F_{\mb{E}_\mathcal{E}}(\psi)$ is 
equal to $\val{f(\psi,e)}_{\mc{M}}$.

For the sake readability, we will omit the subscript when it causes no confusion.
\end{definition}

\begin{lemma}
\label{lemma: newf(psi)}
For $\mathcal{M}$ and $\mathcal{E}$ as above, \[\val{\psi}_{\mathcal{M}^{\mb{E}_\mathcal{E}}} = [F(\psi)]\] where $F(\psi)(e)=\val{f(\psi,e)}_\mathcal{M}$.
\end{lemma}
\begin{proof}
The proof is by induction on the complexity of $\psi$ with $(\mathtt{IH}_\psi) \; : \;  \val{\psi}_{\mathcal{M}^{\mb{E}_\mathcal{E}}} = [F(\psi)]$. 
The statement is trivially true in the base cases and if the main connective are $\wedge$, $\vee$ or  $\rightarrow$.\\ 
If $\psi = \lozenge_i \psi'$, then 
\begin{align*}
\val{\lozenge_i\psi'}_{\mathcal{M}^{\mb{E}_\mathcal{E}}} & = \lozenge^{\mb{E}_\mathcal{E}}\val{\psi'}_{\mathcal{M}^{\mb{E}_\mathcal{E}}}\\
 & = \lozenge_i^{\mb{E}_\mathcal{E}} [F(\psi')]\tag{$\mathtt{IH}_{\psi'}$}\\
 & = [\lozenge_i^{\prod}(F(\psi')\wedge \overline{pre}_\mathcal{M})]
\end{align*}
and
\begin{align*}
\lozenge_i^{\prod}(F(\psi')\wedge \overline{pre}_\mathcal{M})(e) & =\bigvee_{e'\sim_ie}\{\lozenge_i(F(\psi')(e')\land\overline{pre}(e'))\}\tag{\autoref{def:support intermediate APE}}\\
 & = \bigvee_{e'\sim_ie}\{\lozenge_i(\val{f(\psi',e')}_{\mathcal{M}}\land\overline{pre}(e'))\}\\
 & = \val{\bigvee_{e'\sim_ie}\lozenge_i(f(\psi',e')\land pre(e'))}_\mathcal{M}\\
 & = \val{f(\lozenge_i\psi',e)}_\mathcal{M}\tag{Definition \ref{definition: fonestep}}\\
 & = F(\lozenge_i\psi')(e)
\end{align*}
If $\psi=\Box_i\psi'$, then
\begin{align*}
	\val{\Box_i\psi'}_{\mathcal{M}^{\mb{E}_\mathcal{E}}} & = \Box_i^{\mb{E}_\mathcal{E}}\val{\psi'}_{\mathcal{M}^{\mb{E}_\mathcal{E}}}\\
	& = \Box_i^{\mb{E}_\mathcal{E}} [F(\psi')]\tag{$\mathtt{IH}_{\psi'}$}\\
	& = [\Box_i^{\prod}(\overline{pre}_\mathcal{M}\to F(\psi'))]
\end{align*}
and
\begin{align*}
	\Box_i^{\prod}(\overline{pre}_\mathcal{M}\to F(\psi'))(e) & =\bigwedge_{e'\sim_ie}\{\Box_i(\overline{pre}(e')\to F(\psi')(e'))\}\tag{\autoref{def:support intermediate APE}}\\
	& = \bigwedge_{e'\sim_ie}\{\Box_i(\overline{pre}(e')\to \val{f(\psi',e')}_{\mathcal{M}})\}\\
	& = \val{\bigwedge_{e'\sim_ie}\Box_i(pre(e')\to f(\psi',e'))}_\mathcal{M}\\
	& = \val{f(\Box_i\psi',e)}_\mathcal{M}\tag{Definition \ref{definition: fonestep}}\\
	& = F(\Box_i\psi')(e).
\end{align*}
If $\psi = \alpha\mu_i( \psi')\geq \beta$, then 
\begin{align*}
& \val{\alpha\mu_i( \psi')\geq \beta}_{\mathcal{M}^{\mb{E}_\mathcal{E}}} \\
& = \bigvee \left\{[f_{e, a}] \; \middle\vert \; \alpha \mu_i^{\mb{E}_\mathcal{E}}(\val{\psi'}_{\mathcal{M}^{\mb{E}_\mathcal{E}}}\wedge [f_{e, a}])\geq \beta\right\}\\
& = \bigvee \left\{[f_{e, a}] \; \middle\vert \; \alpha \mu_i^{\mb{E}_\mathcal{E}}([F(\psi')]\wedge [f_{e, a}])\geq \beta\right\}\tag{$\mathtt{IH}_{\psi'}$}\\
& = \bigvee \left\{[f_{e, a}] \; \middle\vert \; \alpha\frac{\mu'_i(F(\psi')\wedge f_{e,a})}{\mu'_i(f_{e,a})}\geq \beta\right\}\tag{\autoref{def: updated APE structure F of E}}\\
& = \left[\bigvee \left\{f_{e, a} \; \middle\vert \; \alpha\frac{\mu'_i(F(\psi')\wedge f_{e,a})}{\mu'_i(f_{e,a})} \geq \beta\right\}\right]\\
& = \left[\bigvee \left\{f_{e, a} \; \middle\vert \; \alpha\mu'_i(F(\psi')\wedge f_{e,a})-\beta\mu'_i(f_{e,a}) \geq 0\right\}\right]\\
& = \left[\bigvee \left\{f_{e, a} \; \middle\vert \; \alpha\sum_{
	\begin{smallmatrix}
        e'\sim_i e\\
		\phi \in \rmPhi
		%\val{\phi}_{\mathcal{M}}\wedge (\val{p}_{\mathcal{M}}\wedge \overline{\pre}_{\mathcal{M}}(e')\wedge a)\neq \bot
	\end{smallmatrix}
}
\mu^\phi_i(F(\psi')(e')\wedge a)\cdot P_i(e')\cdot \pre(e'\mid\phi) +\sum_{
	\begin{smallmatrix}
        e'\sim_i e\\
		\phi \in \rmPhi
		%\val{\phi}_{\mathcal{M}}\wedge (\val{p}_{\mathcal{M}}\wedge \overline{\pre}_{\mathcal{M}}(e')\wedge a)\neq \bot
	\end{smallmatrix}
}
-\beta\mu^\phi_i(a)\cdot P_i(e')\cdot \pre(e'\mid\phi)\geq 0\right\}\right]
\end{align*}
and
\begin{align*}
& \left(\bigvee \left\{f_{e, a} \; \middle\vert \;  \alpha\sum_{
	\begin{smallmatrix}
        e'\sim_i e\\
		\phi \in \rmPhi
		%\val{\phi}_{\mathcal{M}}\wedge (\val{p}_{\mathcal{M}}\wedge \overline{\pre}_{\mathcal{M}}(e')\wedge a)\neq \bot
	\end{smallmatrix}
}
\mu^\phi_i(F(\psi')(e')\wedge a) P_i(e') \pre(e'\mid\phi) +\sum_{
	\begin{smallmatrix}
        e'\sim_i e\\
		\phi \in \rmPhi
		%\val{\phi}_{\mathcal{M}}\wedge (\val{p}_{\mathcal{M}}\wedge \overline{\pre}_{\mathcal{M}}(e')\wedge a)\neq \bot
	\end{smallmatrix}
}
-\beta\mu^\phi_i(a) P_i(e') \pre(e'\mid\phi)\geq 0 \right\}\right)(d)\\
& = \left( \bigvee \left\{f_{e, a} \; \middle\vert \;  \alpha\sum_{
	\begin{smallmatrix}
        e'\sim_i e\\
		\phi \in \rmPhi
		%\val{\phi}_{\mathcal{M}}\wedge (\val{p}_{\mathcal{M}}\wedge \overline{\pre}_{\mathcal{M}}(e')\wedge a)\neq \bot
	\end{smallmatrix}
}
\mu^\phi_i( \val{f(\psi',e')}_{\mathcal{M}}\wedge a)P_i(e') \pre(e'\mid\phi)+\sum_{
	\begin{smallmatrix}
        e'\sim_i e\\
		\phi \in \rmPhi
		%\val{\phi}_{\mathcal{M}}\wedge (\val{p}_{\mathcal{M}}\wedge \overline{\pre}_{\mathcal{M}}(e')\wedge a)\neq \bot
	\end{smallmatrix}
}
-\beta\mu^\phi_i(a)P_i(e') \pre(e'\mid\phi)\geq 0 \right\}\right)(d)\\
& = \bigvee \left\{a  \; \middle\vert \;  \alpha\sum_{
	\begin{smallmatrix}
        e'\sim_i d\\
		\phi \in \rmPhi
		%\val{\phi}_{\mathcal{M}}\wedge (\val{p}_{\mathcal{M}}\wedge \overline{pre}_{\mathcal{M}}(e')\wedge a)\neq \bot
	\end{smallmatrix}
}
\mu^\phi_i(\val{f(\psi',e')}_{\mathcal{M}}\wedge a)P_i(e') \pre(e'\mid\phi)+\sum_{
	\begin{smallmatrix}
        e'\sim_i d\\
		\phi \in \rmPhi
		%\val{\phi}_{\mathcal{M}}\wedge (\val{p}_{\mathcal{M}}\wedge \overline{\pre}_{\mathcal{M}}(e')\wedge a)\neq \bot
	\end{smallmatrix}
}
-\beta\mu^\phi_i(a)P_i(e') \pre(e'\mid\phi)\geq 0 \right\}\\
& = \val{\alpha\sum_{
	\begin{smallmatrix}
        e'\sim_i d\\
		\phi \in \rmPhi
		%\val{\phi}_{\mathcal{M}}\wedge (\val{p}_{\mathcal{M}}\wedge \overline{\pre}_{\mathcal{M}}(e')\wedge a)\neq \bot
	\end{smallmatrix}
}
\mu^\phi_i(f(\psi',e'))P_i(e') \pre(e'\mid\phi)+\sum_{
	\begin{smallmatrix}
        e'\sim_i d\\
		\phi \in \rmPhi
		%\val{\phi}_{\mathcal{M}}\wedge (\val{p}_{\mathcal{M}}\wedge \overline{\pre}_{\mathcal{M}}(e')\wedge a)\neq \bot
	\end{smallmatrix}
}
-\beta\mu^\phi_i(\top)P_i(e') \pre(e'\mid\phi)\geq 0}_\mathcal{M}\\
& = \val{f(\alpha\mu_i(\psi')\geq\beta,d)}_\mathcal{M}\tag{Definition \ref{definition: fonestep}}\\
& = F(\alpha\mu_i(\psi')\geq\beta)(d).
\end{align*}
If $\psi = \langle\mathcal{E}',e'\rangle\psi'$ and $\mathcal{N}=\mathcal{M}^{\mb{E}_\mathcal{E}}$, then 
\begin{align*}
\val{\langle\mathcal{E}',e'\rangle\psi'}_\mathcal{N} & =\val{pre(e')}_{\mathcal{N}}\land\pi_{e'}\circ i'(\val{\psi'}_{\mathcal{N}^{\mb{E}_{\mathcal{E}'}}})\\
& = \val{pre(e')}_{\mathcal{N}}\land\pi_{e'}\circ i'([F(\psi')])\tag{$\mathtt{IH}_{\psi'}$}\\
& = \val{pre(e')}_{\mathcal{N}}\land\pi_{e'}(F(\psi')\land\overline{pre})\\
& = \val{pre(e')}_{\mathcal{N}}\land F(\psi')(e')\land\val{pre(e')}_\mathcal{N}\\
& = \val{pre(e')}_{\mathcal{N}}\land\val{f(\psi',e')}_{\mathcal{N}}\land\val{pre(e')}_\mathcal{N}\\
& = \val{pre(e')}_{\mathcal{N}}\land \val{f(\psi',e')}_{\mathcal{N}}\\
& = \val{pre(e')\land f(\psi',e')}_{\mathcal{N}}\\
& = [F(pre(e')\land f(\psi',e'))]
\tag{$\mathtt{IH}_{{pre(e')\land f(\psi',e')}_{\mathcal{N}}}$}
\end{align*}
and
\begin{align*}
F(pre(e')\land f(\psi',e'))(e) & =\val{f(pre(e')\land f(\psi',e'),e)}_\mathcal{M}\\
& =\val{f(\langle\mathcal{E}',e'\rangle\psi',e)}_\mathcal{M}\tag{Definition \ref{definition: fonestep}}\\
& = F(\langle\mathcal{E}',e'\rangle\psi')(e).
\end{align*}
Finally, if $\psi = \left[\mathcal{E}',e'\right]\psi'$ and $\mathcal{N}=\mathcal{M}^{\mb{E}_\mathcal{E}}$, then
\begin{align*}
\val{\left[\mathcal{E}',e'\right]\psi'}_\mathcal{N} & =\val{pre(e')}_{\mathcal{N}}\rightarrow\pi_{e'}\circ i'(\val{\psi'}_{\mathcal{N}^{\mb{E}_{\mathcal{E}'}}})\\
& = \val{pre(e')}_{\mathcal{N}}\rightarrow\pi_{e'}\circ i'([F(\psi')])
\tag{$\mathtt{IH}_{\psi'}$}\\
& = \val{pre(e')}_{\mathcal{N}}\rightarrow\pi_{e'}(F(\psi')\land\overline{pre})\\
& = \val{pre(e')}_{\mathcal{N}}\rightarrow F(\psi')(e')\land\val{pre(e')}_\mathcal{N}\\
& = \val{pre(e')}_{\mathcal{N}}\rightarrow\val{f(\psi',e')}_{\mathcal{N}}\land\val{pre(e')}_\mathcal{N}\\
& = \val{pre(e')}_{\mathcal{N}}\rightarrow\val{f(\psi',e')}_{\mathcal{N}}\tag{$a\to(a\land b)=a\to b$}\\
& = \val{pre(e')\rightarrow f(\psi',e')}_{\mathcal{N}}\\
& = [F(pre(e')\rightarrow f(\psi',e'))]
\tag{$\mathtt{IH}_{pre(e')\rightarrow f(\psi',e')}$}
\end{align*}
and
\begin{align*}
F(pre(e')\rightarrow f(\psi',e'))(e) & =\val{f(pre(e')\rightarrow f(\psi',e'),e)}_\mathcal{M}\\
& =\val{f(\left[\mathcal{E}',e'\right]\psi',e)}_\mathcal{M}\tag{Definition \ref{definition: fonestep}}\\
& = F(\left[\mathcal{E}',e'\right]\psi')(e).
\end{align*}
\end{proof}
\begin{lemma}
\label{lemma: newf(psi) second property}
For every $\mathcal{M}$, $\mathcal{E}$, $e$ and $\psi$, \[\val{\langle\mathcal{E}, e\rangle \psi}_{\mathcal{M}} = \overline{pre}_{\mathcal{M}}(e)\wedge\val{f(\psi,e)}_{\mathcal{M}}
\quad \quad \text{and} \quad\quad
\val{\left[\mathcal{E}, e\right] \psi}_{\mathcal{M}} =\overline{pre}_{\mathcal{M}}(e)\rightarrow\val{f(\psi,e)}_{\mathcal{M}}.\]
\end{lemma}
\begin{proof}
We have
\begin{align*}
\val{\langle\mathcal{E}, e\rangle\psi}_{\mathcal{M}} &  = \val{pre(e)}_{\mathcal{M}}\wedge \pi_e\circ i'(\val{\psi}_{\mathcal{M}^{\mb{E}_\mathcal{E}}})\\&  =\overline{pre}_{\mathcal{M}}(e)\wedge \pi_e\circ i'([F(\psi)])\tag{Lemma \ref{lemma: newf(psi)}}\\
&  =\overline{pre}_{\mathcal{M}}(e)\wedge \pi_e(F(\psi)\land\overline{pre}_{\mathcal{M}})\\
&  =\overline{pre}_{\mathcal{M}}(e)\wedge F(\psi)(e)\land\overline{pre}_{\mathcal{M}}(e)\\
&  =\overline{pre}_{\mathcal{M}}(e)\wedge\val{f(\psi,e)}_{\mathcal{M}}
\end{align*}
and 
\begin{align*}
\val{\left[\mathcal{E}, e\right]\psi}_{\mathcal{M}} &  = \val{pre(e)}_{\mathcal{M}}\rightarrow \pi_e\circ i'(\val{\psi}_{\mathcal{M}^{\mb{E}_\mathcal{E}}})\\&  =\overline{pre}_{\mathcal{M}}(e)\rightarrow \pi_e\circ i'([F(\psi)])\tag{Lemma \ref{lemma: newf(psi)}}\\
&  =\overline{pre}_{\mathcal{M}}(e)\rightarrow \pi_e(F(\psi)\land\overline{pre}_{\mathcal{M}})\\
&  =\overline{pre}_{\mathcal{M}}(e)\rightarrow F(\psi)(e)\land\overline{pre}_{\mathcal{M}}(e)\\
&  =\overline{pre}_{\mathcal{M}}(e)\rightarrow\val{f(\psi,e)}_{\mathcal{M}}\land\overline{pre}_\mathcal{M}(e)\\
&  =\overline{pre}_{\mathcal{M}}(e)\rightarrow\val{f(\psi,e)}_{\mathcal{M}}.\tag{$a\to(a\land b)=a\to b$}
\end{align*}
\end{proof}

%%%%%%%%%%%%%%%%%%%%%%%%%%%%%%%%

\subsubsection*{Proof of the soundness of the reduction axioms}~
\label{ssec:Appendix:soundness}

\bigskip

%\noindent
%$\rule{146.2mm}{0.5pt}$

\textbf{Axiom I1.} 
$\left[\mathcal{E},e\right] p  = pre(e)
\rightarrow \sub(e,p).$
\begin{align*}
\val{\left[\mathcal{E},e\right] p}_\mathcal{M} & = \overline{pre}_\mathcal{M}(e)\rightarrow\val{f(p,e)}_\mathcal{M}
\tag{Lemma \ref{lemma: newf(psi) second property}}\\
& = \overline{pre}_\mathcal{M}(e)\rightarrow\val{\sub(e,p)}_\mathcal{M}.
\end{align*}

\noindent
$\rule{146.2mm}{0.5pt}$

\textbf{Axiom I2.} $\langle\mathcal{E},e\rangle p  = pre(e)
\land \sub(e,p).$
\begin{align*}
\val{\langle\mathcal{E},e\rangle p}_\mathcal{M} & = \overline{pre}_\mathcal{M}(e)\land\val{f(p,e)}_\mathcal{M}
\tag{Lemma \ref{lemma: newf(psi) second property}}\\
& = \overline{pre}_\mathcal{M}(e)\land\val{\sub(e,p)}_\mathcal{M}.
\end{align*}
\noindent
$\rule{146.2mm}{0.5pt}$

\textbf{Axiom I3.} 
$\left[\mathcal{E},e\right] \top= \top.$
\begin{align*}
\val{\left[\mathcal{E},e\right] \top}_\mathcal{M} & = \overline{pre}_\mathcal{M}(e)\rightarrow\val{f(\top,e)}_\mathcal{M}
\tag{Lemma \ref{lemma: newf(psi) second property}}\\
& = \overline{pre}_\mathcal{M}(e)\rightarrow\val{\top}_\mathcal{M}\\
& = \val{\top}_\mathcal{M}.
\end{align*}

\noindent
$\rule{146.2mm}{0.5pt}$

\textbf{Axiom I4.} $\langle\mathcal{E},e\rangle \top  = pre(e).$
\begin{align*}
\val{\langle\mathcal{E},e\rangle \top}_\mathcal{M} & = \overline{pre}_\mathcal{M}(e)\land\val{f(\top,e)}_\mathcal{M}
\tag{Lemma \ref{lemma: newf(psi) second property}}\\
& = \overline{pre}_\mathcal{M}(e)\land\val{\top}_\mathcal{M}\\
& = \overline{pre}_\mathcal{M}(e).
\end{align*}
\noindent
$\rule{146.2mm}{0.5pt}$

\textbf{Axiom I5.} 
$\left[\mathcal{E},e\right] \bot  
= \lnot pre(e).$
\begin{align*}
\val{\left[\mathcal{E},e\right] \bot}_\mathcal{M} & = \overline{pre}_\mathcal{M}(e)\rightarrow\val{f(\bot,e)}_\mathcal{M}
\tag{Lemma \ref{lemma: newf(psi) second property}}\\
& = \overline{pre}_\mathcal{M}(e)\rightarrow\val{\bot}_\mathcal{M}\\
& = \val{\lnot pre(e)}_\mathcal{M}.
\end{align*}

\noindent
$\rule{146.2mm}{0.5pt}$

\textbf{Axiom I6.} $\langle\mathcal{E},e\rangle \bot = \bot.$
\begin{align*}
\val{\langle\mathcal{E},e\rangle \bot}_\mathcal{M} & = \overline{pre}_\mathcal{M}(e)\land\val{f(\bot,e)}_\mathcal{M}
\tag{Lemma \ref{lemma: newf(psi) second property}}\\
& = \overline{pre}_\mathcal{M}(e)\land\val{\bot}_\mathcal{M}\\
& = \bot.
\end{align*}
\noindent
$\rule{146.2mm}{0.5pt}$

\textbf{Axiom I7.} 
$\left[\mathcal{E},e\right](\psi_1\land\psi_2)
= \left[\mathcal{E},e\right]\psi_1
\land \left[\mathcal{E},e\right]\psi_2.$
\begin{align*}
\val{\left[\mathcal{E},e\right](\psi_1\land\psi_2)}_\mathcal{M} & =\overline{pre}_\mathcal{M}(e)\rightarrow\val{f(\psi_1\land\psi_2,e)}_\mathcal{M}
\tag{Lemma \ref{lemma: newf(psi) second property}}\\
& = \overline{pre}_\mathcal{M}(e)\rightarrow\val{f(\psi_1,e)\land f(\psi_2,e)}_\mathcal{M}\tag{Definition \ref{definition: fonestep}}\\
& = \overline{pre}_\mathcal{M}(e)\rightarrow\val{f(\psi_1,e)}_\mathcal{M}\land\val{f(\psi_2,e)}_\mathcal{M}\\
& = (\overline{pre}_\mathcal{M}(e)\rightarrow\val{f(\psi_1,e)}_\mathcal{M})\land(\overline{pre}_\mathcal{M}(e)\rightarrow\val{f(\psi_2,e)}_\mathcal{M})\tag{$a\to b\land c=(a\to b)\land(a\to c)$}\\
& = \val{\left[\mathcal{E},e\right]\psi_1}_\mathcal{M}\land\val{\left[\mathcal{E},e\right]\psi_2}_\mathcal{M}\tag{Lemma \ref{lemma: newf(psi) second property}}
\end{align*}

\noindent
$\rule{146.2mm}{0.5pt}$

\textbf{Axiom I8.} $\langle\mathcal{E},e\rangle
(\psi_1\land\psi_2) = \langle\mathcal{E},e\rangle\psi_1
\land \langle\mathcal{E},e\rangle\psi_2.$
\begin{align*}
\val{\langle\mathcal{E},e\rangle(\psi_1\land\psi_2)}_\mathcal{M} & =\overline{pre}_\mathcal{M}(e)\land\val{f(\psi_1\land\psi_2,e)}_\mathcal{M}
\tag{Lemma \ref{lemma: newf(psi) second property}}\\
& = \overline{pre}_\mathcal{M}(e)\land\val{f(\psi_1,e)\land f(\psi_2,e)}_\mathcal{M}\tag{Definition \ref{definition: fonestep}}\\
& = \overline{pre}_\mathcal{M}(e)\land\val{f(\psi_1,e)}_\mathcal{M}\land\overline{pre}_\mathcal{M}(e)\land \val{f(\psi_2,e)}_\mathcal{M}\\
& = \val{\langle\mathcal{E},e\rangle\psi_1}_\mathcal{M}\land\val{\langle\mathcal{E},e\rangle\psi_2}_\mathcal{M}\tag{Lemma \ref{lemma: newf(psi) second property}}
\end{align*}

\noindent
$\rule{146.2mm}{0.5pt}$

\textbf{Axiom I9.} 
$\left[\mathcal{E},e\right](\psi_1\lor\psi_2)
 =
pre(e)\rightarrow
\langle\mathcal{E},e\rangle\psi_1
\lor \langle\mathcal{E},e\rangle\psi_2.
$
\begin{align*}
\val{\left[\mathcal{E},e\right](\psi_1\lor\psi_2)}_\mathcal{M} & =\overline{pre}_\mathcal{M}(e)\rightarrow\val{f(\psi_1\lor\psi_2,e)}_\mathcal{M}
\tag{Lemma \ref{lemma: newf(psi) second property}}\\
& = \overline{pre}_\mathcal{M}(e)\rightarrow\val{f(\psi_1,e)\lor f(\psi_2,e)}_\mathcal{M}\tag{Definition \ref{definition: fonestep}}\\
& = \overline{pre}_\mathcal{M}(e)\rightarrow\overline{pre}_\mathcal{M}(e)\land(\val{f(\psi_1,e)}_\mathcal{M}\lor\val{f(\psi_2,e)}_\mathcal{M})\tag{$a\to b=a\to a\land b$}\\
& = \overline{pre}_\mathcal{M}(e)\rightarrow(\overline{pre}_\mathcal{M}\land\val{f(\psi_1,e)}_\mathcal{M})\lor(\overline{pre}_\mathcal{M}(e)\land\val{f(\psi_2,e)}_\mathcal{M})\tag{distributivity}\\
& = \overline{pre}_\mathcal{M}(e)\rightarrow\val{\langle\mathcal{E},e\rangle\psi_1}_\mathcal{M}\lor\val{\langle\mathcal{E},e\rangle\psi_2}_\mathcal{M}\tag{Lemma \ref{lemma: newf(psi) second property}}
\end{align*}

\noindent
$\rule{146.2mm}{0.5pt}$

\textbf{Axiom I10.} $\langle\mathcal{E},e\rangle(\psi_1\lor\psi_2) = \langle\mathcal{E},e\rangle\psi_1
\lor \langle\mathcal{E},e\rangle\psi_2.$
\begin{align*}
\val{\langle\mathcal{E},e\rangle(\psi_1\lor\psi_2)}_\mathcal{M} & =\overline{pre}_\mathcal{M}(e)\land\val{f(\psi_1\lor\psi_2,e)}_\mathcal{M}
\tag{Lemma \ref{lemma: newf(psi) second property}}\\
& = \overline{pre}_\mathcal{M}(e)\land\val{f(\psi_1,e)\lor f(\psi_2,e)}_\mathcal{M}\tag{Definition \ref{definition: fonestep}}\\
& = (\overline{pre}_\mathcal{M}(e)\land\val{f(\psi_1,e)}_\mathcal{M})\lor(\overline{pre}_\mathcal{M}(e)\land \val{f(\psi_2,e)}_\mathcal{M})\tag{distributivity}\\
& = \val{\langle\mathcal{E},e\rangle\psi_1}_\mathcal{M}\lor\val{\langle\mathcal{E},e\rangle\psi_2}_\mathcal{M}\tag{Lemma \ref{lemma: newf(psi) second property}}
\end{align*}

\noindent
$\rule{146.2mm}{0.5pt}$

\textbf{Axiom I11.} 
$\left[\mathcal{E},e\right](\psi_1\rightarrow\psi_2) = \langle\mathcal{E},e\rangle\psi_1
\rightarrow \langle\mathcal{E},e\rangle\psi_2.
$
\begin{align*}
\val{\left[\mathcal{E},e\right](\psi_1\rightarrow\psi_2)}_\mathcal{M} & =\overline{pre}_\mathcal{M}(e)\rightarrow\val{f(\psi_1\rightarrow\psi_2,e)}_\mathcal{M}
\tag{Lemma \ref{lemma: newf(psi) second property}}\\
& = \overline{pre}_\mathcal{M}(e)\rightarrow\val{f(\psi_1,e)\rightarrow f(\psi_2,e)}_\mathcal{M}\tag{Definition \ref{definition: fonestep}}\\
& = \overline{pre}_\mathcal{M}(e)\rightarrow(\val{f(\psi_1,e)}_\mathcal{M}\rightarrow\val{f(\psi_2,e)}_\mathcal{M})\\
& = \overline{pre}_\mathcal{M}(e)\land\val{f(\psi_1,e)}_\mathcal{M}\rightarrow\val{f(\psi_2,e)}_\mathcal{M}\tag{$a\to(b\to c)=a\land b\to c$}\\
& = \overline{pre}_\mathcal{M}(e)\land\val{f(\psi_1,e)}_\mathcal{M}\rightarrow\overline{pre}_\mathcal{M}(e)\land\val{f(\psi_1,e)}_\mathcal{M}\land\val{f(\psi_2,e)}_\mathcal{M}\tag{$b\to c=b\to b\land c)$}\\
& = \val{\langle\mathcal{E},e\rangle\psi_1}_\mathcal{M}\rightarrow\val{\langle\mathcal{E},e\rangle\psi_1}_\mathcal{M}\land\val{\langle\mathcal{E},e\rangle\psi_2}_\mathcal{M}\tag{Lemma \ref{lemma: newf(psi) second property}}\\
& = \val{\langle\mathcal{E},e\rangle\psi_1}_\mathcal{M}\rightarrow\val{\langle\mathcal{E},e\rangle\psi_2}_\mathcal{M}\tag{$b\to c=b\to b\land c)$}.
\end{align*}

\noindent
$\rule{146.2mm}{0.5pt}$

\textbf{Axiom I12.} 
$\langle\mathcal{E},e\rangle
(\psi_1\rightarrow\psi_2)= pre(e)
\land( \langle\mathcal{E},e\rangle\psi_1
\rightarrow
\langle\mathcal{E},e\rangle\psi_2).$

\begin{align*}
\val{\langle\mathcal{E},e\rangle(\psi_1\rightarrow\psi_2)}_\mathcal{M} & =\overline{pre}_\mathcal{M}(e)\land\val{f(\psi_1\rightarrow\psi_2,e)}_\mathcal{M}
\tag{Lemma \ref{lemma: newf(psi) second property}}\\
& = \overline{pre}_\mathcal{M}(e)\land\val{f(\psi_1,e)\rightarrow f(\psi_2,e)}_\mathcal{M}\tag{Definition \ref{definition: fonestep}}\\
& = \overline{pre}_\mathcal{M}(e)\land(\val{f(\psi_1,e)}_\mathcal{M}\rightarrow\val{f(\psi_2,e)}_\mathcal{M})\\
& = \overline{pre}_\mathcal{M}(e)\land(\overline{pre}_\mathcal{M}(e)\land\val{f(\psi_1,e)}_\mathcal{M}\rightarrow\val{f(\psi_2,e)}_\mathcal{M})\tag{$a\land(b\to c)=a\land(a\land b\to c)$}\\
& = \overline{pre}_\mathcal{M}(e)\land(\overline{pre}_\mathcal{M}(e)\land\val{f(\psi_1,e)}_\mathcal{M}\rightarrow\val{f(\psi_2,e)}_\mathcal{M}\land\overline{pre}_\mathcal{M}(e)\land\val{f(\psi_1,e)}_\mathcal{M})\tag{$b\to c=b\to b\land c)$}\\
& = \overline{pre}_\mathcal{M}(e)\land(\val{\langle\mathcal{E},e\rangle\psi_1}_\mathcal{M}\rightarrow\val{\langle\mathcal{E},e\rangle\psi_1}_\mathcal{M}\land\val{\langle\mathcal{E},e\rangle\psi_2}_\mathcal{M})\tag{Lemma \ref{lemma: newf(psi) second property}}\\
& = \overline{pre}_\mathcal{M}(e)\land(\val{\langle\mathcal{E},e\rangle\psi_1}_\mathcal{M}\rightarrow\val{\langle\mathcal{E},e\rangle\psi_2}_\mathcal{M})\tag{$b\to c=b\to b\land c)$}.
\end{align*}

\noindent
$\rule{146.2mm}{0.5pt}$

\textbf{Axiom I13} 
$\left[\mathcal{E},e\right]\lozenge_i\psi = pre(e)\rightarrow \bigvee_{e'\sim_i e}\lozenge_i(\langle\mathcal{E},e'\rangle\psi).
$
\begin{align*}
\val{\left[\mathcal{E},e\right]\lozenge_i\psi}_\mathcal{M} & =\overline{pre}_\mathcal{M}(e)\rightarrow\val{f(\lozenge_i\psi,e)}_\mathcal{M}
\tag{Lemma \ref{lemma: newf(psi) second property}}\\
& = \overline{pre}_\mathcal{M}(e)\rightarrow\val{\bigvee_{e'\sim_i e}\lozenge_i (f(\psi,e')\land pre(e'))}_\mathcal{M}\tag{Definition \ref{definition: fonestep}}\\
& = \overline{pre}_\mathcal{M}(e)\rightarrow\bigvee_{e'\sim_i e}\lozenge_i(\val{f(\psi,e')}_\mathcal{M}\land\overline{pre}_\mathcal{M}(e'))\\
& = \overline{pre}_\mathcal{M}(e)\rightarrow\bigvee_{e'\sim_i e}\lozenge_i(\val{\langle\mathcal{E},e'\rangle\psi}_\mathcal{M})\tag{Lemma \ref{lemma: newf(psi) second property}}\\
& = \overline{pre}_\mathcal{M}(e)\rightarrow\val{\bigvee_{e'\sim_i e}\lozenge_i(\langle\mathcal{E},e'\rangle\psi)}_\mathcal{M}.
\end{align*}

\noindent
$\rule{146.2mm}{0.5pt}$

\textbf{Axiom I14.} 
$\langle\mathcal{E},e\rangle\lozenge_i\psi = pre(e)
\land \bigvee_{e'\sim_i e}\lozenge_i(\langle\mathcal{E},e'
\rangle\psi).
$
\begin{align*}
\val{\langle\mathcal{E},e\rangle\lozenge_i\psi}_\mathcal{M} & =\overline{pre}_\mathcal{M}(e)\land\val{f(\lozenge_i\psi,e)}_\mathcal{M}
\tag{Lemma \ref{lemma: newf(psi) second property}}\\
& = \overline{pre}_\mathcal{M}(e)\land\val{\bigvee_{e'\sim_i e}\lozenge_i (f(\psi,e')\land pre(e'))}_\mathcal{M}\tag{Definition \ref{definition: fonestep}}\\
& = \overline{pre}_\mathcal{M}(e)\land\bigvee_{e'\sim_i e}\lozenge_i(\val{f(\psi,e')}_\mathcal{M}\land\overline{pre}_\mathcal{M}(e'))\\
& = \overline{pre}_\mathcal{M}(e)\land\bigvee_{e'\sim_i e}\lozenge_i(\val{\langle\mathcal{E},e'\rangle\psi}_\mathcal{M})\tag{Lemma \ref{lemma: newf(psi) second property}}\\
& = \overline{pre}_\mathcal{M}(e)\land\val{\bigvee_{e'\sim_i e}\lozenge_i(\langle\mathcal{E},e'\rangle\psi)}_\mathcal{M}.
\end{align*}

\noindent
$\rule{146.2mm}{0.5pt}$

\textbf{Axiom I15.} $\left[\mathcal{E},e\right]\Box_i\psi = pre(e)\rightarrow \bigwedge_{e'\sim_i e}\Box_i([\mathcal{E},e']\psi).$

\begin{align*}
	\val{\left[\mathcal{E},e\right]\Box_i\psi}_\mathcal{M} & =\overline{pre}_\mathcal{M}(e)\rightarrow\val{f(\Box_i\psi,e)}_\mathcal{M}
	\tag{Lemma \ref{lemma: newf(psi) second property}}\\
	& = \overline{pre}_\mathcal{M}(e)\rightarrow\val{\bigwedge_{e'\sim_i e}\Box_i(pre(e')\to f(\psi,e'))}_\mathcal{M}\tag{Definition \ref{definition: fonestep}}\\
	& = \overline{pre}_\mathcal{M}(e)\rightarrow\bigwedge_{e'\sim_i e}\Box_i(\overline{pre}_\mathcal{M}(e')\to \val{f(\psi,e')}_\mathcal{M})\\
	& = \overline{pre}_\mathcal{M}(e)\rightarrow\bigwedge_{e'\sim_i e}\Box_i(\val{\left[\mathcal{E},e'\right]\psi}_\mathcal{M})\tag{Lemma \ref{lemma: newf(psi) second property}}\\
	& = \overline{pre}_\mathcal{M}(e)\rightarrow\val{\bigwedge_{e'\sim_i e}\Box_i(\left[\mathcal{E},e'\right]\psi)}_\mathcal{M}.
\end{align*}

\noindent
$\rule{146.2mm}{0.5pt}$

\textbf{Axiom I16.} $\langle\mathcal{E},e\rangle\Box_i\psi = pre(e)
\land \bigwedge_{e'\sim_i e}\Box_i([\mathcal{E},e'
]\psi).$

\begin{align*}
	\val{\langle\mathcal{E},e\rangle\Box_i\psi}_\mathcal{M} & =\overline{pre}_\mathcal{M}(e)\land\val{f(\Box_i\psi,e)}_\mathcal{M}
	\tag{Lemma \ref{lemma: newf(psi) second property}}\\
	& = \overline{pre}_\mathcal{M}(e)\land\val{\bigwedge_{e'\sim_i e}\Box_i(pre(e')\to f(\psi,e'))}_\mathcal{M}\tag{Definition \ref{definition: fonestep}}\\
	& = \overline{pre}_\mathcal{M}(e)\land\bigwedge_{e'\sim_i e}\Box_i(\overline{pre}_\mathcal{M}(e')\to\val{f(\psi,e')}_\mathcal{M})\\
	& = \overline{pre}_\mathcal{M}(e)\land\bigwedge_{e'\sim_i e}\Box_i(\val{\left[\mathcal{E},e'\right]\psi}_\mathcal{M})\tag{Lemma \ref{lemma: newf(psi) second property}}\\
	& = \overline{pre}_\mathcal{M}(e)\land\val{\bigwedge_{e'\sim_i e}\Box_i(\left[\mathcal{E},e'\right]\psi)}_\mathcal{M}.
\end{align*}

\noindent
$\rule{146.2mm}{0.5pt}$

\begin{landscape}
\textbf{Axiom I17.} 
$\left[\mathcal{E},e\right](\alpha\mu_i(\psi)\geq\beta)
= pre(e)\rightarrow\sum_{
	\begin{smallmatrix}
        e'\sim_i e\\
		\phi \in \rmPhi
		%\val{\phi}_{\mathcal{M}}\wedge (\val{\langle \mb{E}, e'\rangle p}_{\mathcal{M}}\wedge a)\neq \bot
	\end{smallmatrix}
}
\alpha P_i(e') \pre(e'\mid\phi)\mu^\phi_i(\left[ \mathcal{E}, e'\right] \psi) + \sum_{
	\begin{smallmatrix}
        e'\sim_i e\\
		\phi \in \rmPhi
		%\val{\phi}_{\mathcal{M}}\wedge  a\neq \bot
	\end{smallmatrix}
}
-\beta P_i(e') \pre(e'\mid\phi) \mu^\phi_i(\top)\geq 0$
\begin{align*}
& \val{\left[\mathcal{E},e\right](\alpha\mu_i(\psi)\geq\beta)}_\mathcal{M}  
 = \overline{pre}_\mathcal{M}(e)\rightarrow\val{f(\alpha\mu_i(\psi)\geq\beta,e)}_\mathcal{M}  \\
& = \overline{pre}_\mathcal{M}(e)\rightarrow\val{\alpha\sum_{
	\begin{smallmatrix}
        e'\sim_i e\\
		\phi \in \rmPhi
		%\val{\phi}_{\mathcal{M}}\wedge (\val{p}_{\mathcal{M}}\wedge \overline{\pre}_{\mathcal{M}}(e')\wedge a)\neq \bot
	\end{smallmatrix}
}
\mu^\phi_i(f(\psi,e'))P_i(e') \pre(e'\mid \phi) - \beta\sum_{
	\begin{smallmatrix}
        e'\sim_i e\\
		\phi \in \rmPhi
		%\val{\phi}_{\mathcal{M}}\wedge (\val{p}_{\mathcal{M}}\wedge \overline{\pre}_{\mathcal{M}}(e')\wedge a)\neq \bot
	\end{smallmatrix}
}
\mu^\phi_i(\top)P_i(e') \pre(e'\mid \phi)\geq 0}_\mathcal{M}\tag{Definition \ref{definition: fonestep}}  \\
& = \overline{pre}_\mathcal{M}(e)\rightarrow\bigvee\left\{a   \; \middle\vert \;   \sum_{
	\begin{smallmatrix}
        e'\sim_i e\\
		\phi \in \rmPhi
		%\val{\phi}_{\mathcal{M}}\wedge (\val{\langle \mb{E}, e'\rangle p}_{\mathcal{M}}\wedge a)\neq \bot
	\end{smallmatrix}
}
\alpha P_i(e') \pre(e'\mid\phi)\mu^\phi_i( \val{f(\psi,e')}_{\mathcal{M}}\wedge a) + \sum_{
	\begin{smallmatrix}
        e'\sim_i e\\
		\phi \in \rmPhi
		%\val{\phi}_{\mathcal{M}}\wedge  a\neq \bot
	\end{smallmatrix}
}
-\beta P_i(e') \pre(e'\mid\phi) \mu^\phi_i(a)\geq 0\right\} \\
& = \overline{pre}_\mathcal{M}(e)\rightarrow\bigvee\left\{a  \; \middle\vert \;   \sum_{
	\begin{smallmatrix}
        e'\sim_i e\\
		\phi \in \rmPhi
		%\val{\phi}_{\mathcal{M}}\wedge (\val{\langle \mb{E}, e'\rangle p}_{\mathcal{M}}\wedge a)\neq \bot
	\end{smallmatrix}
}
\alpha P_i(e') \pre(e'\mid\phi)\mu^\phi_i( \val{f(\psi,e')}_{\mathcal{M}}\land\overline{pre}_\mathcal{M}(e')\wedge a) + \sum_{
	\begin{smallmatrix}
        e'\sim_i e\\
		\phi \in \rmPhi
		%\val{\phi}_{\mathcal{M}}\wedge  a\neq \bot
	\end{smallmatrix}
}
-\beta P_i(e') \pre(e'\mid\phi) \mu^\phi_i(a)\geq 0\right\}\tag{ $\val{\varphi}_\mathcal{M}\leq\overline{pre}_\mathcal{M}(e')$ if $\pre(e'\mid\varphi)\neq 0$, cf.\ Proposition \ref{prop: muai properties}}\\
& = \overline{pre}_\mathcal{M}(e)\rightarrow\bigvee\left\{a  \; \middle\vert \;   \sum_{
	\begin{smallmatrix}
        e'\sim_i e\\
		\phi \in \rmPhi
		%\val{\phi}_{\mathcal{M}}\wedge (\val{\langle \mb{E}, e'\rangle p}_{\mathcal{M}}\wedge a)\neq \bot
	\end{smallmatrix}
}
\alpha P_i(e') \pre(e'\mid\phi)\mu^\phi_i((\overline{pre}_\mathcal{M}(e')\rightarrow\val{f(\psi,e')}_{\mathcal{M}})\land\overline{pre}_\mathcal{M}(e')\wedge a) + \sum_{
	\begin{smallmatrix}
        e'\sim_i e\\
		\phi \in \rmPhi
		%\val{\phi}_{\mathcal{M}}\wedge  a\neq \bot
	\end{smallmatrix}
}
-\beta P_i(e') \pre(e'\mid\phi) \mu^\phi_i(a)\geq 0\right\}\tag{ $a\land(a\to b)=a\land b$}\\
& = \overline{pre}_\mathcal{M}(e)\rightarrow\bigvee\left\{a  \; \middle\vert \;   \sum_{
	\begin{smallmatrix}
        e'\sim_i e\\
		\phi \in \rmPhi
		%\val{\phi}_{\mathcal{M}}\wedge (\val{\langle \mb{E}, e'\rangle p}_{\mathcal{M}}\wedge a)\neq \bot
	\end{smallmatrix}
}
\alpha P_i(e') \pre(e'\mid\phi)\mu^\phi_i(\val{ \left[\mathcal{E},e'\right]\psi}_{\mathcal{M}}\land\overline{pre}_\mathcal{M}(e')\wedge a) + \sum_{
	\begin{smallmatrix}
        e'\sim_i e\\
		\phi \in \rmPhi
		%\val{\phi}_{\mathcal{M}}\wedge  a\neq \bot
	\end{smallmatrix}
}
-\beta P_i(e') \pre(e'\mid\phi) \mu^\phi_i(a)\geq 0\right\}\tag{Lemma \ref{lemma: newf(psi) second property}}\\
& = \overline{pre}_\mathcal{M}(e)\rightarrow\bigvee\left\{a  \; \middle\vert \;   \sum_{
	\begin{smallmatrix}
        e'\sim_i e\\
		\phi \in \rmPhi
		%\val{\phi}_{\mathcal{M}}\wedge (\val{\langle \mb{E}, e'\rangle p}_{\mathcal{M}}\wedge a)\neq \bot
	\end{smallmatrix}
}
\alpha P_i(e') \pre(e'\mid\phi)\mu^\phi_i(\val{ \left[\mathcal{E},e'\right]\psi}_{\mathcal{M}}\wedge a) + \sum_{
	\begin{smallmatrix}
        e'\sim_i e\\
		\phi \in \rmPhi
		%\val{\phi}_{\mathcal{M}}\wedge  a\neq \bot
	\end{smallmatrix}
}
-\beta P_i(e') \pre(e'\mid\phi) \mu^\phi_i(a)\geq 0\right\}\tag{ $\val{\varphi}_\mathcal{M}\leq\overline{pre}_\mathcal{M}(e')$ if $\pre(e'\mid\varphi)\neq 0$}\\
& = \overline{pre}_\mathcal{M}(e)\rightarrow\val{ \sum_{
	\begin{smallmatrix}
        e'\sim_i e\\
		\phi \in \rmPhi
		%\val{\phi}_{\mathcal{M}}\wedge (\val{\langle \mb{E}, e'\rangle p}_{\mathcal{M}}\wedge a)\neq \bot
	\end{smallmatrix}
}
\alpha P_i(e') \pre(e'\mid\phi)\mu^\phi_i(\left[ \mathcal{E}, e'\right] \psi) + \sum_{
	\begin{smallmatrix}
        e'\sim_i e\\
		\phi \in \rmPhi
		%\val{\phi}_{\mathcal{M}}\wedge  a\neq \bot
	\end{smallmatrix}
}
-\beta P_i(e') \pre(e'\mid\phi) \mu^\phi_i(\top)\geq 0}_{\mathcal{M}}\\
\end{align*}	
%\noindent
%$\rule{240mm}{0.5pt}$

\end{landscape}

\begin{landscape}
\textbf{Axiom I18.} 
$\langle\mathcal{E},e\rangle(\alpha\mu_i(\psi)
\geq\beta)
= pre(e)\land \sum_{
	\begin{smallmatrix}
        e'\sim_i e\\
		\phi \in \rmPhi
		%\val{\phi}_{\mathcal{M}}\wedge (\val{\langle \mb{E}, e'\rangle p}_{\mathcal{M}}\wedge a)\neq \bot
	\end{smallmatrix}
}
\alpha P_i(e') \pre(e'\mid\phi)\mu^\phi_i(\langle \mathcal{E}, e'\rangle \psi) + \sum_{
	\begin{smallmatrix}
        e'\sim_i e\\
		\phi \in \rmPhi
		%\val{\phi}_{\mathcal{M}}\wedge  a\neq \bot
	\end{smallmatrix}
}
-\beta P_i(e') \pre(e'\mid\phi) \mu^\phi_i(\top)\geq 0
$
\begin{align*}
& \val{\langle\mathcal{E},e\rangle(\alpha\mu_i(\psi)\geq\beta)}_\mathcal{M}  \\
 & = \overline{pre}_\mathcal{M}(e)\land\val{f(\alpha\mu_i(\psi)\geq\beta,e)}_\mathcal{M}  \\
 & = \overline{pre}_\mathcal{M}(e)\land\val{\alpha\sum_{
	\begin{smallmatrix}
        e'\sim_i e\\
		\phi \in \rmPhi
		%\val{\phi}_{\mathcal{M}}\wedge (\val{p}_{\mathcal{M}}\wedge \overline{\pre}_{\mathcal{M}}(e')\wedge a)\neq \bot
	\end{smallmatrix}
}
\mu^\phi_i(f(\psi,e'))P_i(e') \pre(e'\mid \phi) - \beta\sum_{
	\begin{smallmatrix}
        e'\sim_i e\\
		\phi \in \rmPhi
		%\val{\phi}_{\mathcal{M}}\wedge (\val{p}_{\mathcal{M}}\wedge \overline{\pre}_{\mathcal{M}}(e')\wedge a)\neq \bot
	\end{smallmatrix}
}
\mu^\phi_i(\top)P_i(e') \pre(e'\mid \phi)\geq 0}_\mathcal{M}\tag{Definition \ref{definition: fonestep}}  \\
 & = \overline{pre}_\mathcal{M}(e)\land\bigvee\left\{a  \; \middle\vert \;   \sum_{
	\begin{smallmatrix}
        e'\sim_i e\\
		\phi \in \rmPhi
		%\val{\phi}_{\mathcal{M}}\wedge (\val{\langle \mb{E}, e'\rangle p}_{\mathcal{M}}\wedge a)\neq \bot
	\end{smallmatrix}
}
\alpha P_i(e') \pre(e'\mid\phi)\mu^\phi_i(\val{f(\psi,e')}_{\mathcal{M}}\wedge a) + \sum_{
	\begin{smallmatrix}
        e'\sim_i e\\
		\phi \in \rmPhi
		%\val{\phi}_{\mathcal{M}}\wedge  a\neq \bot
	\end{smallmatrix}
}
-\beta P_i(e') \pre(e'\mid\phi) \mu^\phi_i(a)\geq 0\right\} \\
& = \overline{pre}_\mathcal{M}(e)\land\bigvee\left\{a  \; \middle\vert \;   \sum_{
	\begin{smallmatrix}
        e'\sim_i e\\
		\phi \in \rmPhi
		%\val{\phi}_{\mathcal{M}}\wedge (\val{\langle \mb{E}, e'\rangle p}_{\mathcal{M}}\wedge a)\neq \bot
	\end{smallmatrix}
}
\alpha P_i(e') \pre(e'\mid\phi)\mu^\phi_i(\val{f(\psi,e')}_{\mathcal{M}}\land\overline{pre}_\mathcal{M}(e')\wedge a) + \sum_{
	\begin{smallmatrix}
        e'\sim_i e\\
		\phi \in \rmPhi
		%\val{\phi}_{\mathcal{M}}\wedge  a\neq \bot
	\end{smallmatrix}
}
-\beta P_i(e') \pre(e'\mid\phi) \mu^\phi_i(a)\geq 0\right\}\tag{ $\val{\varphi}_\mathcal{M}\leq\overline{pre}_\mathcal{M}(e')$ if $\pre(e'\mid\varphi)\neq 0$, cf.\ Proposition \ref{prop: muai properties}}\\
& = \overline{pre}_\mathcal{M}(e)\land\bigvee\left\{a  \; \middle\vert \;   \sum_{
	\begin{smallmatrix}
        e'\sim_i e\\
		\phi \in \rmPhi
		%\val{\phi}_{\mathcal{M}}\wedge (\val{\langle \mb{E}, e'\rangle p}_{\mathcal{M}}\wedge a)\neq \bot
	\end{smallmatrix}
}
\alpha P_i(e') \pre(e'\mid\phi)\mu^\phi_i(\val{ \langle\mathcal{E},e'\rangle\psi}_{\mathcal{M}}\wedge a) + \sum_{
	\begin{smallmatrix}
        e'\sim_i e\\
		\phi \in \rmPhi
		%\val{\phi}_{\mathcal{M}}\wedge  a\neq \bot
	\end{smallmatrix}
}
-\beta P_i(e') \pre(e'\mid\phi) \mu^\phi_i(a)\geq 0\right\}\tag{Lemma \ref{lemma: newf(psi) second property}}
\\
& = \overline{pre}_\mathcal{M}(e)\land\val{ \sum_{
	\begin{smallmatrix}
        e'\sim_i e\\
		\phi \in \rmPhi
		%\val{\phi}_{\mathcal{M}}\wedge (\val{\langle \mb{E}, e'\rangle p}_{\mathcal{M}}\wedge a)\neq \bot
	\end{smallmatrix}
}
\alpha P_i(e') \pre(e'\mid\phi)\mu^\phi_i(\langle \mathcal{E}, e'\rangle \psi) + \sum_{
	\begin{smallmatrix}
        e'\sim_i e\\
		\phi \in \rmPhi
		%\val{\phi}_{\mathcal{M}}\wedge  a\neq \bot
	\end{smallmatrix}
}
-\beta P_i(e') \pre(e'\mid\phi) \mu^\phi_i(\top)\geq 0}_{\mathcal{M}} \\
\end{align*}

\noindent
$\rule{240mm}{0.5pt}$

\end{landscape}

\section{Proof of the Completeness of IPDEL}
\label{Appendix:completeness}

In the present section, we  prove the weak completeness of IPDEL w.r.t.\ APE-models. 
Recall that a calculus is \textit{weakly complete} w.r.t.\ a semantics if it provides a proof for every validity, namely, for any formula $\phi$, if $\models \phi$ then $\vdash \phi$.
Similarly to akin logical systems (cf.\ \cite{BMS,KP13,MPS14,BEK06} \cite{cabrer2016lukasiewicz,bakhtiarinoodeh2015epistemic,Achimescu14}), the proof relies on a reduction procedure  of IPDEL-formulas to formulas of the static fragment of IPDEL (referred to in what follows as IPEL), which preserves provable equivalence. 
This reduction procedure is effected using the interaction axioms and the rule of substitution of equivalent formulas.  We omit the details since this procedure is standard (see for instance \cite{baltag2004logics,BMS,WanCao13} for details). 
In the reminder of the present section, we prove the weak completeness of IPEL w.r.t.\ APE-models, i.e., we show that every APE-validity in the language of IPEL is a theorem of IPEL. 
By contraposition, this is equivalent to proving that for any IPEL-formula $\varphi$ which is not an IPEL-theorem, 
there exists an APE-model $\mathcal{M}$ that does not satisfy $\varphi$ in the sense that $\val{\varphi}_\mathcal{M}\neq\top$. 

The proof will proceed as follows.
In \autoref{ssec:cpt:epist:HA}, we extract a finite sublattice of the Lindenbaum-Tarski algebra of the logic that contains $\varphi$ and we  prove that it is an Epistemic Heyting Algebra satisfying certain properties akin to those described in \cite{fischer1978finite}.
Then, in \autoref{ssec:cpt:measure:Astar}, following  ideas from \cite{fagin1990logic} adapted to the algebraic setting, we  define appropriate $i$-measures over the finite Epistemic Heyting Algebra to turn it into an APE-model that does not satisfy $\varphi$.

%Before moving to the proof we present some basic facts about Epistemic Heyting Algebras.\redfootnote{This should be moved someplace else, when the algebras are defined}

\subsection*{The epistemic Heyting algebra $\mathbb{A}^\star_\varphi$}
\label{ssec:cpt:epist:HA}
In this subsection, we construct the finite epistemic Heyting algebra on which the counter-model for $\varphi$ is based. The construction consists of a number of steps, starting with the Lindenbaum-Tarski algebra of $\mathcal{L}$ and restricting it accordingly. 

Henceforth, we let 
\begin{equation}
\label{eq:linden}
\mathbb{A} = \left( A,\top_\mathbb{A},\bot_\mathbb{A},\lor_\mathbb{A},\land_\mathbb{A},\to_\mathbb{A},(\lozenge_i)_{i\in\Ag},(\Box_i)_{i\in\Ag}\right)
\end{equation} 
denote the Lindenbaum-Tarski algebra of IPEL. 
We will use $\lnot_\mathbb{A}(\bullet)$ as shorthand for $\bullet\to_\mathbb{A}\bot_\mathbb{A}$. 
For any agent $i$,  we define: 
$$\lozenge_i\mathbb{A}
:= \{\lozenge_i a\in\mathbb{A}\mid a\in\mathbb{A}\}.$$ 
For any formula $\sigma\in\mathcal{L}_{IPEL}$, we let $\sigma^{\mathbb{A}}\in\mathbb{A}$ denote the equivalence class  of $\sigma$ modulo provable equivalence in IPEL. 
Let $$\mathbb{B}:=\left( B,\top_\mathbb{B},\bot_\mathbb{B},\lor_\mathbb{B},\land_\mathbb{B},\lnot_\mathbb{B}\right)$$ 
be the Boolean Extension of the Heyting algebra reduct of $\mathbb{A}$ (see \cite[Section 13, page 450]{macneille1937partially}).\footnote{The Boolean extension of $\mathbb{A}$ can be identified with the algebra of clopens of the Esakia space dual to $\mathbb{A}$. Notice that this is exactly the same construction semantically underlying the G\"odel-Tarski translation (cf.\ \cite[Section 3]{CPZ:Trans} for an expanded discussion).} 
To enhance readability, we identify $\mathbb{A}$ with its image through the embedding $A\hookrightarrow B$. 
Recall that  $\mathbb{A}$ is a  sublattice of $\mathbb{B}$. 
Henceforth, we will use $\lor$ and $\land$ and $\top$ and $\bot$ ambiguously to denote the operations on both algebras. 
Since $\lozenge_i\mathbb{A}$ is a Boolean algebra (see \autoref{lem:epist-Hey-alg:diamond-A}) and, in every Boolean algebra, negation is unique, we have that $\lnot_\mathbb{A}a=\lnot_\mathbb{B}a$ for every $a\in\lozenge_i\mathbb{A}$ and for every agent $i\in\Ag$.

Let $\varphi$  be an IPEL-formula that is not  a theorem. 
Let $$S_\varphi:=\{\sigma^\mathbb{A}\mid \sigma\text{ is a subformula of }\varphi\},$$ 
let $\Ag_\varphi$ be the set of agents that appear in $\varphi$ and let $S_\varphi^\lozenge\supseteq S_\varphi$ be 
$$S_\varphi^\lozenge
:=S_\varphi\cup\{(\lozenge_i\sigma)^\mathbb{A},(\lnot\lozenge_i\sigma)^\mathbb{A}\mid \sigma\in S_\varphi\text{ and }i\in\Ag_\varphi\}.$$ 
Notice that the sets $S_\varphi$ and $S^\lozenge_\varphi$
are finite. 
Now, let $\mathbb{B}_\varphi\subseteq\mathbb{B}$ be the Boolean subalgebra of $\mathbb{B}$ generated by $S^\lozenge_\varphi$. 
Since $S_\varphi^\lozenge$ is finite, %\redbf{and of cardinality at least two}, 
so will be the domain of $\mathbb{B}_\varphi$ (which we denote with $B_\varphi$). In addition, since $\mathbb{B}_\varphi$ is a sub-lattice of the non-trivial boolean algebra $\mathbb{B}$, it contains at least two elements and has at least one atom.
Plus, it follows that $\mathbb{B}_\varphi$ is generated by its atoms. In view of what will follow, let us endow $\mathbb{B}_\varphi$ with a measure $\mu_\mathbb{B}$ as follows. 
Let $n \geq 1$ be the number of atoms of $\mathbb{B}_\varphi$. 
%\redbf{Notice that $n \geq 1$ ; indeed, since $\mathbb{B}_\varphi$ contains at least two elements, it has at least one atom.}
For every $a\in\mathbb{B}_\varphi$ that is above exactly $m$ atoms, let 
\begin{align}
\mu_\mathbb{B}(a)=\frac{m}{n}.
\label{align:muB}
\end{align}
Now, let 
$$\mathbb{A}_\varphi := \left( A_\varphi, \top, \bot, \land,\lor \right)$$
with $A_\varphi := A\cap B_\varphi$.
Notice that, since both $\mathbb{A}$ and $\mathbb{B}_\varphi$ are distributive lattices, so is  $\mathbb{A}_\varphi$.  For every agent $i\in\Ag_\varphi$, we define 
$$A^{\lozenge_i}_\varphi
:= \{ a\in\mathbb{A}_\varphi\mid \text{ there exists }\sigma\in\mathcal{L}\text{ such that }\lozenge_i\sigma\in a\}=A_\varphi\cap\lozenge_i\mathbb{A}.$$ 
Notice that, if $a\in A^{\lozenge_i}_\varphi$, then $\lnot_\mathbb{A} a\in A^{\lozenge_i}_\varphi$ as well (since $\lnot_\mathbb{B}a\in B_\varphi$ and $\lnot_\mathbb{B}a=\lnot_\mathbb{A}a$). 
Hence, for every agent $i\in\Ag_\varphi$, $\left( A^{\lozenge_i}_\varphi,\top,\bot,\land,\lor,\lnot_\mathbb{A}\right)$ is a Boolean subalgebra of $\mathbb{A}_\varphi$. We are now ready to endow $\mathbb{A}_\varphi$ with an epistemic Heyting algebra structure.

\begin{definition}
Let $$\mathbb{A}^\star_\varphi
:= \left( \mathbb{A}_\varphi,\to^\star,(\lozenge^\star_i)_{i\in\Ag},(\Box^\star_i)_{i\in\Ag}\right)$$ 
where, for all $a,b \in \mathbb{A}_\varphi$,
$$a\to^\star b := \bigvee\{c\in \mathbb{A}_\varphi\mid c\leq a\to_\mathbb{A} b\}=\bigvee\{c\in\mathbb{A}_\varphi\mid c\land a\leq b\},$$
for all  $i\in\Ag_\varphi$ and $a \in \mathbb{A}_\varphi$,
\begin{align*}
\lozenge^\star_i a  := \bigwedge\{ b\in A^{\lozenge_i}_\varphi \mid a\leq b\} && \text{and} &  & \Box^\star_i a := \bigvee\{b\in A^{\lozenge_i}_\varphi \mid b\leq a\} ,
\end{align*}
for all  $i\notin\Ag_\varphi$ and $a \in \mathbb{A}_\varphi$,
\begin{align*}
\lozenge^\star_i a :=
\begin{cases}
\top & \text{if } a\neq\bot,\\
\bot \quad\quad & \text{if } a=\bot,
\end{cases}
&& \text{and}
&&
\Box^\star_i a :=
\begin{cases}
\bot & \text{if } a\neq\top,\\
\top \quad\quad & \text{if } a=\top.
\end{cases}
\end{align*}

%\begin{align*}
%%\begin{array}{lllr}
%\lozenge^\star_i a  := \bigwedge\{ b\in A^{\lozenge_i}_\varphi \mid a\leq b\} & \qquad\phantom{} & \Box^\star_i a := \bigvee\{b\in A^{\lozenge_i}_\varphi \mid b\leq a\} & \qquad\text{for $i\in\Ag_\varphi$}\\
%\lozenge^\star_i a :=
%\left\{ \begin{array}{ll}
%\top & \text{if } a\neq\bot,\\
%\bot \quad\quad & \text{if } a=\bot
%\end{array} 
%\right.  &\qquad &\Box^\star_i a :=
%\left\{ \begin{array}{ll}
%\bot & \text{if } a\neq\top,\\
%\top \quad\quad & \text{if } a=\top
%\end{array} 
%\right.  & \qquad\text{for $i\notin\Ag_\varphi$}
%%\end{array}
%\end{align*}

The operations above are well-defined since $\mathbb{A}_\varphi$ is a finite distributive lattice and hence all the joins and meets exist. %Furthermore, $\lozenge_i^\star a$ is the infimum of the elements in $\mathbb{A}^{\lozenge_i}_\varphi$ above $a$ and $\Box^\star_i a$ is the supremum of the elements in  $\mathbb{A}^{\lozenge_i}_\varphi$ below $a$.
\end{definition}

\begin{lemma}\label{lemma1:completeness}
For every $i\in\Ag_\varphi$, the algebra $\mathbb{A}^\star_\varphi$	satisfies the following properties:
\begin{enumerate}
	\item $\lozenge^\star_i\mathbb{A}^\star_\varphi=\{\lozenge^\star_ia\mid a\in\mathbb{A}^\star_\varphi\}\subseteq A^{\lozenge_i}_\varphi$;
	\item $\lozenge^\star_i\mathbb{A}^\star_\varphi=\Box^\star_i\mathbb{A}^\star_\varphi$;
	\item for all $a\in A^{\lozenge_i}_\varphi$, it holds that $\lozenge^\star_ia=a$ and $\Box^\star_ia=a$;
	\item for all $a,b\in \mathbb{A}^\star_\varphi$, if $a\to_\mathbb{A}b\in\mathbb{A}^\star_\varphi$, then $a\to^\star b=a\to_\mathbb{A}b$;
	\item for all $a\in \mathbb{A}^\star_\varphi$, if $\lozenge_ia\in\mathbb{A}^\star_\varphi$ (resp.\ $\Box_ia\in\mathbb{A}^\star_\varphi$), then $\lozenge_i^\star a=\lozenge_i a$ (resp.\ $\Box_i^\star a=\Box_i a$);
	\item for all formulas $\psi, \varphi \in \mc{L}$,
	if  $(\lozenge_i\psi)^\mathbb{A} \in S^\lozenge_\varphi$ (resp.\ $(\Box_i\psi)^\mathbb{A} \in S^\lozenge_\varphi$ or $(\psi\to\chi)^\mathbb{A}\in S^\lozenge_\varphi$), then $\lozenge_i^\star\psi^\bbA=\lozenge_i\psi^\bbA$ (resp.\ $\Box_i^\star\psi^\bbA=\Box_i\psi^\bbA$ or $\psi^\bbA\to^\star\chi^\bbA=\psi^\bbA\to_\mathbb{A}\chi^\bbA$).
	\item $\lozenge^\star_i\mathbb{A}^\star_\varphi=\{\lozenge^\star_ia\mid a\in\mathbb{A}^\star_\varphi\} = A^{\lozenge_i}_\varphi$;
\end{enumerate}  
\end{lemma}
\begin{proof}
The first five items follow immediately from the definition of $\lozenge_i^\star$ and $\Box_i^\star$. Item 6 is an application of items 4 and 5. Item 7 follows from items 1 and 3.
\end{proof}

\begin{lemma}\label{lemma2:completeness}
The algebra $\mathbb{A}^\star_\varphi$ is an epistemic Heyting algebra.
\end{lemma}
\begin{proof}
	As mentioned early on, $\mathbb{A}_\varphi$ is a distributive lattice. Moreover, by definition, $\to^\star$ is the right residual of $\land$ in $\mathbb{A}_\varphi$. This shows that $\mathbb{A}^\star_\varphi$ is a Heyting algebra. 
To prove that $\mathbb{A}^\star_\varphi$ is an epistemic Heyting algebra, it remains to show that $\mathbb{A}^\star_\varphi$ satisfies the following axioms 
(c.f.\ \autoref{def: epist algebra} and \autoref{def:epist-Heyting-algebra}): 
%axioms \eqref{axiom:epist-alg:refl} to \eqref{axiom:epist-alg:distribb2} listed in \autoref{def: epist algebra} and 
%the axiom \eqref{axiom:epist-alg:boolean} mentioned in .
%	For every $i\notin\Ag_\varphi$ it is routine to check that the axioms \eqref{axiom:epist-alg:refl} to \eqref{axiom:epist-alg:trans} and \eqref{axiom:epist-alg:boolean} hold.
\begin{align}
& a\leq\lozenge_ia
\tag{M1}
\label{proof:axiom:epist-alg:refl}
\\
& \Box_ia\leq a
\tag{M2}
\label{proof:axiom:epist-alg:refl2}
\\
& \lozenge_i(a\lor b)\leq\lozenge_ia\lor\lozenge_ib
\tag{M3}
\label{proof:axiom:epist-alg:distribd1}
\\
& \Box_i(a\to b)\leq\Box_ia\to\Box_ib
\tag{M4}
\label{proof:axiom:epist-alg:distribb1}
\\
& \lozenge_i a \leq \Box_i\lozenge_i a
\tag{M5}
\label{proof:axiom:epist-alg:sym}
\\
& \lozenge_i\Box_ia\leq\Box_ia
\tag{M6}
\label{proof:axiom:epist-alg:trans}
\\
& \Box_i(a\to b)\leq\lozenge_ia\to\lozenge_ib
\tag{M7}
\label{proof:axiom:epist-alg:subalgebra}
\\
& \lozenge_i\bot\leq\bot
\tag{M8}
\label{proof:axiom:epist-alg:distribd2}
\\
& \top\leq\Box_i\top
\tag{M9}
\label{proof:axiom:epist-alg:distribb2}
\\
& \lozenge_ia\lor\lnot\lozenge_ia=\top.
\tag{E}
\label{proof:axiom:epist-alg:boolean}
\end{align}
	
Let $i\in\Ag_\varphi$. By definition, it immediately follows that $\lozenge_i^\star$ and $\Box_i^\star$ verify axioms \ref{proof:axiom:epist-alg:refl} and \ref{proof:axiom:epist-alg:refl2}. Axiom \ref{proof:axiom:epist-alg:distribd1} holds because $\lozenge^\star_ia\lor\lozenge^\star_ib\in\lozenge_i\mathbb{A}_\varphi$ and $a\lor b\leq\lozenge^\star_ia\lor\lozenge^\star_ib$ (and similarly for axiom \ref{proof:axiom:epist-alg:distribb1}). 
	
	\medskip
	
	As for axioms \ref{proof:axiom:epist-alg:sym} and \ref{proof:axiom:epist-alg:trans}, since $\lozenge^\star_ia,\Box^\star_ia\in\lozenge_i\mathbb{A}_\varphi$, by item 3 of Lemma \ref{lemma1:completeness}, we obtain that $\lozenge^\star_i\Box_i^\star a=\Box^\star_i a$ and $\lozenge^\star_ia=\Box_i^\star\lozenge_i^\star a$, which imply the axioms.
	
	\medskip
	
	In the context of axioms \ref{proof:axiom:epist-alg:refl} through \ref{proof:axiom:epist-alg:trans}, axiom \ref{proof:axiom:epist-alg:subalgebra} is equivalent to $\lozenge_i(\lozenge_i p\to\lozenge_i q)\rightarrow(\lozenge_i p\to\lozenge_i q)$ (see \cite[Lemma 2]{bezhanishvili1998varieties}), so let us show that $\mathbb{A}^\star_\varphi$ satisfies $\lozenge_i(\lozenge_i p\to\lozenge_i q)\rightarrow(\lozenge_i p\to\lozenge_i q)$. 
	Observe that for all $a,b\in\mathbb{A}^\star_\varphi$, since $\lozenge^\star_ia,\lozenge_i^\star b\in A^{\lozenge_i}_\varphi$ and $A^{\lozenge_i}_\varphi$ is a Boolean algebra (and hence contains $\lnot_\mathbb{A}\lozenge_i^\star a$), we have that 
	$$\lozenge^\star_ia\to_\mathbb{A} \lozenge_i^\star b=\lnot_\mathbb{A}\lozenge^\star_ia\lor\lozenge_i^\star b \;\ \in \; A^{\lozenge_i}_\varphi$$ 
	 which implies by item 4 of Lemma \ref{lemma1:completeness} that \begin{align}
	 \lozenge^\star_ia\to^\star\lozenge^\star_i b=\lozenge^\star_ia\to_\mathbb{A}\lozenge^\star_ib.\label{equations:completenessequation1}\end{align}
	 Now, by item 3 of Lemma \ref{lemma1:completeness}, we have that  
	 $$\lozenge^\star_i(\lozenge^\star_ia\to_\mathbb{A}\lozenge^\star_ib)=\lozenge^\star_ia\to_\mathbb{A}\lozenge^\star_ib$$ 
	 which by the equation \eqref{equations:completenessequation1} is equivalent to
	 $$\lozenge^\star_i(\lozenge^\star_ia\to^\star\lozenge^\star_ib)=\lozenge^\star_ia\to^\star\lozenge^\star_ib,$$ that is, $\mathbb{A}^\star_\varphi$ satisfies $\lozenge_i(\lozenge_i p\to\lozenge_i q)\rightarrow(\lozenge_i p\to\lozenge_i q)$.

\medskip

Axioms \ref{proof:axiom:epist-alg:distribd2} and \ref{proof:axiom:epist-alg:distribb2} follow from the fact that $\top,\bot\in A_\varphi\cap\lozenge_i\bbA$ and item 3 of Lemma \ref{lemma1:completeness}.

\medskip
	
	Finally, axiom \ref{proof:axiom:epist-alg:boolean} follows immediately from item 4 of Lemma \ref{lemma1:completeness} and from the fact that $A^{\lozenge_i}_\varphi$ is a Boolean algebra. Hence if $a\in A^{\lozenge_i}_\varphi$ then $(a\to_\mathbb{A}\bot_\mathbb{A})\in A^{\lozenge_i}_\varphi$.
\end{proof}

\subsection*{Measures on $\mathbb{A}^\star_\varphi$}
\label{ssec:cpt:measure:Astar}
In this section, for each agent $i\in\Ag_\varphi$, we will define an $i$-measure on the algebra $\mathbb{A}^\star_\varphi$ and a valuation on $\mathbb{A}^\star_\varphi$, so as to define an APE-model $\mathcal{M}_\varphi$ 
such that $\val{\sigma}_{\mathcal{M}_\varphi}=\sigma^\mathbb{A}$ for every subformula $\sigma$ of $\varphi$.
%\redfootnote{@Apostolos: I am not sure I really understand this sentence. 
%Can we take the time to discuss it, please?\\
%How do we get $\val{\sigma}_{\mathbb{A}^\star_\varphi}=\sigma^\mathbb{A}$
%when $\mb{A}$ is not equiped with probability measures?\\
%In $\mb{A}$, $\sigma$ is just some equivalence class, in a model (with a fixed probability measure such as $\mathcal{M}_\varphi$) inequalities behave classically and are either true or false.
%\\
%Moreover, shouldn't it be $\val{\sigma}_{\mathcal{M}_\varphi}$ instead of $\val{\sigma}_{\mathbb{A}^\star_\varphi}$.} 
Before defining the measures, we will state some auxiliary results.

\begin{lemma}\label{lemma:measures:basic}
	The system IPEL proves all classical truths about linear inequalities.
\end{lemma}
\begin{proof}
	See \cite{fagin1990logic} for an explanation of why axioms N0 to N6 are enough. 
	Notice that, even though the result is proven for classical logic, it still holds for IPEL. Indeed, the fragment of the logic involving inequalities is classical because of the axiom N5: $(\tau\geq\beta)\lor(\lnot\tau\geq\beta)$.
\end{proof}

\begin{lemma}\label{lemma:measuers:newaxiom}
The formulas 
\begin{equation}
\label{eq:lemma:measuers:newaxiom:1}
\left( 
\lozenge_i\psi \land \left( \sum_m\alpha_m\mu_i(\phi_m)\geq\beta\right)\right)
\to \left( \sum_m\alpha_m\mu_i(\phi_m\land\lozenge_i\psi)\geq\beta\right)
\end{equation}
and 

\begin{equation}
\label{eq:lemma:measuers:newaxiom:2}
\left( \lozenge_i\psi
\land \left( \sum_m\alpha_m\mu_i(\phi_m)<\beta \right)\right) 
\to \left( \sum_m\alpha_m\mu_i(\phi_m\land\lozenge_i\psi)<\beta \right)
\end{equation}
are provable in IPEL.
\end{lemma}
\begin{proof}
We only prove \eqref{eq:lemma:measuers:newaxiom:1}, the proof of   \eqref{eq:lemma:measuers:newaxiom:2} being almost verbatim. Early on we observed  (see  Lemma \ref{lemma:axiomreplacement}) that axiom P4 implies the validity of $\Box_i \varphi \leftrightarrow (\mu_i(\varphi)=1)$. This and axiom M5 
(i.e.\ $\lozenge_i\psi\leftrightarrow\Box_i\lozenge_i\psi$) imply 
\begin{equation}
\label{eq:completeness:1}
\vdash_{\text{IPEL}}\lozenge_i\psi\leftrightarrow(\mu_i(\lozenge_i\psi)=1).
\end{equation}
Since $\vdash_{\text{IPEL}}\lozenge_i\psi\to (\lozenge_i\psi\lor\phi_m)$ for every $\phi_m \in \mc{L}$, by rule Sub$_\mu$ we obtain
\begin{equation}
\label{eq:completeness:2}
\vdash_{\text{IPEL}}\mu_i(\lozenge_i\psi)\leq \mu_i(\lozenge_i\psi \lor \phi_m).
\end{equation}
From \eqref{eq:completeness:2} and Lemma \ref{lemma:measures:basic},
we deduce that 
\begin{equation}
\label{eq:completeness:3}
\vdash_{\text{IPEL}}\mu_i(\lozenge_i\psi)=1\to\mu_i(\phi_m\lor\lozenge_i\psi)=1.
\end{equation}
Lemma \ref{lemma:measures:basic} and axiom P3 
(i.e.\ $\mu_i(\phi_m) = \mu_i(\phi_m \lor \lozenge_i\psi) + \mu_i(\phi_m \land \lozenge_i\psi )- \mu_i(\lozenge_i\psi) $) entail
\begin{equation}
\label{eq:completeness:4}
\vdash_{\text{IPEL}}
\left( \sum_m
\alpha_m\mu_i(\phi_m)\geq\beta
\right)
\leftrightarrow
\left( \sum_m \alpha_m\Big(\mu_i(\phi_m\land\lozenge_i\psi)+\mu_i(\phi_m\lor\lozenge_i\psi)-\mu_i(\lozenge_i\psi)\Big)\geq\beta \right).
\end{equation}
Combining \eqref{eq:completeness:1}, \eqref{eq:completeness:3} and \eqref{eq:completeness:4}, we obtain 
\begin{equation}
\label{eq:completeness:5}
\vdash_{\text{IPEL}}
\Big( \lozenge_i\psi\land A \Big)
\to  \Big( 
\big( \mu_i(\lozenge_i\psi)=1 \big) \land \bigwedge_{m}\big( \mu_i(\phi_m\lor\lozenge_i\psi)=1 \big) 
  \land B  \Big)
\end{equation}
with 
$$ A := \sum_m \alpha_m\mu_i(\phi_m)\geq\beta,$$
and 
$$ B := \sum_m\alpha_m(\mu_i(\phi_m\land\lozenge_i\psi)+\mu_i(\phi_m\lor\lozenge_i\psi)-\mu_i(\lozenge_i\psi))\geq\beta.$$
Again, by using Lemma \ref{lemma:measures:basic}, we obtain that 
\begin{equation}
\label{eq:completeness:6}
\vdash_{\text{IPEL}}
\left((\mu_i(\lozenge_i\psi)=1)
\land\bigwedge_m(\mu_i(\phi_m\lor\lozenge_i\psi)=1)
\land
B\right)
\to D
\end{equation}
with
$$ D:=\sum_m\alpha_m\mu_i(\phi_m\land\lozenge_i\psi)\geq\beta .$$
%{\small $$\vdash_{\text{IPDEL}}(\mu_i(\lozenge_i\psi)=1)\land(\mu_i(\phi_m\lor\lozenge_i\psi)=1)\land(\sum\alpha_m(\mu_i(\phi_m\land\lozenge_i\psi)+\mu_i(\psi_m\lor\lozenge_i\psi)-\mu_i(\lozenge_i\psi))\geq\beta)\to(\sum_m\alpha_m\mu_i(\phi_m\land\lozenge_i\psi)\geq\beta).$$} 

Putting \eqref{eq:completeness:5} and \eqref{eq:completeness:6}  together, we finally get: 
$$\vdash_{\text{IPEL}}
\left( \lozenge_i\psi\land \sum_m\alpha_m\mu_i(\phi_m)\geq\beta
\right)
\to \left( \sum_m\alpha_m\mu_i(\phi_m\land\lozenge_i\psi)\geq\beta \right)$$ 
as desired.\end{proof}

Observe that for any agent $i\in\Ag_\varphi$, since $\mathbb{A}^\star_\varphi$ is finite and $\lozenge^\star_i\mathbb{A}^\star_\varphi=A^{\lozenge_i}_\varphi$ is a Boolean algebra, it is the case that the $i$-minimal elements are the atoms of this Boolean algebra and every element of $\mathbb{A}^{\lozenge_i}_\varphi$ can be written as the union of some of these $i$-minimal elements. 
Let $n_i$ be the number of $i$-minimal elements of $\mathbb{A}^\star_\varphi$. 
Let us call $a_k^i$, for $1 \leq k\leq n_i$, the $i$-minimal elements of $\mathbb{A}^\star_\varphi$. 
Now, for each $i$-probability formula $\sigma$ with $\sigma^\mathbb{A}\in S^\lozenge_\varphi$, we have that  $\sigma^\mathbb{A}\in A^{\lozenge_i}_\varphi$. Hence, we have that $(\lnot\sigma)^\mathbb{A}\in A^{\lozenge_i}_\varphi$.  
This implies  that there exists a function $f_\sigma: \{1, 2, \dots, n_i\}\to \{0,1\}$ 
such that 
$$\sigma^\mathbb{A}=\bigvee_{f_\sigma(k)=1}a_k^i \quad \text{and} \quad (\lnot\sigma)^\mathbb{A}=\bigvee_{f_\sigma(k)=0}a_k^i.$$ 
It should be stressed that since $\lor$ and $\land$ in $\mathbb{A}^\star_\varphi$ are inherited by $\mathbb{A}$,  these equalities hold in $\mathbb{A}$ as well.

Now, let us fix $i\in\Ag_\varphi$. For every $k\in n_i$, we  define a system of equations $E_{a^i_k}$, with variables $x_b$ for every $b\leq a^i_k$ as follows\footnote{The sums in system of equations $E_{a^i_k}$ range over $m$.}:

$$E_{a^i_k} := \left( \begin{array}{l l}
\sum\alpha_m \cdot x^i_{\psi_m^\mathbb{A}\land a^i_k}\geq \beta, & \text{for all $\sigma:=(\sum\alpha_m \cdot \mu_i(\psi_m)\geq \beta)$ with $\sigma^\mathbb{A}\in S^\lozenge_\varphi$ and $f_\sigma(k)=1$ }\\
\sum\alpha_m \cdot x^i_{\psi_m^\mathbb{A}\land a^i_k}< \beta, & \text{for all $\sigma:=(\sum\alpha_m \cdot \mu_i(\psi_m)\geq \beta)$ with $\sigma^\mathbb{A}\in S^\lozenge_\varphi$ and $f_\sigma(k)=0$ }\\
x^i_{b} \geq 0 \text{ and } x_{b}\leq 1, & \text{for all $b\in\mathbb{A}^\star_\varphi$  with $b\leq a^i_k$}\\
x^i_{b} + x^i_{c} = x^i_{b\land c}+x^i_{b \lor c},\ & \text{for all $b,c\in\mathbb{A}^\star_\varphi$ with $b,c\leq a^i_k$}\\
x^i_{b}\leq x^i_{c}, & \text{for all $b,c\in\mathbb{A}^\star_\varphi$  with $b\leq c\leq a^i_k$}\\
x^i_\bot=0\\
x^i_{a^i_k} = 1\\
\end{array} \right). $$

For a solution $s$ of the above system, we denote with $(x^i_b)^s$ the solution according to $s$ of $x^i_b$.

Notice that the system is designed in such a way that any particular solution (cf.\ Lemma \ref{lemma:measures:key3}) provides an $i$-measure on $\mathbb{A}^\star_\varphi$ that  guarantees that the valuation of an $i$-probability formula $\sigma$ is $\sigma^\mathbb{A}$. 
Indeed, the first two types of inequalities in the system will guarantee that exactly the $i$-minimal elements of $\mathbb{A}^\star_\varphi$ below $\sigma^\mathbb{A}$ will constitute $\val{\sigma}$ (see Definition \ref{def:semantics:IPDEL}). The rest of the inequalities will guarantee that the solution satisfies the basic properties of $i$-measures. 

Observe that, for every $b\leq a^i_k$, there exists a formula $\tau_b$ such that $b=\tau_b^\mathbb{A}$ and  if $b\leq c$ then $\vdash_{\text{IPEL}}\tau_b\to\tau_c$. 
Let $E_{a^i_k}^\tau$ be the system of equations where each $x^i_b$ is replaced by $\mu_i(\tau_b)$. 
Since $a^i_k$ is $i$-minimal, we can assume without loss of generality that $\tau_{a^i_k}$ is of the form $\lozenge_i\tau'$. 
Furthermore, let $PS_i\subseteq S^\lozenge_\varphi$ be the set of $i$-probability formulas that are subformulas of $\varphi$.  
For every $\sigma^\mathbb{A}\in PS_i$ such that $\sigma:=(\sum\alpha_m \cdot \mu_i(\psi_m)\geq \beta)$, let $\sigma[a_k^i]$ be the formula $\sum\alpha_m \cdot \mu_i(\psi_m\land\tau_{a^i_k})\geq \beta$.

\begin{lemma}\label{lemma:measures:key1}
	For every $k\in n_i$, the system $E_{a^i_k}$ has a solution.
\end{lemma}
\begin{proof}
	Notice that all but the first two types of inequalities in  $E_{a^i_k}^\tau$ are provable in IPEL as they are immediate consequences of axioms P1, P2, P3 and the rule Sub$_\mu$. 
	Heading towards a contradiction, let us first assume that $E_{a^i_k}$ does not have a solution at all. 
	This is a truth about linear inequalities of rational numbers, hence, by Lemma \ref{lemma:measures:basic}, it is provable in IPDEL. As mentioned above, since some inequalities are provable this is tantamount to saying that 
\begin{align}
\vdash_{\text{IPEL}}\lnot  
\left( \left( \bigwedge_{\begin{smallmatrix}
		\sigma^\mathbb{A} \in PS_i\\
		f_\sigma(k)=1
		\end{smallmatrix}}\sigma[a_k^i] \right)\land \left( \bigwedge_{\begin{smallmatrix}
		\sigma^\mathbb{A} \in PS_i\\
		f_\sigma(k)=0
		\end{smallmatrix}}\lnot\sigma[a_k^i] \right) \right).
		\label{proof:cpt:align:Eak:1}
\end{align}	 
		Notice that, by Lemma \ref{lemma:measuers:newaxiom},  we have: for every $\sigma^\mathbb{A}\in PS_i$,  
		$$\vdash_{\text{IPEL}} \left(\sigma\land\tau_{a^i_k} \right) \to\sigma[a^i_k]$$ 
		and 
		$$\vdash_{\text{IPEL}} \left( \lnot \sigma\land\tau_{a^i_k} \right)\to\lnot\sigma[a^i_k].$$ 
		%%%%%
		
		Therefore,
			\begin{align}\vdash_{\text{IPEL}}
	\left( \left(  \left(   
	\bigwedge_{\begin{smallmatrix}
		\sigma^\mathbb{A} \in PS_i\\
		f_\sigma(k)=1
		\end{smallmatrix}}\sigma \right)
		\land \left(\bigwedge_{\begin{smallmatrix}
		\sigma^\mathbb{A} \in PS_i\\
		f_\sigma(k)=0
		\end{smallmatrix}}\lnot\sigma
		\right) \right) \land\tau_{a^i_k} \right) \rightarrow \left( \left( \bigwedge_{\begin{smallmatrix}
		\sigma^\mathbb{A} \in PS_i\\
		f_\sigma(k)=1
		\end{smallmatrix}}\sigma[a_k^i] \right)\land \left( \bigwedge_{\begin{smallmatrix}
		\sigma^\mathbb{A} \in PS_i\\
		f_\sigma(k)=0
		\end{smallmatrix}}\lnot\sigma[a_k^i] \right) \right).
		\label{proof:firsttouselater}
	\end{align} 
	%\redbf{/!$\setminus$ Apostolos, use left( and right) instead of big( and Big(.}
	
	Since one direction of contraposition is provable in intiontionistic logic we obtain that: 
\begin{align}
\vdash_{\text{IPEL}}
\left( \lnot \left( \left( 
\bigwedge_{\begin{smallmatrix}
		\sigma^\mathbb{A} \in PS_i\\
		f_\sigma(k)=1
		\end{smallmatrix}}
		\sigma[a_k^i] \right) 
		\land \left( \bigwedge_{\begin{smallmatrix}
		\sigma^\mathbb{A} \in PS_i\\
		f_\sigma(k)=0
		\end{smallmatrix}}
		\lnot\sigma[a_k^i] \right) \right) \right) \to\lnot
		\left( \left( \left( \bigwedge_{\begin{smallmatrix}
		\sigma^\mathbb{A} \in PS_i\\
		f_\sigma(k)=1
		\end{smallmatrix}}\sigma \right) \land \left( \bigwedge_{\begin{smallmatrix}
		\sigma^\mathbb{A} \in PS_i\\
		f_\sigma(k)=0
		\end{smallmatrix}}\lnot\sigma
		\right)   \right) \land\tau_{a^i_k} \right).
		\label{proof:cpt:align:Eak:2}
\end{align}	

		%%%%%
	
\eqref{proof:cpt:align:Eak:1} and \eqref{proof:cpt:align:Eak:2} imply that
	\begin{align}
	\vdash_{\text{IPEL}} 
	\lnot\left( \left( \left( \bigwedge_{\begin{smallmatrix}
		\sigma^\mathbb{A} \in PS_i\\
		f_\sigma(k)=1
		\end{smallmatrix}}\sigma \right)\land \left( \bigwedge_{\begin{smallmatrix}
		\sigma^\mathbb{A} \in PS_i\\
		f_\sigma(k)=0
		\end{smallmatrix}}\lnot\sigma
		\right) \right) \land\tau_{a^i_k} \right).
		\label{proof:cpt:align:Eak:2,5}
	\end{align}
In addition, $\mathbb{A}^\star_\varphi$ inherits the order from $\mathbb{A}$ and by construction $a_k^i\leq\sigma^\mathbb{A}$ when  $f_\sigma(k)=1$ and $a_k^i\leq(\lnot\sigma)^\mathbb{A}$ when $f_\sigma(k)=0$. Hence,
we have that, for all $\sigma\in PS_i$, if $f_\sigma(k)=1$ then $\vdash_{\text{IPEL}}\tau_{a^i_k}\to\sigma$ and if $f_\sigma(k)=0$ then $\vdash_{\text{IPEL}}\tau_{a^i_k}\to\lnot\sigma$. 
Therefore, we have 
$$\vdash_{\text{IPEL}} \tau_{a^i_k} \rightarrow 
\left( \left( \bigwedge_{\begin{smallmatrix}
		\sigma^\mathbb{A} \in PS_i\\
		f_\sigma(k)=1
		\end{smallmatrix}}\sigma \right)\land \left( \bigwedge_{\begin{smallmatrix}
		\sigma^\mathbb{A} \in PS_i\\
		f_\sigma(k)=0
		\end{smallmatrix}}\lnot\sigma \right)  \right).$$
Hence,
$$\vdash_{\text{IPEL}}\lnot\tau_{a^i_k} \leftrightarrow 
\lnot
\left( \left( \left( \bigwedge_{\begin{smallmatrix}
		\sigma^\mathbb{A} \in PS_i\\
		f_\sigma(k)=1
		\end{smallmatrix}}\sigma \right) \land \left( \bigwedge_{\begin{smallmatrix}
		\sigma^\mathbb{A} \in PS_i\\
		f_\sigma(k)=0
		\end{smallmatrix}}\lnot\sigma \right) \right) \land\tau_{a^i_k}\right)$$ 
and by \eqref{proof:cpt:align:Eak:2,5}
$$\vdash_{\text{IPEL}}\lnot\tau_{a^i_k}.$$ 
We have reached a contradiction because $a_k^i$ is an element of $\mathbb{A}$ different from $\bot$ and hence each formula corresponding to it is consistent. Therefore $E_{a^i_k}$ has a solution.
\end{proof}

\begin{lemma}\label{lemma:measures:key2}
	 For every $k\in n_i$ and every $b<c\leq a_k^i$, the system $E_{a^i_k}$ has a solution $s_{b,c}$ such that $(x^i_b)^{s_{b,c}}<(x^i_c)^{s_{b,c}}$.
\end{lemma}
\begin{proof}
	Heading towards a contradiction, let $b < c\leq a_k^i$ such that,  for every solution $s$ of  $E_{a^i_k}$, we have $(x^i_b)^s=(x^i_c)^s$. This is a fact of inequalities of real numbers and therefore, by Lemma \ref{lemma:measures:basic}, it is provable in IPEL. 
Since all but the first two types of inequalities in $E_{a^i_k}$ are provable in IPEL, we have that  	
	$$\vdash_{\text{IPEL}}
	\left( \left( \bigwedge_{\begin{smallmatrix}
		\sigma^\mathbb{A} \in PS_i\\
		f_\sigma(k)=1
		\end{smallmatrix}}\sigma[a_k^i] \right) \land \left( \bigwedge_{\begin{smallmatrix}
		\sigma^\mathbb{A} \in PS_i\\
		f_\sigma(k)=0
		\end{smallmatrix}}\lnot\sigma[a_k^i] \right) \right) \to\mu_i(\tau_b)=\mu_i(\tau_c).$$ 
	
	Since $\vdash_{\text{IPEL}}\tau_b\to\tau_c$, necessitation implies $\vdash_{\text{IPEL}}\Box_i(\tau_b\to\tau_c)$. Using axiom P4
	$$\Big( \big(\Box_i(\phi\to\psi) \big) \land \big( \mu_i(\phi)=\mu_i(\psi) \big) \Big) \leftrightarrow\Box_i(\psi\leftrightarrow\phi),$$ we obtain that 
	\begin{align}
	\vdash_{\text{IPEL}}
	\left( \left( \bigwedge_{\begin{smallmatrix}
		\sigma^\mathbb{A} \in PS_i\\
		f_\sigma(k)=1
		\end{smallmatrix}}\sigma[a_k^i] \right) \land \left( \bigwedge_{\begin{smallmatrix}
		\sigma^\mathbb{A} \in PS_i\\
		f_\sigma(k)=0
		\end{smallmatrix}}\lnot\sigma[a_k^i]
		\right) \right)\to\Box_i(\tau_c\to\tau_b).
		\label{proof:cpl:align:tau:1}
	\end{align}
Recall that\footnote{see proof of \autoref{lemma:measures:key1}.}
\begin{align}
\vdash_{\text{IPEL}} \tau_{a^i_k} \rightarrow 
\left(\left( \bigwedge_{\begin{smallmatrix}
		\sigma^\mathbb{A} \in PS_i\\
		f_\sigma(k)=1
		\end{smallmatrix}}\sigma \right) \land \left( \bigwedge_{\begin{smallmatrix}
		\sigma^\mathbb{A} \in PS_i\\
		f_\sigma(k)=0
		\end{smallmatrix}}\lnot\sigma \right) \right).
\label{proof:cpl:align:tau:2}
\end{align}
		 Using Lemma \ref{lemma:measuers:newaxiom} and \eqref{proof:cpl:align:tau:2} (cf.\ \eqref{proof:firsttouselater}), we get that 
		 \begin{align}
		 		 \vdash_{\text{IPEL}}\tau_{a^i_k}\to\left( \left( \bigwedge_{\begin{smallmatrix}
		\sigma^\mathbb{A} \in PS_i\\
		f_\sigma(k)=1
		\end{smallmatrix}}\sigma[a_k^i] \right) \land \left( \bigwedge_{\begin{smallmatrix}
		\sigma^\mathbb{A} \in PS_i\\
		f_\sigma(k)=0
		\end{smallmatrix}}\lnot\sigma[a_k^i] \right) \right).
	\label{eq:proof:oneforme}
	\end{align}
 
From \eqref{proof:cpl:align:tau:1} and \eqref{eq:proof:oneforme}, we deduce that 
	$$\vdash_{\text{IPEL}}\tau_{a^i_k}\to\Box_i(\tau_c\to\tau_b).$$ 
	
	By axiom M2 ($\Box_i p\to p$), we have 
	
	$$\vdash_{\text{IPEL}}\tau_{a^i_k}\to(\tau_c\to\tau_b),$$ 
	 which is equivalent to 
	 $$\vdash_{\text{IPEL}}(\tau_{a^i_k}\land\tau_c)\to\tau_b.$$ 
	 Since $\vdash_{\text{IPEL}}\tau_c\to\tau_{a^i_k}$, the equation above implies that $$\vdash_{\text{IPEL}}\tau_c\to\tau_b.$$
This last equation is	 
	   a contradiction since in $\mathbb{A}$, the Lindenbaum-Tarski algebra of IPEL, we have that $c\nleq b$. Therefore, for every such pair $b<c\leq a^i_k$, there exists a solution $s_{b,c}$ of  $E_{a^i_k}$ such that $(x^i_b)^{s_{b,c}}<(x^i_c)^{s_{b,c}}$.
\end{proof}

\begin{lemma}\label{lemma:measures:key3}
 For every $k\in n_i$, the system $E_{a^i_k}$ has a solution $s$, such that $(x^i_b)^s<(x^i_c)^s$ for all $b,c\leq a_k^i$ with $b<c$.
\end{lemma}
\begin{proof}
By Lemma \ref{lemma:measures:key2}, for every pair $b<c\leq a^i_k$ there exists a solution $s_{b,c}$ of  $E_{a^i_k}$ such that $(x^i_b)^{s_{b,c}}<(x^i_c)^{s_{b,c}}$. 
Notice that the solution space of $E_{a^i_k}$ is a convex subspace of $\mathbb{R}^l$, for some natural number $l$. 
Indeed, it is immediate that the solutions of each linear inequality define a convex space and the intersection of convex spaces is a convex space (cf.\ \cite[Chapter 12]{lang2013linear}). Let $n$ be the number of aforementioned solutions. Then it is the case that $\sum_{b<c\leq a^i_k}\frac{1}{n}s_{b,c}$ is also a solution of $E_{a^i_k}$ (see e.g.\ \cite[Chapter 12, Theorem 1.2]{lang2013linear}). Let us call this solution $s$
and show that if $d<e$ then $(x^i_{d})^s<(x^i_{e})^s$. 

Let $d<e$.
Notice that, for every $s_{b,c}$, it is the case that $(x^i_{d})^{s_{b,c}}\leq (x^i_{e})^{s_{b,c}}$ by the restraints of the system $E_{a^i_k}$. Moreover, we have $(x^i_{d})^{s_{d,e}}< (x^i_{e})^{s_{d,e}}$. Hence, $$(x^i_d)^s=\sum_{b<c}\frac{1}{n}(x^i_d)^{s_{b,c}}<\sum_{b<c}\frac{1}{n}(x^i_e)^{s_{b,c}}=(x^i_e)^s.$$ 
Therefore, we have that, for every pair $d<e \leq a^i_k$, we have $(x^i_{d})^s<(x^i_{e})^s$  as required.
\end{proof}

For every agent $i\in\Ag_\varphi$ and every system $E_{a^i_k}$, pick a solution $s$ satisfying the conditions of Lemma \ref{lemma:measures:key3} and define $\mu_i(b)=(x^i_b)^s$, for every $b\in {\mathsf{Min}(\mathbb{A}^\star_\varphi)}{\downarrow}$. For agents $j\notin\Ag_\varphi$, let $\mu_j(b)=\mu_\mathbb{B}(b)$ (see \eqref{align:muB}).  Now, we define an APE-model 
\begin{equation}\label{eq:modelintheend}
\mathcal{M}_\varphi=\langle\mathbb{A}^\star_\varphi,(\mu_i)_{i\in\Ag}, v\rangle
\end{equation}
such that, for every $p\in AtProp\cap S^\lozenge_\varphi$, it holds that $v(p)=p^\mathbb{A}$. 

\begin{lemma}\label{lemma:completenes:model}
The model $\mathcal{M}_\varphi$ is an APE-model.
\end{lemma}
\begin{proof} 
For any $i\in\Ag_\varphi$,
 the restrictions imposed by the systems of inequalities and the conditions of Lemma \ref{lemma:measures:key3} immediately yield that $\mu_i$ is an $i$-measure. 
 For $j\notin\Ag_\varphi$, the only $j$-minimal element is $\top$. Furthermore, $\mu_\mathbb{B}$ is satisfies the restrictions of $j$-measures by definition. Hence, each $\mu_i$ is an $i$-measure, and by Lemma \ref{lemma2:completeness} and Definition \ref{def: alg probab epist structure} we have that $\mathcal{M}_\varphi$ is an APE-model.
\end{proof}
\begin{lemma}[Truth Lemma]\label{lemma:measures:truth}
	For every $\psi\in\mathcal{L}$ such that $\psi^\mathbb{A}\in S^\lozenge_\varphi$, it is the case that $$\val{\psi}_{\mathcal{M}_\varphi}=\psi^\mathbb{A}.$$

\end{lemma}
\begin{proof}
	By definition, $ S^\lozenge_\varphi$ is closed under subformulas. The proof proceeds by induction on the complexity of $\psi$. For the atomic variables, this follows immediately from the definition of $v$. For formulas of the form $\psi\land\tau$ and $\psi\lor\tau$ this follows from the fact that $\mathbb{A}^\star_\varphi$ inherits $\lor$ and $\land$ from $\mathbb{A}$. For formulas of the form $\psi\to\tau$, $\lozenge_i\psi$ and $\Box_i\psi$ it follows from item 6 of Lemma \ref{lemma1:completeness}. 
	Finally, for probability formulas of the form $\sigma:=\sum\alpha_m\mu_i(\psi_m)\geq\beta$, notice that, by the choice of $\mu_i$ as particular solutions of the systems $E_{a^i_k}$, exactly the $i$-minimal elements $a^i_k\leq\sigma^\mathbb{A}$ are such that $\sum\alpha_m\mu_i(\val{\psi_m}_{\mathcal{M}_\varphi}\land a^i_k)\leq\beta$. Hence, $\val{\sigma}_{\mathcal{M}_\varphi}=\sigma^\mathbb{A}$ by definition (cf.\ Definition \ref{def:semantics:IPDEL}). This concludes the proof.
\end{proof}

\begin{proposition}[Completeness]
	The axiomatisation for IPDEL given in Table \ref{table:IPDEL} is weakly complete w.r.t.\ APE-models.
\end{proposition}
\begin{proof}
	As discussed in the beginning of this section, the problem is reduced to proving the weak completeness of IPEL. Let $\varphi$ be an IPEL formula that is not a theorem. This means that $\varphi^\mathbb{A}\neq\top^\mathbb{A}$, where $\mathbb{A}$ is the Lindembaum-Tarski algebra of IPEL (see \eqref{eq:linden}). By Lemma \ref{lemma:completenes:model}, the model $\mathcal{M}_\varphi$ based on the algebra $\mathbb{A}^\star_\varphi$ defined in \eqref{eq:modelintheend} is an APE-model. By Lemma \ref{lemma:measures:truth},  $\val{\varphi}_{\mathcal{M}_\varphi}=\varphi^\mathbb{A}$. Since $\top^{\mathbb{A}^\star_\varphi}=\top^\mathbb{A}$, this shows that $\val{\varphi}_{\mathcal{M}_\varphi}\neq\top^{\mathbb{A}^\star_\varphi}$, which means that $\mathcal{M}_\varphi$ does not satisfies $\varphi$   as required.
\end{proof}

%\newpage
%
%
%\setcounter{tocdepth}{3}
%\tableofcontents

\end{document}